\documentclass{article}

\usepackage[export]{adjustbox}
\usepackage{amssymb, amsmath, amsthm}
\usepackage{microtype}
\usepackage{graphicx}
\usepackage{subfigure}
\usepackage{booktabs} % for professional tables
\usepackage{mdframed}

\usepackage{hyperref}
\usepackage{tcolorbox}
\tcbset{boxsep=0mm,boxrule=0pt,colframe=white,arc=0mm,left=0.5mm,right=0.5mm}
\usepackage{tikz}
\newcommand*\circled[1]{\tikz[baseline=(char.base)]{\node[shape=circle,draw,inner sep=1.0pt, minimum size=0.1cm] (char) {#1};}}

\allowdisplaybreaks[4]

\usepackage[accepted]{icml2020}
\usepackage{enumitem}
\icmltitlerunning{An Accelerated DFO Algorithm for Finite-sum Convex Functions}

\newcommand{\E}{{\mathbb E}}

\newcommand{\R}{{\mathbb{R}}}

\newtheorem{theorem}{Theorem}
\newtheorem{lemma}[theorem]{Lemma}

\newtheorem{corollary}[theorem]{Corollary}

\newtheorem*{remark}{Remark}

\begin{document}

\twocolumn[
\icmltitle{An Accelerated DFO Algorithm for Finite-sum Convex Functions}

\begin{icmlauthorlist}
\icmlauthor{Yuwen Chen}{eth}
\icmlauthor{Antonio Orvieto}{eth}
\icmlauthor{Aurelien Lucchi}{eth}
\end{icmlauthorlist}

\icmlaffiliation{eth}{Computer Science Department, ETH Z{\"u}rich, Switzerland}

\icmlcorrespondingauthor{Chen, Yuwen}{aaronchenyuwen@gmail.com}
\icmlcorrespondingauthor{Orvieto, Antonio}{antonio.orvieto@inf.ethz.ch}
\icmlcorrespondingauthor{Lucchi, Aurelien}{aurelien.lucchi@inf.ethz.ch}

% You may provide any keywords that you
% find helpful for describing your paper; these are used to populate
% the "keywords" metadata in the PDF but will not be shown in the document
\icmlkeywords{Machine Learning, ICML}

\vskip 0.3in
]

% this must go after the closing bracket ] following \twocolumn[ ...

% This command actually creates the footnote in the first column
% listing the affiliations and the copyright notice.
% The command takes one argument, which is text to display at the start of the footnote.
% The \icmlEqualContribution command is standard text for equal contribution.
% Remove it (just {}) if you do not need this facility.

%\printAffiliationsAndNotice{}  % leave blank if no need to mention equal contribution
\printAffiliationsAndNotice{} % otherwise use the standard text.

\begin{abstract}
Derivative-free optimization (DFO) has recently gained a lot of momentum in machine learning, spawning interest in the community to design faster methods for problems where gradients are not accessible. While some attention has been given to the concept of acceleration in the DFO literature, existing stochastic algorithms for objective functions with a finite-sum structure have not been shown theoretically to achieve an accelerated rate of convergence. Algorithms that use acceleration in such a setting are prone to instabilities, making it difficult to reach convergence. In this work, we exploit the finite-sum structure of the objective in order to design a variance-reduced DFO algorithm that provably yields acceleration. We prove rates of convergence for both smooth convex and strongly-convex finite-sum objective functions. Finally, we validate our theoretical results empirically on several tasks and datasets.
\end{abstract}

% !TEX root = main.tex

\section{Introduction}

While gradient-based techniques are extremely popular in machine learning, there are applications where derivatives are too expensive to compute or might not even be accessible (black-box optimization). In such cases, an alternative is to use derivative-free methods which rely on function values instead of explicitly computing gradients. These methods date to the 1960's, including e.g.~\cite{matyas1965random,nelder1965simplex} and have recently gained more attention in machine learning in areas such as black-box adversarial attacks~\cite{chen2017zoo}, reinforcement learning~\cite{salimans2017evolution}, online learning~\cite{bubeck2012regret}, etc.

We focus our attention on optimizing finite-sum objective functions which are commonly encountered in machine learning and which can be formulated as:
\begin{equation}
\min_{x \in \R^d} \;\; \left[ f(x) := \frac{1}{n}\sum_{i=1}^n f_i(x) \right],
\label{eq:min_f}
\end{equation}
where each function $f_i: \R^d \to \R$ is convex and differentiable, but its derivatives are not directly accessible.

The problem of optimizing Eq.~\eqref{eq:min_f} has been addressed in a seminal work by~\cite{nesterov2017random} who introduced a deterministic random\footnote{The term random refers to the use of randomly sampled directions to estimate derivatives.} gradient-free method (RGF) using a two-point Gaussian random gradient estimator. The authors derived a rate of convergence for RGF for both convex and strongly-convex functions and they also introduced a variant with a provably accelerated rate of convergence.
Subsequently, ~\cite{ghadimi2013stochastic} developed a stochastic variant of RGF, proving a nearly\footnote{For a precise definition, see~\cite{ghadimi2013stochastic}.} optimal rate of convergence for convex functions.

In the field of first-order gradient-based methods, gradient descent has long been known to achieve a suboptimal convergence rate. In a seminal paper,~\citet{nesterov1983method} showed that one can construct an optimal -- i.e. accelerated -- algorithm that achieves faster rates of convergence for both convex and strongly-convex functions. Accelerated methods have attracted a lot of attention in machine learning, pioneering some popular momentum-based methods such as Adam~\cite{kingma2014adam} which is commonly used to train deep neural networks. It therefore seems natural to ask whether \textit{provably accelerated} methods can be designed in a derivative-free setting. While this question has been considered in a deterministic setting in~\cite{nesterov2017random} as well as in a stochastic setting~\cite{gorbunov2018accelerated, gorbunov2019stochastic}, none of these works provably derived an accelerated rate of convergence for the finite-sum setting presented in Eq.~\eqref{eq:min_f}.

The inherent difficulty of designing a stochastic algorithm with an accelerated rate of convergence is due to the instability of the momentum term~\cite{allen2017katyusha,orvieto2019role}. One way to reduce instabilities is to rely on stochastic variance reduction~\cite{johnson2013accelerating, defazio2014saga} which allows to achieve a linear rate of convergence for smooth and strongly convex functions in a gradient-based setting and then extended to nonconvex functions~\cite{fang2018spider,zhou2018nested}. This rate is however still suboptimal (see e.g.~\cite{lan2018optimal}) and there has been some recent effort to design an optimal variance-reduced method, including~\cite{lin2015universal, allen2017katyusha, lan2018optimal, lan2019unified}. We will build on the approach of~\cite{lan2019unified} as it relies on less restrictive assumptions than other methods (see discussion in Section~\ref{sec:related_work}). We design a novel algorithm that estimates derivatives using the Gaussian smoothing approach of~\cite{nesterov2017random} as well as the coordinate-wise approach of~\cite{ji2019improved}. We prove an accelerated rate of convergence for this algorithm in the case of convex and strongly-convex functions. Our experimental results on several datasets support our theoretical findings.

% !TEX root = main.tex

\section{Related Work}
\label{sec:related_work}
\paragraph{Momentum in gradient-based setting.}
The first accelerated proof of convergence for the deterministic setting dates back to~\citet{polyak1964some} who proved a local linear rate of convergence for Heavy-ball (with constant momentum) for twice continuously differentiable, $\tau$-strongly convex and $L$-smooth functions, with a constant of geometric decrease which is smaller than the one for gradient descent. A similar method, Nesterov's Accelerated Gradient~(NAG), was introduced by~\cite{nesterov1983method}. It achieves the optimal $\mathcal{O}(1/t^2)$ rate of convergence for convex functions and, with small modifications, an accelerated linear convergence rate for smooth and strongly-convex functions.

Prior work has shown that vanilla momentum methods lack stability in stochastic settings, where the evaluation of the gradients is affected by noise~(see e.g. motivation in~\cite{allen2017katyusha}). Various solutions have been suggested in the literature, including using a regularized auxiliary objective that enjoys a better condition number than the original objective~\cite{lin2015universal} or applying variance-reduction to obtain more stable momentum updates~\cite{allen2017katyusha, lan2019unified}.
We here build on the Varag approach presented in~\cite{lan2019unified} as it presents several advantages over prior work, including the ability to accelerate for smooth convex finite-sum problems as well as for strongly-convex problems without requiring an additional strongly-convex regularization term. Unlike Katyusha~\cite{allen2017katyusha}, Varag also only requires the solution of one, rather than two, subproblems per iteration~(discussed in~\cite{lan2019unified}). 

\vspace{-2mm}
\paragraph{Variance-reduced DFO.}
In the finite-sum setting introduced in Eq.~\eqref{eq:min_f}, variance-reduction techniques have become popular in machine learning. These techniques were originally developed for gradient-based methods and later adapted to the derivative-free setting in~\cite{liu2018stochastic} and~\cite{liu2018stochasticb}. Various improvements were later made in~\cite{ji2019improved} such as allowing for a larger constant stepsize, as well as extending the analysis of~\cite{fang2018spider} to a broader class of functions in a DFO setting. Finally, \cite{ji2019improved} introduced a coordinate-wise approach to estimate the gradients instead of the Gaussian smoothing method. This yields a more accurate estimation of the gradient at the price of a higher computational complexity. We rely on this technique to estimate the gradient at the pivot point in our analysis~(see details in Section~\ref{sec:algo_analysis}).

\vspace{-2mm}
\paragraph{Momentum in DFO.} As mentioned earlier,~\cite{nesterov2017random} proved a rate in a deterministic setting. ~\cite{gorbunov2018accelerated}, analyzes acceleration in a stochastic setting for general objective functions without explicitly exploiting any finite-sum structure~(hence, assuming finite variance). Closer to our setting, ~\cite{gorbunov2019stochastic} analyzes a stochastic momentum DFO method based on the three point estimation technique proposed in~\cite{bergou2019stochastic}. Although they do theoretically analyze the convergence of such algorithms, they only prove a suboptimal rate of convergence instead of the accelerated rate of convergence derived in our work.

% !TEX root = main.tex
\section{Background and Notation}
\label{sec:background}
In this paper we work in $\R^d$ with the standard Euclidean norm $\|\cdot\|$ and scalar product $\langle\cdot,\cdot\rangle$. Our goal, as stated in the introduction, is to minimize a convex function $f=\frac{1}{n}\sum_{i=1}^n f_i$ (with $f_i:\R^d\to\R$ for each $i=1\dots n$) without using gradient information. For our theoretical analysis, we will need the following standard assumption.
\begin{tcolorbox}
\textbf{(A1)} \ \ Each $f_i$ is convex, differentiable and $L$-smooth\footnote{for all $x,y\in\R^d$ we have $\|\nabla f_i(x)-\nabla f_i(y)\|\le L\|x-y\|$.}.\\ Hence, also $f=\frac{1}{n}\sum_{i=1}^n f_i$ is convex and $L$-smooth.
\end{tcolorbox}
To estimate gradients, we will use and combine two different gradient estimation techniques, with different properties.
\vspace{-2mm}
\paragraph{Estimation by Gaussian smoothing.} This technique was first presented by~\citet{nesterov2017random}: let $\mu$ be the smoothing parameter, then $f_\mu:\R^d\to\R$, the smoothed version of $f$, is defined to be such that for all $x\in\R^d$
  $$f_\mu(x) = \frac{1}{(2\pi)^{\frac{d}{2}}}\int_{\R^d} f(x+\mu u)e^{-\frac{1}{2}\|u\|^2} du.$$
In our setting, it is easy to see that (see discussion in the appendix) $f_\mu$ is still convex and $L$-smooth. Crucially, the integral in the definition of $f_\mu$ can be approximated by sampling random directions $u \in \R^d$ with a Gaussian distribution: $f_\mu(x) = \E_u[f(x + \mu u)]$. Note that, for $\mu\ll 1$, we have $f_\mu\approxeq f$.\\
The gradient of $f_\mu$ can be written as
\begin{align}
\nabla f_\mu(x) = \E_u \left[\frac{(f(x + \mu u) - f(x)) u}{\mu}\right]=: \E_u[g_{\mu}(x, u)]. \nonumber
\end{align}
A stochastic estimate of $g_{\mu}(x, u)$ using data-point $i$, which we denote by $g_{\mu}(x, u, i)$, can be then calculated as follows:
\vspace{-1mm}
\begin{align}
g_{\mu}(x,u,i) := \frac{f_{i}(x+\mu u)-f_{i}(x)}{\mu} u.
\label{DFO-framework-gaussian-smoothing}
\end{align}
As we will see more in more detail in the next section, this cheap estimate is not appropriate if we seek a solid approximation. Fortunately, for such a task, we can use the coordinate-wise finite difference method.
\vspace{-1mm}
\paragraph{Estimation by coordinate-wise finite difference.} This approach, introduced in~\cite{ji2019improved}, estimates $\nabla f_i (x)$ without introducing a smoothing distortion, by directly evaluating the function value in each coordinate:
 \vspace{-2mm}
\begin{align}
	g_{\nu}(x,i) = \sum_{j=1}^{d}\frac{f_{i}(x+\nu \text{e}_{j})-f_{i}(x-\nu \text{e}_{j})}{2\nu} \text{e}_{j},
\label{DFO-framework-cord-finite-difference}
\end{align}
where $\text{e}_{j}$ is the unit vector with only one non-zero entry $1$ at its $j^{th}$ coordinate. Note that, $g_{\nu}$ is $d$ times more expensive to compute compared to $g_{\mu}$. Besides, the coordinate-wise estimator of $\nabla f(x)$ is denoted as $g_{\nu}(x)$ where we remove the subscript $i$ from Eq.~\eqref{DFO-framework-cord-finite-difference}.

\section{Algorithm and Analysis}
\label{sec:algo_analysis}
The method we propose is presented as Algorithm~\ref{algorithm-VARAG} (ZO-Varag), and is an adaptation of Varag~\cite{lan2019unified} to the DFO setting. At it's core, ZO-Varag has the same structure of SVRG~\cite{johnson2013accelerating}, but profits from the mechanism of accelerated stochastic approximation~\cite{lan2012optimal} combined with the two different zero-order gradient estimators presented in the last section. We highlight some important details below:
 \vspace{-3mm}
\begin{enumerate}
\item At the beginning of epoch $s$, we compute a \textit{full} zero-order gradient $\tilde g^s$ at the \textit{pivotal point} $\tilde{x}^{s-1}$~(i.e. the approximation of the solution provided by the preceding epoch). Since the accuracy in $\tilde g^s$ drastically influences the progress made in the epoch, we choose for its approximation the coordinate-wise estimator in Eq.~\eqref{DFO-framework-cord-finite-difference}.  The estimate $\tilde g^s$ will then be used to perform $T_s$ inner-iterations and to compute the next approximation $\tilde x^s$ to the problem solution.

\item Each inner-iteration~(within an epoch) uses three sequences: $\{ x_t \}, \{ \underline x_t \}, \{ \bar x_t \}$. Each of these sequences play an important role in the acceleration mechanics~(see discussion in~\cite{lan2019unified}).
\item In the inner loop, at iteration $t$, a cheap \textit{variance-reduced gradient estimate} of $\nabla f_{\mu}(\underline{x}_t)$ is computed using the same technique as SVRG~\cite{johnson2013accelerating} combined with Gaussian smoothing~(see Eq.~\eqref{DFO-framework-gaussian-smoothing})
\begin{equation}
    G_t := g_{\mu}(\underline x_{t}, u_t, i_t) - g_{\mu}(\tilde{x}, u_t, i_t)  + \tilde{g}^s,
    \label{eq:svrg-gradient-estimation}
\end{equation}
where $u_t$ is a sample from a standard multivariate Gaussian, as required by the estimator definition and $\tilde{x}$ is the pivotal point for this inner loop (epoch $s$).
\item The choice of the additional parameters $\{T_s\}$, $\{\gamma_s\}$, $\{\alpha_s\}$, $\{p_s\}$, $\{\theta_t\}$ will be specified in the convergence theorems depending on each function class being considered (smooth, convex or strongly-convex).
\end{enumerate}
\begin{algorithm*}[ht!]
	\caption{ZO-Varag}
	\label{algorithm-VARAG}
	\renewcommand{\algorithmicoutput}{\textbf{Output:}}
	\begin{algorithmic}[H]
		\REQUIRE $x^0\in \mathbb{R}^d, \{T_s\}, \{\gamma_s\}, \{\alpha_s\}, \{p_s\}, \{\theta_t\}$.
		\STATE Set $\tilde x^0 = \bar{x}^0 = x^0$.
		\FOR {$s=1,2,\ldots, S$}
		\STATE \textbf{Option I \ :} $\tilde{x}=\tilde{x}^{s-1}$
		\STATE \textbf{Option II:} $\tilde{x}=\bar{x}^{s-1}$
		\STATE Set $x_0=x^{s-1}$, $\bar x_0 = \tilde x$.
		\STATE \textbf{Pivotal ZO gradient} $\tilde{g}^s = g_{\nu}(\tilde{x})$ using the coordinate-wise approach by Eq.~\eqref{DFO-framework-cord-finite-difference}.
		% \STATE probability $Q=\{q_1,\ldots,q_m\}$ on $\{1,\ldots,m\}$.
		\FOR {$t=1,2,\ldots,T_s$}
		\STATE $\underline{x}_t = \big[(1+\tau \gamma_s)(1  - \alpha_s - p_s) \bar{x}_{t-1}  + \alpha_s x_{t-1}+ (1+\tau \gamma_s) p_s \tilde{x}\big]/[1+\tau \gamma_s(1-\alpha_s)]$\label{eqn:underlinex}.
		\STATE Pick $i_t\in \{1,\dots,m\}$ uniformly and generate $u_t$ from $\mathcal{N}(0, I_d)$. 
        \STATE $G_t = g_{\mu}(\underline x_{t}, u_t, i_t) - g_{\mu}(\tilde{x}, u_t, i_t)  + \tilde{g}^s$\label{eqn:estgradient}.
		\STATE $x_{t} =[\gamma_s G_t + \gamma_s \tau \underline{x}_t + x_{t-1}]/[1+\gamma_s \tau] \ \ \ \ \ \ \ \ \ \ \diamond = \arg \min_{x \in \R^d} \left\{ \gamma_s \left[\langle G_t, x \rangle + \frac{\tau}{2}\|\underline{x_t}- x \|^2\right] + \frac{1}{2}\|x_{t-1} - x\|^2 \right\}$\label{eqn:xt}.
		\STATE $\bar x_{t} =   (1 - \alpha_s-p_s) \bar x_{t-1}  + \alpha_s x_{t} + p_s \tilde x$\label{eqn:barx}.
		\ENDFOR
		\STATE Set  $x^s = x_{T_s}, \bar{x}^s = \bar{x}_{T_s}$ and $\tilde{x}^s =\sum_{t=1}^{T_s}(\theta_t \bar x_t)/\big(\sum_{t=1}^{T_s} \theta_t\big)$.
		\ENDFOR
		\OUTPUT {$\tilde{x}^{S}$}
	\end{algorithmic}
\end{algorithm*}

\subsection{Variance of the Gradient Estimators}
From our discussion above, it is clear the following \textit{error term} $\delta_t$ will heavily influence the analysis: the error in the estimation of the per-iteration direction $G_t$.
\begin{align*}
    \delta_t := G_t - \nabla f_{\mu}(\underline{x}_t)  \ \  & \text{\ \ \ (iteration gradient error).}
\end{align*}
The expectation of $\delta_t$, over $u_t, i_t$, is $e^s$ defined below:
\begin{align*}
       e^s := \tilde g^s -\nabla f_\mu(\tilde{x}) \ \ & \text{\ \ \ (pivotal gradient error).}
\end{align*}
This is different from the standard SVRG as the pivotal gradient error $e^s$ vanishes for gradient-based methods. The rest of this section is dedicated to the fundamental properties of this error.
\paragraph{Pivotal gradient error bound.} Crucially, note that $e^s$ is measured with respect to $f_\mu$ (the smoothed version of $f$). This provides consistency with $\delta_t$, at the price of a well-behaved additional error coming from the smoothing distortion. A necessary first step to start our analysis is to bound $\|e^s\|^2$ \textit{uniformly} by a problem-dependent constant $E$:
\begin{align}
    \|e^s\|^2 \le \ & \|\tilde g^s -\nabla f_\mu(\tilde x^{s-1})\|^2 \notag\\ 
    \le \ & 2 \big[\|\tilde{g}^{s}- \nabla f(\tilde{x}^{s-1})\|^2 \notag\\
    & + \|\nabla f(\tilde{x}^{s-1}) - \nabla f_{\mu}(\tilde{x}^{s-1})\|^2 \big] \notag\\
    \le \ &  2L^2d\nu^2 + \frac{\mu^2L^2(d+3)^3}{2} =: E,
    \label{definition_E}
\end{align}
where the last inequality is a combination of Lemma 3 from~\cite{ji2019improved} and Lemma 3 from~\cite{nesterov2017random}. Note that a similar inequality would not be possible by using an estimate obtained from sampling a random direction for pivotal $\tilde{g}^s$ --- as the strength of the error would depend on the gradient magnitude, i.e. cannot be uniformly bounded (see~Theorem 3 in~\cite{nesterov2017random}).

\paragraph{Iteration gradient error bound.}  Unfortunately, as ZO-Varag is a DFO algorithm, the expectation of $\delta_t$ is \textit{not vanishing}~(in contrast to standard SVRG and Varag). However, the next lemma shows that it is still possible to bound the (trace of the) variance of $G_t$.

\begin{lemma} (Variance of $G_t$)\label{VARAG-lemma-1}
Assume \textbf{(A1)}. Then, at any epoch $s\ge 1$ and iteration $1\le t\le T_s$ we have
 \vspace{-1mm}
\begin{align}
& \mathbb{E}_{u_t,i_t| \mathcal{F}_{t-1}} \big[\| G_t - \mathbb{E}_{u_t,i_t| \mathcal{F}_{t-1}}[G_t]\|^2\big] \\
&\le \ \ 18\mu^2L^2(d+6)^3 \nonumber \\
& \ \ \ \ \  + 8(d+4)L \big[f_{\mu}(\tilde{x})-f_{\mu}(\underline{x}_t) - \langle \nabla f_{\mu}(\underline{x}_t), \tilde{x} - \underline{x}_{t} \rangle \big], \nonumber
\end{align}
and 
\vspace{-3mm}
\begin{equation}
\mathbb{E}_{u_t, i_t| \mathcal{F}_{t-1}}\big[\delta_t\big] = \tilde{g}^s - \nabla f_{\mu} (\tilde{x}) \neq 0,
\vspace{-1mm}
\end{equation}
where $\mathcal{F}_t$ is the $\sigma$-algebra generated by the previous iterates in the current epoch, i.e. $\mathcal{F}_t := \{u_{t}, i_{t}, \ldots, u_{1}, i_{1} \}$, and $\tilde{x}$ is the pivotal point for the epoch $s$.

\end{lemma}
Compared to Lemma 3 in \cite{lan2019unified}, the bound on the variance of the gradient in the DFO case is dimension-dependent and has an extra error $18\mu^2L^2(d+6)^3$ due to Gaussian smoothing (it comes from the fact that we also take the expectation over $u_t$).

\subsection{Analysis for Smooth and Convex Functions}
For our final complexity result in this section to hold, we need all the sequences generated by Algorithm~\ref{algorithm-VARAG} to be bounded \textit{in expectation}.

\begin{tcolorbox}
\textbf{(A2$_{\boldsymbol{\mu}}$)} \ \ Let $x_\mu^* \in \text{argmin}_{x} f_\mu(x)$ and consider the sequence of approximations $\{\tilde x^s\}$ returned by Algorithm~\ref{algorithm-VARAG}. There exists a \textit{finite} constant $Z<\infty$, potentially dependent on $L$ and $d$, such that, for $\mu$ small enough,
$$\sup_{s\ge0}\E\left[\|\tilde x^s-x_\mu^*\|\right]\le Z.$$
\end{tcolorbox}
Using an argument similar to~\cite{gadat2018stochastic}, it is possible to show that this assumption holds under the requirement that $f$ is coercive, i.e. $f(x)\to\infty$ as $\|x\|\to\infty$.
We are ready to state the main theorem of this section.
\begin{theorem}\label{VARAG-theorem-1}
	Assume \textbf{(A1)} and \textbf{(A2$_{\boldsymbol{\mu}}$)}. If we define $s_0 := \lfloor \log (d+4)n \rfloor+1$ and set $\{T_s\}$, $\{\gamma_s\}$ and $\{p_s\}$ as
	\begin{align}\label{parameter-deter-smooth1}
	T_s = \begin{cases}
	2^{s-1}, & s \le s_0\\
	T_{s_0}, & s > s_0
	\end{cases}, \ \ \ \ \
	\gamma_s = \tfrac{1}{12 (d+4)L \alpha_s}, \ \ \ \ \
	p_s = \tfrac{1}{2},
	\end{align}
	% with 
	\vspace{-1.5mm}
	with
	\vspace{-5mm}
	\begin{align}\label{parameter-deter-alpha-sm}
	\alpha_s =
	\begin{cases}
	\tfrac{1}{2}, & s \le s_0\\
	\tfrac{2}{s-s_0+4},& s > s_0
	\end{cases}.
	\end{align}
	If we set
	\vspace{-3mm}
	\begin{align}\label{VARAG-def-theta-paper}
\theta_t =
\begin{cases}
\tfrac{\gamma_{s}}{\alpha_{s}} (\alpha_{s} + p_{s}) & 1 \le t \le T_s-1\\
\tfrac{\gamma_s}{\alpha_s} & t=T_s.
\end{cases}
\end{align} 
	we obtain
	\begin{align*}
	&\mathbb{E} \big[f_{\mu}(\tilde{x}^s)-f^*_{\mu}\big] \le \\
	& \ \ \ \ \ \ \begin{cases}
	\cfrac{(d+4)D_0}{2^{s+1}}  + 2\varsigma_1 +3\varsigma_2, & 1 \le s \le s_0,\\
	\cfrac{16 D_0}{n(s-s_0+4)^2} + \delta_s \cdot (\varsigma_1 +\varsigma_2),&  s > s_0,\\
	\end{cases}
	\end{align*}
	where $\varsigma_1 = \mu^2L(d+4)^2$, $\varsigma_2 = Z\sqrt{E}$, $\delta_s = \mathcal{O}(s-s_0)$ and $D_0$ is defined as
	\begin{align} \label{VARAG-def-D_0}
	D_0:= \frac{2}{(d+4)}[f_{\mu}(x^0) - f_{\mu}(x_{\mu}^*)] + 6L \|x^0-x_{\mu}^* \|^2,
	\end{align}
	\vspace{-2mm}
	where $x^*_\mu$ is any finite minimizer of $f_\mu$.
\end{theorem}

Compared to the gradient-based analysis of~\cite{lan2019unified}, \textit{two additional errors terms appear} because of the DFO framework: $\varsigma_1$ is the error due to the Gaussian smooth estimation and $\varsigma_2$ is an error due to the approximation made at the pivot point. It is essential to note that, in the bound for $s>s_0$, the error $\delta_s (\varsigma_1 +\varsigma_2)$ grows linearly with the number of epochs. In Corollary~\ref{VARAG-corollary-1} we show how it is possible to tune our zeroth-order estimators to make these errors vanish by choosing sufficiently small smoothing parameters $\mu$ and $\nu$, with an argument similar to the one used in Theorem 9 from~\citet{nesterov2017random}.

Based on Theorem~\ref{VARAG-theorem-1}, we obtain the following complexity bound.
\begin{tcolorbox}
\begin{corollary}\label{VARAG-corollary-1}
Assume \textbf{(A1)} and \textbf{(A2$_{\boldsymbol{\mu}}$)}. The total number $\bar{N}_\epsilon$ of function queries performed by Algorithm~\ref{algorithm-VARAG} to find a stochastic $\epsilon$-solution, i.e. a point $\bar{x} \in \R^d$ s.t. $\mathbb{E} [f(\bar{x})-f^*]\le \epsilon$,  can be bounded by
\begin{align*}
\bar{N}_\epsilon := \begin{cases}
\mathcal{O}\left\{  dn \log \tfrac{d D_0}{\epsilon} \right\}, & n \ge D_0/\epsilon\\
{\cal O} \left\{dn \log dn + d\sqrt{\tfrac{n D_0}{\epsilon}} \right\}, &  n < D_0/\epsilon.
\end{cases}
\end{align*}
\end{corollary}
\end{tcolorbox}

The reasoning behind the proof is quite standard in the DFO literature~\cite{nesterov2017random}, yet it contains some important ideas. Hence we include a proof sketch in order to give some additional intuition to the reader, who might wonder how to control the error terms from the theorem. Details can be found in the appendix.
\begin{proof}[Proof sketch]
The procedure consists in deriving three bounds, that combined give the desired suboptimality $\epsilon$:
\vspace{-4mm}
\begin{enumerate}[wide, labelwidth=1pt, labelindent=1pt]
    \itemsep0em
    \item In general, $x_\mu^*\ne x^*\in\text{argmin}_x f(x)$. Yet, Theorem~\ref{VARAG-theorem-1} gives us a procedure to approximate $x_\mu^*$. Hence we have to show that $f_{\mu}(\tilde{x}^s) - f_{\mu}(x_{\mu}^*)$ and $ f(\tilde{x}^s) - f(x^*)$ are close enough.
    In particular, we have $f_{\mu}(\tilde{x}^s) - f_{\mu}(x_{\mu}^*) \ge f(\tilde{x}^s) - f(x^*) - \mu^2 Ld$, directly from Theorem 1 in~\cite{nesterov2017random}. Hence, we get the following sufficient condition: $\frac{\epsilon}{4}\ge {\mu^2Ld}$. Therefore, the desired bound holds for $\mu^2\le\frac{\epsilon}{4Ld}$. Since $\mu$ is a design parameter, \textit{which does not affect the convergence speed but just the error}, we can choose it small enough so that this requirement is satisfied.
    \item Next, assume $\varsigma_1 = \varsigma_2 = 0$ --- we will deal with these terms at the end of the proof. We can then follow the proof of Theorem 1 in~\cite{lan2019unified}, but with the requirement of $\frac{\epsilon}{2}$ accuracy. This gives us the desired number $\bar N_\epsilon$ of function queries, which correspond to $\bar s_\epsilon$ epochs. 
    \item Last, we spend the last $\frac{\epsilon}{4}$ accuracy to bound the error terms, now that we know we need to be running the algorithm only for $\bar s_\epsilon$ epochs. First, we group together the error terms in $\varsigma = (1+\delta_{\bar s_\epsilon})(\varsigma_1 + \varsigma_2)$. We recall that, by Eq.~\eqref{definition_E}, $(\varsigma_1 + \varsigma_2) \propto \mu^2 + \sqrt{\mu^2 + \nu^2}$. Hence, again as for the first point of this proof, we can choose $\mu$ and $\nu$ small enough such that $\varsigma\le\frac{\epsilon}{4}$. Note that it is exactly in this step that we need \textbf{(A2$_{\boldsymbol{\mu}}$)}.
\end{enumerate}
 \vspace{-1mm}
Hence, we can reach accuracy $\epsilon = \frac{\epsilon}{4} + \frac{\epsilon}{2} + \frac{\epsilon}{4}.$
\end{proof}
We make two important remarks.
\begin{remark}[Error terms]
As we discussed in the proof of Corollary~\ref{VARAG-corollary-1}, smaller smoothing parameters yields smaller additional errors. Thus, in line with the previous literature~\cite{nesterov2017random,ji2019improved} we can choose the smoothing parameters $\mu,\nu$ arbitrarily small as long as they are less than the upper bounds derived in "Proof of Corollary~\ref{VARAG-corollary-1}" in the appendix. Theoretically, $\mu,\nu$ relies on a good estimation of $Z$ in A2$_{\mu}$. However, from a more practical side, we note in our experimental results in Section~\ref{sec:experiments} that the worst-case guarantees are not necessarily tight since we do not observe any significant error accumulation.
\end{remark}

\vspace{-2mm}
\begin{remark}[Dependency on the problem dimension] 
The overall dependency of $\bar N_\epsilon$ on the problem dimension is $\mathcal{O}(d\log(d))$. This complexity is comparable\footnote{As an interesting side-note, if $d$ is the ratio between the diameter of the universe and the diameter of a proton (i.e. $\approx 10^{42}$), we have $\log(d)<100$.} to the usual $\mathcal{O}(d)$ found in the classical literature~\cite{nesterov2017random,ghadimi2013stochastic}.
\end{remark}

We conclude by observing that, in the case $n \ge D_0/\epsilon$, Algorithm~\ref{algorithm-VARAG} achieves a linear rate of convergence when the desired accuracy is low ($\epsilon$ has a large value) and/or $n$ is large. In the other case $n < D_0/\epsilon$ (i.e. high accuracy), Algorithm~\ref{algorithm-VARAG} achieves \textit{acceleration}.

\subsection{Analysis for Smooth and Strongly-convex Functions}

We now analyze the case where $f$ is strongly-convex.

\begin{tcolorbox}
\textbf{(A3)} \ \ $f=\frac{1}{n}\sum_{i=1}^n f_i$ is $\tau$-strongly convex. That is, for all $x,y\in\R^d, \ f(y)\ge f(x)+\langle\nabla f(x),y-x\rangle +\frac{\tau}{2}\|y-x\|^2$.
\end{tcolorbox}

For this case, we do not need \textbf{(A2$_{\boldsymbol{\mu}}$)} since we will leverage on strong convexity~(which implies coercivity) to include $Z$ directly into our analysis.

\begin{theorem}\label{VARAG-theorem-2}
	Assume \textbf{(A1)} and \textbf{(A3)}. Let us denote $s_0 := \lfloor \log (d+4)n\rfloor+1$ and assume that the
	weights $\{\theta_t\}$ are set to Eq.~\eqref{VARAG-def-theta-paper} if $1\le s \le s_0$. Otherwise, they are set to
	\vspace{-1mm}
	\begin{align}\label{VARAG-def-theta-2}
	\theta_t =
	\begin{cases}
	\Gamma_{t-1} - (1 - \alpha_s - p_s) \Gamma_{t}, & 1 \le t \le T_s-1,\\
	\Gamma_{t-1}, & t = T_s,
	\end{cases}
	\end{align}
	where $\Gamma_t= \big(1+\frac{\tau\gamma_s}{2}\big)^t$.
	If the parameters $\{T_s\}$, $\{\gamma_s\}$ and $\{p_s\}$ are set to Eq.~\eqref{parameter-deter-smooth1} with %$\{\alpha_s\}$ being set as
	\begin{align}\label{parameter-deter-alpha-unified}
	\alpha_s =
	\begin{cases}
	\tfrac{1}{2}, & s \le s_0,\\
	\min\{\sqrt{\frac{n \tau}{24L}}, \tfrac{1}{2}\},& s > s_0,
	\end{cases}
	\end{align}
	we obtain
	\begin{align*}
	& \mathbb{E} \big[f_{\mu}(\tilde{x}^{s})-f_{\mu}^{*}\big] \le \\
	& \begin{cases}
	\cfrac{1}{2^{s+1}} (d+4)D_0 + 2\varsigma_1 + 0.5\varsigma_2, & \ \ \ \ \ 1 \le s \le s_0\\
	& \\
	 (4/5)^{s-s_0}\cfrac{D_0}{n}+ 12\varsigma_1 + 5\varsigma_2, &\ \ \ \ \ s > s_0 \text{ and } n \ge \frac{6L}{\tau} \\
	& \\
	  \left(1+\frac{1}{4} \sqrt{\frac{n\tau}{6L}}\right)^{-(s-s_0)} \cfrac{D_0}{n}  & \ \ \ \ \ s > s_0 \text{ and } n < \frac{6L}{\tau} \\
	\ \ \ \ \ \ + \Big(8\sqrt{\frac{6L}{n\tau}}+4\Big)\varsigma_1 + 5\varsigma_2,
	\end{cases}
	\end{align*}
where $\varsigma_1 = \mu^2L(d+4)^2$, $\varsigma_2 = E/\tau$ and $D_0$ is defined as in Eq.~\eqref{VARAG-def-D_0}.
\end{theorem}

\begin{remark}
Unlike the result in Theorem~\ref{VARAG-theorem-1} for convex functions, the error term in Theorem~\ref{VARAG-theorem-2} for the strongly-convex case does not increase with the epoch $s$. This is consistent with Theorem 9 in~\cite{nesterov2017random}.
\end{remark}

Using the same technique as for the proof of Corollary~\ref{VARAG-corollary-1}, we get the following complexity bound.

\begin{tcolorbox}
\begin{corollary}\label{VARAG-corollary-2}
	Assume \textbf{(A1)} and \textbf{(A3)}. The total number $\bar{N}_\epsilon$ of function queries performed by Algorithm~\ref{algorithm-VARAG} to find a stochastic $\epsilon$-solution, i.e. a point $\bar{x} \in \R^d$ s.t. $\mathbb{E} [f(\bar{x})-f^*]\le \epsilon$,  can be bounded by
	\begin{align*}
	\bar N_\epsilon := 
	\begin{cases}
	\mathcal{O}\big\{dn\log \big(\frac{d D_0}{\epsilon}\big)\big\}, &  \scriptstyle n \ge D_0/\epsilon \mathrm{~or~} n \ge 6L/\tau,\\
	& \\
	\mathcal{O} \bigg\{dn \log(dn)  & \scriptstyle n <D_0/\epsilon \text{ and } n < 6L/\tau  \\
	+ d \sqrt{\frac{nL}{\tau}} \log \big(\frac{D_0}{n\epsilon}\big)\bigg\},&  \\
	\end{cases}
	\end{align*}
	% $\theta_t$ are chosen as in \eqref{def_theta_acc_SVRG} for the first and third case, and $\theta_t$ are chosen as in the following \eqref{def_theta_acc_SVRG_sc} for the second and fourth case,
\end{corollary}
\end{tcolorbox}

We conclude this subsection by commenting on the optimality of this complexity result, following the discussion in~\cite{lan2019unified}.
When $\tau$ and $\epsilon$ are small enough (i.e. the second case in Corollary~\ref{VARAG-corollary-2}, ill-conditioned), ZO-Varag exhibits an accelerated linear rate of convergence which depends on the square root of the condition number $\sqrt{L/\tau}$. Else, if $\epsilon$ or $\tau$ are relatively large~(first case), ZO-Varag treats the problem as if it was not strongly convex and retrieves the complexity bound of Corollary~\ref{VARAG-corollary-1}. Again, similarly to the smooth convex case, the dependency of $\bar N_\epsilon$ on the problem dimension is $\mathcal{O}(d\log(d))$.

\section{A Coordinate-wise Variant}
\label{sec:coordinate-wise}
In this section, we study the effect of replacing the gradient estimator $g_{\mu}(x, u, i)$ in the inner-loop of Algorithm \ref{algorithm-VARAG} with the coordinate wise variant $g_{\nu}(x,i)$ proposed in \cite{ji2019improved} and already used in the last section for the computation of the pivotal gradient $\tilde g^s$. More precisely, we consider the following modification ($g_{\nu}$ defined in Eq.~\eqref{DFO-framework-cord-finite-difference}): at each inner-loop iteration $t$,
 \vspace{-2mm}
$$G_t = g_{\nu}(\underline x_{t}, i_t) - g_{\nu}(\tilde{x}, i_t)  + \tilde{g}^s.$$
As we are not using a smoothed version of $f$ anymore, we need to introduce a slight modification on \textbf{(A2$_{\boldsymbol{\mu}}$)}. 
\begin{tcolorbox}
\textbf{(A2$_{\boldsymbol{\nu}}$)} \ \ Let $x^*\in \text{argmin}_{x\in\R^d} f(x)$. For any epoch $s$ of Algorithm~\ref{algorithm-VARAG}, consider the inner-loop sequences $\{ \underline x_t \}$ and $\{ \bar x_t \}$. There exist a \textit{finite} constant $Z<\infty$, potentially dependent on $L$ and $d$, such that, for $\nu$ small enough,
$$\sup_{s\ge0}\max_{\ \ x\in\{\bar x_t  \}\cup\{\underline x_t\}}\E\left[\|x-x^*\|\right]\le Z.$$
\end{tcolorbox}
Again, as mentioned in the context of \textbf{(A2$_{\boldsymbol{\mu}}$)} in the last section, it is possible to show that this assumption holds under the requirement that $f$ is coercive.

\subsection{Modified Analysis for Smooth and Convex Functions}
We follow the same proof procedure from last section, and comment the results with a remark at the end of this section.
\begin{theorem}\label{VARAG-cord-theorem-1}
	Consider the coordinate-wise variant of Algorithm~\ref{algorithm-VARAG} we just discussed. Assume \textbf{(A1)} and \textbf{(A2$_{\boldsymbol{\nu}}$)}. Let us denote $s_0 := \lfloor \log n \rfloor+1$. Suppose the weights $\{\theta_t\}$ are set as in Eq.~\eqref{VARAG-def-theta-paper}
	and parameters $\{T_s\}$, $\{\gamma_s\}$, $\{p_s\}$ are set as
	 \vspace{-1mm}
	\begin{align}\label{cord-parameter-deter-smooth1}
	T_s = \begin{cases}
	2^{s-1}, & s \le s_0\\
	T_{s_0}, & s > s_0
	\end{cases}, \
	\gamma_s = \tfrac{1}{12 L \alpha_s}, \
	\ \
	p_s = \tfrac{1}{2}, \ \mbox{with}
	\end{align}
	\vspace{-5mm}
	\begin{align}\label{cord-parameter-deter-alpha-sm}
	\alpha_s =
	\begin{cases}
	\tfrac{1}{2}, & s \le s_0\\
	\tfrac{2}{s-s_0+4},& s > s_0
	\end{cases}.
	\end{align}
	Then, we have
	\begin{align*}
	& \mathbb{E} \big[f(\tilde{x}^s)-f^*\big] \le \\
	& \ \ \ \ \ \ \ \begin{cases} 
	\cfrac{D_0'}{2^{s+1}}  + \varsigma_1 + 4\varsigma_2, &  \ \ \ \ \ \  1 \le s \le s_0\\
	\cfrac{16 D_0'}{n(s-s_0+4)^2} + \delta_s \cdot (\varsigma_1+\varsigma_2), & \ \ \ \ \ \ s > s_0 \\
	\end{cases}
	\end{align*}
	where $\varsigma_1=\nu^2Ld$, $\varsigma_2=L\sqrt{d}Z\nu$, $\delta_s = \mathcal{O}(s-s_0)$ and $D_0'$ is defined as
	\begin{align} \label{VARAG-cord-def-D_0}
	D_0':= 2[f(x^0) - f(x^*)] + 6L \|x^0-x^* \|^2,
	\end{align}
	where $x^*$ is any finite minimizer of $f$.
\end{theorem}

\begin{tcolorbox}
\begin{corollary}\label{VARAG-cord-corollary-1}
		Consider the coordinate-wise variant of Algorithm~\ref{algorithm-VARAG}. Assume \textbf{(A1)} and \textbf{(A2$_{\boldsymbol{\nu}}$)}. The total number $\bar{N}_\epsilon$ of function queries performed by Algorithm~\ref{algorithm-VARAG} to find a stochastic $\epsilon$-solution, i.e. a point $\bar{x} \in \R^d$ s.t. $\mathbb{E} [f(\bar{x})-f^*]\le \epsilon$,  can be bounded by
	\begin{align*}
	\bar{N}_\epsilon := \begin{cases}
	\mathcal{O}\left\{  dn \log \tfrac{D_0'}{\epsilon} \right\}, & n \ge D_0'/\epsilon,\\
	{\cal O} \left\{dn \log n + d\sqrt{\frac{n D_0'}{\epsilon}} \right\}, &  n < D_0'/\epsilon.
	\end{cases}
	\end{align*}
\end{corollary}
\end{tcolorbox}

\subsection{Modified Analysis under Strong Convexity}

\begin{theorem}\label{VARAG-cord-theorem-2}
	Consider the coordinate-wise variant of Algorithm~\ref{algorithm-VARAG}. Assume \textbf{(A1)}, \textbf{(A2$_{\boldsymbol{\nu}}$)} and \textbf{(A3)}. Let us denote $s_0 := \lfloor \log n \rfloor+1$ and assume that the
	weights $\{\theta_t\}$ are set to Eq.~\eqref{VARAG-def-theta-paper} if $1\le s \le s_0$. Otherwise, they are set to
	\vspace{-1mm}
	\begin{align}\label{VARAG-cord-def-theta-2}
	\theta_t =
	\begin{cases}
	\Gamma_{t-1} - (1 - \alpha_s - p_s) \Gamma_{t}, & 1 \le t \le T_s-1,\\
	\Gamma_{t-1}, & t = T_s,
	\end{cases}
	\end{align}
	where $\Gamma_t= \big(1+\tau\gamma_s\big)^t$.
	If the parameters $\{T_s\}$, $\{\gamma_s\}$ and $\{p_s\}$ set to	Eq.~\eqref{cord-parameter-deter-smooth1} with 
	\begin{align}\label{cord-parameter-deter-alpha-unified}
	\alpha_s =
	\begin{cases}
	\tfrac{1}{2}, & s \le s_0,\\
	\min\{\sqrt{\frac{n \tau}{12L}}, \tfrac{1}{2}\},& s > s_0,
	\end{cases}
	\end{align}
	\vspace{-1.5mm}
	We obtain
	\begin{align*}
	& \mathbb{E} \big[f(\tilde{x}^{s})-f^*\big] \le \\
	& \begin{cases}
	\cfrac{1}{2^{s+1}} D_0' + 1.5\varsigma_1 + 4\varsigma_2, & \ \ \ \ 1 \le s \le s_0\\
	& \\
	(4/5)^{s-s_0}\cfrac{D_0'}{n} + 9\varsigma_1 + 24\varsigma_2,  &  \ \ \ \ s > s_0 \text{ and } n \ge \frac{3L}{\tau}\\
	& \\
	\left(1+\frac{1}{4} \sqrt{\frac{n\tau}{3L}}\right)^{-(s-s_0)} \cfrac{D_0'}{n} & \\
	+ \big(2\sqrt{\frac{3L}{n\tau}}+1\big)(3\varsigma_1 + 8\varsigma_2), &  \ \ \ \  s > s_0 \text{ and } n < \frac{3L}{\tau}
	\end{cases}
	\end{align*}
	where $\varsigma_1=\nu^2Ld$, $\varsigma_2 = L\sqrt{d}Z\nu$ and $D_0'$ is defined as in Eq.~\eqref{VARAG-cord-def-D_0}.
\end{theorem}

\begin{tcolorbox}
\begin{corollary}\label{VARAG-cord-corollary-2}
    	Consider the coordinate-wise variant of Algorithm~\ref{algorithm-VARAG}. Assume \textbf{(A1)}, \textbf{(A2$_{\boldsymbol{\nu}}$)} and \textbf{(A3)}. The total number $\bar{N}_\epsilon$ of function queries performed by Algorithm~\ref{algorithm-VARAG} to find a stochastic $\epsilon$-solution, i.e. a point $\bar{x} \in \R^d$ s.t. $\mathbb{E} [f(\bar{x})-f^*]\le \epsilon$,  can be bounded by
	\begin{align*}
	\bar N := 
	\begin{cases}
	\mathcal{O}\big\{dn\log \big(\frac{D_0'}{\epsilon}\big)\big\}, & n \ge \frac{D_0'}{\epsilon} \text{ or } n \ge \frac{3L}{\tau},\\
	& \\
	\mathcal{O} \bigg\{dn \log(n)  &  n < \frac{D_0'}{\epsilon} \text{ and } n < \frac{3L}{\tau}\\
	+ d \sqrt{\frac{nL}{\tau}} \log \big(\frac{D_0'}{n\epsilon}\big)\bigg\}, & \\
	\end{cases}
	\end{align*}
\end{corollary}
\end{tcolorbox}
%\vspace{-12mm}
\begin{remark}[Complexity using the coordinate-wise variant]
Note that, the complexity results found in Corollary~\ref{VARAG-cord-corollary-1} and~\ref{VARAG-cord-corollary-2} are comparable to the ones for Gaussian smoothing (Corollary~\ref{VARAG-corollary-1} and~\ref{VARAG-corollary-2}) while the dimensional dependency here is $d$ rather than $d \log(d)$.
\end{remark}

% !TEX root = main.tex

\section{Experiments}
\label{sec:experiments}

\begin{figure}[t!]
 \centering          \begin{tabular}{c@{}c@{}c@{}c@{}}%c@{}}
        &   \scriptsize{Diabetes, S = 100, b = 5} & \scriptsize{ijcnn1, S = 50, b = 500} \\%& \tiny{autoencoder} \\
        \rotatebox[origin=c]{90}{\scriptsize{Logistic, $\lambda = 0$}} &    
             \includegraphics[width=0.475\linewidth,valign=c]{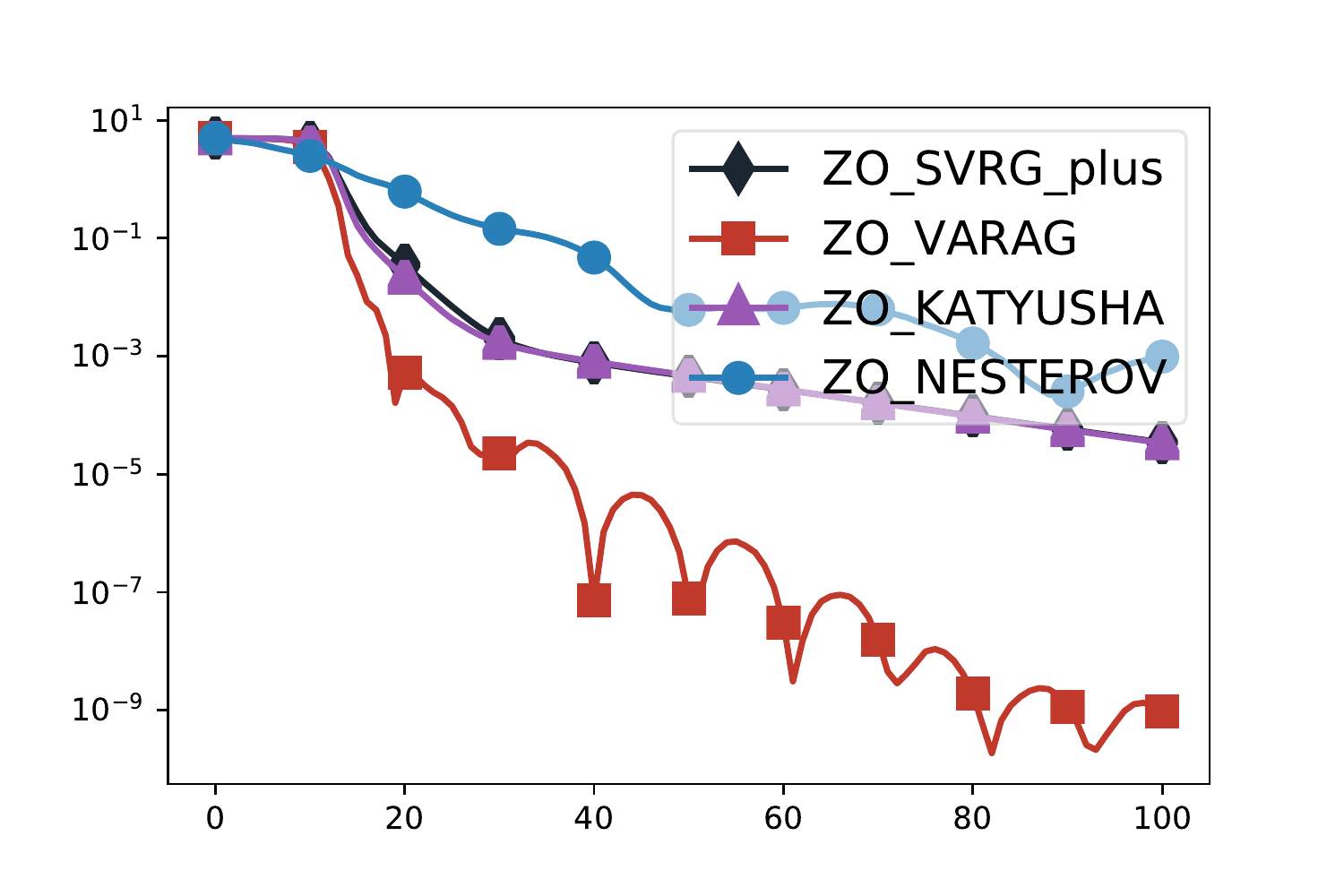}  &  \includegraphics[width=0.475\linewidth,valign=c]{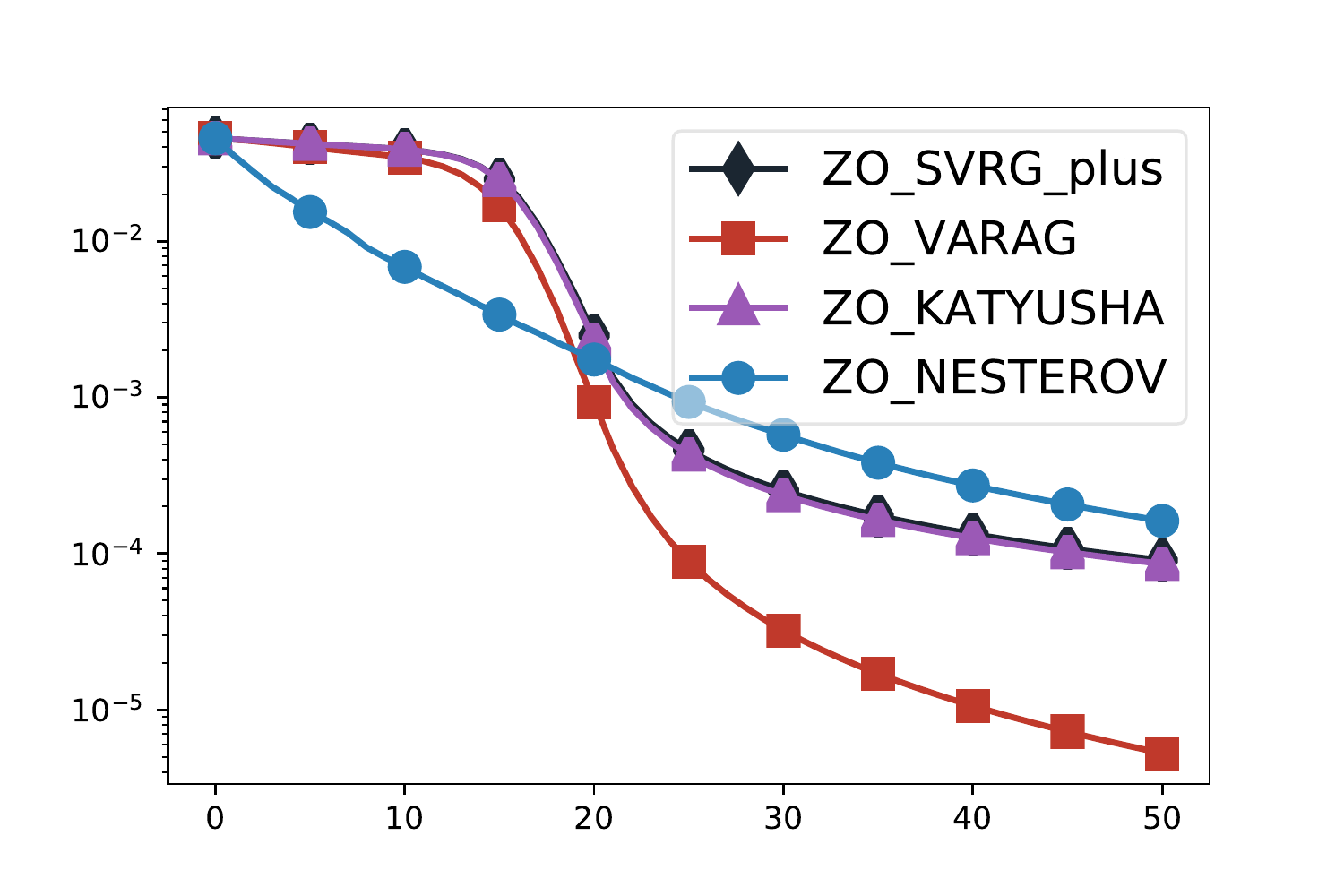}
             \\
             
        \rotatebox[origin=c]{90}{\scriptsize{Logistic, $\lambda = 1e^{-5}$}}&
             
              \includegraphics[width=0.475\linewidth,valign=c]{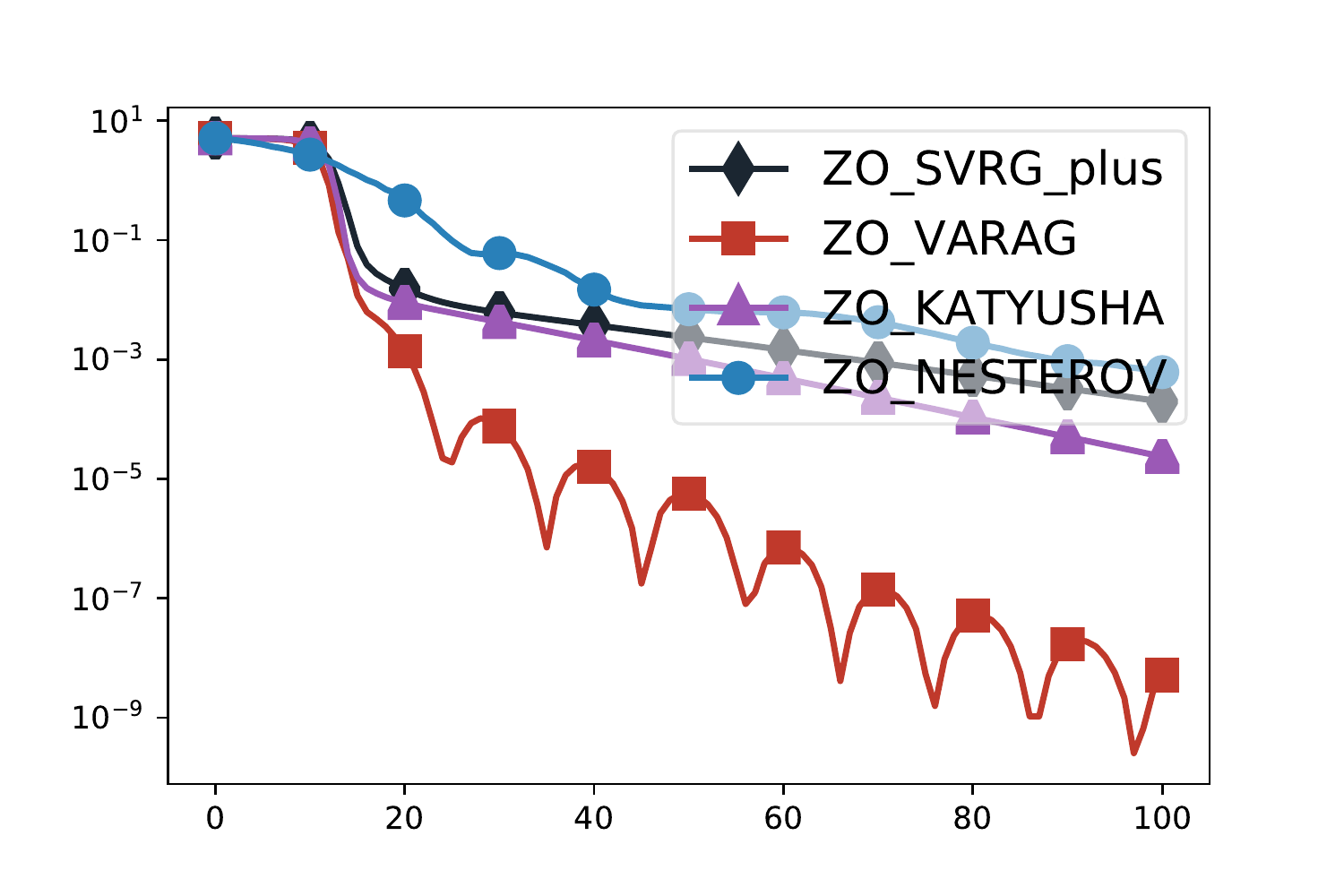}& \includegraphics[width=0.475\linewidth,valign=c]{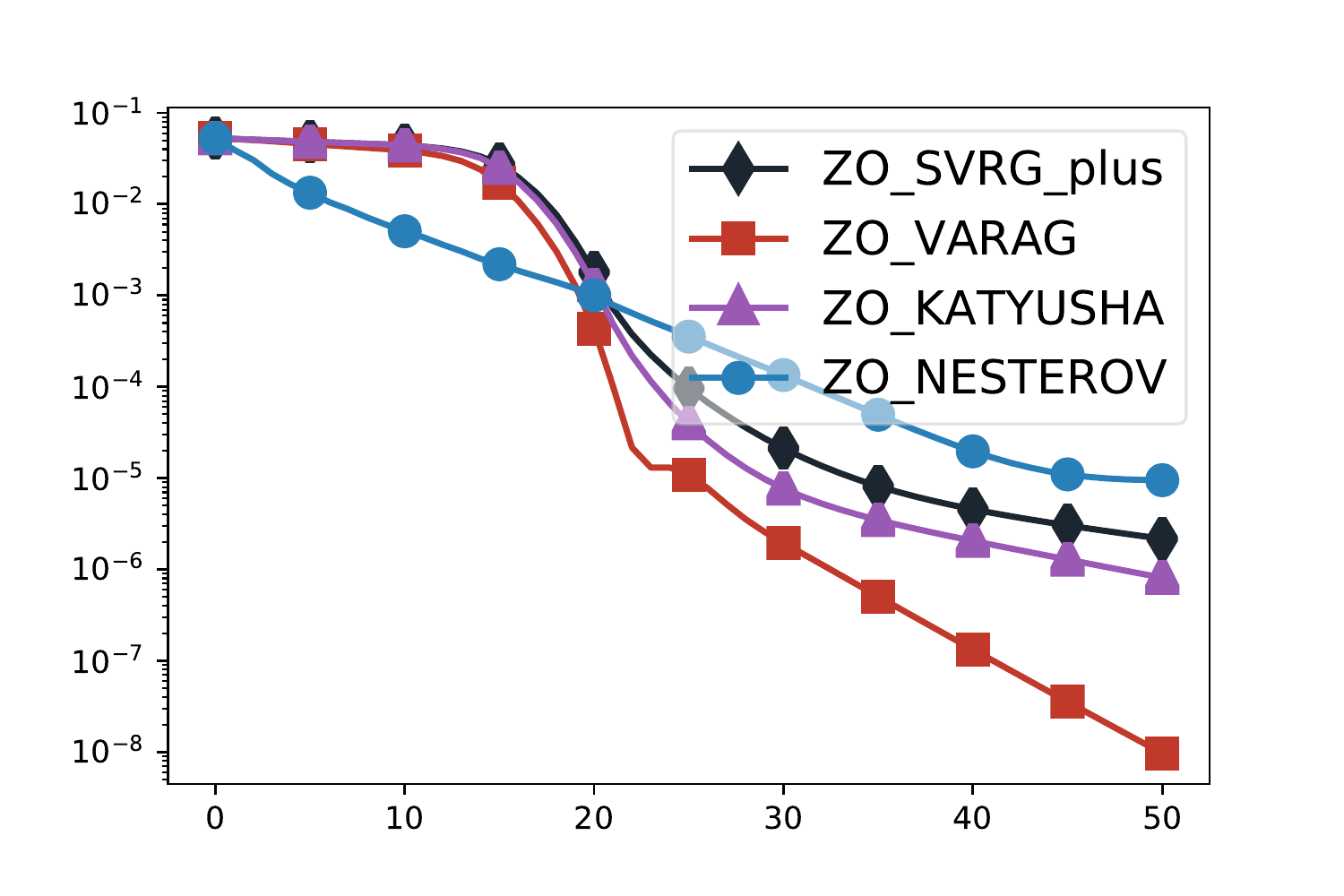}
             \\
        \rotatebox[origin=c]{90}{\scriptsize{Ridge, $\lambda = 0$}} &    
             \includegraphics[width=0.475\linewidth,valign=c]{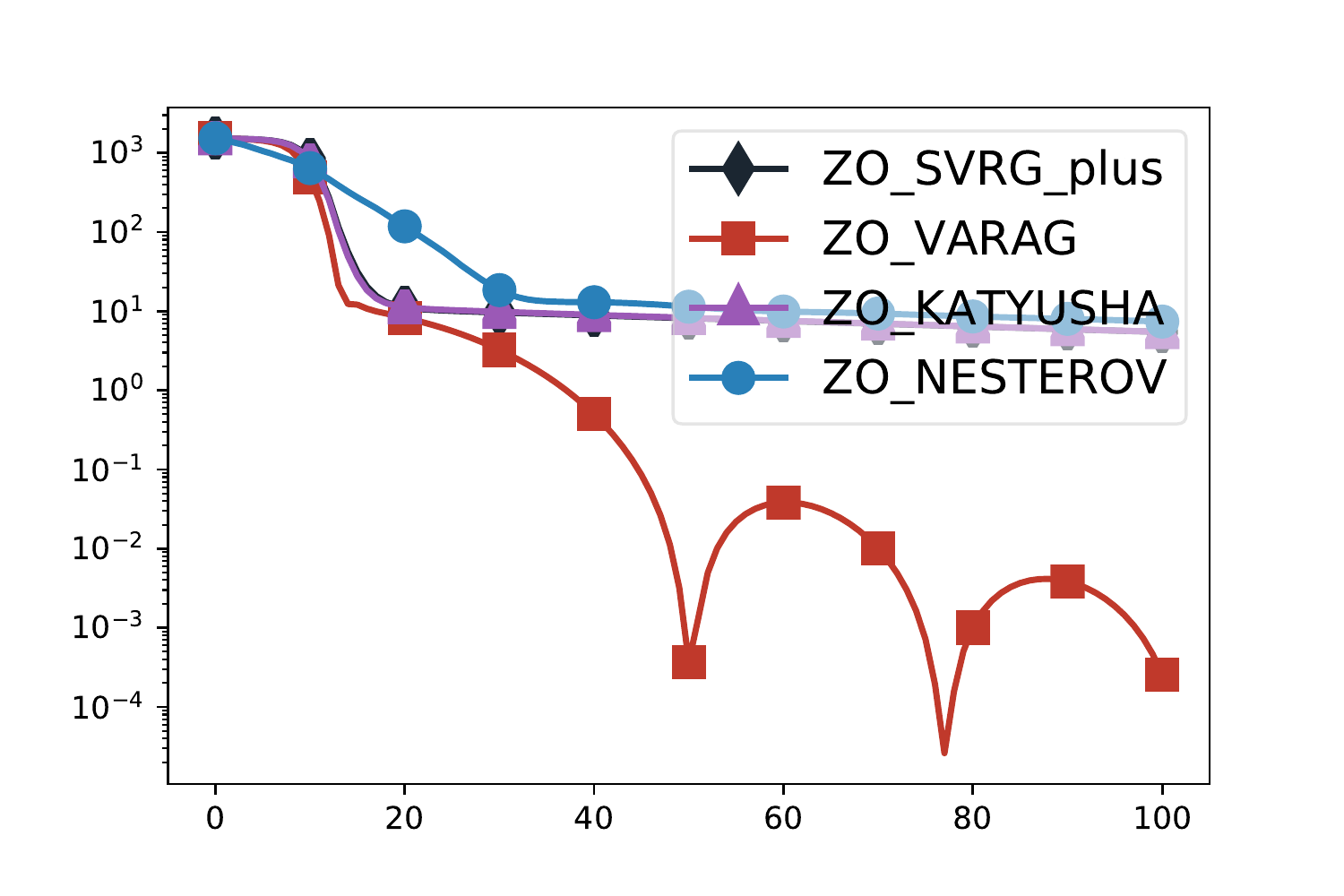}  &  \includegraphics[width=0.475\linewidth,valign=c]{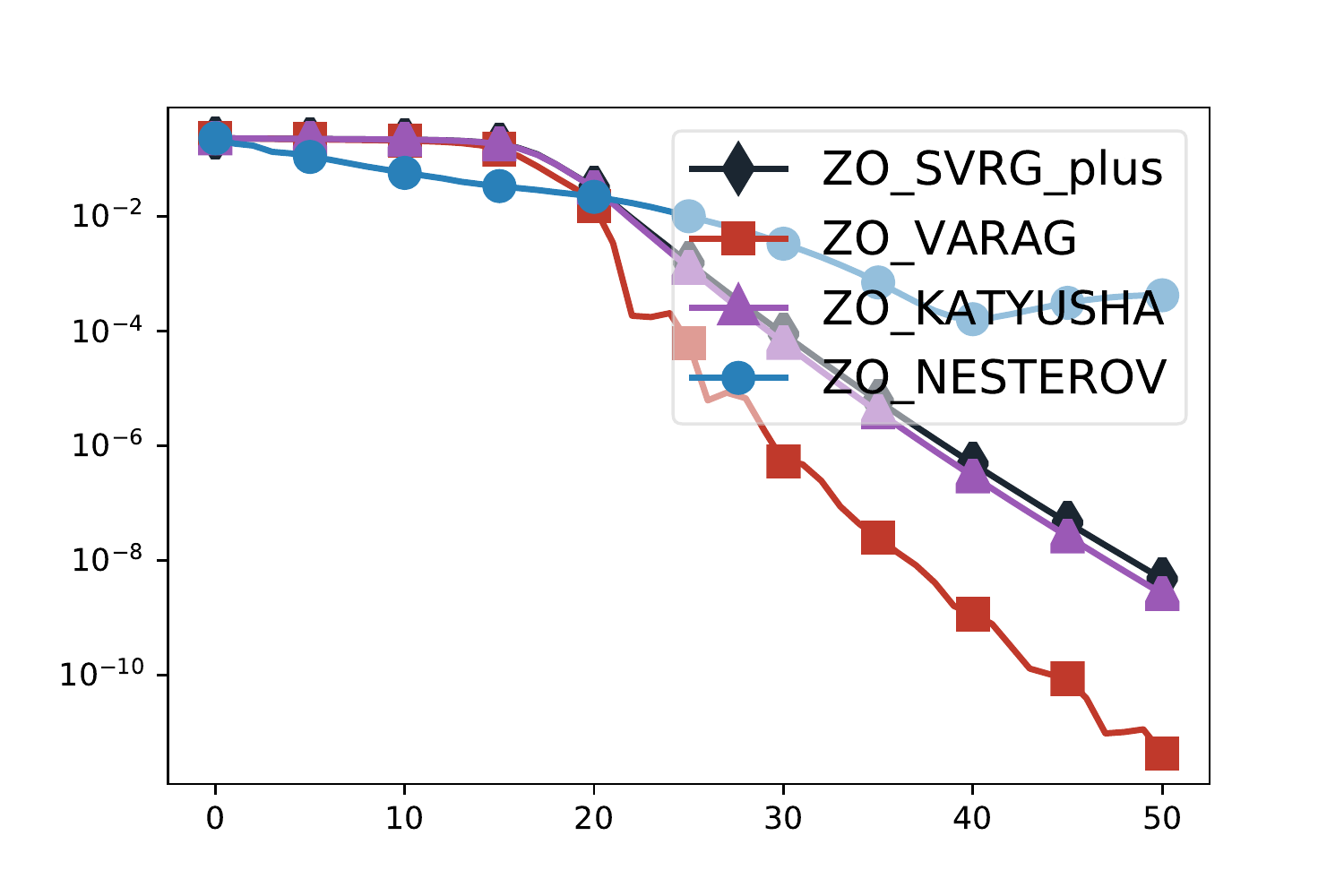}
             \\
             
        \rotatebox[origin=c]{90}{\scriptsize{Ridge, $\lambda = 1e^{-5}$}}&
             
              \includegraphics[width=0.475\linewidth,valign=c]{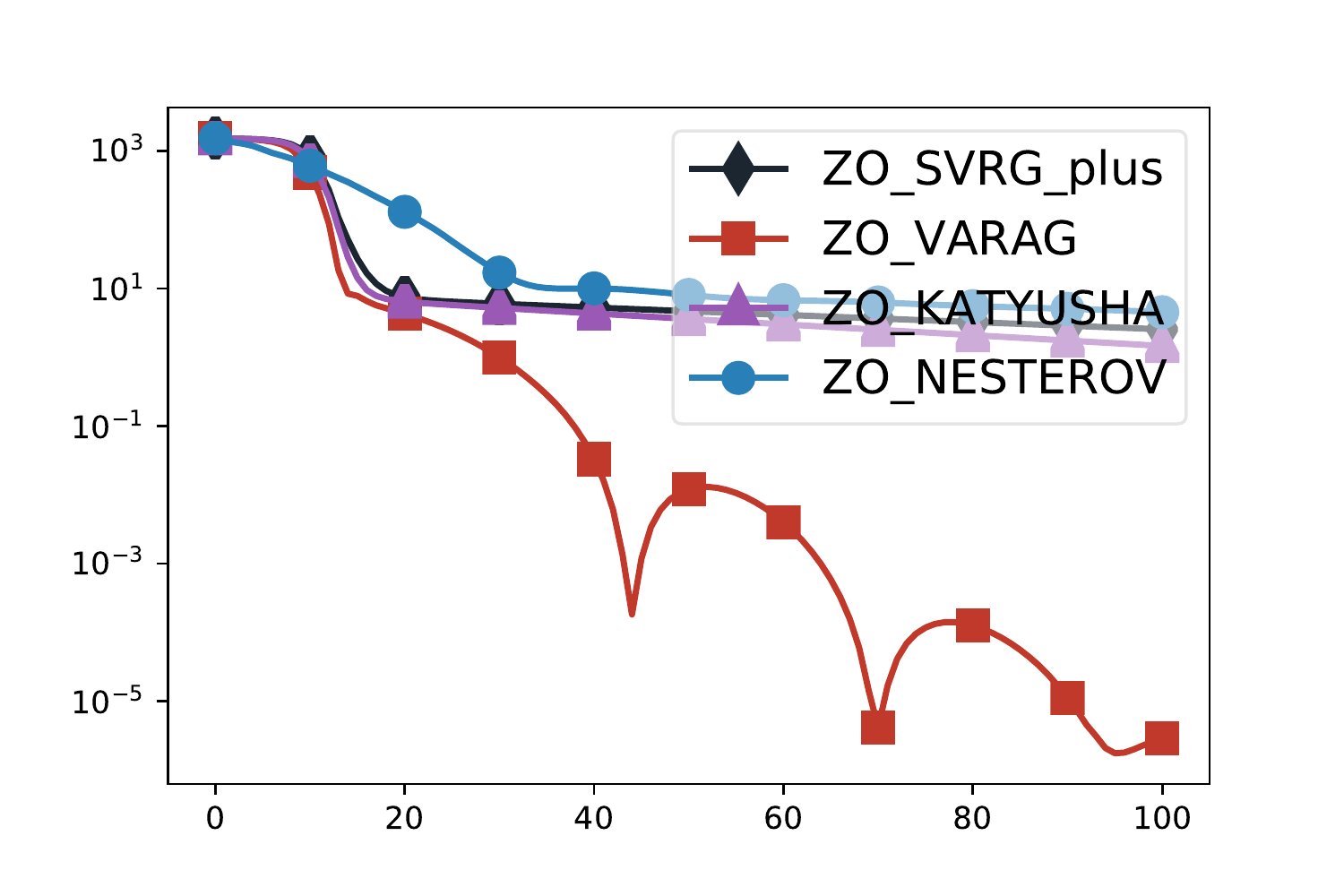}& \includegraphics[width=0.475\linewidth,valign=c]{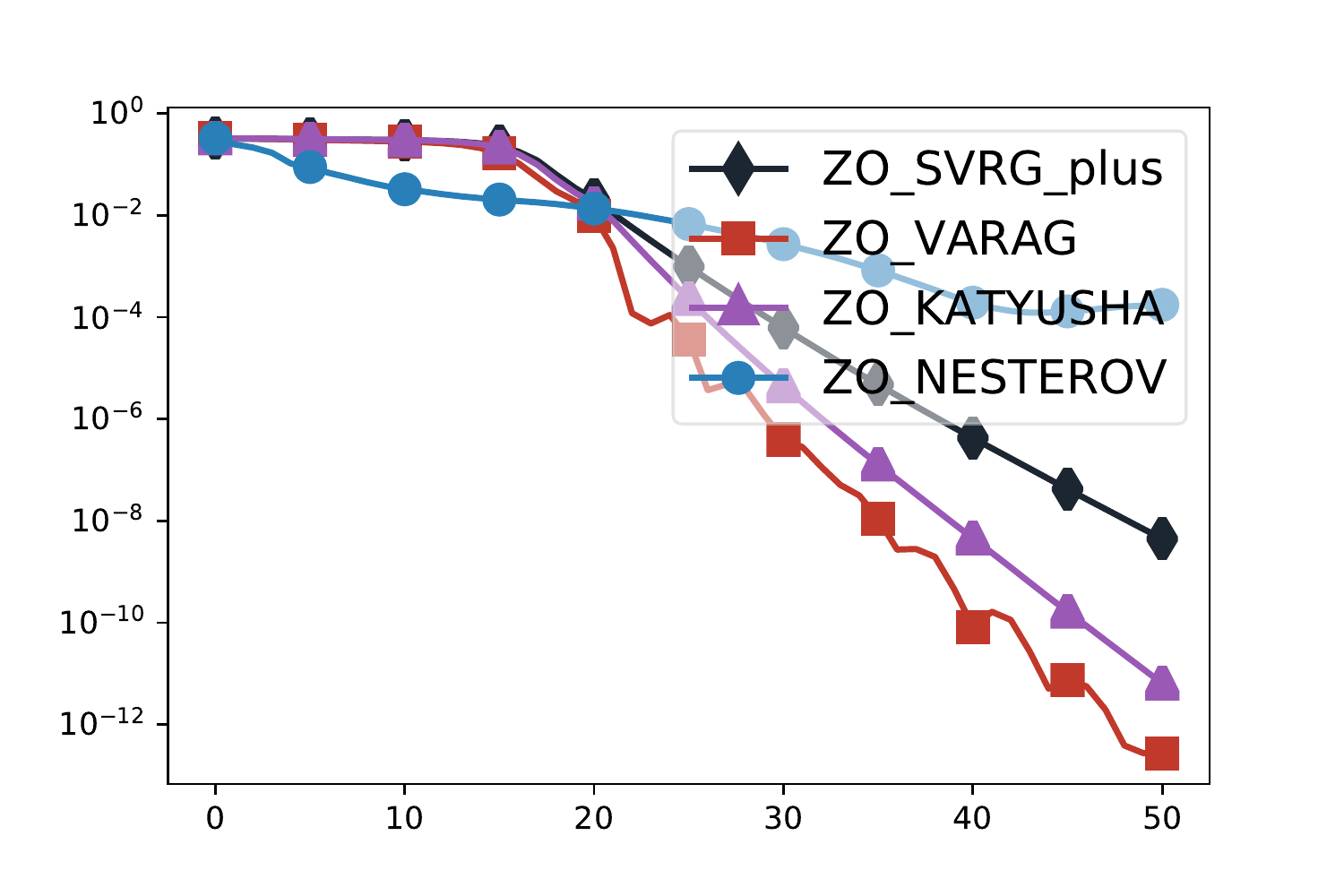}
            \\
        &   \scriptsize{epochs} & \scriptsize{epochs}
   \end{tabular}
          \caption{ \footnotesize{Loss $\log (f -f^*)$ over epochs. We show more results for various hyperparameters in the appendix.}}
          \label{fig:basic_test}
\end{figure}

In this section, we compare the empirical performance of ZO-Varag with ZO-SVRG-Coord-Rand in \cite{ji2019improved} and a simplified ZO-Katyusha which is the ZO-version of the simplified Katyusha algorithm in \cite{fanhua2017simKatyusha}, see Algorithm \ref{algorithm-Katyusha} in the appendix. We conduct experiments for both logistic regression and ridge regression~\footnote{Note that logistic regression with $\lambda = 0$ is not guaranteed to be coercive. However, this does not appear to be a problem in practice.} with and without $\ell_2$ regularization on the diabete dataset ($n=442, d=10$) from sklearn and the ijcnn1 dataset ($n=49990, d=22$) from LIBSVM. The choice of the hyperparameters chosen for each algorithm is detailed in the appendix.

Based on our theoretical analysis, we require $\frac{(d+4)n}{2} \le k \le (d+4)n$ iterations per epoch for ZO-Varag. However, we can lower the computational cost by using a batch update with $b$ samples per iteration and decrease the number of iterations to be $b$ times smaller, i.e. $\frac{(d+4)n}{b}$ for each epoch.

\subsection{Overall Performance}

We first compare the performance of ZO-Varag to the baselines for two different regularizers: $\lambda = \{ 0, 1e^{-5} \}$ (i.e. adding $\lambda \| x \|^2$ to the loss). In this part, we set the Katyusha momentum to a constant $p_0$ such that $p_0 + \alpha_0 = 1$. We then set the Katyusha momentum to $p_0 = 0.5$ ({see additional results for different values of $p_0$ in the appendix}).
The results shown in Figure~\ref{fig:basic_test} demonstrate that ZO-Varag does achieve an accelerated rate for all settings. While the zero-th order adaptation of simplified Katyusha does seem to be faster than the other two approaches, its performance is still close to the ZO-SVRG-Coord-Rand introduced in \cite{ji2019improved}. Finally, we note that Nesterov's ZO (ZO-Nesterov) method \cite{nesterov2017random} is a deterministic approach and it therefore has a much higher complexity per step. Indeed, while one step of ZO-Nesterov requires $2n$ queries, all the other methods require $2b$ queries. In order to establish a fair comparison, we plot the results of ZO-Nesterov with the nearest functional queries w.r.t. the results at the pivotal points for other stochastic methods.

\subsection{Options for Pivotal Point} 
\begin{figure}[t!]
 \centering          \begin{tabular}{c@{}c@{}c@{}c@{}}%c@{}}
        &   \scriptsize{ijcnn1, S = 50, b = 500, $\lambda = 0$} & \scriptsize{ijcnn1, S = 50, b = 500, $\lambda = 1e^{-5}$} \\%& \tiny{autoencoder} \\
        \rotatebox[origin=c]{90}{\scriptsize{Logistic}} &    
             \includegraphics[width=0.475\linewidth,valign=c]{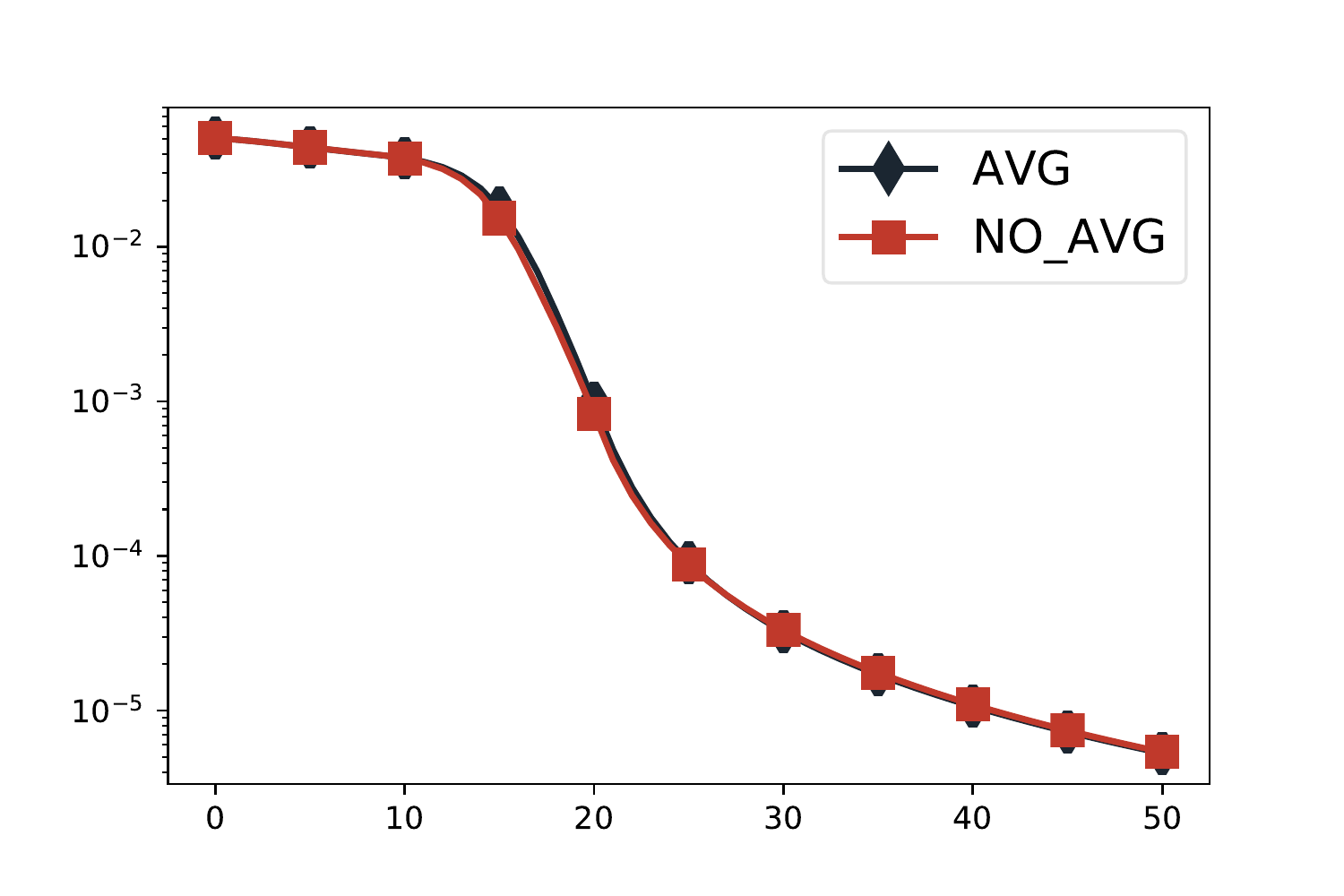}  &  \includegraphics[width=0.475\linewidth,valign=c]{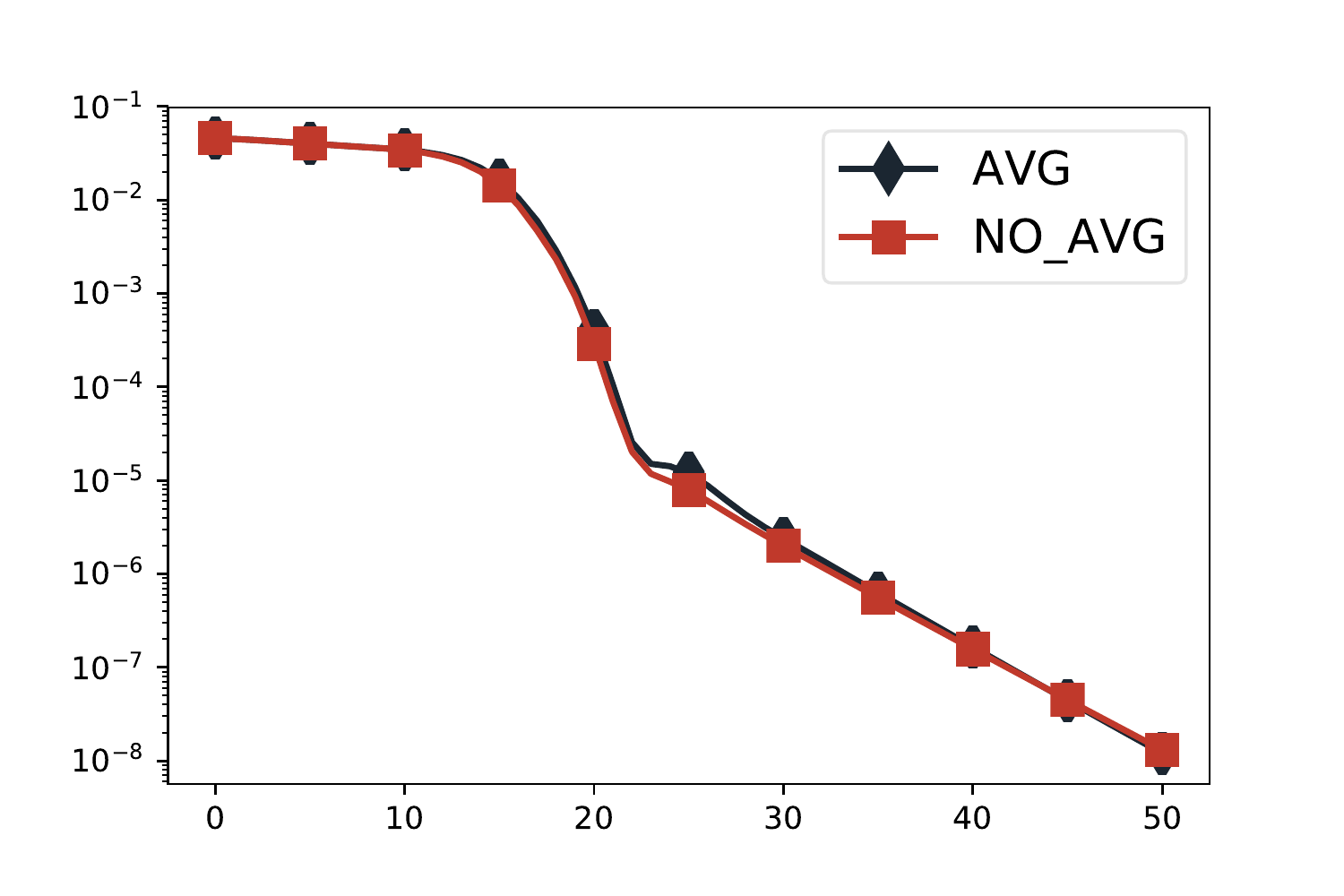}
             \\
             
        \rotatebox[origin=c]{90}{\scriptsize{Ridge}}&
             
              \includegraphics[width=0.475\linewidth,valign=c]{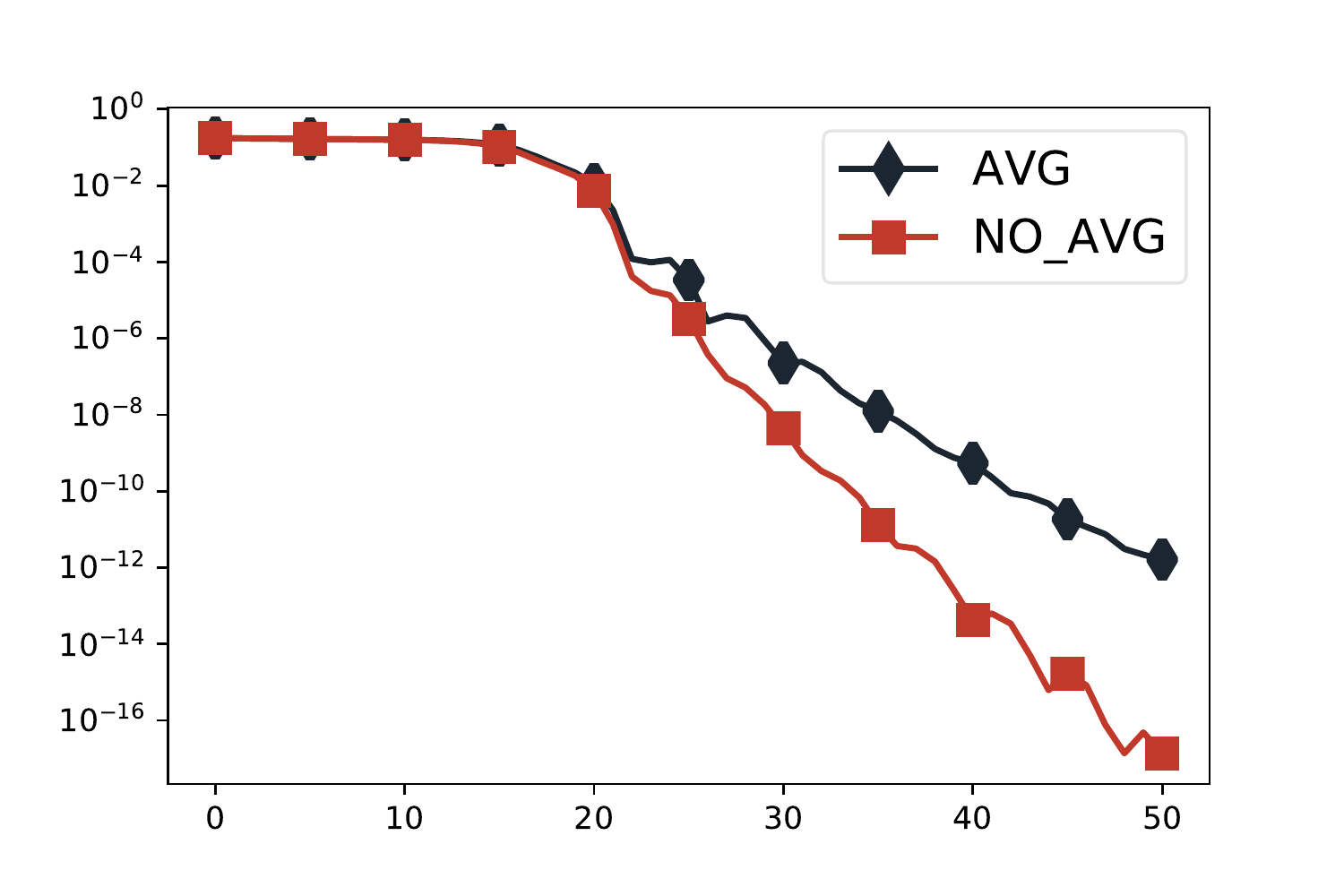}& \includegraphics[width=0.475\linewidth,valign=c]{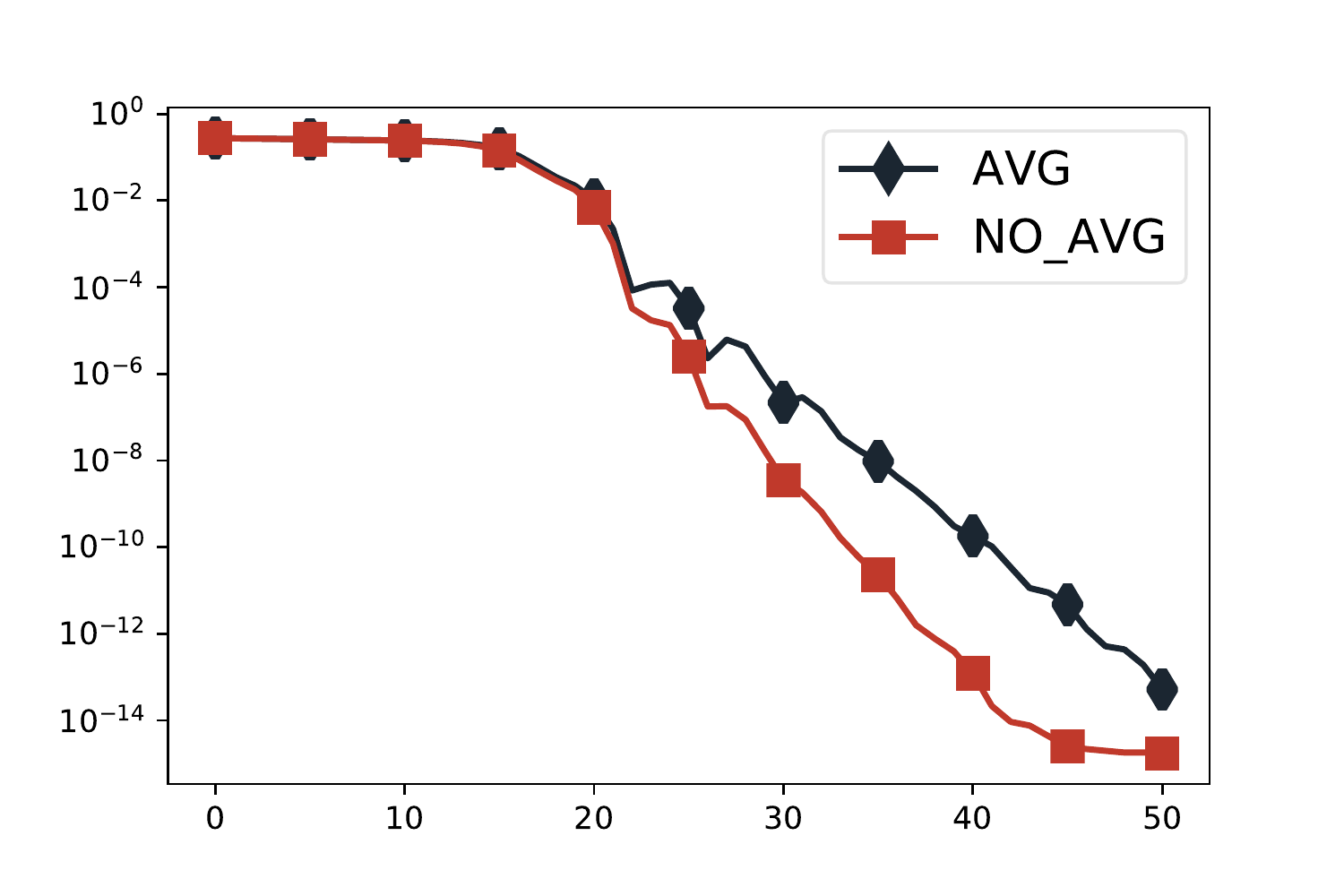}
            \\
        &   \scriptsize{epochs} & \scriptsize{epochs}
   \end{tabular}
          \caption{ \footnotesize{ZO-Varag, averaging vs. no-averaging}}
          \label{fig:pivot_selection_ijcnn1}
\end{figure}

As in~\cite{johnson2013accelerating}, we consider two options for specifying the pivot point: i) $\tilde{x} = \tilde{x}^{s-1}$ (as used in our analysis), or ii)  $\tilde{x} = \bar{x}^{s-1}$. The comparison of these two options is shown in Fig. \ref{fig:pivot_selection_ijcnn1} as well as in the appendix. Although option ii) does not have any theoretical guarantee, it empirically converges at a slightly faster rate than i) for logistic regression and significantly more for ridge regression.

\subsection{Effect of the Regularizer $\lambda$}
\begin{figure}[t!]
 \centering          \begin{tabular}{c@{}c@{}c@{}c@{}}%c@{}}
        &   \scriptsize{Diabetes, S = 300, b = 5} & \scriptsize{ijcnn1, S = 100, b = 500} \\%& \tiny{autoencoder} \\
        \rotatebox[origin=c]{90}{\scriptsize{Logistic}} &    
             \includegraphics[width=0.475\linewidth,valign=c]{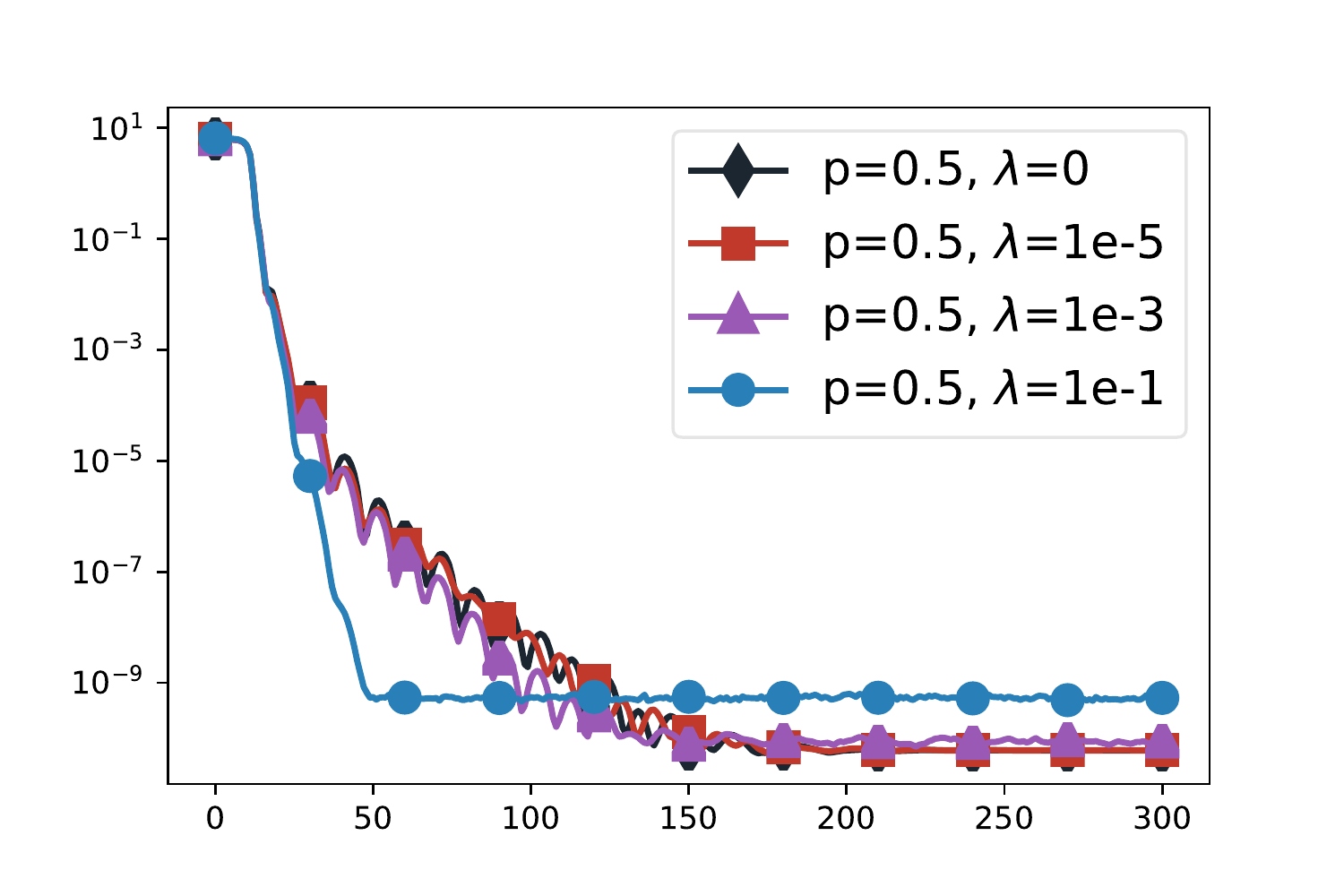}  &  \includegraphics[width=0.475\linewidth,valign=c]{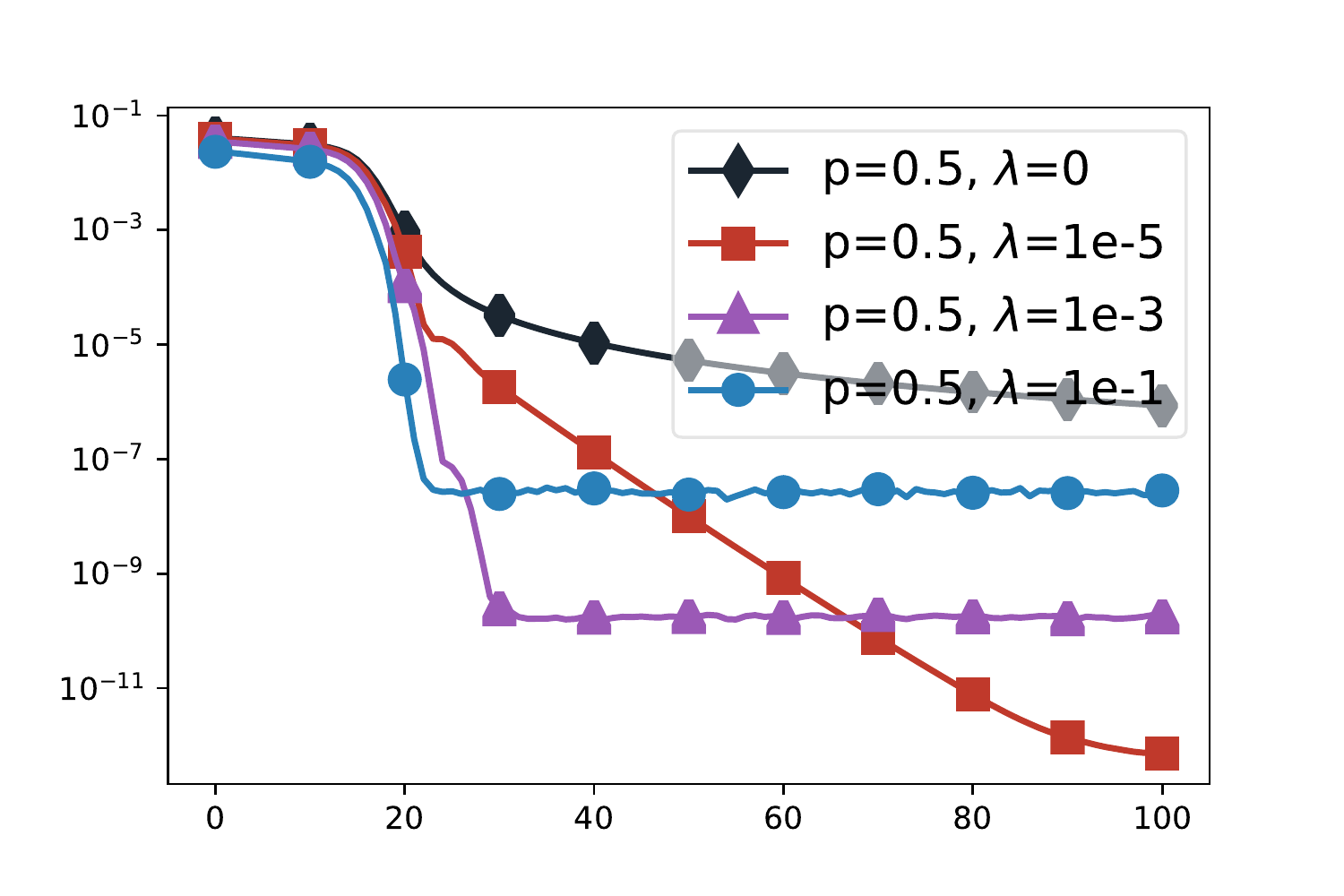}
             \\
             
        \rotatebox[origin=c]{90}{\scriptsize{Ridge}}&
             
              \includegraphics[width=0.475\linewidth,valign=c]{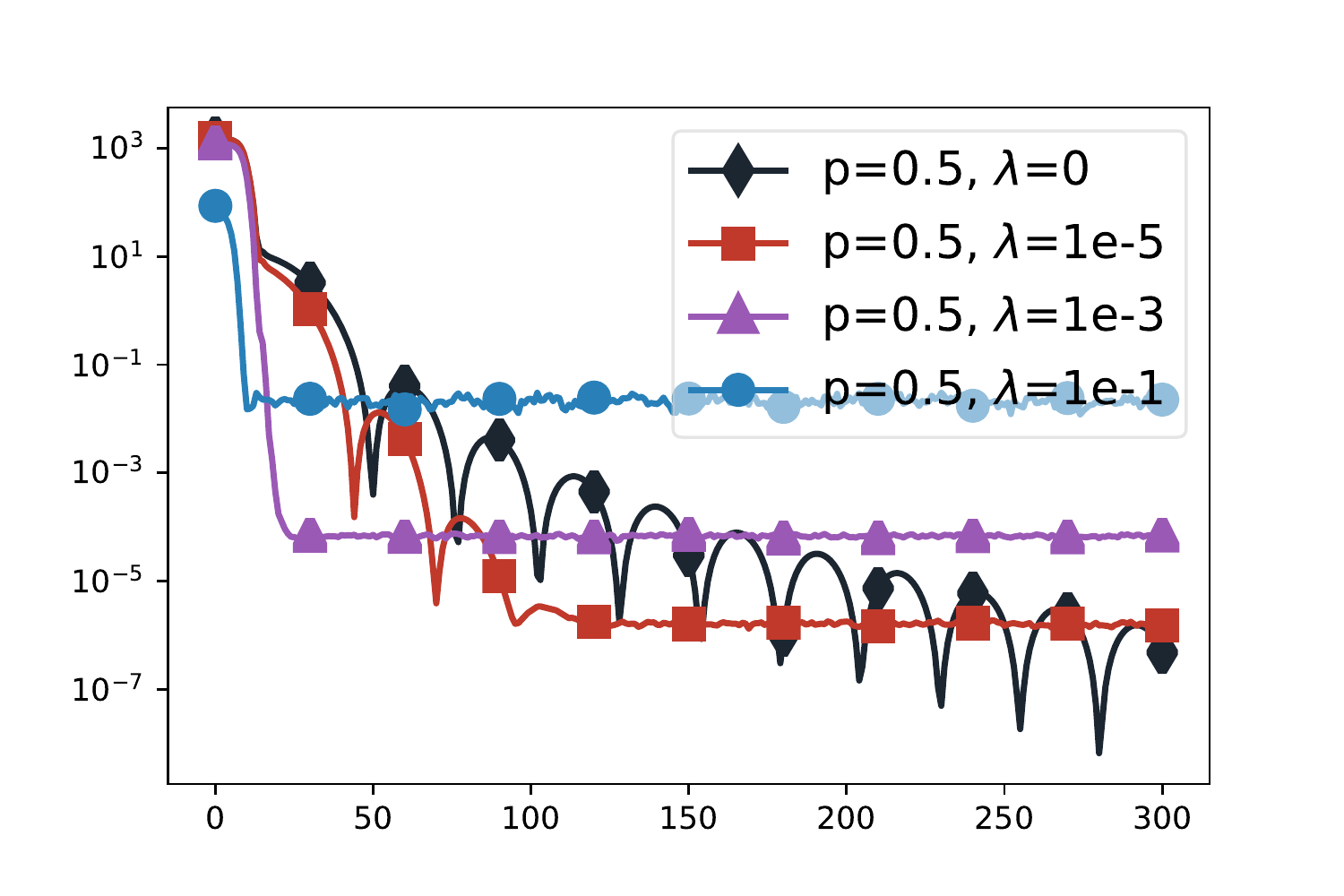}& \includegraphics[width=0.475\linewidth,valign=c]{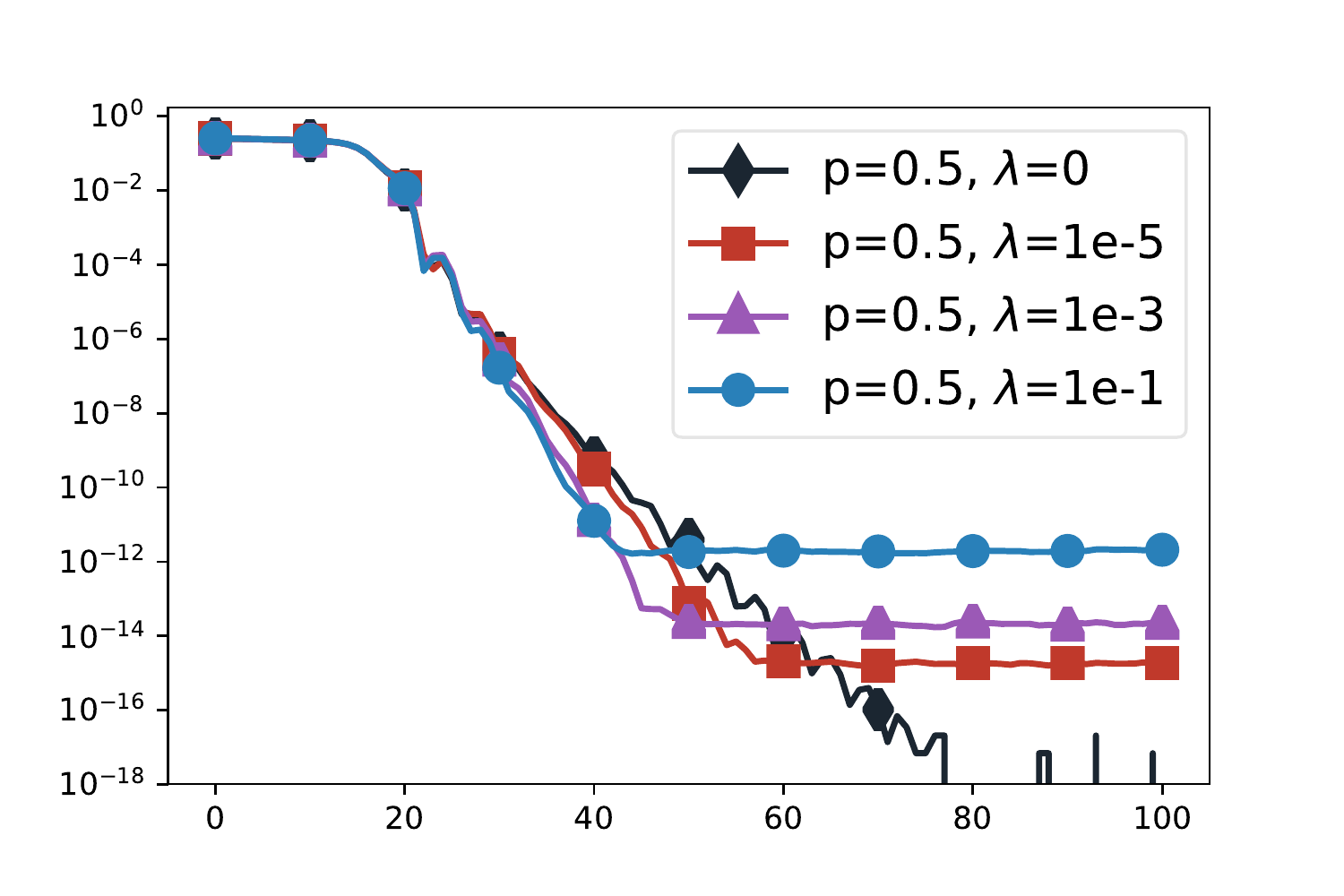}
            \\
        &   \scriptsize{epochs} & \scriptsize{epochs}
   \end{tabular}
          \caption{ \footnotesize{ZO-Varag, effect of varying the regularizer $\lambda$.}}
          \label{fig:vary_regularizer}
\end{figure}

We vary the strength of the regularizer to understand the behavior of the algorithms for objectives with stronger convexity constants and also to observe the convergence of the algorithm to the optimal solution. These results are shown in Figure \ref{fig:vary_regularizer} for increasing values of $\lambda$. ZO-Varag is faster in the initial stage but for all values of $\lambda$, we observe that it converges to a ball around the optimum. At first, one could expect that this might be due to the DFO errors $\varsigma_1, \varsigma_2$ shown in our convergence theorems, which would only appears to be a problem in high-accuracy regimes. However, the reason may come from other two additional sources: 1) the non-vanishing SVRG variance problem raised in the SARAH paper (see Fig. 1 in \cite{nguyen2017sarah} and our discussion of Fig.~\ref{fig:varying_step} in the appendix) and 2) the fact that stronger convexity constants increase the approximation error of Gaussian smoothing.

% !TEX root = main.tex

\section{Conclusion}

We presented a derivative-free algorithm that achieves the first accelerated rate of convergence for stochastic optimization of a convex finite-sum objective function. We also extended our analysis to the case of strongly-convex functions and included a variant of our algorithm for a coordinate-wise estimation of the gradient based on~\cite{ji2019improved}. Besides, a proximal variant of our approach could probably be derived, to deal with non-smooth problems as in the original Varag algorithm~\cite{lan2019unified}. Finally, we conducted experiments on several datasets demonstrating that our algorithm performs better than all existing non-accelerated DFO algorithms.

\clearpage %  flush all the figures up to that point.

\bibliography{main}
\bibliographystyle{icml2020}

% !TEX root = main.tex

\newpage
\onecolumn
\appendix
%\part*{Appendix}

%%%%%%%%%%%%%%%%%%%%%%%%%%%%%%%%%%%%%%%%%%%%%%%%%%%%%%%%%%%%
%%%%%%%%%%%%%%%%%%%%%%%%%%%%%%%%%%%%%%%%%%%%%%%%%%%%%%%%%%%%

{\huge Appendix 
}\\ \ \\
This supplementary material is organized as follows:
\begin{itemize}
    \item In Appendix~\ref{app:grad} we discuss some fundamental properties of zero-order gradient estimation techniques.
    \item In Appendix~\ref{app:proofs_gauss} we give detailed proofs for the results in Section~\ref{sec:algo_analysis}.
    \item In Appendix~\ref{app:proofs_coord} we give detailed proofs for the results in Section~\ref{sec:coordinate-wise}.
    \item In Appendix~\ref{app:exp} we give further details for the experiments in Section~\ref{sec:experiments} and provide additional empirical results.
\end{itemize}

We summarize some notation used in the paper. The vector $x \in \mathbb{R}^d$ is the variable to optimize and $n$ is the cardinality of the dataset. The Gaussian smoothed gradient estimator is denoted as $g_{\mu}$~( Eq.~\eqref{DFO-framework-gaussian-smoothing}) while the coordinate-wise gradient estimator is denoted as $g_{\nu}$~(Eq.~\eqref{DFO-framework-cord-finite-difference}). The variable $i$ indexes the data-point, and sometimes we specify it in the subscript of our estimators, e.g. $g_{\mu, i}, g_{\nu, i}$. The vector $u\in\R^d$ is the random direction generated from $\mathcal{N}(0, I_d)$ for Gaussian smoothing estimator ($I_d$ is the identity matrix in $\R^d$). $D_0, D_0'$ are some suboptimality measures for the initial states (see main paper). $s$ is the index of epoch, and we usually omit the superscript when discussing inner iterations inside each epoch, e.g. $x_t$ (which should be denoted as $x_t^s$ rigorously). The pivotal information always has a superscript ``$\sim$'', e.g. $\tilde{x}$ denotes the current pivotal point and $\tilde{g}$ denotes the pivotal gradient estimation at epoch $s$.

\section{Zero-Order Gradient estimation with variance reduction}
\label{app:grad}
We discuss here some fundamental properties of zero-order gradient estimation. We will use these properties heavily in Appendix~\ref{app:proofs_gauss} and Appendix~\ref{app:proofs_coord}.
\subsection{Gaussian smoothing approach}
We start by recalling some definitions presented in Section~\ref{sec:background} of the main paper.
Consider a differentiable function $f:\R^d\to\R$; its smoothed version $f_\mu:\R^d\to\R$ is defined pointwise as
$$f_\mu(x) = \frac{1}{(2\pi)^{\frac{d}{2}}}\int_{\R^d} f(x+\mu u)e^{-\frac{1}{2}\|u\|^2} du, \quad \forall x\in\R^d.$$

We list some useful properties of $f_\mu$ in the next lemma.\\ We recall that we say $f$ is $L$\textit{-smooth} if, $\forall x,y\in\R^d$, $\|\nabla f(x)-\nabla f(y)\|\le L\|x-y\|$.

\begin{mdframed}
\begin{lemma} The following properties hold :\label{Nes-DFO-property}
\\
(1) \  If $f$ is convex, then $f_{\mu}$ is also convex.\\
\\
(2) \  If $f$ is $L$-smooth, then $f_{\mu}$ is also $L$-smooth.\\
\\
(3) \ If $f$ is $\tau$-strongly convex, then $f_{\mu}$ is also $\tau$-strongly convex.\\
\\
(4) \ (Lemma 1 in \cite{nesterov2017random}) Let $u\sim\mathcal{N}(0,I_d)$, the standard normal distribution in $\R^d$. For $p \ge 2$, $d^{\frac{p}{2}}\le \mathbb{E}_u\big[\|u\|^p\big] \le (d+p)^{\frac{p}{2}}$.
\end{lemma}
\end{mdframed}

We give a proof of the third property, since it is not explicitly carried out in~\cite{nesterov2017random}.

\begin{proof}[Proof of Lemma~\ref{Nes-DFO-property}]
    $f$ is $\tau$-strongly convex if and only if (see e.g. Theorem 2.1.9 in \cite{nesterov2004introbook}) for all $x',y'\in\R^d$ and $\alpha\in[0,1],$
    $$f(\alpha x' + (1-\alpha) y')\le \alpha f(x') +(1-\alpha)f(y') - \frac{\alpha(1-\alpha)\tau}{2}\|x'-y'\|^2.$$
    We want to prove the same inequality for $f_\mu$. Let $x,y\in\R^d$ and $\alpha\in[0,1]$:
    \begin{align*}
        f_\mu(\alpha x + (1-\alpha) y)& = \frac{1}{(2\pi)^{\frac{d}{2}}}\int_{\R^d} f(\alpha x + (1-\alpha) y+\mu u)e^{-\frac{1}{2}\|u\|^2} du\\
        &=\frac{1}{(2\pi)^{\frac{d}{2}}}\int_{\R^d} f(\alpha (x+\mu u) + (1-\alpha) (y+\mu u))e^{-\frac{1}{2}\|u\|^2} du.
    \end{align*}
By picking $x' = x+\mu u$ and $y' = y+\mu u$ in the definition of strong convexity for $f$, by linearity of integration and noting that $x'-y' = x - y$, we get the desired result:
\begin{align*}
    f_\mu(\alpha x + (1-\alpha) y)&\le\frac{1}{(2\pi)^{\frac{d}{2}}}\int_{\R^d} \left(\alpha f(x+\mu u) + (1-\alpha) f(y+\mu u) - \frac{\alpha(1-\alpha)\tau}{2}\|x-y\|^2\right)e^{-\frac{1}{2}\|u\|^2} du\\
    &= \alpha f_\mu(x) + (1-\alpha) f_\mu(y) - \frac{\alpha(1-\alpha)\tau}{2}\|x-y\|^2,
\end{align*}
where in the last equality we used the fact that $\frac{1}{(2\pi)^{\frac{d}{2}}}\int_{\R^d} e^{-\frac{1}{2}\|u\|^2} du=1$.
\end{proof}

\paragraph{Properties of the smoothed gradient field.} Note that $f_\mu(x) = \E_u[f(x + \mu u)]$, with $u\sim\mathcal{N}(0,I_d)$, the standard normal distribution in $\R^d$. Hence, the gradient of $f_\mu$ can be written as
\begin{align}
\nabla f_\mu(x) = \E_u[g_{\mu}(x,u)],\quad g_{\mu}(x,u):= \frac{(f(x + \mu u) - f(x)) u}{\mu}.\nonumber
\end{align}

We list below some useful bounds from \cite{nesterov2017random}.
\begin{mdframed}
\begin{lemma}\label{Nes-DFO-error}
	If $f:\R^d\to\R$ is $L$-smooth, then\\
	(1) \ (Theorem 1 from \cite{nesterov2017random})
	\begin{align*}
	|f_{\mu}(x)-f(x)|\le \frac{\mu^2Ld}{2};
	\end{align*}
	and, if $f$ is convex (inequality (11) in \cite{nesterov2017random})
	\begin{align*}
	f_{\mu}(x) \ge f(x);
	\end{align*}
	(2) \  (Lemma 3 from \cite{nesterov2017random})
	\begin{align*}
	\|\nabla f_{\mu}(x)-\nabla f(x)\| \le \frac{\mu L (d+3)^{\frac{3}{2}}}{2};
	\end{align*}
	(3) \ (Lemma 4 from \cite{nesterov2017random})
	\begin{align*}
	\|\nabla f(x)\|^2 \le 2\|\nabla f_{\mu}(x)\|^2 + \frac{\mu^2 L^2 (d+6)^3}{2};
	\end{align*}
	(4) \ (Theorem 4 from \cite{nesterov2017random})
	\begin{align*}
	\mathbb{E}_{u}\big[\|g_{\mu}(x,u)\|^2\big] \le \frac{\mu^2 L^2 (d+6)^3}{2}+2(d+4)\|\nabla f(x)\|^2;
	\end{align*}
	(5) \ (Lemma 5 from \cite{nesterov2017random})
	\begin{align*}
	\mathbb{E}_{u}\big[\|g_{\mu}(x,u)\|^2\big] \le 3\mu^2 L^2 (d+4)^3+4(d+4)\|\nabla f_{\mu}(x)\|^2.
	\end{align*}
\end{lemma}
\end{mdframed}

\paragraph{Stochastic approximation of $\boldsymbol{\nabla f_\mu}$.} In the context of this paper, $f:=\frac{1}{n}\sum_i f_i$. A stochastic estimate of $g_{\mu}(x,u)$ using data-point $i$ can be then calculated as follows:
$$g_{\mu}(x,u,i) := \frac{f_{i}(x+\mu u)-f_{i}(x)}{\mu} u,\quad u\sim\mathcal{N}(0,I_d).$$

In the inner loop of Algorithm~\ref{algorithm-VARAG}, at iteration $t$, we use $g_{\mu}(x,u,i)$ to get a \textit{variance-reduced gradient estimate} of $\nabla f_{\mu}(\underline{x}_t)$:
$$G_t := g_{\mu}(\underline x_{t}, u_t, i_t) - g_{\mu}(\tilde{x}, u_t, i_t)  + \tilde{g},$$
where \textit{we dropped the epoch index} (i.e. $s$) for simplicity, as we will often do in the next pages. To study Algorithm~\ref{algorithm-VARAG}, it is necessary to get an estimate of $\mathbb{E}_{u_t,i_t| \mathcal{F}_{t-1}} \big[\| G_t - \mathbb{E}_{u_t,i_t| \mathcal{F}_{t-1}}[G_t]\|^2\big]$, where $\mathcal{F}_{t-1}$ denotes the past iterates in the current epoch. Such a bound is provided by Lemma \ref{VARAG-lemma-1} --- our main lemma for DFO variance reduction. Before proving this bound, we need a result from \cite{nesterov2017random}.
\begin{mdframed}
\begin{lemma}\label{Nes-DFO-lemma-directional-derivative}
(Theorem 3 from \cite{nesterov2017random}) Denote $f'(x,u)$ the directional derivative of $f$ at $x$ along direction $u$:
\begin{align*}
f'(x,u) = \lim_{\alpha \downarrow 0} \frac{1}{\alpha} \big[f(x+\alpha u)-f(x)\big].
\end{align*}
Let $g_{0}(x,u) := f'(x,u) \cdot u$. If f is differentiable at $x$, then $f'(x,u) = \langle \nabla f(x), u \rangle$, and $g_{0}(x,u) = \langle \nabla f(x), u \rangle \cdot u$. Also, the following inequality holds:
\begin{align*}
\mathbb{E}_{u}\big[\|g_{0}(x,u)\|^2\big] \le (d+4)\|\nabla f(x)\|^2.
\end{align*}
\end{lemma}
\end{mdframed}

Now, let us start the proof of Lemma~\ref{VARAG-lemma-1} in the main paper. \textit{Note that this lemma requires each $f_i$ to be $L$-smooth.}
\begin{proof}[Proof of Lemma \ref{VARAG-lemma-1}]
    According to $\mathbb{E}\big[||\xi - \mathbb{E}[\xi]||^2 \big] \le \mathbb{E}\big[||\xi||^2 \big]$, we have
    \begin{align*}
        \mathbb{E}_{u_t,i_t| \mathcal{F}_{t-1}} \big[\| G_t - \mathbb{E}_{u_t,i_t| \mathcal{F}_{t-1}}[G_t]\|^2\big] & = \mathbb{E}_{u_t,i_t| \mathcal{F}_{t-1}} \big[ \|g_{\mu}(\underline{x}_t, u_t, i_t) - g_{\mu}(\tilde{x}, u_t, i_t) - \nabla f_{\mu}(\underline{x}_t)+\nabla f_{\mu}(\tilde{x})\|^2 \big]\\
        & \le \mathbb{E}_{u_t,i_t| \mathcal{F}_{t-1}} \big[\| g_{\mu} \left( \underline{x}_t,u_t,i_t\right) - g_{\mu} \left(\tilde{x},u_t,i_t\right)\|^2 \big].
    \end{align*}
    The term $\| g_{\mu} \left( \underline{x}_t,u_t,i_t\right) - g_{\mu} \left(\tilde{x},u_t,i_t\right)\|^2$ can be bounded as follows:
	\begin{align*}
	& \| g_{\mu} \left( \underline{x}_t,u_t,i_t\right) - g_{\mu} \left(\tilde{x},u_t,i_t\right)\|^2\\
	= \  & \left\|\frac{f_{i_t}(\underline{x}_t+\mu u_t)-f_{i_t}(\underline{x}_t)}{\mu} \cdot u_t - \frac{f_i(\tilde{x}+\mu u_t)-f_{i_t}(\tilde{x})}{\mu} \cdot u_t\right\|^2\\
	= \  & \left\|\frac{f_{\mu,i_t}(\underline{x}_t+\mu u_t)-e_1-f_{\mu,i_t}(\underline{x}_t)+e_2}{\mu} \cdot u_t - \frac{f_{\mu,i}(\tilde{x}+\mu u_t)-e_3-f_{\mu,i_t}(\tilde{x})+e_4}{\mu} \cdot u_t\right\|^2,\\
	\end{align*}
	where $e_1, e_2, e_3, e_4$ denote some errors due to the \textit{small} difference between $f_{i_t}$ and $f_{\mu,i_t}$, which we will bound shortly. We proceed with some additional algebraic manipulations.
    \begin{align*}
	& \| g_{\mu} \left( \underline{x}_t,u_t,i_t\right) - g_{\mu} \left(\tilde{x},u_t,i_t\right)\|^2\\
	= \  & \bigg\|\frac{f_{\mu,i_t}(\underline{x}_t+\mu u_t)-f_{\mu,i_t}(\underline{x}_t)}{\mu} \cdot u_t - \frac{f_{\mu,i}(\tilde{x}+\mu u_t)-f_{\mu,i_t}(\tilde{x})}{\mu} \cdot u_t+\frac{e_1-e_2-e_3+e_4}{\mu}\cdot u_t\bigg\|^2\\
	= \  & \bigg\|\frac{f_{\mu,i_t}(\underline{x}_t+\mu u_t)-f_{\mu,i_t}(\underline{x}_t)-\mu \langle \nabla f_{\mu,i_t}(\underline{x}_t),u_t\rangle}{\mu} \cdot u_t - \frac{f_{\mu,i_t}(\tilde{x}+\mu u_t)-f_{\mu,i_t}(\tilde{x})-\mu \langle \nabla f_{\mu,i_t}(\tilde{x}),u_t\rangle}{\mu} \cdot u_t\\
	& + \langle \nabla f_{\mu,i_t}(\underline{x}_t)-\nabla f_{\mu,i_t}(\tilde{x}), u_t \rangle \cdot u_t++\frac{e_1-e_2-e_3+e_4}{\mu}\cdot u_t\bigg\|^2\\
	\le \  & 4\bigg( \left\|\frac{f_{\mu,i_t}(\underline{x}_t+\mu u_t)-f_{\mu,i_t}(\underline{x}_t)-\mu \langle \nabla f_{\mu,i_t}(\underline{x}_t),u_t\rangle}{\mu} \cdot u_t \right\|^2 + \left\| \frac{f_{\mu,i_t}(\tilde{x}+\mu u_t)-f_{\mu,i_t}(\tilde{x})-\mu \langle \nabla f_{\mu,i_t}(\tilde{x}),u_t\rangle}{\mu} \cdot u_t \right\|^2\\
	& + \left\|\langle \nabla f_{\mu,i_t}(\underline{x}_t)-\nabla f_{\mu,i_t}(\tilde{x}), u_t \rangle \cdot u_t\right\|^2 + \left\|\frac{e_1-e_2-e_3+e_4}{\mu}\cdot u_t\right\|^2 \bigg)\\
	\le \  & 4\bigg( \left(\frac{\mu}{2}L \|u_t\|^3\right)^2 + \left(\frac{\mu}{2}L \|u_t\|^3\right)^2 + \|\langle \nabla f_{\mu,i_t}(\underline{x}_t)-\nabla f_{\mu,i_t}(\tilde{x}), u_t \rangle \cdot u_t\|^2 +\left\|\frac{e_1-e_2-e_3+e_4}{\mu}\cdot u_t\right\|^2  \bigg)\\
	\le \  & 4\bigg( \frac{\mu^2}{2}L^2 \|u_t\|^6 + \|\langle \nabla f_{\mu,i_t}(\underline{x}_t)-\nabla f_{\mu,i_t}(\tilde{x}), u_t \rangle \cdot u_t\|^2 +4\mu^2L^2d^2\|u_t\|^2 \bigg),
	\end{align*}
	 where the second last inequality comes from the smoothness of $f_{\mu,i}$ and the last inequality is from (1) in Lemma \ref{Nes-DFO-error}. 
	Now, we define a new function $f_{\mu,i_t}^{e}(x) = f_{\mu,i_t}(x) - \langle \nabla f_{\mu,i_t}(\tilde{x}), x \rangle $ and 
	\begin{align*}
	\nabla f_{\mu,i_t}^e(x) = \nabla f_{\mu,i_t}(x)-\nabla f_{\mu,i_t}(\tilde{x}).
	\end{align*}
	Also, we define $g_{0,\mu,i_t}^e(x,u)$ as
	\begin{align}
	g_{0,\mu,i_t}^e(x,u) := {(f_{\mu,i_t}^{e})}^{'}(x,u) \cdot u.
	\end{align}
	Note that $g_{0,\mu,i_t}^e(x,u_t)$ is related to the second term in the inequality before:
	\begin{align*}
	\|g_{0,\mu,i_t}^e(x,u_t)\|^2 = \langle \nabla f_{\mu,i_t}^e(x),u_t\rangle^2 \cdot \|u_t\|^2 = \|\langle \nabla f_{\mu,i_t}(x)-\nabla f_{\mu,i_t}(\tilde{x}), u_t \rangle \cdot u_t\|^2.
	\end{align*}
	Next, we apply Lemma \ref{Nes-DFO-lemma-directional-derivative} to $f_{\mu,i_t}^e$:
	\begin{align*}
	\mathbb{E}_{u_t}\big[\|g_{0,\mu,i_t}^e(x,u_t)\|^2 \big] \le \  & (d+4)\|\nabla f_{\mu,i_t}^e(x)\|^2\\
	= \  & (d+4)\|\nabla f_{\mu,i_t}(x)-\nabla f_{\mu,i_t}(\tilde{x}) \|^2.
	\end{align*}
	Putting it all together, we obtain the desired bound:
	\begin{align}
	& \mathbb{E}_{u_t,i_t| \mathcal{F}_{t-1}} \big[\| G_t - \mathbb{E}_{u_t,i_t| \mathcal{F}_{t-1}}[G_t]\|^2\big] \notag\\
	\le \  & \mathbb{E}_{u_t,i_t| \mathcal{F}_{t-1}} \big[\| g_{\mu} \left( \underline{x}_t,u_t,i_t\right) - g_{\mu} \left(\tilde{x},u_t,i_t\right)\|^2\big] \notag\\
	\le \  & 2\mu^2 L^2 \mathbb{E}_{u_t| \mathcal{F}_{t-1}} \big[\|u_t\|^6\big] + 4(d+4)\mathbb{E}_{i_t| \mathcal{F}_{t-1}} \big[ \|\nabla f_{\mu,i_t}(\underline{x}_t)-\nabla f_{\mu,i_t}(\tilde{x}) \|^2 \big] + 16\mu^2L^2d^2\mathbb{E}_{u_t| \mathcal{F}_{t-1}}\big[\|u_t\|^2\big] \notag\\
	\le \  & 2\mu^2L^2(d+6)^3+ 16\mu^2L^2d^3+ 8(d+4)L\mathbb{E}_{i_t| \mathcal{F}_{t-1}}\big[f_{\mu,i_t}(\tilde{x})-f_{\mu,i_t}(\underline{x}_t) - \langle \nabla f_{\mu,i_t}(\underline{x}_t), \tilde{x} - \underline{x}_t \rangle \big] \notag\\
	\le \  & 18\mu^2L^2(d+6)^3 + 8(d+4)L \big[f_{\mu}(\tilde{x})-f_{\mu}(\underline{x}_t) - \langle \nabla f_{\mu}(\underline{x}_t), \tilde{x} - \underline{x}_t \rangle \big]. \notag
	\end{align}
	The second last inequality holds thanks to Theorem 2.1.5 in \cite{nesterov2004introbook}.
\end{proof}

\subsection{Coordinate-wise approach}

In Section~\ref{sec:coordinate-wise} we replace the Gaussian smoothing estimator of Eq.~\eqref{DFO-framework-gaussian-smoothing} with the coordinate-wise approach of~\cite{ji2019improved} for computing $G_t$ in Algorithm \ref{algorithm-VARAG}. That is, we set
$$G_t = g_{\nu}(\underline{x}_t,i_t) - g_{\nu}(\tilde{x},i_t) + g_{\nu}(\tilde{x}),$$
with, as we specified in Eq.~\eqref{DFO-framework-cord-finite-difference} of the main paper:
$$g_{\nu}(x,i) = \sum_{j=1}^{d}\frac{f_{i}(x+\nu \text{e}_{j})-f_{i}(x-\nu \text{e}_{j})}{2\nu} \text{e}_{j},$$
where $\text{e}_{j}$ is the unit vector with only one non-zero entry $1$ at its $j^{th}$ coordinate. Note that, $g_{\nu}$\textit{ is $d$ times more expensive to compute compared compared to }$g_{\mu}$, which we discussed before.\\
The following lemma gives an useful approximation error bound.

\begin{mdframed}
\begin{lemma} (Lemma 3~(Appendix D) from~\cite{ji2019improved}) \label{VARAG-cord-lemma-0}
Suppose each $f_i$ is $L$-smooth and that we use the coordinate-wise gradient estimation in Eq.~\eqref{DFO-framework-cord-finite-difference}. For any smoothing parameter $\nu > 0$ and any $x \in \mathbb{R}^d$, we have
\begin{align*}
\|g_{\nu}(x,i) - \nabla f_i(x) \|^2 \le L^2 d \nu^2.
\end{align*}
Also, if we define $g_{\nu}(x):=\frac{1}{n}\sum_{i=1}^n g_{\nu}(x,i)$, we clearly have $\|g_{\nu}(x) - \nabla f(x) \|^2 \le L^2 d \nu^2.$
\end{lemma}
\end{mdframed}

In the next lemma, we bound the variance of $G_t$. As the reader will soon notice, compared to Lemma~\ref{VARAG-lemma-1}, the proof in the coordinate-wise case is simpler and closely related to the standard variance reduction analysis \footnote{See e.g. Lemma 2.4 in \cite{allen2017katyusha}.}.

\begin{mdframed}
\begin{lemma} \label{VARAG-cord-lemma-1}
	When we use coordinate-wise gradient estimator Eq.~\eqref{DFO-framework-cord-finite-difference} for computing $G_t$, we can obtain a DFO variance reduction as follows:\\
	\begin{align}
    \mathbb{E}_{i_t} \big[\|G_t - \nabla f(\underline{x}_t)\|^2\big] \le 12L^2d\nu^2 + 8L \big[f(\tilde{x})-f(\underline{x}_t) - \langle \nabla f(\underline{x}_t), \tilde{x} - \underline{x}_t \rangle \big],
	\end{align}
	where $ G_t$ is defined as
	\begin{align*}
	G_t = g_{\nu}(\underline{x}_t,i_t) - g_{\nu}(\tilde{x},i_t) + g_{\nu}(\tilde{x})
	\end{align*} 
	and $g_{\nu}$ is the gradient estimator as defined by Eq.~\eqref{DFO-framework-cord-finite-difference}. Moreover, the expectation of the gradient estimation is
	\begin{align}
	\mathbb{E}_{i_t}\big[\delta_t\big] = g_{\nu}(\underline{x}_t) - \nabla f(\underline{x}_t) \neq 0,
	\end{align}
	which is different from $\mathbb{E}_{i_t| \mathcal{F}_{t-1}}\big[\delta_t\big] = g_{\nu}(\tilde{x}) - \nabla f_{\mu}(\tilde{x})$ in Lemma \ref{VARAG-lemma-1}.
\end{lemma}
\end{mdframed}

\begin{proof}
Note that $G_t - \nabla f(\underline{x}_t)$ can be decoupled as
\begin{align*}
G_t - \nabla f(\underline{x}_t) = \  & \nabla f_{i_t}(\underline{x}_t) - \nabla f_{i_t}(\tilde{x})  - \big(\nabla f(\underline{x}_t) - \nabla f(\tilde{x})\big) + g_{\nu}(\underline{x}_t, i_t) - \nabla f_{i_t}(\underline{x}_t) \\
& - g_{\nu}(\tilde{x}, i_t) + \nabla f_{i_t}(\tilde{x}) + g_{\nu}(\tilde{x}) - \nabla f(\tilde{x}).
\end{align*}
Therefore, we have
\begin{align*}
 \mathbb{E}_{i_t} \big[\|G_t - \nabla f(\underline{x}_t)\|^2\big] \le \  & 4 \mathbb{E}_{i_t} \biggl[ \|\nabla f_{i_t}(\underline{x}_t) - \nabla f_{i_t}(\tilde{x})  - \big(\nabla f(\underline{x}_t) - \nabla f(\tilde{x})\big)\|^2 + \| g_{\nu}(\underline{x}_t, i_t) - \nabla f_{i_t}(\underline{x}_t)\|^2\\
& + \|g_{\nu}(\tilde{x}, i_t) - \nabla f_{i_t}(\tilde{x})\|^2 + \| g_{\nu}(\tilde{x}) - \nabla f(\tilde{x})\|^2
\biggr]\\
\le \  & 4 \mathbb{E}_{i_t} \big[ \|\nabla f_{i_t}(\underline{x}_t) - \nabla f_{i_t}(\tilde{x})\|^2 + 3L^2 d \nu^2
\big]\\
\le \  & 8L \mathbb{E}_{i_t}\big[f_{i_t}(\tilde{x})-f_{i_t}(\underline{x}_t) - \langle \nabla f_{i_t}(\underline{x}_t), \tilde{x} - \underline{x}_t \rangle \big] + 12L^2 d \nu^2\\
= \  & 8L \big[f(\tilde{x})-f(\underline{x}_t) - \langle \nabla f(\underline{x}_t), \tilde{x} - \underline{x}_t \rangle \big] + 12L^2 d \nu^2.
\end{align*}
The second inequality holds because of $\mathbb{E}\big[\|\xi - \mathbb{E}[\xi]\|^2\big] \le \mathbb{E}\big[\|\xi\|^2\big]$ and thanks to Lemma \ref{VARAG-cord-lemma-0}. The last inequality holds thanks to Theorem 2.1.5 in \cite{nesterov2004introbook}.
\end{proof}

\section{Proofs for Section~\ref{sec:algo_analysis}}
\label{app:proofs_gauss}
The proofs of Theorem \ref{VARAG-theorem-1} and Theorem~\ref{VARAG-theorem-2} follow the same structure as in \cite{lan2019unified}, with some modifications due to the zero-order gradient estimation techniques, supported by the bounds in Appendix~\ref{app:grad}. We recall our basic assumption:
\begin{tcolorbox}
\textbf{(A1)} \ \ Each $f_i$ is convex, differentiable and $L$-smooth. Hence, also $f=\frac{1}{n}\sum_{i=1}^n f_i$ is convex and $L$-smooth.
\end{tcolorbox}

To make the notation compact, we define, again in analogy with~\cite{lan2019unified}:
\begin{align}\label{VARAG-def-x+}
x_{t-1}^{+} := \frac{1}{1+\tau \gamma_s}(x_{t-1}+\tau \gamma_s \underline{x}_t)
\end{align}
and 
\begin{align}\label{VARAG-def-lf}
    l_{f}(z,x) := f(z) + \langle \nabla f(z), x-z \rangle.
\end{align}

Using the definition of $\bar{x}_t$ and $x_t$ in Algorithm~\ref{algorithm-VARAG}, we have:
\begin{align}\label{VARAG-eq-1}
\bar{x}_t - \underline{x}_t = \alpha_{s}(x_t - x_{t-1}^+).
\end{align}

The first result is simply an adaptation of Lemma 5 in~\cite{lan2019unified} for the non-regularized Euclidean case. Hence, it does not require a proof.
\begin{mdframed}
\begin{lemma}
\label{VARAG-lemma-2}
Assume \textbf{(A1)}. For any $x \in \R^d$, we have
\begin{multline*}
\gamma_s [l_{f_{\mu}}(\underline x_t, x_t) - l_{f_{\mu}}(\underline x_t, x)] \le \\ \frac{\tau\gamma_s}{2}\|\underline{x_t} - x\|^2 + \frac{1}{2}\|x_{t-1} - x\|^2 - \frac{1+\tau \gamma_s}{2}\|x_t - x\|^2  - \frac{1+\tau \gamma_s}{2} \|x_t - x^+_{t-1}\|^2 - \gamma_s \langle \delta_t, x_t - x \rangle,
%- \gamma_t \langle \delta_t, x^+_{t-1} - x \rangle
\end{multline*}
which can be rewritten as
\begin{multline*}
\gamma_s \langle \nabla f_{\mu}(\underline{x}_t), x_t - x \rangle \le \frac{\tau\gamma_s}{2}\|\underline{x_t} - x\|^2 + \frac{1}{2}\|x_{t-1} - x\|^2 - \frac{1+\tau \gamma_s}{2}\|x_t - x\|^2 - \frac{1+\tau \gamma_s}{2} \|x_t - x^+_{t-1}\|^2 - \gamma_s \langle \delta_t, x_t - x \rangle.
%- \gamma_t \langle \delta_t, x^+_{t-1} - x \rangle
\end{multline*}
\end{lemma}
\end{mdframed}

The following lemma bounds the progress made at each inner iteration, and is similar to Lemma 6 in~\cite{lan2019unified}, but with some additional error terms coming from the zero-order estimation error for the gradients.
\begin{mdframed}
\begin{lemma}
\label{VARAG-lemma-3}
Assume \textbf{(A1)}. Assume that $\alpha_{s} \in [0,1]$, $p_s \in [0,1]$ and $\gamma_{s} > 0$ satisfy
\begin{align}\label{VARAG-lemma-3-assump-1}
1+\tau \gamma_{s} - L\alpha_{s}\gamma_{s} > 0,
\end{align}
\begin{align}\label{VARAG-lemma-3-assump-2}
p_s - \frac{4(d+4)L\alpha_{s}\gamma_{s}}{1+\tau \gamma_{s} - L\alpha_{s}\gamma_{s}} \ge 0.
\end{align}
Conditioned on past events $\mathcal{F}_{t-1}$ and taking the expectation of $u_t, i_t$, we have
\begin{align}
& \mathbb{E}_{u_t, i_t| \mathcal{F}_{t-1}}\bigg[\frac{\gamma_{s}}{\alpha_{s}}\big[f_{\mu}(\bar{x}_t)-f_{\mu}(x)\big] + \frac{(1+\tau \gamma_{s})}{2}\|x_t-x\|^2\bigg] \notag\\
\le \  & \frac{\gamma_{s}}{\alpha_{s}}(1-\alpha_{s}-p_s)\big[f_{\mu}(\bar{x}_{t-1})-f_{\mu}(x)\big]+\frac{\gamma_{s}p_s}{\alpha_{s}} \big[f_{\mu}(\tilde{x})-f_{\mu}(x)\big] + \frac{1}{2}\|x_{t-1}-x\|^2 + \frac{\gamma_{s}}{\alpha_{s}} \cdot \frac{9\alpha_{s}\gamma_{s}\mu^2L^2(d+6)^3}{1+\tau\gamma_{s}-L\alpha_{s}\gamma_{s}}\notag\\
& - \frac{\gamma_{s}}{\alpha_{s}} \cdot \alpha_{s} \mathbb{E}_{u_t, i_t| \mathcal{F}_{t-1}}\big[\langle \tilde{g}-\nabla f_{\mu}(\tilde{x}), x_t - x \rangle \big] \label{VARAG-lemma-3-result}
\end{align}
for any $x \in \mathbb{R}^d$.
\end{lemma}
\end{mdframed}

\begin{remark}
The second term in Eq.~\eqref{VARAG-lemma-3-assump-2} has a dependency on $(d+4)$, due to the Gaussian smoothing distortion.
\end{remark}

\begin{proof}[Proof of Lemma \ref{VARAG-lemma-3}]
By the $L$-smoothness of $f_{\mu}$ (from Lemma \ref{Nes-DFO-property}),
\begin{align*}
f_{\mu}(\bar{x}_t) \le \ & f_{\mu}(\underline{x}_t) + \langle \nabla f_{\mu}(\underline{x}_t), \bar{x}_t - \underline{x}_t \rangle + \frac{L}{2}\|\bar{x}_t-\underline{x}_t\|^2\\
= \ & (1-\alpha_{s}-p_s)\big[f_{\mu}(\underline{x}_t)+\langle \nabla f_{\mu}(\underline{x}_t), \bar{x}_{t-1} - \underline{x}_t \rangle\big] + \alpha_{s}\big[f_{\mu}(\underline{x}_t)+\langle \nabla f_{\mu}(\underline{x}_t), x_t - \underline{x}_t \rangle\big]\\
& + p_s\big[f_{\mu}(\underline{x}_t)+\langle \nabla f_{\mu}(\underline{x}_t), \tilde{x} - \underline{x}_t \rangle\big] + \frac{L\alpha_{s}^2}{2}\|x_t - x_{t-1}^+\|^2.
\end{align*}
The equality above holds because of the update rule of $\bar{x}_t$ in Algorithm \ref{algorithm-VARAG} and the Eq.~\eqref{VARAG-eq-1}. Next, applying Lemma \ref{VARAG-lemma-2} for the inequality above, we have
\begin{align*}
f_{\mu}(\bar{x}_t)
\le \ & (1-\alpha_{s}-p_s)\big[f_{\mu}(\underline{x}_t)+\langle \nabla f_{\mu}(\underline{x}_t), \bar{x}_{t-1} - \underline{x}_t \rangle\big] + \alpha_{s}\big[f_{\mu}(\underline{x}_t)+\langle \nabla f_{\mu}(\underline{x}_t), x - \underline{x}_t \rangle\big] \notag\\
& \ +\alpha_{s}\big[\frac{\tau}{2}\|\underline{x_t} - x\|^2 + \frac{1}{2\gamma_{s}}\|x_{t-1} - x\|^2 - \frac{1+\tau \gamma_s}{2\gamma_{s}}\|x_t - x\|^2 - \frac{1+\tau \gamma_s}{2\gamma_{s}} \|x_t - x^+_{t-1}\|^2 - \langle \delta_t, x_t - x \rangle\big] \notag\\
& \ + p_s\big[f_{\mu}(\underline{x}_t)+\langle \nabla f_{\mu}(\underline{x}_t), \tilde{x} - \underline{x}_t \rangle\big] + \frac{L\alpha_{s}^2}{2}\|x_t - x_{t-1}^+\|^2 \notag\\
\le \ & (1-\alpha_{s}-p_s)\big[f_{\mu}(\bar{x}_{t-1}) - \frac{\tau}{2}\|\bar{x}_{t-1}-\underline{x}_t\|^2\big] + \alpha_{s}\big[f_{\mu}(x) - \frac{\tau}{2}\|x-\underline{x}_t\|^2\big] \notag\\
& \ +\alpha_{s}\big[\frac{\tau}{2}\|\underline{x_t} - x\|^2 + \frac{1}{2\gamma_{s}}\|x_{t-1} - x\|^2 - \frac{1+\tau \gamma_s}{2\gamma_{s}}\|x_t - x\|^2 \big] \notag\\
& + p_s\big[f_{\mu}(\underline{x}_t)+\langle \nabla f_{\mu}(\underline{x}_t), \tilde{x} - \underline{x}_t \rangle\big] - \frac{\alpha_{s}}{2\gamma_{s}}(1+\tau \gamma_{s} - L\alpha_{s} \gamma_{s})\|x_t - x_{t-1}^+\|^2 -\alpha_{s}\langle \delta_t, x_t - x \rangle \notag\\
= \ & (1-\alpha_{s}-p_s)\big[f_{\mu}(\bar{x}_{t-1}) - \frac{\tau}{2}\|\bar{x}_{t-1}-\underline{x}_t\|^2\big] +\alpha_{s}\big[f_{\mu}(x) + \frac{1}{2\gamma_{s}}\|x_{t-1} - x\|^2 - \frac{1+\tau \gamma_s}{2\gamma_{s}}\|x_t - x\|^2 \big] \notag\\
& + p_s\big[f_{\mu}(\underline{x}_t)+\langle \nabla f_{\mu}(\underline{x}_t), \tilde{x} - \underline{x}_t \rangle\big] - \frac{\alpha_{s}}{2\gamma_{s}}(1+\tau \gamma_{s} - L\alpha_{s} \gamma_{s})\|x_t - x_{t-1}^+\|^2 \notag\\
& -\alpha_{s}\langle \delta_t - \tilde{g} + \nabla f_{\mu}(\tilde{x}), x_t - x_{t-1}^{+} \rangle -\alpha_{s}\langle \tilde{g} - \nabla f_{\mu}(\tilde{x}), x_t - x_{t-1}^{+} \rangle -\alpha_{s}\langle \delta_t, x_{t-1}^{+} - x \rangle  \notag\\
= \ & (1-\alpha_{s}-p_s)\big[f_{\mu}(\bar{x}_{t-1}) - \frac{\tau}{2}\|\bar{x}_{t-1}-\underline{x}_t\|^2\big] +\alpha_{s}\big[f_{\mu}(x) + \frac{1}{2\gamma_{s}}\|x_{t-1} - x\|^2 - \frac{1+\tau \gamma_s}{2\gamma_{s}}\|x_t - x\|^2 \big] \notag\\
& + p_s\big[f_{\mu}(\underline{x}_t)+\langle \nabla f_{\mu}(\underline{x}_t), \tilde{x} - \underline{x}_t \rangle\big] - \frac{\alpha_{s}}{2\gamma_{s}}(1+\tau \gamma_{s} - L\alpha_{s} \gamma_{s})\|x_t - x_{t-1}^+\|^2 \notag\\
& -\alpha_{s}\langle g_{\mu}(\underline{x}_t,u_t,i_t) - g_{\mu}(\tilde{x},u_t, i_t) - \nabla f_{\mu}(\underline{x}_t)+\nabla f_{\mu}(\tilde{x}), x_t - x_{t-1}^{+} \rangle \notag\\
& -\alpha_{s}\langle \tilde{g} - \nabla f_{\mu}(\tilde{x}), x_t - x_{t-1}^{+} \rangle -\alpha_{s}\langle \delta_t, x_{t-1}^{+} - x \rangle.\\
\end{align*}
The second inequality holds thanks to the strong convexity~(with $\tau\ge0$) of $f_{\mu}$ (see again Lemma \ref{Nes-DFO-property}) and the last equality comes from the definition
$$\delta_t = G_t - \nabla f_{\mu}(\underline{x}_t).$$

Next, note that for any $a>0, b\in\R$ and $u,v\in \R^d$, it holds that $b \langle u,v \rangle - \frac{a}{2}\|v\|^2 \le \frac{b^2}{2a}\|u\|^2$. If we set $a = \frac{\alpha_{s}}{\gamma_{s}}(1+\tau \gamma_{s} - L\alpha_{s} \gamma_{s})$ and $b = -\alpha_{s}$ (requiring $1+\tau \gamma_{s} - L\alpha_{s} \gamma_{s} > 0$), we get 
\begin{align}
f_{\mu}(\bar{x}_t)\le \ & (1-\alpha_{s}-p_s)\big[f_{\mu}(\bar{x}_{t-1}) - \frac{\tau}{2}\|\bar{x}_{t-1}-\underline{x}_t\|^2\big] +\alpha_{s}\big[f_{\mu}(x) + \frac{1}{2\gamma_{s}}\|x_{t-1} - x\|^2 - \frac{1+\tau \gamma_s}{2\gamma_{s}}\|x_t - x\|^2 \big] \notag\\
& + p_s\big[f_{\mu}(\underline{x}_t)+\langle \nabla f_{\mu}(\underline{x}_t), \tilde{x} - \underline{x}_t \rangle\big] +\frac{\alpha_{s}\gamma_{s}}{2(1+\tau \gamma_{s}-L\alpha_{s} \gamma_{s})}\|g_{\mu}(\underline{x}_t,u_t,i_t) - g_{\mu}(\tilde{x},u_t,i_t) - \nabla f_{\mu}(\underline{x}_t)+\nabla f_{\mu}(\tilde{x})\|^2 \notag\\
& -\alpha_{s}\langle \tilde{g} - \nabla f_{\mu}(\tilde{x}), x_t - x_{t-1}^{+} \rangle -\alpha_{s}\langle \delta_t, x_{t-1}^{+} - x \rangle \notag\\
= \ & (1-\alpha_{s}-p_s)\big[f_{\mu}(\bar{x}_{t-1}) - \frac{\tau}{2}\|\bar{x}_{t-1}-\underline{x}_t\|^2\big] +\alpha_{s}\big[f_{\mu}(x) + \frac{1}{2\gamma_{s}}\|x_{t-1} - x\|^2 - \frac{1+\tau \gamma_s}{2\gamma_{s}}\|x_t - x\|^2 \big] \notag\\
& + p_s\big[f_{\mu}(\underline{x}_t)+\langle \nabla f_{\mu}(\underline{x}_t), \tilde{x} - \underline{x}_t \rangle\big] +\frac{\alpha_{s}\gamma_{s}}{2(1+\tau \gamma_{s}-L\alpha_{s} \gamma_{s})}\|G_t - \mathbb{E}_{u_t,i_t}[G_t]\|^2 \notag\\
& -\alpha_{s}\langle \tilde{g} - \nabla f_{\mu}(\tilde{x}), x_t - x_{t-1}^{+} \rangle -\alpha_{s}\langle \delta_t, x_{t-1}^{+} - x \rangle. \label{VARAG-smooth-ineq-1}
\end{align}

Taking the expectation w.r.t. $u_t, i_t$ conditional on past iterates and applying Lemma \ref{VARAG-lemma-1},
\begin{align}
& p_s \big[f_{\mu}(\underline{x}_t)+\langle \nabla f_{\mu}(\underline{x}_t), \tilde{x} - \underline{x}_t \rangle\big] +\frac{\alpha_{s}\gamma_{s}}{2(1+\tau \gamma_{s}-L\alpha_{s} \gamma_{s})}\mathbb{E}_{u_t, i_t| \mathcal{F}_{t-1}}\big[\|G_t - \mathbb{E}_{u_t,i_t}[G_t]\|^2\big] \notag\\
& -\alpha_{s}\mathbb{E}_{u_t, i_t| \mathcal{F}_{t-1}} \big[\langle \tilde{g} - \nabla f_{\mu}(\tilde{x}), x_t - x_{t-1}^{+} \rangle\big] -\alpha_{s}\mathbb{E}_{u_t, i_t| \mathcal{F}_{t-1}} \big[\langle \delta_t, x_{t-1}^{+} - x \rangle\big] \notag\\
\le \ & p_s \big[f_{\mu}(\underline{x}_t)+\langle \nabla f_{\mu}(\underline{x}_t), \tilde{x} - \underline{x}_t \rangle\big] + \frac{9\alpha_{s}\gamma_{s}\cdot \mu^2L^2(d+6)^3}{1+\tau \gamma_{s} - L\alpha_{s}\gamma_{s}} +\frac{4\alpha_{s}\gamma_{s}(d+4)L}{1+\tau \gamma_{s}-L\alpha_{s} \gamma_{s}}\big[f_{\mu}(\tilde{x})-f_{\mu}(\underline{x}_t) - \langle \nabla f_{\mu}(\underline{x}_t), \tilde{x} - \underline{x}_t \rangle \big] \notag\\
& -\alpha_{s}\mathbb{E}_{u_t, i_t| \mathcal{F}_{t-1}} \big[\langle \tilde{g} - \nabla f_{\mu}(\tilde{x}), x_t - x \rangle\big] \notag\\
= \ & \big(p_s-\frac{4\alpha_{s}\gamma_{s}(d+4)L}{1+\tau \gamma_{s}-L\alpha_{s} \gamma_{s}}\big) \big[f_{\mu}(\underline{x}_t)+\langle \nabla f_{\mu}(\underline{x}_t), \tilde{x} - \underline{x}_t \rangle\big] + \frac{9\alpha_{s}\gamma_{s}\cdot \mu^2L^2(d+6)^3}{1+\tau \gamma_{s} - L\alpha_{s}\gamma_{s}} +\frac{4\alpha_{s}\gamma_{s}(d+4)L}{1+\tau \gamma_{s}-L\alpha_{s} \gamma_{s}}f_{\mu}(\tilde{x})\notag\\ & -\alpha_{s}\mathbb{E}_{u_t, i_t| \mathcal{F}_{t-1}} \big[\langle \tilde{g} - \nabla f_{\mu}(\tilde{x}), x_t - x \rangle\big] \notag\\
\le \ & \big(p_s-\frac{4\alpha_{s}\gamma_{s}(d+4)L}{1+\tau \gamma_{s}-L\alpha_{s} \gamma_{s}}\big) \big[f_{\mu}(\tilde{x}) - \frac{\tau}{2}\|\tilde{x} - \underline{x}_t\|^2\big] + \frac{9\alpha_{s}\gamma_{s}\cdot \mu^2L^2(d+6)^3}{1+\tau \gamma_{s} - L\alpha_{s}\gamma_{s}} +\frac{4\alpha_{s}\gamma_{s}(d+4)L}{1+\tau \gamma_{s}-L\alpha_{s} \gamma_{s}}f_{\mu}(\tilde{x})\notag\\
& -\alpha_{s}\mathbb{E}_{u_t, i_t| \mathcal{F}_{t-1}} \big[\langle \tilde{g} - \nabla f_{\mu}(\tilde{x}), x_t - x
\rangle\big] \notag\\
= \ & p_s f_{\mu}(\tilde{x})-\big(p_s-\frac{4\alpha_{s}\gamma_{s}(d+4)L}{1+\tau \gamma_{s}-L\alpha_{s} \gamma_{s}}\big) \cdot \frac{\tau}{2}\|\tilde{x} - \underline{x}_t\|^2 + \frac{9\alpha_{s}\gamma_{s}\cdot \mu^2L^2(d+6)^3}{1+\tau \gamma_{s} - L\alpha_{s}\gamma_{s}}\notag \\ & -\alpha_{s}\mathbb{E}_{u_t, i_t| \mathcal{F}_{t-1}} \big[\langle \tilde{g} - \nabla f_{\mu}(\tilde{x}), x_t - x \rangle\big]. \label{VARAG-smooth-ineq-2}
\end{align}
The last inequality holds when $p_s-\frac{4\alpha_{s}\gamma_{s}(d+4)L}{1+\tau \gamma_{s}-L\alpha_{s} \gamma_{s}} \ge 0$. Combining Eq.~\eqref{VARAG-smooth-ineq-1} with Eq.~\eqref{VARAG-smooth-ineq-2}, we obtain
\begin{align}
&\mathbb{E}_{u_t, i_t| \mathcal{F}_{t-1}}\big[f_{\mu}(\bar{x}_t) + \frac{\alpha_{s}(1+\tau \gamma_{s})}{2\gamma_{s}}\|x_t-x\|^2\big] \notag\\
\le \ & (1-\alpha_{s}-p_s)f_{\mu}(\bar{x}_{t-1})+p_s f_{\mu}(\tilde{x}) + \alpha_{s} f_{\mu}(x) + \frac{\alpha_{s}}{2\gamma_{s}}\|x_{t-1}-x\|^2 + \frac{9\alpha_{s}\gamma_{s}\mu^2L^2(d+6)^3}{1+\tau\gamma_{s}-L\alpha_{s}\gamma_{s}} \notag\\
&- \alpha_{s} \mathbb{E}_{u_t, i_t| \mathcal{F}_{t-1}} \big[\langle \tilde{g}-\nabla f_{\mu}(\tilde{x}), x_t - x \rangle \big] - \frac{(1-\alpha_{s}-p_s)\tau}{2}\|\bar{x}_{t-1}-\underline{x}_t\|^2 -\big(p_s-\frac{4\alpha_{s}\gamma_{s}(d+4)L}{1+\tau \gamma_{s}-L\alpha_{s} \gamma_{s}}\big) \cdot \frac{\tau}{2}\|\tilde{x} - \underline{x}_t\|^2 \notag\\
\le \ & (1-\alpha_{s}-p_s)f_{\mu}(\bar{x}_{t-1})+p_s f_{\mu}(\tilde{x}) + \alpha_{s} f_{\mu}(x) + \frac{\alpha_{s}}{2\gamma_{s}}\|x_{t-1}-x\|^2 + \frac{9\alpha_{s}\gamma_{s}\mu^2L^2(d+6)^3}{1+\tau\gamma_{s}-L\alpha_{s}\gamma_{s}} \notag\\
& - \alpha_{s} \mathbb{E}_{u_t, i_t| \mathcal{F}_{t-1}} \big[\langle \tilde{g}-\nabla f_{\mu}(\tilde{x}), x_t - x \rangle \big]. \notag
\end{align}
Multiplying both sides by $\frac{\gamma_{s}}{\alpha_{s}}$ and then rearranging the inequality, we finish the proof of this lemma, i.e. Eq.~\eqref{VARAG-lemma-3-result}.
\end{proof}

\subsection{Proof of Theorem~\ref{VARAG-theorem-1}}

Before giving the convergence result for convex smooth $f$, we provide a lemma for the \textit{epoch-wise analysis}. This lemma needs an additional technical assumption.
\begin{tcolorbox}
\textbf{(A2$_{\boldsymbol{\mu}}$)} \ \ Let $x_\mu^* \in \text{argmin}_{x} f_\mu(x)$ and consider the sequence of approximations $\{\tilde x^s\}$ returned by Algorithm~\ref{algorithm-VARAG}. There exist a \textit{finite} constant $Z<\infty$, potentially dependent on $L$ and $d$, such that, for $\mu$ small enough,
$$\sup_{s\ge0}\E\left[\|\tilde x^s-x_\mu^*\|\right]\le Z.$$
Using an argument similar to~\cite{gadat2018stochastic}, it is possible to show that this assumption holds under the requirement that $f$ is coercive, i.e. $f(x)\to\infty$ as $\|x\|\to\infty$.
\end{tcolorbox}

\begin{mdframed}
\begin{lemma}\label{VARAG-lemma-4}
Assume \textbf{(A1)}, \textbf{(A2$_{\boldsymbol{\mu}}$)}. Suppose that the weights $\{\theta_t\}$ are set as
\begin{align}\label{VARAG-def-theta-1}
\theta_t =
\begin{cases}
\tfrac{\gamma_{s}}{\alpha_{s}} (\alpha_{s} + p_{s}) & 1 \le t \le T_s-1\\
\tfrac{\gamma_s}{\alpha_s} & t=T_s.
\end{cases}
\end{align} 
Define:
\begin{align}
\mathcal{L}_s :=\frac{\gamma_{s}}{\alpha_{s}}+(T_s-1)\frac{\gamma_{s}(\alpha_{s}+p_s)}{\alpha_{s}},
\end{align}
\begin{align}
\mathcal{R}_s := \frac{\gamma_{s}}{\alpha_{s}}(1-\alpha_{s})+(T_s-1)\frac{\gamma_{s}p_s}{\alpha_{s}}.
\end{align}
Under the conditions in Eq.~\eqref{VARAG-lemma-3-assump-1} and Eq.~\eqref{VARAG-lemma-3-assump-2}, we have:
\begin{multline*}
\mathcal{L}_s \mathbb{E}_{\mathcal{F}_{T_s}} \big[f_{\mu}(\tilde{x}^s)-f_{\mu}(x_{\mu}^*)\big]\\ 
\le \mathcal{R}_s \cdot \big[f_{\mu}(\tilde{x}^{s-1})-f_{\mu}(x_{\mu}^*)\big]
+ \big(\frac{1}{2}\|x^{s-1}-x_{\mu}^*\|^2-\frac{1}{2}\|x^s-x_{\mu}^*\|^2\big) + T_s \frac{\gamma_{s}}{\alpha_{s}} \cdot \frac{9\alpha_{s}\gamma_{s}\mu^2L^2(d+6)^3}{1-L\alpha_{s}\gamma_{s}} + (\mathcal{L}_s+\mathcal{R}_s)Z||e^s||\\
\le \mathcal{R}_s \cdot \big[f_{\mu}(\tilde{x}^{s-1})-f_{\mu}(x_{\mu}^*)\big]
+ \big(\frac{1}{2}\|x^{s-1}-x_{\mu}^*\|^2-\frac{1}{2}\|x^s-x_{\mu}^*\|^2\big) + \underbrace{T_s \frac{\gamma_{s}}{\alpha_{s}} \cdot \frac{9\alpha_{s}\gamma_{s}\mu^2L^2(d+6)^3}{1-L\alpha_{s}\gamma_{s}}}_{\circled{1}} + \underbrace{(\mathcal{L}_s+\mathcal{R}_s)Z\sqrt{E}}_{\circled{2}}.
\end{multline*}
where $e^s = \tilde{g}^s-\nabla f_{\mu}(\tilde{x}^{s-1})$ and $||e^s||^2 \le E$, which is consistent with the definition of $E$ in Section 4.1. Here, the expectation is taken over $\mathcal{F}_{T_s}$ inside the epoch $s$.
\end{lemma}
\end{mdframed}
\begin{remark}
Compared to the corresponding result by~\citet{lan2019unified} (Lemma 7 in their paper), we note that two additional errors terms appear. \circled{1} is the error due to the Gaussian smooth estimation and \circled{2} is an error due to the approximation made at the pivot point.
We will later see that the coordinate-wise approach introduced in Eq.~\eqref{DFO-framework-cord-finite-difference} yields a constant error bound for \circled{2}, which is independent of the gradient information.
\end{remark}

\begin{proof}[Proof of Lemma \ref{VARAG-lemma-4}]
For $f$ convex and $L$-smooth, we have that $f_{\mu}$ is $L$-smooth and $\tau$-strongly-convex with $\tau = 0$ from Lemma \ref{Nes-DFO-property}. Hence, Lemma \ref{VARAG-lemma-3} can be written as
\begin{align*}
& \mathbb{E}_{u_t, i_t| \mathcal{F}_{t-1}}\bigg[\frac{\gamma_{s}}{\alpha_{s}}\big[f_{\mu}(\bar{x}_t)-f_{\mu}(x)\big] + \frac{1}{2}\|x_t-x\|^2\bigg]\notag\\
\le \  & \frac{\gamma_{s}}{\alpha_{s}}(1-\alpha_{s}-p_s)\big[f_{\mu}(\bar{x}_{t-1})-f_{\mu}(x)\big]+\frac{\gamma_{s}p_s}{\alpha_{s}} \big[f_{\mu}(\tilde{x})-f_{\mu}(x)\big] + \frac{1}{2}\|x_{t-1}-x\|^2 + \frac{\gamma_{s}}{\alpha_{s}} \cdot \frac{9\alpha_{s}\gamma_{s}\mu^2L^2(d+6)^3}{1-L\alpha_{s}\gamma_{s}}\notag\\
& - \frac{\gamma_{s}}{\alpha_{s}} \cdot \alpha_{s} \mathbb{E}_{u_t, i_t| \mathcal{F}_{t-1}}\big[\langle \tilde{g}-\nabla f_{\mu}(\tilde{x}), x_t - x \rangle \big].
\end{align*}
Summing up these inequalities over $t = 1, \dots, T_s$, using the definition of $\theta_t$ and $\bar{x}_0 = \tilde{x}$,
\begin{align*}
\sum_{t=1}^{T_s}\theta_t \mathbb{E}_{\mathcal{F}_t} \big[f_{\mu}(\bar{x}_t)-f_{\mu}(x)\big]
\le \  & \big[\frac{\gamma_{s}}{\alpha_{s}}(1-\alpha_{s})+ (T_s-1)\frac{\gamma_{s}p_s}{\alpha_{s}}\big] \cdot \big[f_{\mu}(\tilde{x})-f_{\mu}(x)\big] + \big(\frac{1}{2}\|x_0-x\|^2-\frac{1}{2}\|x_{T_s}-x\|^2\big) \notag\\
& + T_s \cdot \frac{\gamma_{s}}{\alpha_{s}} \cdot \frac{9\alpha_{s}\gamma_{s}\mu^2L^2(d+6)^3}{1-L\alpha_{s}\gamma_{s}} - \frac{\gamma_{s}}{\alpha_{s}} \sum_{t=1}^{T_s} \alpha_{s} \mathbb{E}_{\mathcal{F}_t}\big[\langle \tilde{g}-\nabla f_{\mu}(\tilde{x}), x_t - x \rangle \big].
\end{align*}
Using the fact that $\tilde{x}^s = \sum_{t=1}^{T_s}\big(\theta_t \bar{x}_t\big)/\sum_{t=1}^{T_s}\theta_t$, $\tilde{x} = \tilde{x}^{s-1}$, $x_0 = x^{s-1}$, $x_{T_s} = x^s$ and using convexity of $f_{\mu}$, the inequality above implies
\begin{align*}
\sum_{t=1}^{T_s}\theta_t \mathbb{E}_{\mathcal{F}_t} \big[f_{\mu}(\tilde{x}^s)-f_{\mu}(x)\big] \le \  & \big[\frac{\gamma_{s}}{\alpha_{s}}(1-\alpha_{s})+ (T_s-1)\frac{\gamma_{s}p_s}{\alpha_{s}}\big] \cdot \big[f_{\mu}(\tilde{x}^{s-1})-f_{\mu}(x)\big] + \big(\frac{1}{2}\|x^{s-1}-x\|^2-\frac{1}{2}\|x^s-x\|^2\big) \notag\\
& + T_s \cdot \frac{\gamma_{s}}{\alpha_{s}} \cdot \frac{9\alpha_{s}\gamma_{s}\mu^2L^2(d+6)^3}{1-L\alpha_{s}\gamma_{s}} - \frac{\gamma_{s}}{\alpha_{s}} \sum_{t=1}^{T_s} \alpha_{s} \mathbb{E}_{\mathcal{F}_t}\big[\langle \tilde{g}^s-\nabla f_{\mu}(\tilde{x}^{s-1}), x_t - x \rangle \big],
\end{align*}
which is equivalent to
\begin{align}
\mathcal{L}_s \mathbb{E}_{\mathcal{F}_t} \big[f_{\mu}(\tilde{x}^s)-f_{\mu}(x)\big] \le \  & \mathcal{R}_s \cdot \big[f_{\mu}(\tilde{x}^{s-1})-f_{\mu}(x)\big] + \big(\frac{1}{2}\|x^{s-1}-x\|^2-\frac{1}{2}\|x^s-x\|^2\big)  \notag\\
& + T_s \cdot \frac{\gamma_{s}}{\alpha_{s}} \cdot \frac{9\alpha_{s}\gamma_{s}\mu^2L^2(d+6)^3}{1-L\alpha_{s}\gamma_{s}} - \frac{\gamma_{s}}{\alpha_{s}} \sum_{t=1}^{T_s} \alpha_{s} \mathbb{E}_{\mathcal{F}_t}\big[\langle \tilde{g}^s-\nabla f_{\mu}(\tilde{x}^{s-1}), x_t - x \rangle \big]. \label{VARAG-lemma-4-ineq-1}
\end{align}
Notice that, since $\bar{x}_0 = \tilde{x} = \tilde{x}^{s-1}$ in the epoch $s$,
\begin{align}
& \frac{\gamma_{s}}{\alpha_{s}}\sum_{t=1}^{T_s} \alpha_{s} \mathbb{E}_{\mathcal{F}_t}\big[\langle \tilde{g}^s-\nabla f_{\mu}(\tilde{x}^{s-1}), x_t - x \rangle \big] \notag\\
= \  & \frac{\gamma_{s}}{\alpha_{s}}\sum_{t=1}^{T_s} \mathbb{E}_{\mathcal{F}_t}\big[\langle \tilde{g}^s-\nabla f_{\mu}(\tilde{x}^{s-1}), \bar{x}_t-(1-\alpha_{s}-p_s)\bar{x}_{t-1}-p_s\tilde{x}^{s-1} - \alpha_{s} x \rangle \big] \notag\\
= \  & \frac{\gamma_{s}}{\alpha_{s}} \mathbb{E}_{\mathcal{F}_{T_s}}\big[\langle \tilde{g}^s-\nabla f_{\mu}(\tilde{x}^{s-1}), \bar{x}_{T_s}+\sum_{t=1}^{T_s-1}(\alpha_{s}+p_s)\bar{x}_{t}-\big[(1-\alpha_{s})+(T_s-1)p_s\big]\tilde{x}^{s-1} - \alpha_{s} T_sx \rangle \big] \notag\\
= \  & \frac{\gamma_{s}}{\alpha_{s}} \mathbb{E}_{\mathcal{F}_{T_s}}\big[\langle \tilde{g}^s-\nabla f_{\mu}(\tilde{x}^{s-1}), \frac{\alpha_{s}}{\gamma_{s}}\sum_{t=1}^{T_s}\theta_t \tilde{x}^s -\big[(1-\alpha_{s})+(T_s-1)p_s\big]\tilde{x}^{s-1} - \alpha_{s} T_sx \rangle \big] \notag\\
= \  & \mathbb{E}_{\mathcal{F}_{T_s}}\big[\langle \tilde{g}^s-\nabla f_{\mu}(\tilde{x}^{s-1}), \mathcal{L}_s \tilde{x}^s -\mathcal{R}_s\tilde{x}^{s-1} - \gamma_{s} T_sx \rangle \big] \notag\\
= \  & \mathbb{E}_{\mathcal{F}_{T_s}}\big[\langle \tilde{g}^s-\nabla f_{\mu}(\tilde{x}^{s-1}), \mathcal{L}_s \big(\tilde{x}^s-x\big) -\mathcal{R}_s\big(\tilde{x}^{s-1}-x\big) \rangle \big]. \label{VARAG-pivotal-error}
\end{align}
The first equality is indeed the update rule of $\bar{x}_t$, the thrid equality is the definition of $\tilde{x}^s$ and the second last equality comes from the definition of $\mathcal{L}_s$ and $\mathcal{R}_s$.

Then, we set $x = x_{\mu}^*$ to the inequality above. Based on the assumption \textbf{(A2$_{\boldsymbol{\mu}}$)} and combining the previous inequality with Eq.~\eqref{VARAG-lemma-4-ineq-1}, we have
\begin{align*}
\mathcal{L}_s \mathbb{E}_{\mathcal{F}_{T_s}} \big[f_{\mu}(\tilde{x}^s)-f_{\mu}(x_{\mu}^*)\big] \le \  & \mathcal{R}_s \cdot \big[f_{\mu}(\tilde{x}^{s-1})-f_{\mu}(x_{\mu}^*)\big] + \big(\frac{1}{2}\|x^{s-1}-x_{\mu}^*\|^2-\frac{1}{2}\|x^s-x_{\mu}^*\|^2\big) \notag\\
& + T_s \cdot \frac{\gamma_{s}}{\alpha_{s}} \cdot \frac{9\alpha_{s}\gamma_{s}\mu^2L^2(d+6)^3}{1-L\alpha_{s}\gamma_{s}} + (\mathcal{L}_s+\mathcal{R}_s)Z\|e^s\|.
\end{align*}
\end{proof}

\noindent Finally, we derive Theorem~\ref{VARAG-theorem-1} directly from Lemma~\ref{VARAG-lemma-4}. For convenience of the reader, we re-write the theorem here.

\begin{mdframed}
\textbf{Theorem 2. }Assume \textbf{(A1)} and \textbf{(A2$_{\boldsymbol{\mu}}$)}. If we define $s_0 := \lfloor \log (d+4)n \rfloor+1$ and set $\{T_s\}$, $\{\gamma_s\}$ and $\{p_s\}$ as
	\begin{align}
	T_s = \begin{cases}
	2^{s-1}, & s \le s_0\\
	T_{s_0}, & s > s_0
	\end{cases}, \ \ \ \ \
	\gamma_s = \tfrac{1}{12 (d+4)L \alpha_s}, \ \ \ \ \
	p_s = \tfrac{1}{2},
	\end{align}
	% with 
	\vspace{-1.5mm}
	with
	\vspace{-5mm}
	\begin{align}
	\alpha_s =
	\begin{cases}
	\tfrac{1}{2}, & s \le s_0\\
	\tfrac{2}{s-s_0+4},& s > s_0
	\end{cases}.
	\end{align}
	If we set
	\vspace{-3mm}
	\begin{align}
\theta_t =
\begin{cases}
\tfrac{\gamma_{s}}{\alpha_{s}} (\alpha_{s} + p_{s}) & 1 \le t \le T_s-1\\
\tfrac{\gamma_s}{\alpha_s} & t=T_s.
\end{cases}
\end{align} 
	we obtain
	\begin{align*}
	&\mathbb{E} \big[f_{\mu}(\tilde{x}^s)-f^*_{\mu}\big] \le \begin{cases}
	\cfrac{(d+4)D_0}{2^{s+1}}  + \varsigma_1 +\varsigma_2, & 1 \le s \le s_0\\
	\cfrac{16 D_0}{n(s-s_0+4)^2} + \delta_s \cdot (\varsigma_1 +\varsigma_2),&  s > s_0\\
	\end{cases}
	\end{align*}
	where $\varsigma_1 = 2\mu^2L(d+4)^2$, $\varsigma_2 = 3Z\sqrt{E}$, $\delta_s = \mathcal{O}(s-s_0)$ and $D_0$ is defined as
	\begin{align}
	D_0:= \frac{2}{(d+4)}[f_{\mu}(x^0) - f_{\mu}(x_{\mu}^*)] + 6L \|x^0-x_{\mu}^* \|^2,
	\end{align}
	where $x^*_\mu$ is any finite minimizer of $f_\mu$.
\end{mdframed}

\begin{proof}[Proof of Theorem \ref{VARAG-theorem-1}]
First, note that, with the parameter choices described in the theorem statement, \textit{the restrictions in Eq.~\eqref{VARAG-lemma-3-assump-1} and Eq.~\eqref{VARAG-lemma-3-assump-2} are satisfied}:
\begin{align}
1+\tau \gamma_{s} - L\alpha_{s}\gamma_{s} = 1 - \frac{1}{12(d+4)} > 0, 
\end{align}
\begin{align}
p_s - \frac{4(d+4)L\alpha_{s}\gamma_{s}}{1+\tau \gamma_{s} - L\alpha_{s}\gamma_{s}} = \frac{1}{2}-\frac{1}{3}\cdot \frac{1}{1-\frac{1}{12(d+4)}} > 0. 
\end{align}
We further define
\begin{align}
w_s := \mathcal{L}_s - \mathcal{R}_{s+1}.
\end{align}
As in \cite{lan2019unified}, if $1\le s <s_0$, 
\begin{align*}
w_s & = \mathcal{L}_s - \mathcal{R}_{s+1} = \frac{\gamma_{s}}{\alpha_s}\big[1+(T_s-1)(\alpha_{s}+p_s)-(1-\alpha_{s})-(2T_s-1)p_s\big] = \frac{\gamma_{s}}{\alpha_{s}}\big[T_s(\alpha_{s}-p_{s})\big] = 0.
\end{align*}
Otherwise, if $s \ge s_0$, we have $\frac{\gamma_s}{\alpha_s} = \frac{1}{12(d+4)L\alpha_s^2} = \frac{(s-s_0+4)^2}{48(d+4)L}$ and
\begin{align*}
w_s & = \mathcal{L}_s - \mathcal{R}_{s+1} = \frac{\gamma_{s}}{\alpha_{s}} -  \frac{\gamma_{s+1}}{\alpha_{s+1}}(1-\alpha_{s+1})+(T_{s_0}-1)\big[\frac{\gamma_{s}(\alpha_{s}+p_s)}{\alpha_{s}}-\frac{\gamma_{s+1}p_{s+1}}{\alpha_{s+1}}\big]\\
& = \frac{(s-s_0+4)^2}{48(d+4)L} - \frac{(s-s_0+5)^2}{48(d+4)L}\big(1- \frac{2}{s-s_0+5}\big)\\
& \quad +(T_{s_0}-1)\bigg[\frac{(s-s_0+4)^2}{48(d+4)L} \cdot \big(\frac{2}{s-s_0+4}+\frac{1}{2}\big)-\frac{(s-s_0+5)^2}{48(d+4)L}\cdot \frac{1}{2}\bigg]\\
& = \frac{1}{48(d+4)L}+\frac{T_{s_0}-1}{96(d+4)L}\big[2(s-s_0+4)-1\big] > 0.
\end{align*}
Hence, $w_s \ge 0$ for all $s$. We can therefore use Lemma \ref{VARAG-lemma-4} iteratively as follows,
\begin{align}
& \mathcal{L}_s \mathbb{E} \big[f_{\mu}(\tilde{x}^s)-f_{\mu}(x_{\mu}^*)\big]  + \bigg(\sum_{j=1}^{s-1}w_j\mathbb{E}\big[f_{\mu}(\tilde{x}^j) - f_{\mu}(x_{\mu}^*)\big]\bigg)\notag\\
\le \  & \mathcal{R}_1 \cdot \mathbb{E} \big[f_{\mu}(\tilde{x}^{0})-f_{\mu}(x_{\mu}^*)\big] + \mathbb{E}\big[\frac{1}{2}\|x^{0}-x_{\mu}^*\|^2-\frac{1}{2}\|x^s-x_{\mu}^*\|^2\big] + \sum_{j=1}^{s} T_j \cdot \frac{\gamma_{j}}{\alpha_{j}} \cdot \frac{9\alpha_{j}\gamma_{j}\mu^2L^2(d+6)^3}{1-L\alpha_{j}\gamma_{j}} \notag\\
& + \sum_{j=1}^{s}(\mathcal{L}_j+\mathcal{R}_j)Z\|e^j\| \notag\\
\le \  & \frac{1}{6(d+4)L} \big[f_{\mu}(\tilde{x}^{0})-f_{\mu}(x_{\mu}^*)\big] + \frac{1}{2}\|x^{0}-x_{\mu}^*\|^2 + \sum_{j=1}^{s} T_j \cdot \frac{\gamma_{j}}{\alpha_{j}} \cdot \frac{9\alpha_{j}\gamma_{j}\mu^2L^2(d+6)^3}{1-L\alpha_{j}\gamma_{j}}\\ & + \sum_{j=1}^{s}(\mathcal{L}_j+\mathcal{R}_j)Z\|e^j\| \notag\\
= \  & \frac{1}{12L} D_0 + \sum_{j=1}^{s} T_j \cdot \frac{\gamma_{j}}{\alpha_{j}} \cdot \frac{9\alpha_{j}\gamma_{j}\mu^2L^2(d+6)^3}{1-L\alpha_{j}\gamma_{j}} + \sum_{j=1}^{s}(\mathcal{L}_j+\mathcal{R}_j)Z\|e^j\| \notag\\
\le \  & \frac{1}{12L} D_0 + \sum_{j=1}^{s} T_j \cdot \frac{\gamma_{j}}{\alpha_{j}} \cdot \frac{3\mu^2L(d+6)^3}{4(d+4)} + \sum_{j=1}^{s}(\mathcal{L}_j+\mathcal{R}_j)Z\|e^j\|\notag\\
\le \  & \frac{1}{12L} D_0 + \sum_{j=1}^{s} T_j \cdot \frac{\gamma_{j}}{\alpha_{j}} \cdot \mu^2L(d+4)^2 + \sum_{j=1}^{s}(\mathcal{L}_j+\mathcal{R}_j)Z\|e^j\|\notag\\
\le \  & \frac{1}{12L} D_0 + \sum_{j=1}^{s} T_j \cdot \frac{\gamma_{j}}{\alpha_{j}} \cdot \mu^2L(d+4)^2 + \sum_{j=1}^{s}(\mathcal{L}_j+\mathcal{R}_j)Z\sqrt{E}.
\label{VARAG-theorem-1-ineq-1}
\end{align}
The second last equality holds when $\alpha_{j}\gamma_{j} = \frac{1}{12(d+4)L}$ and $x = x_{\mu}^*$, the optimal solution for $f_{\mu}$. We proceed with two cases:\\
\\
\textbf{Case I:} If $s \le s_0$, $\mathcal{L}_s = \frac{2^{s+1}}{12(d+4)L}$, $\mathcal{R}_s = \frac{2^{s}}{12(d+4)L} = \frac{\mathcal{L}_s}{2}$, $\frac{\gamma_{s}}{\alpha_{s}} = \frac{1}{3(d+4)L}$, $T_s = 2^{s-1}$. Hence, we have
\begin{align}
& \mathbb{E} \big[f_{\mu}(\tilde{x}^s)-f_{\mu}(x_{\mu}^*)\big] \le \frac{1}{2^{s+1}} (d+4)D_0 + 2\mu^2L(d+4)^2 + 3Z\sqrt{E}, \qquad 1 \le s \le s_0.
\label{VARAG-theorem-1-ineq-2}
\end{align}
\textbf{Case II:} If $s > s_0$, we have
\begin{align*}
\mathcal{L}_s = \  & \frac{1}{12(d+4)\alpha_s^2}\big[(T_s-1)\alpha_{s} + \frac{1}{2}(T_s+1)\big] = \frac{(s-s_0+4)^2}{48(d+4)L} \cdot \big[(T_{s_0}-1)\alpha_{s} + \frac{1}{2}(T_{s_0}+1)\big]\\
\ge \ & \frac{(s-s_0+4)^2}{96(d+4)L} \cdot (T_{s_0}+1)
\ge \frac{n \cdot (s-s_0+4)^2}{192L},
\end{align*}
where the last inequality holds since $T_{s_0} = 2^{\left\lfloor \log_{2}[(d+4)n]\right\rfloor} \ge \frac{(d+4)n}{2}$, i.e. $2^{s_0} \ge (d+4)n$. Hence, based on $\sum_{i=1}^{n}i^2 = \frac{n(n+1)(2n+1)}{6}$ and Eq.~\eqref{VARAG-theorem-1-ineq-2}, Eq.~\eqref{VARAG-theorem-1-ineq-1} implies
\begin{align}
\mathbb{E} \big[f_{\mu}(\tilde{x}^s)-f_{\mu}(x_{\mu}^*)\big] \le \  & \frac{16 D_0}{n(s-s_0+4)^2} + \mathcal{O}(s-s_0) \cdot \mu^2L(d+4)^2+ \mathcal{O}(s-s_0) \cdot Z\sqrt{E}.
\label{VARAG-theorem-1-ineq-3}
\end{align}
\end{proof}

We conclude by deriving the final complexity result, stated in the main paper.

\begin{proof}[Proof of Corollary~\ref{VARAG-corollary-1}] We pick up from the proof presented in the main paper, which we summarize in the next lines. Note that the analysis we performed in the last pages is based on $f_{\mu}$ rather than $f$. Hence, we first need to ensure that the error between these two functions is sufficiently small. We can bound $f_{\mu}(\tilde{x}^s) - f_{\mu}(x_{\mu}^*)$ from $f(\tilde{x}^s) - f(x^*)$ as follows:
\begin{align*}
    f_{\mu}(\tilde{x}^s) - f_{\mu}(x_{\mu}^*) & = f_{\mu}(\tilde{x}^s) - f(\tilde{x}^s) + f(\tilde{x}^s) - f_{\mu}(x_{\mu}^*) + f_{\mu}(x^*) - f_{\mu}(x^*) + f(x^*) - f(x^*)\\
    & \ge -\frac{\mu^2Ld}{2} + f(\tilde{x}^s) - f_{\mu}(x_{\mu}^*) + f_{\mu}(x^*) -\frac{\mu^2Ld}{2} - f(x^*)\\
    & \ge f(\tilde{x}^s) - f(x^*) - \mu^2Ld.
\end{align*}
The first inequality comes from Lemma \ref{Nes-DFO-error} and the second inequality comes from the definition of $x_{\mu}$.

We want the error term $\mu^2Ld$ we just derived to be small, say $\le \frac{\epsilon}{4}$, i.e. $\mu = \mathcal{O}\big(\sqrt{\frac{\epsilon}{4Ld}}\big)$. From this, we get an upper bound on $\mu$~(\textit{choosing $\mu$ small does not affect the convergence rate}). Next, we bound in the same way the additional~(non-vanishing) error terms in Eq.~\eqref{VARAG-theorem-1-ineq-2} and Eq.~\eqref{VARAG-theorem-1-ineq-3}. This requires $\mu = \mathcal{O}\big(\frac{\epsilon^{1/2}}{L^{1/2} d} \big)$, $\mu = \mathcal{O}\big(\frac{\epsilon}{ZLd^{3/2}}\big)$ and $\nu = \mathcal{O}\big( \frac{\epsilon}{ZL d^{1/2}}\big)$ for  Eq.~\eqref{VARAG-theorem-1-ineq-2} while $\mu = \mathcal{O}\big(\frac{n^{1/4}\epsilon^{3/4}}{L^{1/2} d D_0^{1/4}} \big)$, $\mu = \mathcal{O}\big(\frac{n^{1/2}\epsilon^{3/2}}{Z L d^{3/2} D_0^{1/2}}\big)$ and $\nu = \mathcal{O}\big( \frac{n^{1/2}\epsilon^{3/2}}{Z L d^{1/2} D_0^{1/2}}\big)$ for  Eq.~\eqref{VARAG-theorem-1-ineq-3} to ensure $\epsilon$-optimality, $\frac{\epsilon}{4}$ more specifically. Therefore, if we bound the term which contains $D_0$ by $\frac{\epsilon}{2}$, $f(\tilde{x}^s) - f(x^*)$ would achieve $\epsilon$-optimality in expectation. This is what we do next (following the proof in~\cite{lan2012optimal}), for the two cases in Theorem~\ref{VARAG-theorem-1}.

If $n \ge \frac{D_0}{\epsilon}$, i.e. in the region of relatively low accuracy and/or large number of components, we have
\begin{align*}
\frac{(d+4)D_0}{2^{s_0+1}} \le \frac{D_0}{2n}\le \frac{\epsilon}{2} \Rightarrow \log\frac{(d+4)D_0}{\epsilon} \le s_0.
\end{align*}
Therefore, the number of epochs is at most $s_0$ for the first term in Eq.~\eqref{VARAG-theorem-1-ineq-2} to achieve $\frac{\epsilon}{2}$ optimality inside Case I. Hence, the total number of function queries is bounded by
\begin{multline*}
{d nS_l + \sum_{s=1}^{S_l}T_s= \mathcal{O} \bigg\{\min \bigg(d n\log\frac{(d+4)D_0}{\epsilon}, d n \log (dn), dn\bigg)\bigg\} = \mathcal{O} \bigg\{\min \left(d n \log \frac{d D_0}{\epsilon},dn \right)\bigg\} = \mathcal{O} \left\{d n \log \frac{d D_0}{\epsilon} \right\}.}
\end{multline*}
If instead $n < \frac{D_0}{\epsilon}$, at epoch $S_h = \left\lceil \sqrt{\frac{32D_0}{n\epsilon}} + s_0 -4 \right\rceil$ (ensuring the first term in Eq.~\eqref{VARAG-theorem-1-ineq-3} to be not bigger than $\frac{\epsilon}{2}$), we can achieve $\epsilon$ optimality. Hence, the total number of function queries is
\begin{align*}
dn s_0 + \sum_{s=1}^{s_0}T_s + (T_{s_0}+dn)(S_h-s_0) \le \sum_{s=1}^{s_0}T_s + (T_{s_0}+dn)S_h= \mathcal{O}\bigg\{d\sqrt{\frac{nD_0}{\epsilon}} + dn \log(dn) \bigg\}.
\end{align*}
\end{proof}

\subsection{Proof of Theorem~\ref{VARAG-theorem-2}}

In this section, we assume $f$ to be strongly convex. Hence, $f_{\mu}$ is also strongly convex by Lemma \ref{Nes-DFO-property}. 

\begin{tcolorbox}
\textbf{(A3)} \ \ $f=\frac{1}{n}\sum_{i=1}^n f_i$ is $\tau$-strongly convex. That is, for all $x,y\in\R^d, \ f(y)\ge f(x)+\langle\nabla f(x),y-x\rangle +\frac{\tau}{2}\|y-x\|^2$.
\end{tcolorbox}

We rewrite below Theorem \ref{VARAG-theorem-2}, for convenience of the reader:
\begin{mdframed}
\textbf{Theorem 4.} Assume \textbf{(A1)} and \textbf{(A3)}. Let us denote $s_0 := \lfloor \log (d+4)n\rfloor+1$ and assume that the
	weights $\{\theta_t\}$ are set to Eq.~\eqref{VARAG-def-theta-paper} if $1\le s \le s_0$. Otherwise, they are set to
	\vspace{-1mm}
	\begin{align}
	\theta_t =
	\begin{cases}
	\Gamma_{t-1} - (1 - \alpha_s - p_s) \Gamma_{t}, & 1 \le t \le T_s-1,\\
	\Gamma_{t-1}, & t = T_s,
	\end{cases}
	\end{align}
	where $\Gamma_t= \big(1+\frac{\tau\gamma_s}{2}\big)^t$.
	If the parameters $\{T_s\}$, $\{\gamma_s\}$ and $\{p_s\}$ are set to Eq.~\eqref{parameter-deter-smooth1} with %$\{\alpha_s\}$ being set as
	\begin{align}
	\alpha_s =
	\begin{cases}
	\tfrac{1}{2}, & s \le s_0,\\
	\min\{\sqrt{\frac{n \tau}{24L}}, \tfrac{1}{2}\},& s > s_0,
	\end{cases}
	\end{align}
	we obtain
	\begin{align*}
	\mathbb{E} \big[f_{\mu}(\tilde{x}^{s})-f_{\mu}^{*}\big] \le \begin{cases}
	\cfrac{1}{2^{s+1}} (d+4)D_0 + \varsigma_1 + \varsigma_2, & \ \ \ \ \ 1 \le s \le s_0\\
	& \\
	 (4/5)^{s-s_0}\cfrac{D_0}{n}+ \varsigma_1 + \varsigma_2, &\ \ \ \ \ s > s_0 \text{ and } n \ge \frac{6L}{\tau} \\
	& \\
	  \left(1+\frac{1}{4} \sqrt{\frac{n\tau}{6L}}\right)^{-(s-s_0)} \cfrac{D_0}{n}+ \Big(\sqrt{\frac{6L}{n\tau}}+1\Big)\varsigma_1 + \varsigma_2  & \ \ \ \ \ s > s_0 \text{ and } n < \frac{6L}{\tau}
	\end{cases}
	\end{align*}
where $\varsigma_1 = 12\mu^2L(d+4)^2$, $\varsigma_2 = 5E/\tau$ and $D_0$ is defined as in Eq.~\eqref{VARAG-def-D_0}.
\end{mdframed}

\begin{remark}
Compared with smooth convex case, we can drop the assumption \textbf{(A2$_{\boldsymbol{\mu}}$)}.
\end{remark}

We start with the following result.
\begin{mdframed}
\begin{lemma}\label{VARAG-lemma-5}
	Assume \textbf{(A1)}, \textbf{(A3)}. Under the choice of parameters from Theorem~\ref{VARAG-theorem-2}, for any $0 < c \le 1$,
	\begin{align}
	& \mathbb{E}_{u_t, i_t| \mathcal{F}_{t-1}}\bigg[\frac{\gamma_{s}}{\alpha_{s}}\big[f_{\mu}(\bar{x}_t)-f_{\mu}^{*}\big] + \big(1+(1-c)\tau \gamma_{s}\big) \cdot \frac{1}{2}\|x_t-x_{\mu}^*\|^2\bigg] \notag\\
	\le \  & \frac{\gamma_{s}}{\alpha_{s}}(1-\alpha_{s}-p_s)\big[f_{\mu}(\bar{x}_{t-1})-f_{\mu}^{*}\big]+\frac{\gamma_{s}p_s}{\alpha_{s}} \big[f_{\mu}(\tilde{x})-f_{\mu}^{*}\big] + \frac{1}{2}\|x_{t-1}-x_{\mu}^*\|^2 + \frac{\gamma_{s}}{\alpha_{s}} \cdot \mu^2L(d+4)^2 + \frac{\gamma_{s}}{2c\tau} \cdot E, \label{VARAG-lemma-5-result}
	\end{align}
	where $E$ is defined as in Lemma \ref{VARAG-lemma-4}.
\end{lemma}
\end{mdframed}

\begin{proof}[Proof of Lemma \ref{VARAG-lemma-5}]
First, note that, with the parameter choices described in the Theorem~\ref{VARAG-theorem-2}, \textit{the restrictions in Eq.~\eqref{VARAG-lemma-3-assump-1} and Eq.~\eqref{VARAG-lemma-3-assump-2} are satisfied}. Hence, Eq.~\eqref{VARAG-lemma-3-result} becomes, when setting $x = x_{\mu}^*$,
	\begin{align*}
	& \mathbb{E}_{u_t, i_t| \mathcal{F}_{t-1}}\bigg[\frac{\gamma_{s}}{\alpha_{s}}\big[f_{\mu}(\bar{x}_t)-f_{\mu}^{*}\big] + \frac{(1+\tau \gamma_{s})}{2}\|x_t-x_{\mu}^*\|^2\bigg] \notag\\
	\le \  & \frac{\gamma_{s}}{\alpha_{s}}(1-\alpha_{s}-p_s)\big[f_{\mu}(\bar{x}_{t-1})-f_{\mu}^{*}\big]+\frac{\gamma_{s}p_s}{\alpha_{s}} \big[f_{\mu}(\tilde{x})-f_{\mu}^{*}\big] + \frac{1}{2}\|x_{t-1}-x_{\mu}^*\|^2 + \frac{\gamma_{s}}{\alpha_{s}} \cdot \frac{9\alpha_{s}\gamma_{s}\mu^2L^2(d+6)^3}{1+\tau\gamma_{s}-L\alpha_{s}\gamma_{s}} \notag\\
	& - \gamma_{s} \mathbb{E}_{u_t, i_t| \mathcal{F}_{t-1}}\big[\langle \tilde{g}-\nabla f_{\mu}(\tilde{x}), x_t - x_{\mu}^* \rangle \big]\\
	= \  & \frac{\gamma_{s}}{\alpha_{s}}(1-\alpha_{s}-p_s)\big[f_{\mu}(\bar{x}_{t-1})-f_{\mu}^{*}\big]+\frac{\gamma_{s}p_s}{\alpha_{s}} \big[f_{\mu}(\tilde{x})-f_{\mu}^{*}\big] + \frac{1}{2}\|x_{t-1}-x_{\mu}^*\|^2 + \frac{\gamma_{s}}{\alpha_{s}} \cdot \frac{3\mu^2L(d+6)^3}{4(1+\tau\gamma_{s}-\frac{1}{12})(d+4)} \notag\\
	& - \gamma_{s} \mathbb{E}_{u_t, i_t| \mathcal{F}_{t-1}}\big[\langle \tilde{g}-\nabla f_{\mu}(\tilde{x}), x_t - x_{\mu}^* \rangle \big]\\
	\le \  & \frac{\gamma_{s}}{\alpha_{s}}(1-\alpha_{s}-p_s)\big[f_{\mu}(\bar{x}_{t-1})-f_{\mu}^{*}\big]+\frac{\gamma_{s}p_s}{\alpha_{s}} \big[f_{\mu}(\tilde{x})-f_{\mu}^{*}\big] + \frac{1}{2}\|x_{t-1}-x_{\mu}^*\|^2 + \frac{\gamma_{s}}{\alpha_{s}} \cdot \mu^2L(d+4)^2 \notag\\
	& - \gamma_{s} \mathbb{E}_{u_t, i_t| \mathcal{F}_{t-1}}\big[\langle \tilde{g}-\nabla f_{\mu}(\tilde{x}), x_t - x_{\mu}^* \rangle \big].
	\end{align*}
	The equality above holds since $\alpha_{s}$ and $\gamma_{s} $ are defined as in Theorem \ref{VARAG-theorem-2} since $\alpha_{s}\gamma_{s} = \frac{1}{12(d+4)L}$.\\
	Moreover, for any $0 < c \le 1$, we have
	\begin{align*}
	- \gamma_{s} \langle \tilde{g}-\nabla f_{\mu}(\tilde{x}), x_t - x_{\mu}^* \rangle - \frac{c\tau\gamma_{s}}{2}\|x_t-x_{\mu}^*\|^2 \le \frac{\gamma_{s}}{2c\tau}\|\tilde{g}-\nabla f_{\mu}(\tilde{x})\|^2,
	\end{align*}
	since $b\langle u, v \rangle - \frac{a}{2} \|v\|^2 \le \frac{b^2}{2a}\|u\|^2$ when $a >0$. Hence, plugging this in, 
	\begin{align*}
	& \mathbb{E}_{u_t, i_t| \mathcal{F}_{t-1}}\bigg[\frac{\gamma_{s}}{\alpha_{s}}\big[f_{\mu}(\bar{x}_t)-f_{\mu}^{*}\big] + \big(1+(1-c)\tau \gamma_{s}\big) \cdot \frac{1}{2}\|x_t-x_{\mu}^*\|^2\bigg] \notag\\
	\le \  & \frac{\gamma_{s}}{\alpha_{s}}(1-\alpha_{s}-p_s)\big[f_{\mu}(\bar{x}_{t-1})-f_{\mu}^{*}\big]+\frac{\gamma_{s}p_s}{\alpha_{s}} \big[f_{\mu}(\tilde{x})-f_{\mu}^{*}\big] + \frac{1}{2}\|x_{t-1}-x_{\mu}^*\|^2 + \frac{\gamma_{s}}{\alpha_{s}} \cdot \mu^2L(d+4)^2 \notag\\
	& + \frac{\gamma_{s}}{2c\tau} \mathbb{E}_{u_t, i_t| \mathcal{F}_{t-1}}\big[\|\tilde{g}-\nabla f_{\mu}(\tilde{x})\|^2\big] \notag\\
	\le \  & \frac{\gamma_{s}}{\alpha_{s}}(1-\alpha_{s}-p_s)\big[f_{\mu}(\bar{x}_{t-1})-f_{\mu}^{*}\big]+\frac{\gamma_{s}p_s}{\alpha_{s}} \big[f_{\mu}(\tilde{x})-f_{\mu}^{*}\big] + \frac{1}{2}\|x_{t-1}-x_{\mu}^*\|^2 + \frac{\gamma_{s}}{\alpha_{s}} \cdot \mu^2L(d+4)^2 + \frac{\gamma_{s}}{2c\tau} \cdot E.
	\end{align*}
\end{proof}

We divide the proof of Theorem \ref{VARAG-theorem-2} into three cases, corresponding to the three lemmas below.

\begin{mdframed}
\begin{lemma}\label{VARAG-lemma-6}
	Assume \textbf{(A1)}, \textbf{(A3)}. Under the choice of parameters from Theorem~\ref{VARAG-theorem-2}, if $s \le s_0$, then for any $x \in \mathbb{R}^d$,
	\begin{align*}
	\mathbb{E}\big[f_{\mu}(\tilde{x}^s) - f_{\mu}^{*}\big] \le \frac{1}{2^{s+1}} (d+4)D_0 + 2 \mu^2L(d+4)^2 + \frac{E}{2\tau},
	\end{align*}
	where $D_0$ is defined in Eq.~\eqref{VARAG-def-D_0}.
\end{lemma}
\end{mdframed}

\begin{proof}[Proof of Lemma \ref{VARAG-lemma-6}]
	For this case, $\alpha_s = p_s = \frac{1}{2}$, $\gamma_{s} = \frac{1}{6(d+4)L}$, $T_s = 2^{s-1}$.
	Starting from Lemma \ref{VARAG-lemma-5}, if we set $c = 1$ in Eq.~\eqref{VARAG-lemma-5-result} and sum it up from $t = 1$ to $T_s$, we have 
	\begin{align*}
	& \sum_{t=1}^{T_s} \frac{\gamma_{s}}{\alpha_{s}}\mathbb{E}_{u_t, i_t| \mathcal{F}_{t-1}} \big[f_{\mu}(\bar{x}_t)-f_{\mu}^{*}\big] +  \frac{1}{2} \mathbb{E} \big[\|x_{T_s}-x_{\mu}^*\|^2\big]\notag\\
	\le \  & \frac{\gamma_{s}}{2\alpha_{s}} \cdot T_s \big[f_{\mu}(\tilde{x})-f_{\mu}^{*}\big] + \frac{1}{2}\|x_{0}-x_{\mu}^*\|^2 + \frac{\gamma_{s}}{\alpha_{s}} \cdot T_s \cdot \mu^2L(d+4)^2 + T_s \cdot \frac{\gamma_{s}}{2\tau} \cdot E.
	\end{align*}
	Thanks to convexity of $f_{\mu}$, we have $f \left( \frac{1}{T_s} \sum_{t=1}^{T_s} \bar{x}_t \right) \leq \frac{1}{T_s} \sum_{t=1}^{T^s} f(\bar{x}_t)$. Hence, the last inequality implies
	\begin{align*}
	& \frac{1}{3(d+4)L} \cdot T_s \cdot \mathbb{E}_{\mathcal{F}_{T_s}} \big[f_{\mu}(\tilde{x}^{s})-f_{\mu}^{*}\big] +  \frac{1}{2} \mathbb{E}_{\mathcal{F}_{T_s}} \big[\|x^{s}-x_{\mu}^*\|^2\big] \notag\\
	\le \  & \frac{1}{6(d+4)L} \cdot T_s \big[f_{\mu}(\tilde{x}^{s-1})-f_{\mu}^{*}\big] + \frac{1}{2}\|x^{s-1}-x_{\mu}^*\|^2 + \frac{1}{3(d+4)L} \cdot T_s \cdot \mu^2L(d+4)^2 + T_s \cdot \frac{1}{12(d+4)L\tau} \cdot E \notag\\
	= \  & \frac{1}{3(d+4)L} \cdot T_{s-1} \big[f_{\mu}(\tilde{x}^{s-1})-f_{\mu}^{*}\big] + \frac{1}{2}\|x^{s-1}-x_{\mu}^*\|^2 + \frac{1}{3(d+4)L} \cdot T_s \cdot \mu^2L(d+4)^2  + T_s \cdot \frac{1}{12(d+4)L\tau} \cdot E,
	\end{align*}
	where $x_{T_s} = x^s$, $x_{0} = x^{s-1}$, $\tilde{x} = \tilde{x}^{s-1}$. Applying the last inequality iteratively, we obtain
	\begin{align}
	& \frac{1}{3(d+4)L} \cdot T_s \cdot \mathbb{E} \big[f_{\mu}(\tilde{x}^{s})-f_{\mu}^{*}\big] +  \frac{1}{2} \mathbb{E} \big[\|x^{s}-x_{\mu}^*\|^2\big] \notag\\
	\le \  & \frac{1}{3(d+4)L} \cdot T_{0} \big[f_{\mu}(\tilde{x}^{0})-f_{\mu}^{*}\big] + \frac{1}{2}\|x^{0}-x_{\mu}^*\|^2 + \frac{1}{3(d+4)L} \sum_{j=1}^{s} T_j \cdot \mu^2L(d+4)^2 + \sum_{j=1}^{s} T_j \cdot \frac{1}{12(d+4)L\tau} \cdot E, \notag
	\end{align}
	where $T_0 = \frac{1}{2}$ is in accordance with the definition of $T_s = 2^{s-1}, \ 0<s \le s_0$. Finally, we obtain
	\begin{align}
	& \mathbb{E} \big[f_{\mu}(\tilde{x}^{s})-f_{\mu}^{*}\big] + \frac{\alpha_{s}}{\gamma_{s} T_s} \cdot \frac{1}{2} \mathbb{E} \big[\|x^{s}-x_{\mu}^*\|^2\big] \notag\\
	= \  & \mathbb{E} \big[f_{\mu}(\tilde{x}^{s})-f_{\mu}^{*}\big] + \frac{3(d+4)L}{T_s} \cdot \frac{1}{2} \mathbb{E} \big[\|x^{s}-x_{\mu}^*\|^2\big] \notag\\
	\le \  & \frac{1}{2^s} \big[f_{\mu}(\tilde{x}^{0})-f_{\mu}^{*} + 3(d+4)L\|x^{0}-x_{\mu}^*\|^2 \big] + \frac{1}{2^{s-1}} \sum_{j=1}^{s} T_j \cdot \mu^2L(d+4)^2 + \frac{1}{2^{s-1}}\sum_{j=1}^{s} T_j \cdot \frac{1}{4\tau} \cdot E \notag\\
	\le \  & \frac{1}{2^{s+1}} (d+4)D_0 + 2\mu^2L(d+4)^2 + \frac{E}{2\tau}. \label{VARAG-lemma-6-ineq-1}
	\end{align}
	We conclude the proof by observing that $\frac{1}{2^{s-1}}\sum_{j=1}^{s} T_j \leq 2$ when $s \le s_0$.
\end{proof}

\begin{mdframed}
\begin{lemma}\label{VARAG-lemma-7}
	Assume \textbf{(A1)}, \textbf{(A3)}. Under the choice of parameters from Theorem~\ref{VARAG-theorem-2}, if $s \ge s_0$ and $n \ge \frac{6L}{\tau}$, then for any $x \in \mathbb{R}^d$
	\begin{align*}
	\mathbb{E}\big[f_{\mu}(\tilde{x}^s) - f_{\mu}^{*}\big] \le \left(\frac{4}{5}\right)^{s-s_0} \frac{D_0}{n}+ 12\mu^2L(d+4)^2 + \frac{5E}{\tau}.
	\end{align*}
\end{lemma}
\end{mdframed}
\begin{proof}[Proof of Lemma \ref{VARAG-lemma-7}]
	For this case, $\alpha_s = \alpha = p_s = \frac{1}{2}$, $\gamma_{s} = \gamma = \frac{1}{6(d+4)L}$, $T_s = 2^{s_0-1}$ when $s \ge s_0$. Thanks to Lemma \ref{VARAG-lemma-5}, if we set $c = \frac{1}{2}$ in Eq.~\eqref{VARAG-lemma-5-result}, we have
	\begin{multline*}
	\mathbb{E}_{u_t, i_t| \mathcal{F}_{t-1}}\bigg[\frac{\gamma}{\alpha}\big[f_{\mu}(\bar{x}_t)-f_{\mu}^{*}\big] + \big(1+\frac{\tau \gamma}{2}\big) \cdot \frac{1}{2}\|x_t-x_{\mu}^*\|^2\bigg]
	\\ \le \frac{\gamma}{2\alpha} \big[f_{\mu}(\tilde{x})-f_{\mu}^{*}\big] + \frac{1}{2}\|x_{t-1}-x_{\mu}^*\|^2 + \frac{\gamma}{\alpha} \cdot \mu^2L(d+4)^2 + \frac{\gamma}{\tau} \cdot E.
	\end{multline*}
	Multiplying both sides by $\Gamma_{t-1} = (1+\frac{\tau \gamma}{2})^{t-1}$, we obtain
	\begin{align*}
	& \mathbb{E}_{u_t, i_t| \mathcal{F}_{t-1}}\bigg[\frac{\gamma}{\alpha} \Gamma_{t-1} \big[f_{\mu}(\bar{x}_t)-f_{\mu}^{*}\big] + \frac{\Gamma_{t}}{2}\|x_t-x_{\mu}^*\|^2\bigg] \notag\\
	\le \  & \frac{\gamma}{2\alpha} \Gamma_{t-1} \big[f_{\mu}(\tilde{x})-f_{\mu}^{*}\big] + \frac{\Gamma_{t-1}}{2}\|x_{t-1}-x_{\mu}^*\|^2 + \frac{\gamma}{\alpha} \Gamma_{t-1} \cdot \mu^2L(d+4)^2 + \frac{\gamma}{\tau}\Gamma_{t-1} \cdot E.
	\end{align*}
	Since $\theta_t = \Gamma_{t-1}$ as defined in Eq.~\eqref{VARAG-def-theta-2}, the last inequality can be rewritten as
	\begin{multline*}
	\mathbb{E}_{u_t, i_t| \mathcal{F}_{t-1}}\bigg[\frac{\gamma}{\alpha} \theta_{t} \big[f_{\mu}(\bar{x}_t)-f_{\mu}^{*}\big] + \frac{\Gamma_{t}}{2}\|x_t-x_{\mu}^*\|^2\bigg]
	\\ \le \frac{\gamma}{2\alpha} \theta_{t} \big[f_{\mu}(\tilde{x})-f_{\mu}^{*}\big] + \frac{\Gamma_{t-1}}{2}\|x_{t-1}-x_{\mu}^*\|^2 + \frac{\gamma}{\alpha} \theta_{t} \cdot \mu^2L(d+4)^2 + \frac{\gamma}{\tau}\theta_{t}\cdot E.
	\end{multline*}
	Summing up the inequality above from $t = 1$ to $T_s$, we obtain
	\begin{align*}
	& \frac{\gamma}{\alpha} \sum_{t=1}^{T_s}\theta_{t} \mathbb{E}_{\mathcal{F}_t} \big[f_{\mu}(\bar{x}_t)-f_{\mu}^{*}\big] + \frac{\Gamma_{T_s}}{2} \mathbb{E}_{\mathcal{F}_{T_s}} \|x_{T_s}-x_{\mu}^*\|^2\notag\\
	\le \  & \frac{\gamma}{2\alpha} \sum_{t=1}^{T_s}\theta_{t} \mathbb{E}_{\mathcal{F}_{T_s}}\big[f_{\mu}(\tilde{x})-f_{\mu}^{*}\big] + \frac{1}{2}\mathbb{E}_{\mathcal{F}_{T_s}} \big[\|x_{0}-x_{\mu}^*\|^2 \big] + \frac{\gamma}{\alpha} \sum_{t=1}^{T_s}\theta_{t}  \cdot \mu^2L(d+4)^2 + \frac{\gamma}{\tau}\sum_{t=1}^{T_s} \theta_{t}\cdot E,
	\end{align*}
	and then
	\begin{align}
	& \frac{5}{4} \bigg[\frac{\gamma}{2\alpha} \sum_{t=1}^{T_s}\theta_{t} \mathbb{E}_{\mathcal{F}_t} \big[f_{\mu}(\bar{x}_t)-f_{\mu}^{*}\big] + \frac{1}{2} \mathbb{E}_{\mathcal{F}_{T_s}} \|x_{T_s}-x_{\mu}^*\|^2\bigg]\notag\\
	\le \  & \frac{\gamma}{2\alpha} \sum_{t=1}^{T_s}\theta_{t} \mathbb{E}_{\mathcal{F}_{T_s}} \big[f_{\mu}(\tilde{x})-f_{\mu}^{*}\big] + \frac{1}{2}\mathbb{E}_{\mathcal{F}_{T_s}}\big[\|x_{0}-x_{\mu}^*\|^2 \big] + \frac{\gamma}{\alpha} \sum_{t=1}^{T_s}\theta_{t}  \cdot \mu^2L(d+4)^2 + \frac{\gamma}{\tau}\sum_{t=1}^{T_s} \theta_{t}\cdot E, \label{VARAG-lemma-7-ineq-1}
	\end{align}
	The last inequality is based on the fact that, for $s \ge s_0$, $\frac{(d+4)n}{2}\le T_s = T_{s_0} \le (d+4)n$, we have
	\begin{align*}
	\Gamma_{T_s} & = \big(1+\frac{\tau \gamma}{2}\big)^{T_s} = \big(1+\frac{\tau \gamma}{2}\big)^{T_{s_0}} \ge 1+ \frac{\tau \gamma}{2} \cdot T_{s_0} \ge 1 + \frac{\tau \gamma}{2} \cdot \frac{(d+4)n}{2}\\
	& = 1 + \frac{\tau}{12(d+4)L} \cdot \frac{(d+4)n}{2} = 1 + \frac{\tau n}{24 L} \ge \frac{5}{4},
	\end{align*}
	where the last step holds under $n \ge \frac{6L}{\tau}$. Since $\tilde{x}^{s} = \sum_{t=1}^{T_s} (\theta_t \bar{x}_{t})/ \sum_{t=1}^{T_s} \theta_t$, $\tilde{x} = \tilde{x}^{s-1}$, $x_0 = x^{s-1}$, $x_{T_s} = x^s$ in the epoch $s$ and the convexity of $f_{\mu}$, Eq.~\eqref{VARAG-lemma-7-ineq-1} implies
	\begin{align*}
	& \frac{5}{4} \bigg[\frac{\gamma}{2\alpha} \mathbb{E}_{\mathcal{F}_{T_s}} \big[f_{\mu}(\tilde{x}^{s})-f_{\mu}^{*}\big] + \frac{1}{2 \sum_{t=1}^{T_s}\theta_{t}} \mathbb{E}_{\mathcal{F}_{T_s}} \|x^s-x_{\mu}^*\|^2\bigg]\notag\\
	\le \  & \frac{\gamma}{2\alpha} \mathbb{E}_{\mathcal{F}_{T_s}} \big[f_{\mu}(\tilde{x}^{s-1})-f_{\mu}^{*}\big] + \frac{1}{2 \sum_{t=1}^{T_s}\theta_{t}}\mathbb{E}_{\mathcal{F}_{T_s}}\big[\|x^{s-1}-x_{\mu}^*\|^2 \big] + \frac{\gamma}{\alpha} \cdot \mu^2L(d+4)^2 + \frac{\gamma}{\tau} \cdot E.
	\end{align*}
	Applying it recursively for $s \ge s_0$, we obtain
	\begin{align*}
	& \mathbb{E} \big[f_{\mu}(\tilde{x}^{s})-f_{\mu}^{*}\big] + \frac{2\alpha}{\gamma \sum_{t=1}^{T_s}\theta_{t}} \cdot \frac{1}{2}\mathbb{E} \big[\|x^s-x_{\mu}^*\|^2\big] \notag\\
	\le \  & \big(4/5\big)^{s-s_0} \bigg[\mathbb{E} \big[f_{\mu}(\tilde{x}^{s_0})-f_{\mu}^{*}\big] + \frac{2\alpha}{\gamma \sum_{t=1}^{T_s}\theta_{t}} \cdot \frac{1}{2} \mathbb{E} \big[\|x^{s_0}-x_{\mu}^*\|^2 + \sum_{j = s_0+1}^{s} \big( 4/5 \big)^{s+1-j} \bigg[ 2\mu^2L(d+4)^2 + \frac{2\alpha}{\tau} \cdot E\bigg] \notag\\
	\le \  & \big(4/5\big)^{s-s_0} \bigg[\mathbb{E} \big[f_{\mu}(\tilde{x}^{s_0})-f_{\mu}^{*}\big] + \frac{2\alpha}{\gamma T_{s_0}} \cdot \frac{1}{2}\mathbb{E} \big[\|x^{s_0}-x_{\mu}^*\|^2\big]\bigg] + 8\mu^2L(d+4)^2 + \frac{4}{\tau} \cdot E,
	\end{align*}
	where the last inequality holds because $\sum_{j = s_0+1}^{s} \left(\frac{4}{5}\right)^{s+1-j} \le \frac{4}{5} \cdot \frac{1}{1-\frac{4}{5}} = 4$, $\sum_{t=1}^{T_s} \theta_t \ge T_s = T_{s_0}$ and $2\alpha \le 1$. Finally,
	\begin{align*}
	& \mathbb{E} \big[f_{\mu}(\tilde{x}^{s})-f_{\mu}^{*}\big] + \frac{2\alpha}{\gamma \sum_{t=1}^{T_s}\theta_{t}} \cdot \frac{1}{2}\mathbb{E} \big[\|x^s-x_{\mu}^*\|^2\big] \notag\\
	\le \  & \big(4/5\big)^{s-s_0} \bigg[\mathbb{E} \big[f_{\mu}(\tilde{x}^{s_0})-f_{\mu}^{*}\big] + \frac{2\alpha}{\gamma T_{s_0}} \cdot \frac{1}{2}\mathbb{E} \big[\|x^{s_0}-x_{\mu}^*\|^2\big]\bigg] + 8\mu^2L(d+4)^2 + \frac{4}{\tau} \cdot E \notag\\
	\le \  & \big(4/5 \big)^{s-s_0} 2\bigg[\mathbb{E} \big[f_{\mu}(\tilde{x}^{s_0})-f_{\mu}^{*}\big] + \frac{\alpha}{\gamma T_{s_0}} \cdot \frac{1}{2}\mathbb{E} \big[\|x^{s_0}-x_{\mu}^*\|^2\big]\bigg] + 8\mu^2L(d+4)^2 + \frac{4}{\tau} \cdot E \notag\\
	\le \  & \big(4/5 \big)^{s-s_0} 2 \cdot \big[\frac{1}{2^{s_0+1}} (d+4)D_0 + 2\mu^2L(d+4)^2 + \frac{E}{2\tau}\big] + 8\mu^2L(d+4)^2 + \frac{4}{\tau} \cdot E \notag\\
	\le \  & \big( 4/5 \big)^{s-s_0} \cdot \frac{(d+4)D_0}{2^{s_0}} + 12\mu^2L(d+4)^2 + \frac{5E}{\tau} \notag\\
	= \  & \big(4/5\big)^{s-s_0} \frac{(d+4)D_0}{2T_{s_0}}+ 12\mu^2L(d+4)^2 + \frac{5E}{\tau} \notag\\
	\le \  & \big(4/5\big)^{s-s_0} \frac{D_0}{n}+ 12\mu^2L(d+4)^2 + \frac{5E}{\tau},
	\end{align*}
	where the third inequality comes from Eq.~\eqref{VARAG-lemma-6-ineq-1} and the last inquality relies on $T_{s_0} \ge \frac{(d+4)n}{2}$.
\end{proof}
\begin{mdframed}
\begin{lemma}\label{VARAG-lemma-8}
	Assume \textbf{(A1)}, \textbf{(A3)}. Under the choice of parameters from Theorem~\ref{VARAG-theorem-2}, if $s \ge s_0$ and $n < \frac{6L}{\tau}$, then for any $x \in \mathbb{R}^d$,
	\begin{align*}
	\mathbb{E}\big[f_{\mu}(\tilde{x}^s) - f_{\mu}^{*}\big] \le \bigg(1+\frac{1}{4} \sqrt{\frac{n\tau}{6L}}\bigg)^{-(s-s_0)} \frac{D_0}{n} + \left(8\sqrt{\frac{6L}{n\tau}}+4\right)\mu^2L(d+4)^2 + \frac{5E}{\tau}.
	\end{align*}
\end{lemma}
\end{mdframed}

\begin{proof}[Proof of Lemma \ref{VARAG-lemma-8}]
	For this case, $\alpha_s = \alpha = \sqrt{\frac{n \tau}{24L}}$, $p_s = p = \frac{1}{2}$, $\gamma_{s} = \gamma = \frac{1}{(d+4)\sqrt{6nL\tau}}$, $T_s = T_{s_0} = 2^{s_0-1}$ when $s \ge s_0$. Based on Lemma \ref{VARAG-lemma-5}, if we set $c = \frac{1}{2}$ in Eq.~\eqref{VARAG-lemma-5-result}, we have
	\begin{align*}
	& \mathbb{E}_{u_t, i_t| \mathcal{F}_{t-1}}\bigg[\frac{\gamma}{\alpha}\big[f_{\mu}(\bar{x}_t)-f_{\mu}^{*}\big] + \big(1+\frac{\tau \gamma}{2}\big) \cdot \frac{1}{2}\|x_t-x_{\mu}^*\|^2\bigg] \notag\\
	\le \  & \frac{\gamma}{\alpha}(1-\alpha-p)\big[f_{\mu}(\bar{x}_{t-1})-f_{\mu}^{*}\big] + \frac{\gamma}{2\alpha} \big[f_{\mu}(\tilde{x})-f_{\mu}^{*}\big] + \frac{1}{2}\|x_{t-1}-x_{\mu}^*\|^2 + \frac{\gamma}{\alpha} \cdot \mu^2L(d+4)^2 + \frac{\gamma}{\tau} \cdot E.
	\end{align*}
	Multiplying both sides by $\Gamma_{t-1} = (1+\frac{\tau \gamma}{2})^{t-1}$, we obtain
	\begin{align*}
	& \mathbb{E}_{u_t, i_t| \mathcal{F}_{t-1}}\bigg[\frac{\gamma}{\alpha} \Gamma_{t-1} \big[f_{\mu}(\bar{x}_t)-f_{\mu}^{*}\big] + \frac{\Gamma_{t}}{2}\|x_t-x_{\mu}^*\|^2\bigg] \notag\\
	\le \  & \frac{\Gamma_{t-1} \gamma}{\alpha}(1-\alpha-p)\big[f_{\mu}(\bar{x}_{t-1})-f_{\mu}^{*}\big] + \frac{\Gamma_{t-1} \gamma p}{\alpha} \big[f_{\mu}(\tilde{x})-f_{\mu}^{*}\big]\\& + \frac{\Gamma_{t-1}}{2}\|x_{t-1}-x_{\mu}^*\|^2 + \frac{\gamma}{\alpha} \Gamma_{t-1} \cdot \mu^2L(d+4)^2 + \frac{\gamma}{\tau}\Gamma_{t-1} \cdot E.
	\end{align*}
	Summing up the inequality above from $t = 1$ to $T_s$, we obtain
	\begin{align*}
	& \frac{\gamma}{\alpha} \sum_{t=1}^{T_s}\theta_{t} \mathbb{E}_{\mathcal{F}_t} \big[f_{\mu}(\bar{x}_t)-f_{\mu}^{*}\big] + \frac{\Gamma_{T_s}}{2} \mathbb{E}_{\mathcal{F}_t} \|x_{T_s}-x_{\mu}^*\|^2\notag\\
	\le \  & \frac{\gamma}{\alpha} \big[1-\alpha - p + p\sum_{t=1}^{T_s}\Gamma_{t-1}\big] \big[f_{\mu}(\tilde{x})-f_{\mu}^{*}\big] + \frac{1}{2}\|x_{0}-x_{\mu}^*\|^2 + \frac{\gamma}{\alpha} \sum_{t=1}^{T_s} \Gamma_{t-1} \cdot \mu^2L(d+4)^2 + \frac{\gamma}{\tau}\sum_{t=1}^{T_s} \Gamma_{t-1}\cdot E.
	\end{align*}
	Since $\tilde{x}^{s} = \sum_{t=1}^{T_s} (\theta_t \bar{x}_{t})/ \sum_{t=1}^{T_s} \theta_t$, $\tilde{x} = \tilde{x}^{s-1}$, $x_0 = x^{s-1}$, $x_{T_s} = x^s$ in the epoch $s$, and thanks to the convexity of $f_{\mu}$, the last inequality implies, for $s > s_0$:
	\begin{align}
	& \frac{\gamma}{\alpha} \sum_{t=1}^{T_s}\theta_{t} \mathbb{E}_{\mathcal{F}_{T_s}} \big[f_{\mu}(\tilde{x}^{s})-f_{\mu}^{*}\big] + \frac{\Gamma_{T_{s_0}}}{2} \mathbb{E}_{\mathcal{F}_{T_s}} \|x^s-x_{\mu}^*\|^2\notag\\
	\le \  & \frac{\gamma}{\alpha} \big[1-\alpha - p + p\sum_{t=1}^{T_s}\Gamma_{t-1}\big] \big[f_{\mu}(\tilde{x}^{s-1})-f_{\mu}^{*}\big] + \frac{1}{2}\|x^{s-1}-x_{\mu}^*\|^2 + \frac{\gamma}{\alpha} \sum_{t=1}^{T_s} \Gamma_{t-1} \cdot \mu^2L(d+4)^2 + \frac{\gamma}{\tau}\sum_{t=1}^{T_s} \Gamma_{t-1}\cdot E. \label{VARAG-lemma-8-ineq-1}
	\end{align}
	Moreover, we have
	\begin{align*}
	\sum_{t=1}^{T_{s_0}} \theta_{t} = \  & \Gamma_{T_{s_0}-1} + \sum_{t=1}^{T_{s_0}-1} \big(\Gamma_{t-1}-(1-\alpha - p)\Gamma_{t}\big) \\
	= \  & \Gamma_{T_{s_0}}(1-\alpha -p) + \sum_{t=1}^{T_{s_0}} \big(\Gamma_{t-1}-(1-\alpha - p)\Gamma_{t}\big)\\
	= \  & \Gamma_{T_{s_0}}(1-\alpha -p) + \big[1-(1-\alpha -p)(1+\frac{\tau \gamma}{2})\big]\sum_{t=1}^{T_{s_0}} \Gamma_{t-1}.\\
	\end{align*}
	Considering the range of $\alpha_{s}$, since $T_{s_0} \le (d+4)n$,
	\begin{align*} 
	\alpha & = \sqrt{\frac{n\tau}{24L}} \ge \sqrt{\frac{T_{s_0} \tau}{24(d+4)L}} = \frac{1}{2} \cdot \frac{\tau}{d+4} \cdot \frac{1}{\sqrt{6nL\tau}} \cdot \sqrt{T_{s_0}(d+4)n} \\
	& = \frac{\tau \gamma}{2} \cdot \sqrt{T_{s_0}(d+4)n} \ge \frac{\tau \gamma T_{s_0}}{2}.
	\end{align*}
	Also note that, for any $T > 1$ and $0 \le \delta T \le 1$, $(1+T\delta) \le (1+\delta)^T \le (1+2T \delta)$. If we set $\delta = \frac{\tau \gamma}{2}$ and $T = T_{s_0}$ here,
	\begin{align*}
	\delta T = \frac{\tau \gamma T_{s_0}}{2} \le \alpha < 1.
	\end{align*}
	Then, we have
	\begin{align*}
	1-(1-\alpha -p)(1+\frac{\tau \gamma}{2}) = \  & (1+\frac{\tau \gamma}{2})(\alpha+p-\frac{\tau \gamma}{2}) + \frac{\tau^2\gamma^2}{4} \\
	\ge \ & (1+\frac{\tau \gamma}{2})(\frac{\tau \gamma T_{s_0}}{2} + p-\frac{\tau \gamma}{2})\\
	= \  & p(1+\frac{\tau \gamma}{2})(1+ 2(T_{s_0}-1)\cdot \frac{\tau \gamma}{2})\\
	\ge \ &  p(1+\frac{\tau \gamma}{2})^{T_{s_0}} = p\Gamma_{T_{s_0}}.
	\end{align*}
Hence, we obtain $\sum_{t=1}^{T_{s_0}} \theta_{t} \ge \Gamma_{T_{s_0}} \cdot \big[1-\alpha - p + p\sum_{t=1}^{T_s}\Gamma_{t-1}\big]$. Moreover, using Eq.~\eqref{VARAG-lemma-8-ineq-1} and the fact that $f_{\mu}(\tilde{x}^{s}) - f_{\mu}^{*} \ge 0$, we have
	\begin{align*}
	& \Gamma_{T_{s_0}} \cdot \bigg[\frac{\gamma}{\alpha}\big[1-\alpha - p + p\sum_{t=1}^{T_s}\Gamma_{t-1}\big] \mathbb{E}_{\mathcal{F}_{T_s}} \big[f_{\mu}(\tilde{x}^{s})-f_{\mu}^{*}\big] + \frac{1}{2} \mathbb{E}_{\mathcal{F}_{T_s}} \|x^s-x_{\mu}^*\|^2\bigg]\notag\\
	\le \  & \frac{\gamma}{\alpha} \big[1-\alpha - p + p\sum_{t=1}^{T_s}\Gamma_{t-1}\big] \big[f_{\mu}(\tilde{x}^{s-1})-f_{\mu}^{*}\big] + \frac{1}{2}\|x^{s-1}-x_{\mu}^*\|^2 + \frac{\gamma}{\alpha} \sum_{t=1}^{T_s} \Gamma_{t-1} \cdot \mu^2L(d+4)^2 + \frac{\gamma}{\tau}\sum_{t=1}^{T_s} \Gamma_{t-1}\cdot E.
	\end{align*}
	Applying this inequality iteratively for $s > s_0$, we obtain
	\begin{align*}
	& \frac{\gamma}{\alpha}\big[1-\alpha - p + p\sum_{t=1}^{T_s}\Gamma_{t-1}\big] \mathbb{E} \big[f_{\mu}(\tilde{x}^{s})-f_{\mu}^{*}\big] + \frac{1}{2} \mathbb{E} \|x^s-x_{\mu}^*\|^2 \notag\\
	\le \  & \left(\frac{1}{\Gamma_{T_{s_0}}}\right)^{s-s_0} \bigg[\frac{\gamma}{\alpha} \big[1-\alpha - p + p\sum_{t=1}^{T_s}\Gamma_{t-1}\big] \mathbb{E} \big[f_{\mu}(\tilde{x}^{s_0})-f_{\mu}^{*}\big] + \frac{1}{2}\|x^{s_0}-x_{\mu}^*\|^2\bigg]\notag\\
	& + \sum_{j=1}^{s-s_0} \left(\frac{1}{\Gamma_{T_{s_0}}}\right)^{j}\bigg[\frac{\gamma}{\alpha} \sum_{t=1}^{T_s} \Gamma_{t-1} \cdot \mu^2L(d+4)^2 + \frac{\gamma}{\tau}\sum_{t=1}^{T_s} \Gamma_{t-1}\cdot E\bigg].
	\end{align*}
	Note that $\frac{\gamma}{\alpha}\big[1-\alpha - p + p\sum_{t=1}^{T_s}\Gamma_{t-1}\big] \ge \frac{\gamma p}{\alpha}\sum_{t=1}^{T_s}\Gamma_{t-1} \ge \frac{\gamma p T_s}{\alpha} = \frac{\gamma p T_{s_0}}{\alpha}$ and $p = \frac{1}{2}$, the inequality above implies
	\begin{align*}
	& \mathbb{E} \big[f_{\mu}(\tilde{x}^{s})-f_{\mu}^{*}\big] \notag\\
	\le \  & \left(\frac{1}{\Gamma_{T_{s_0}}}\right)^{s-s_0} \bigg[\mathbb{E} \big[f_{\mu}(\tilde{x}^{s_0})-f_{\mu}^{*}\big] + \frac{ \alpha}{\gamma T_{s_0}} \mathbb{E}\big[\|x^{s_0}-x_{\mu}^*\|^2\big]\bigg] + \sum_{j=1}^{s-s_0} \left(\frac{1}{\Gamma_{T_{s_0}}}\right)^{j}\bigg[ 2\mu^2L(d+4)^2 + \frac{2 \alpha}{\tau} \cdot E\bigg]\\
	\le \  & \left(\frac{1}{\Gamma_{T_{s_0}}}\right)^{s-s_0} \bigg[\mathbb{E} \big[f_{\mu}(\tilde{x}^{s_0})-f_{\mu}^{*}\big] + \frac{ \alpha}{\gamma T_{s_0}} \mathbb{E}\big[\|x^{s_0}-x_{\mu}^*\|^2\big]\bigg] + \frac{1}{\Gamma_{T_{s_0}}-1}\bigg[ 2\mu^2L(d+4)^2 + \frac{2 \alpha}{\tau}\cdot  E\bigg].
	\end{align*}
	Next, as
	$$\Gamma_{T_{s_0}} = \big(1+\frac{\tau \gamma}{2}\big)^{T_{s_0}} \ge 1+ \frac{\tau \gamma T_{s_0}}{2} \ge 1+\frac{\tau \gamma (d+4)n}{4} = 1+\frac{1}{4}\cdot \sqrt{\frac{n\tau}{6L}}$$ and $\frac{2 \alpha}{\tau} = \sqrt{\frac{n}{6L\tau}},$ we have that, for $s > s_0,$
	\begin{align*}
	& \mathbb{E} \big[f_{\mu}(\tilde{x}^{s})-f_{\mu}^{*}\big] \notag\\
	\le \  & \left(\frac{1}{\Gamma_{T_{s_0}}}\right)^{s-s_0} \bigg[\mathbb{E} \big[f_{\mu}(\tilde{x}^{s_0})-f_{\mu}^{*}\big] + \frac{ \alpha}{\gamma T_{s_0}} \mathbb{E}\big[\|x^{s_0}-x_{\mu}^*\|^2\big] \bigg] + 4\sqrt{\frac{6L}{n\tau}}\bigg[ 2\mu^2L(d+4)^2 + \sqrt{\frac{n}{6L\tau}} \cdot  E\bigg]\\
	= \  & \left(\frac{1}{\Gamma_{T_{s_0}}}\right)^{s-s_0} \bigg[\mathbb{E} \big[f_{\mu}(\tilde{x}^{s_0})-f_{\mu}^{*}\big] + \frac{ \alpha}{\gamma T_{s_0}} \mathbb{E}\big[\|x^{s_0}-x_{\mu}^*\|^2\big]\bigg] + 8\sqrt{\frac{6L}{n\tau}}\mu^2L(d+4)^2 + \frac{4E}{\tau}\\
	\le \  & \left(\frac{1}{\Gamma_{T_{s_0}}}\right)^{s-s_0} 2\bigg[\mathbb{E} \big[f_{\mu}(\tilde{x}^{s_0})-f_{\mu}^{*}\big] + \frac{ \alpha}{\gamma T_{s_0}} \cdot \frac{1}{2} \mathbb{E}\big[\|x^{s_0}-x_{\mu}^*\|^2\big]\bigg] + 8\sqrt{\frac{6L}{n\tau}}\mu^2L(d+4)^2 + \frac{4E}{\tau}.
	\end{align*}
	Note that, since $n < \frac{6L}{\tau}$, we have $\frac{\alpha}{\gamma} = 12(d+4)L\alpha^2 = \frac{(d+4)\tau n}{2} \le 3(d+4)L$. Finally, for $s > s_0$,
	\begin{align*}
	& \mathbb{E} \big[f_{\mu}(\tilde{x}^{s})-f_{\mu}^{*}\big] \notag\\
	\le \  & \left(\frac{1}{\Gamma_{T_{s_0}}}\right)^{s-s_0} 2\bigg[\mathbb{E} \big[f_{\mu}(\tilde{x}^{s_0})-f_{\mu}^{*}\big] + \frac{ 3(d+4)L}{T_{s_0}} \cdot \frac{1}{2} \mathbb{E}\big[\|x^{s_0}-x_{\mu}^*\|^2\big]\bigg] + 8\sqrt{\frac{6L}{n\tau}}\mu^2L(d+4)^2 + \frac{4E}{\tau}\\
	\le \  & \left(\frac{1}{\Gamma_{T_{s_0}}}\right)^{s-s_0} 2\bigg[\frac{1}{2^{s_0+1}} (d+4)D_0 + 2\mu^2L(d+4)^2 + \frac{E}{2\tau}\bigg] + 8\sqrt{\frac{6L}{n\tau}}\mu^2L(d+4)^2 + \frac{4E}{\tau}\\
	\le \  & \left(\frac{1}{\Gamma_{T_{s_0}}}\right)^{s-s_0} \frac{(d+4)D_0}{2^{s_0}} + \left(8\sqrt{\frac{6L}{n\tau}}+4\right)\mu^2L(d+4)^2 + \frac{5E}{\tau}\\
	= \  & \left(\frac{1}{\Gamma_{T_{s_0}}}\right)^{s-s_0} \frac{(d+4)D_0}{2T_{s_0}} + \left(8\sqrt{\frac{6L}{n\tau}}+4\right)\mu^2L(d+4)^2 + \frac{5E}{\tau}\\
	\le \  & \left(\frac{1}{\Gamma_{T_{s_0}}}\right)^{s-s_0} \frac{D_0}{n} + \left(8\sqrt{\frac{6L}{n\tau}}+4\right)\mu^2L(d+4)^2 + \frac{5E}{\tau}\\
	= \  & \bigg(1+\frac{1}{2(d+4)} \cdot \sqrt{\frac{\tau}{6nL}}\bigg)^{-T_{s_0}(s-s_0)} \frac{D_0}{n} + \left(8\sqrt{\frac{6L}{n\tau}}+4\right)\mu^2L(d+4)^2 + \frac{5E}{\tau}\\
	\le \  & \bigg(1+\frac{1}{2(d+4)} \cdot \sqrt{\frac{\tau}{6nL}}\bigg)^{-\frac{n(d+4)(s-s_0)}{2}} \frac{D_0}{n} + \left(8\sqrt{\frac{6L}{n\tau}}+4\right)\mu^2L(d+4)^2 + \frac{5E}{\tau}\\
	\le \  & \bigg(1+\frac{1}{4} \cdot \sqrt{\frac{n\tau}{6L}}\bigg)^{-(s-s_0)} \frac{D_0}{n} + \left(8\sqrt{\frac{6L}{n\tau}}+4\right)\mu^2L(d+4)^2 + \frac{5E}{\tau}.
	\end{align*}
	The second inequality is based on Eq.~\eqref{VARAG-lemma-6-ineq-1} and the fourth and fifth inequalities rely on $T_{s_0} \ge \frac{(d+4)n}{2}$. The last inequality comes from $1+T\delta \le (1+\delta)^T$ when $\delta \ge 0$.
\end{proof}

Now, we can derive Theorem \ref{VARAG-theorem-2} based on Lemma \ref{VARAG-lemma-6}, Lemma \ref{VARAG-lemma-7}, Lemma \ref{VARAG-lemma-8}.

\begin{proof}[Proof of Theorem \ref{VARAG-theorem-2}]
To summarize, we have obtained
\begin{align}
\mathbb{E} \big[f_{\mu}(\tilde{x}^{s})-f_{\mu}^{*}\big] := 
\begin{cases}
\frac{1}{2^{s+1}} (d+4)D_0 + 2 \mu^2L(d+4)^2 + \frac{E}{2\tau}, &\quad \quad \quad \quad 1 \le s \le s_0\\
& \\
\left(\frac{4}{5}\right)^{s-s_0} \frac{D_0}{n}+ 12\mu^2L(d+4)^2 + \frac{5E}{\tau}, &\quad \quad \quad \quad s > s_0 \text{ and } n \ge \frac{6L}{\tau}\\
& \\
\bigg(1+\frac{1}{4} \sqrt{\frac{n\tau}{6L}}\bigg)^{-(s-s_0)} \frac{D_0}{n}  &\quad \quad \quad \quad  s > s_0 \text{ and } n < \frac{6L}{\tau}\\
 \ \ \ \ \ \ \ \ + \left(8\sqrt{\frac{6L}{n\tau}}+4\right)\mu^2L(d+4)^2 + \frac{5E}{\tau}, &
\end{cases}
\end{align}
from  Lemma \ref{VARAG-lemma-6}, Lemma \ref{VARAG-lemma-7}, Lemma \ref{VARAG-lemma-8}. Hence, the proof of Theorem \ref{VARAG-theorem-2} is completed.
\end{proof} 

We conclude by deriving the final complexity result, stated in the main paper.

\begin{proof}[Proof of Corollary~\ref{VARAG-corollary-2}]
Using the same technique as for the proof of Corollary~\ref{VARAG-corollary-1}, we can make the error terms depending on $E$ or $\mu$ vanish. In addition to $\mu \le \mathcal{O}\big(\sqrt{\frac{\epsilon}{Ld}}\big)$ comeing from functional approximation error (see proof of Corollary 3), we also need $\mu = \mathcal{O} \big(\frac{\epsilon^{1/2}}{L^{1/2}d}\big)$ for the first two cases ($1\le s \le s_0$ or $s > s_0 \text{ and } n \ge \frac{6L}{\tau}$), $\mu = \mathcal{O} \big(\frac{n^{1/4}\tau^{1/4}\epsilon^{1/2}}{L^{3/4}d}\big)$ for the third case ($s > s_0 \text{ and } n < \frac{6L}{\tau}$) and $\mu = \mathcal{O} \big( \frac{\tau^{1/2}\epsilon^{1/2}}{L d^{3/2}}\big)$, $\nu = \mathcal{O} \big( \frac{\tau^{1/2}\epsilon^{1/2}}{L d^{1/2}}\big)$ to ensure $\epsilon$-optimality, $\frac{\epsilon}{4}$ more specifically. Hence, we can proceed as in~\cite{lan2019unified} , neglecting the errors coming from the DFO framework~(note that a similar procedure is adopted also in~\cite{nesterov2017random} and~\cite{liu2018stochastic,liu2018stochasticb})  . For the first case~($1\le s \le s_0$) the total number of function queries is given in Corollary~\ref{VARAG-corollary-1}. Then, in the second case~($s > s_0 \text{ and } n \ge \frac{6L}{\tau}$), the algorithm run at most $S := \mathcal{O}\big\{\log \big(\frac{(d+4)D_0}{\epsilon}\big)\big\}$ epochs to ensure the first error with $\epsilon$-optimality. Thus, the total number of function queries in this case is bounded by
\begin{align}
dnS + \sum_{s=1}^{S}T_s \le dn + S(d+4)n = \mathcal{O}\left\{dn\log \left(\frac{d D_0}{\epsilon}\right)\right\}.
\end{align}
Finally, in the last case ($s > s_0 \text{ and } n < \frac{6L}{\tau}$) to achieve $\epsilon$-error for the first term, the algorithm need to run at most $S^{'} := s_0 + \sqrt{\frac{6L}{n\tau}} \log \left(\frac{D_0}{n\epsilon}\right)$ epochs. Therefore,  the total number of function queries in this case is bounded by
\begin{align}
\sum_{s=1}^{S^{'}}(dn + T_s) = \  & \sum_{s=1}^{s_0}(dn + T_s) + (dn + T_{s_0})(S^{'}-s_0) \notag\\
\le \  & 2dns_0 + \big(dn+(d+4)n\big)\sqrt{\frac{6L}{n\tau}} \log \left(\frac{D_0}{n\epsilon}\right) \notag\\
= \  & \mathcal{O} \bigg\{dn \log(dn) + d \sqrt{\frac{nL}{\tau}} \log \left(\frac{D_0}{n\epsilon}\right)\bigg\}.
\end{align}
\end{proof}

\section{Proofs for Section~\ref{sec:coordinate-wise}: the coordinate-wise variant of Algorithm~\ref{algorithm-VARAG}}
\label{app:proofs_coord}

When we replace the gradient estimator $g_{\mu}(x, u, i)$ in Algorithm \ref{algorithm-VARAG} with Eq.~\eqref{DFO-framework-cord-finite-difference}, the dependency on the problem dimension $d$ gets better (Lemma \ref{VARAG-cord-lemma-1} compared to Lemma \ref{VARAG-lemma-1}), and the analysis looks more like the original Varag analysis \cite{lan2019unified}, with the addition of DFO errors. However, we should notice that Eq.~\eqref{DFO-framework-cord-finite-difference} requires $d$ times computation per iteration compared to Eq.~\eqref{DFO-framework-gaussian-smoothing}. From another point of view, choosing the gradient estimation in derivative-free optimization is a trade off between computation time and numerical accuracy.

The first lemma follows directly from Lemma 5 in \cite{lan2019unified} (we simplify it to the case with $V(z,x) = \frac{1}{2}\|z-x\|^2$, $X = \mathbb{R}^d$ and $h(x) = 0$). Note that this is very similar to Lemma \ref{VARAG-lemma-2}, but the Lemma below is with respect to $f$ rather than $f_{\mu}$. Indeed, for this appendix we define 
$$\delta_t := G_t - \nabla f(\underline{x}_t).$$

Also we recall that, to make the notation compact, we define
\begin{align*}
x_{t-1}^{+} := \frac{1}{1+\tau \gamma_s}(x_{t-1}+\tau \gamma_s \underline{x}_t), \quad \quad l_{f}(z,x) := f(z) + \langle \nabla f(z), x-z \rangle.
\end{align*}
\begin{mdframed}
\begin{lemma}\label{VARAG-cord-lemma-2}
	Consider the coordinate-wise variant of Algorithm~\ref{algorithm-VARAG}. Assume \textbf{(A1)}. For any $x \in \R^d$, we have
	\begin{multline*}
	\gamma_s [l_{f}(\underline x_t, x_t) - l_{f}(\underline x_t, x)] \\ \le \frac{\tau\gamma_s}{2}\|\underline{x_t} - x\|^2 + \frac{1}{2}\|x_{t-1} - x\|^2 - \frac{1+\tau \gamma_s}{2}\|x_t - x\|^2 - \frac{1+\tau \gamma_s}{2} \|x_t - x^+_{t-1}\|^2 - \gamma_s \langle \delta_t, x_t - x \rangle.
	\end{multline*}
	which can be rewritten as
	\begin{multline*}
	\gamma_s \langle \nabla f(\underline{x}_t), x_t - x \rangle \\ \le  \frac{\tau\gamma_s}{2}\|\underline{x_t} - x\|^2 + \frac{1}{2}\|x_{t-1} - x\|^2 - \frac{1+\tau \gamma_s}{2}\|x_t - x\|^2 - \frac{1+\tau \gamma_s}{2} \|x_t - x^+_{t-1}\|^2 - \gamma_s \langle \delta_t, x_t - x \rangle.
	\end{multline*}
\end{lemma}
\end{mdframed}

The next lemma is similar to Lemma 6 in~\cite{lan2019unified}, but with some additional error terms, due to DFO framework.
\begin{mdframed}
\begin{lemma}\label{VARAG-cord-lemma-3}
	Consider the coordinate-wise variant of Algorithm~\ref{algorithm-VARAG}. Assume \textbf{(A1)}. Assume that $\alpha_{s} \in [0,1]$, $p_s \in [0,1]$ and $\gamma_{s} > 0$ satisfy
	\begin{align}\label{VARAG-cord-lemma-3-assump-1}
	1+\tau \gamma_{s} - L\alpha_{s}\gamma_{s} > 0,
	\end{align}
	\begin{align}\label{VARAG-cord-lemma-3-assump-2}
	p_s - \frac{4L\alpha_{s}\gamma_{s}}{1+\tau \gamma_{s} - L\alpha_{s}\gamma_{s}} \ge 0.
	\end{align}
	Under the expectation of $i_t$, we have
	\begin{align}
	& \mathbb{E}_{i_t}\bigg[\frac{\gamma_{s}}{\alpha_{s}}\big[f(\bar{x}_t)-f(x)\big] + \frac{(1+\tau \gamma_{s})}{2}\|x_t-x\|^2\bigg] \notag\\
	\le \  & \frac{\gamma_{s}}{\alpha_{s}}(1-\alpha_{s}-p_s)\big[f(\bar{x}_{t-1})-f(x)\big]+\frac{\gamma_{s}p_s}{\alpha_{s}} \big[f(\tilde{x})-f(x)\big] + \frac{1}{2}\|x_{t-1}-x\|^2 \notag\\
	& + \frac{\gamma_{s}}{\alpha_{s}} \cdot \frac{6\alpha_{s}\gamma_{s}\nu^2L^2d}{1+\tau\gamma_{s}-L\alpha_{s}\gamma_{s}} - \frac{\gamma_{s}}{\alpha_{s}} \cdot \alpha_{s}\langle g_{\nu}(\underline{x}_t) - \nabla f(\underline{x}_t), x_{t-1}^+ - x \rangle. \label{VARAG-cord-lemma-3-result}
	\end{align}
	for any $x \in \mathbb{R}^d$.
\end{lemma}
\end{mdframed}
\begin{proof}[Proof of Lemma \ref{VARAG-cord-lemma-3}]
	By the $L$-smoothness of $f$,
	\begin{align*}
	f(\bar{x}_t) \le \  & f(\underline{x}_t) + \langle \nabla f(\underline{x}_t), \bar{x}_t - \underline{x}_t \rangle + \frac{L}{2}\|\bar{x}_t-\underline{x}_t\|^2\\
	= \  & (1-\alpha_{s}-p_s)\big[f(\underline{x}_t)+\langle \nabla f(\underline{x}_t), \bar{x}_{t-1} - \underline{x}_t \rangle\big] + \alpha_{s}\big[f(\underline{x}_t)+\langle \nabla f(\underline{x}_t), x_t - \underline{x}_t \rangle\big]\\
	& + p_s\big[f(\underline{x}_t)+\langle \nabla f(\underline{x}_t), \tilde{x} - \underline{x}_t \rangle\big] + \frac{L\alpha_{s}^2}{2}\|x_t - x_{t-1}^+\|^2.
	\end{align*}
	The equality above holds because of the update rule of $\bar{x}_t$ in Algorithm \ref{algorithm-VARAG} and Eq.~\eqref{VARAG-eq-1}. Then, applying Lemma \ref{VARAG-cord-lemma-2} for the inequality above, we have
	\begin{align}
	& f(\bar{x}_t) \notag\\
	\le \  & (1-\alpha_{s}-p_s)\big[f(\underline{x}_t)+\langle \nabla f(\underline{x}_t), \bar{x}_{t-1} - \underline{x}_t \rangle\big] + \alpha_{s}\big[f(\underline{x}_t)+\langle \nabla f(\underline{x}_t), x - \underline{x}_t \rangle\big] \notag\\
	& +\alpha_{s}\big[\frac{\tau}{2}\|\underline{x_t} - x\|^2 + \frac{1}{2\gamma_{s}}\|x_{t-1} - x\|^2 - \frac{1+\tau \gamma_s}{2\gamma_{s}}\|x_t - x\|^2 - \frac{1+\tau \gamma_s}{2\gamma_{s}} \|x_t - x^+_{t-1}\|^2 - \langle \delta_t, x_t - x \rangle\big] \notag\\
	& + p_s\big[f(\underline{x}_t)+\langle \nabla f(\underline{x}_t), \tilde{x} - \underline{x}_t \rangle\big] + \frac{L\alpha_{s}^2}{2}\|x_t - x_{t-1}^+\|^2 \notag\\
	\le \  & (1-\alpha_{s}-p_s)\big[f(\bar{x}_{t-1}) - \frac{\tau}{2}\|\bar{x}_{t-1}-\underline{x}_t\|^2\big] + \alpha_{s}\big[f(x) - \frac{\tau}{2}\|x-\underline{x}_t\|^2\big] \notag\\
	& +\alpha_{s}\big[\frac{\tau}{2}\|\underline{x_t} - x\|^2 + \frac{1}{2\gamma_{s}}\|x_{t-1} - x\|^2 - \frac{1+\tau \gamma_s}{2\gamma_{s}}\|x_t - x\|^2 \big] \notag\\
	& + p_s\big[f(\underline{x}_t)+\langle \nabla f(\underline{x}_t), \tilde{x} - \underline{x}_t \rangle\big] - \frac{\alpha_{s}}{2\gamma_{s}}(1+\tau \gamma_{s} - L\alpha_{s} \gamma_{s})\|x_t - x_{t-1}^+\|^2 -\alpha_{s}\langle \delta_t, x_t - x \rangle \notag\\
	= \  & (1-\alpha_{s}-p_s)\big[f(\bar{x}_{t-1}) - \frac{\tau}{2}\|\bar{x}_{t-1}-\underline{x}_t\|^2\big] +\alpha_{s}\big[f(x) + \frac{1}{2\gamma_{s}}\|x_{t-1} - x\|^2 - \frac{1+\tau \gamma_s}{2\gamma_{s}}\|x_t - x\|^2 \big] \notag\\
	& + p_s\big[f(\underline{x}_t)+\langle \nabla f(\underline{x}_t), \tilde{x} - \underline{x}_t \rangle\big] - \frac{\alpha_{s}}{2\gamma_{s}}(1+\tau \gamma_{s} - L\alpha_{s} \gamma_{s})\|x_t - x_{t-1}^+\|^2 -\alpha_{s}\langle \delta_t, x_t - x_{t-1}^{+} \rangle  -\alpha_{s}\langle \delta_t, x_{t-1}^{+} - x \rangle \notag\\
	\le \  & (1-\alpha_{s}-p_s)\big[f(\bar{x}_{t-1}) - \frac{\tau}{2}\|\bar{x}_{t-1}-\underline{x}_t\|^2\big] +\alpha_{s}\big[f(x) + \frac{1}{2\gamma_{s}}\|x_{t-1} - x\|^2 - \frac{1+\tau \gamma_s}{2\gamma_{s}}\|x_t - x\|^2 \big] \notag\\
	& + p_s\big[f(\underline{x}_t)+\langle \nabla f(\underline{x}_t), \tilde{x} - \underline{x}_t \rangle\big] +\frac{\alpha_{s}\gamma_{s}}{2(1+\tau \gamma_{s}-L\alpha_{s} \gamma_{s})}\|\delta_t\|^2 - \alpha_{s}\langle \delta_t, x_{t-1}^{+} - x \rangle. \label{VARAG-cord-smooth-ineq-1}
	\end{align}
	The second inequality holds thanks to (strong) convexity of $f$. The last inequality follows from $b \langle u,v \rangle - \frac{a}{2}\|v\|^2 \le \frac{b^2}{2a}\|u\|^2, \forall a >0$; where we set $a = \frac{\alpha_{s}}{\gamma_{s}}(1+\tau \gamma_{s} - L\alpha_{s} \gamma_{s})$ and $b = -\alpha_{s}$, requiring $1+\tau \gamma_{s} - L\alpha_{s} \gamma_{s} > 0$.\\
	\\
	Note that $\delta_t = G_t - \nabla f(\underline{x}_t)$ for the coordinate-wise variant. Taking the expectation w.r.t $i_t$, according to Lemma \ref{VARAG-cord-lemma-1},
	\begin{align}
	& p_s \big[f(\underline{x}_t)+\langle \nabla f(\underline{x}_t), \tilde{x} - \underline{x}_t \rangle\big] +\frac{\alpha_{s}\gamma_{s}}{2(1+\tau \gamma_{s}-L\alpha_{s} \gamma_{s})}\mathbb{E}_{i_t}\big[\|\delta_t\|^2\big] -\alpha_{s}\mathbb{E}_{i_t} \big[\langle \delta_t, x_{t-1}^{+} - x \rangle\big] \notag\\
	\le \  & p_s \big[f(\underline{x}_t)+\langle \nabla f(\underline{x}_t), \tilde{x} - \underline{x}_t \rangle\big] + \frac{6\alpha_{s}\gamma_{s}\cdot \nu^2L^2d}{1+\tau \gamma_{s} - L\alpha_{s}\gamma_{s}} +\frac{4\alpha_{s}\gamma_{s}L}{1+\tau \gamma_{s}-L\alpha_{s} \gamma_{s}}\big[f(\tilde{x})-f(\underline{x}_t) - \langle \nabla f(\underline{x}_t), \tilde{x} - \underline{x}_t \rangle \big] \notag\\
	& -\alpha_{s}\big[\langle g_{\nu}(\underline{x}_t) - \nabla f(\underline{x}_t), x_{t-1}^+ - x \rangle\big] \notag\\
	= \  & \big(p_s-\frac{4\alpha_{s}\gamma_{s}L}{1+\tau \gamma_{s}-L\alpha_{s} \gamma_{s}}\big) \big[f(\underline{x}_t)+\langle \nabla f(\underline{x}_t), \tilde{x} - \underline{x}_t \rangle\big] + \frac{6\alpha_{s}\gamma_{s}\cdot \nu^2L^2d}{1+\tau \gamma_{s} - L\alpha_{s}\gamma_{s}} +\frac{4\alpha_{s}\gamma_{s}L}{1+\tau \gamma_{s}-L\alpha_{s} \gamma_{s}}f(\tilde{x})\notag\\
	& -\alpha_{s}\langle g_{\nu}(\underline{x}_t) - \nabla f(\underline{x}_t), x_{t-1}^+ - x \rangle \notag\\
	\le \  & \big(p_s-\frac{4\alpha_{s}\gamma_{s}L}{1+\tau \gamma_{s}-L\alpha_{s} \gamma_{s}}\big) \big[f(\tilde{x}) - \frac{\tau}{2}\|\tilde{x} - \underline{x}_t\|^2\big] + \frac{6\alpha_{s}\gamma_{s}\cdot \nu^2L^2d}{1+\tau \gamma_{s} - L\alpha_{s}\gamma_{s}} +\frac{4\alpha_{s}\gamma_{s}L}{1+\tau \gamma_{s}-L\alpha_{s} \gamma_{s}}f(\tilde{x})\notag\\
	& -\alpha_{s}\langle g_{\nu}(\underline{x}_t) - \nabla f(\underline{x}_t), x_{t-1}^+ - x \rangle \notag\\
	= \  & p_s f(\tilde{x})-\big(p_s-\frac{4\alpha_{s}\gamma_{s}L}{1+\tau \gamma_{s}-L\alpha_{s} \gamma_{s}}\big) \cdot \frac{\tau}{2}\|\tilde{x} - \underline{x}_t\|^2 + \frac{6\alpha_{s}\gamma_{s}\cdot \nu^2L^2d}{1+\tau \gamma_{s} - L\alpha_{s}\gamma_{s}} -\alpha_{s}\langle g_{\nu}(\underline{x}_t) - \nabla f(\underline{x}_t), x_{t-1}^+ - x \rangle, \label{VARAG-cord-smooth-ineq-2}
	\end{align}
	where the last inequality holds if $p_s-\frac{4\alpha_{s}\gamma_{s}L}{1+\tau \gamma_{s}-L\alpha_{s} \gamma_{s}} \ge 0$. Combining Eq.~\eqref{VARAG-cord-smooth-ineq-1} with Eq.~\eqref{VARAG-cord-smooth-ineq-2}, we obtain
	\begin{align}
	&\mathbb{E}_{i_t} \big[f(\bar{x}_t) + \frac{\alpha_{s}(1+\tau \gamma_{s})}{2\gamma_{s}}\|x_t-x\|^2\big] \notag\\
	\le \  & (1-\alpha_{s}-p_s)f(\bar{x}_{t-1})+p_s f(\tilde{x}) + \alpha_{s} f(x) + \frac{\alpha_{s}}{2\gamma_{s}}\|x_{t-1}-x\|^2  + \frac{6\alpha_{s}\gamma_{s}\cdot \nu^2L^2d}{1+\tau \gamma_{s} - L\alpha_{s}\gamma_{s}}\notag\\
	& -\alpha_{s}\langle g_{\nu}(\underline{x}_t) - \nabla f(\underline{x}_t), x_{t-1}^+ - x \rangle - \frac{(1-\alpha_{s}-p_s)\tau}{2}\|\bar{x}_{t-1}-\underline{x}_t\|^2 -\big(p_s-\frac{4\alpha_{s}\gamma_{s}(d+4)L}{1+\tau \gamma_{s}-L\alpha_{s} \gamma_{s}}\big) \cdot \frac{\tau}{2}\|\tilde{x} - \underline{x}_t\|^2 \notag\\
	\le \  & (1-\alpha_{s}-p_s)f(\bar{x}_{t-1})+p_s f(\tilde{x}) + \alpha_{s} f(x) + \frac{\alpha_{s}}{2\gamma_{s}}\|x_{t-1}-x\|^2  + \frac{6\alpha_{s}\gamma_{s}\cdot \nu^2L^2d}{1+\tau \gamma_{s} - L\alpha_{s}\gamma_{s}}\notag\\
	& -\alpha_{s}\langle g_{\nu}(\underline{x}_t) - \nabla f(\underline{x}_t), x_{t-1}^+ - x \rangle. \notag
	\end{align}
	Multiplying both sides by $\frac{\gamma_{s}}{\alpha_{s}}$ and then rearranging the inequality, we finish the proof of this lemma, i.e. Eq.~\eqref{VARAG-cord-lemma-3-result}.
\end{proof}

\subsection{Proof of Theorem~\ref{VARAG-cord-theorem-1}}

To proceed, as in the Gaussian smoothing case, we need a technical assumption:
\begin{tcolorbox}
\textbf{(A2$_{\boldsymbol{\nu}}$)} \ \ Let $x^*\in \text{argmin}_{x\in\R^d} f(x)$. For any epoch $s$ of Algorithm~\ref{algorithm-VARAG}, consider the inner-loop sequences $\{ \underline x_t \}$ and $\{ \bar x_t \}$. There exist a \textit{finite} constant $Z<\infty$, potentially dependent on $L$ and $d$, such that, for $\nu$ small enough,
$$\sup_{s\ge0}\max_{\ \ x\in\{\bar x_t  \}\cup\{\underline x_t\}}\E\left[\|x-x^*\|\right]\le Z.$$
\end{tcolorbox}

Again, as mentioned in the context of \textbf{(A2$_{\boldsymbol{\mu}}$)}, it is possible to show that this assumption holds under the requirement that $f$ is coercive. As for Lemma \ref{VARAG-lemma-4}, thanks to \textbf{(A2$_{\boldsymbol{\nu}}$)}, we can get an epoch-wise inequality of the coordinate-wise approach.
\begin{mdframed}
\begin{lemma}\label{VARAG-cord-lemma-4}
Consider the coordinate-wise variant of Algorithm~\ref{algorithm-VARAG}. Assume \textbf{(A1)}, \textbf{(A2$_{\boldsymbol{\nu}}$)}. Set $\{\theta_t\}$ to
	\begin{align}\label{VARAG-cord-def-theta-1}
	\theta_t =
	\begin{cases}
	\tfrac{\gamma_{s}}{\alpha_{s}} (\alpha_{s} + p_{s}) & 1 \le t \le T_s-1\\
	\tfrac{\gamma_s}{\alpha_s} & t=T_s
	\end{cases}
	\end{align} 
	and define
	\begin{align}
	\mathcal{L}_s :=\frac{\gamma_{s}}{\alpha_{s}}+(T_s-1)\frac{\gamma_{s}(\alpha_{s}+p_s)}{\alpha_{s}};
	\end{align}
	\begin{align}
	\mathcal{R}_s := \frac{\gamma_{s}}{\alpha_{s}}(1-\alpha_{s})+(T_s-1)\frac{\gamma_{s}p_s}{\alpha_{s}}.
	\end{align}
	Under the conditions in Eq.~\eqref{VARAG-cord-lemma-3-assump-1} and Eq.~\eqref{VARAG-cord-lemma-3-assump-2}, we have:
	\begin{align*}
	\mathcal{L}_s \mathbb{E} \big[f(\tilde{x}^s)-f(x^*)\big]
	\le \  & \mathcal{R}_s \cdot \big[f(\tilde{x}^{s-1})-f(x^*)\big] + \big(\frac{1}{2}\|x^{s-1}-x^*\|^2-\frac{1}{2}\|x^s-x^*\|^2\big) \notag\\
	& + T_s \cdot \frac{\gamma_{s}}{\alpha_{s}} \cdot \frac{6\alpha_{s}\gamma_{s}\nu^2L^2d}{1-L\alpha_{s}\gamma_{s}} + T_s \cdot \frac{\gamma_{s}}{\alpha_{s}}(2-\alpha_s)L \sqrt{d}Z \nu,
	\end{align*}
	where $x^{*} := \arg \min_{x\in\R^d} f(x)$.
\end{lemma}
\end{mdframed}

\begin{proof}[Proof of Lemma \ref{VARAG-cord-lemma-4}]
	 If we set $x = x^*$, Lemma \ref{VARAG-cord-lemma-3} can be written as
	\begin{align*}
	& \mathbb{E}_{i_t}\bigg[\frac{\gamma_{s}}{\alpha_{s}}\big[f(\bar{x}_t)-f(x^*)\big] + \frac{1}{2}\|x_t-x^*\|^2\bigg] \notag\\
	\le \  & \frac{\gamma_{s}}{\alpha_{s}}(1-\alpha_{s}-p_s)\big[f(\bar{x}_{t-1})-f(x^*)\big]+\frac{\gamma_{s}p_s}{\alpha_{s}} \big[f(\tilde{x})-f(x^*)\big] + \frac{1}{2}\|x_{t-1}-x^*\|^2 \notag\\
	& + \frac{\gamma_{s}}{\alpha_{s}} \cdot \frac{6\alpha_{s}\gamma_{s}\nu^2L^2d}{1-L\alpha_{s}\gamma_{s}} - \frac{\gamma_{s}}{\alpha_{s}} \cdot \alpha_{s}\langle g_{\nu}(\underline{x}_t) - \nabla f(\underline{x}_t), x_{t-1}^+ - x^* \rangle.
	\end{align*}
	Summing up these inequalities over $t = 1, \dots, T_s$, using the definition of $\theta_t$ and $\bar{x}_0 = \tilde{x}$, we get
	\begin{align*}
	\sum_{t=1}^{T_s}\theta_t \mathbb{E} \big[f(\bar{x}_t)-f(x^*)\big]
	\le \  & \left[\frac{\gamma_{s}}{\alpha_{s}}(1-\alpha_{s})+ (T_s-1)\frac{\gamma_{s}p_s}{\alpha_{s}}\right] \cdot \big[f(\tilde{x})-f(x^*)\big] + \left(\frac{1}{2}\|x_0-x^*\|^2-\frac{1}{2}\|x_{T_s}-x^*\|^2\right) \notag\\
	& + T_s \cdot \frac{\gamma_{s}}{\alpha_{s}} \cdot \frac{6\alpha_{s}\gamma_{s}\nu^2L^2d}{1-L\alpha_{s}\gamma_{s}} - \frac{\gamma_{s}}{\alpha_{s}} \sum_{t=1}^{T_s} \alpha_{s} \mathbb{E}\big[\langle g_{\nu}(\underline{x}_t) - \nabla f(\underline{x}_t), x_{t-1}^+ - x^* \rangle \big].
	\end{align*}
	Noticing that $\tilde{x}^s = \sum_{t=1}^{T_s}\big(\theta_t \bar{x}_t\big)/\sum_{t=1}^{T_s}\theta_t$, $\tilde{x} = \tilde{x}^{s-1}$, $x_0 = x^{s-1}$, $x_{T_s} = x^s$ and thanks to the convexity of $f$, the inequality above implies
	\begin{align*}
	\sum_{t=1}^{T_s}\theta_t \mathbb{E} \big[f(\tilde{x}^s)-f(x^*)\big]
	\le \  & \left[\frac{\gamma_{s}}{\alpha_{s}}(1-\alpha_{s})+ (T_s-1)\frac{\gamma_{s}p_s}{\alpha_{s}}\right] \cdot \big[f(\tilde{x}^{s-1})-f(x^*)\big] + \left(\frac{1}{2}\|x^{s-1}-x^*\|^2-\frac{1}{2}\|x^s-x^*\|^2\right) \notag\\
	& + T_s \cdot \frac{\gamma_{s}}{\alpha_{s}} \cdot \frac{6\alpha_{s}\gamma_{s}\nu^2L^2d}{1-L\alpha_{s}\gamma_{s}} - \frac{\gamma_{s}}{\alpha_{s}} \sum_{t=1}^{T_s} \alpha_{s} \mathbb{E}\big[\langle g_{\nu}(\underline{x}_t) - \nabla f(\underline{x}_t), x_{t-1}^+ - x^* \rangle \big],
	\end{align*}
	which is equivalent to
	\begin{align}
	\mathcal{L}_s \mathbb{E} \big[f(\tilde{x}^s)-f(x^*)\big]
	\le \  & \mathcal{R}_s \cdot \big[f(\tilde{x}^{s-1})-f(x^*)\big] + \left(\frac{1}{2}\|x^{s-1}-x^*\|^2-\frac{1}{2}\|x^s-x^*\|^2\right) \notag\\
	& + T_s \cdot \frac{\gamma_{s}}{\alpha_{s}} \cdot \frac{6\alpha_{s}\gamma_{s}\nu^2L^2d}{1-L\alpha_{s}\gamma_{s}} - \frac{\gamma_{s}}{\alpha_{s}} \sum_{t=1}^{T_s} \alpha_{s} \mathbb{E}\big[\langle g_{\nu}(\underline{x}_t) - \nabla f(\underline{x}_t), x_{t-1}^+ - x^* \rangle \big]. \label{VARAG-cord-lemma-4-ineq-1}
	\end{align}
	Now, let us look at the additional term $\langle g_{\nu}(\underline{x}_t) - \nabla f(\underline{x}_t), x_{t-1}^+ - x^* \rangle$, \textit{which can not eliminated by expectation compared to gradient-based Varag }\cite{lan2019unified}. According to Eq.~\eqref{VARAG-eq-1} and the update rule of $\bar{x}_t$ in Algorithm \ref{algorithm-VARAG},
	we have
	\begin{align*}
	\alpha_s(x_{t-1}^+ - x^*) = \  & \alpha_s x_t + \underline{x}_t - \bar{x}_t - \alpha_s x^*\\
	= \  & - (1-\alpha_s-p_s)\bar{x}_{t-1} - p_s \tilde{x} + \underline{x}_t - \alpha_s x^*\\
	= \  & - (1-\alpha_s-p_s)(\bar{x}_{t-1}-x^*) - p_s (\tilde{x}-x^*) + (\underline{x}_t-x^*)\\
	\le \  & (1-\alpha_s-p_s)||\bar{x}_{t-1}-x^*|| + p_s ||\tilde{x}-x^*|| + ||\underline{x}_t-x^*||.
	\end{align*}
	Thanks to assumption \textbf{(A2$_{\boldsymbol{\nu}}$)}, we have
	\begin{align}
	\alpha_s \langle g_{\nu}(\underline{x}_t) - \nabla f(\underline{x}_t), x_{t-1}^+ - x^*\rangle \le \  & \|g_{\nu}(\underline{x}_t) - \nabla f(\underline{x}_t)\| \cdot \|\alpha_s(x_{t-1}^+ - x^*)\| \notag\\
	\le \  & (2-\alpha_s)L \sqrt{d}Z \nu. \label{VARAG-cord-surplus-error}
	\end{align}
	
	The last inequality comes from Lemma \ref{VARAG-cord-lemma-0}. Combining the previous inequality with Eq.~\eqref{VARAG-cord-lemma-4-ineq-1}, we have
	\begin{align*}
	\mathcal{L}_s \mathbb{E} \big[f(\tilde{x}^s)-f(x^*)\big]
	\le \  & \mathcal{R}_s \cdot \big[f(\tilde{x}^{s-1})-f(x^*)\big] + \big(\frac{1}{2}\|x^{s-1}-x^*\|^2-\frac{1}{2}\|x^s-x^*\|^2\big) \notag\\
	& + T_s \cdot \frac{\gamma_{s}}{\alpha_{s}} \cdot \frac{6\alpha_{s}\gamma_{s}\nu^2L^2d}{1-L\alpha_{s}\gamma_{s}} + T_s \cdot \frac{\gamma_{s}}{\alpha_{s}}(2-\alpha_s)L \sqrt{d}Z \nu.
	\end{align*}
\end{proof}

\begin{remark}
Here $\mathbb{E}_{i_t}\big[\delta_t\big] = g_{\nu}(\underline{x}_t) - \nabla f(\underline{x}_t) \neq 0$. Notice that the error terms in Eq.~\eqref{VARAG-cord-surplus-error}, i.e. $\langle g_{\nu}(\underline{x}_t) - \nabla f(\underline{x}_t), x_{t-1}^+ - x^* \rangle$, is different from its counterpart in Eq.~\eqref{VARAG-pivotal-error}, i.e. $\langle \tilde{g}-\nabla f_{\mu}(\tilde{x}), x_t - x^*\rangle$.
\end{remark}

\noindent Then, we can derive Theorem \ref{VARAG-cord-theorem-1} for convex and smooth $f_i$, based on Lemma \ref{VARAG-cord-lemma-4}. For convenience of the reader, we re-write the theorem here.
\begin{mdframed}
\textbf{Theorem 6.}	Consider the coordinate-wise variant of Algorithm~\ref{algorithm-VARAG}. Assume \textbf{(A1)} and \textbf{(A2$_{\boldsymbol{\nu}}$)}. Let us denote $s_0 := \lfloor \log n \rfloor+1$. Suppose the weights $\{\theta_t\}$ are set as in Eq.~\eqref{VARAG-def-theta-paper}
	and parameters $\{T_s\}$, $\{\gamma_s\}$, $\{p_s\}$ are set as
	\begin{align}
	T_s = \begin{cases}
	2^{s-1}, & s \le s_0\\
	T_{s_0}, & s > s_0
	\end{cases}, \
	\gamma_s = \tfrac{1}{12 L \alpha_s}, \
	\ \
	p_s = \tfrac{1}{2}, \ \mbox{with}
	\end{align}
	\begin{align}
	\alpha_s =
	\begin{cases}
	\tfrac{1}{2}, & s \le s_0\\
	\tfrac{2}{s-s_0+4},& s > s_0
	\end{cases}.
	\end{align}
	Then, we have
	\begin{align*}
	& \mathbb{E} \big[f(\tilde{x}^s)-f^*\big] \le\begin{cases} 
	\cfrac{D_0'}{2^{s+1}}  + \varsigma_1 + \varsigma_2, &  \ \ \ \ \ \  1 \le s \le s_0\\
	\cfrac{16 D_0'}{n(s-s_0+4)^2} + \delta_s \cdot (\varsigma_1+\varsigma_2), & \ \ \ \ \ \ s > s_0 \\
	\end{cases}
	\end{align*}
	where $\varsigma_1=\nu^2Ld$, $\varsigma_2=4L\sqrt{d}Z\nu$, $\delta_s = \mathcal{O}(s-s_0)$ and $D_0'$ is defined as
	\begin{align}
	D_0':= 2[f(x^0) - f(x^*)] + 6L \|x^0-x^* \|^2,
	\end{align}
	where $x^*$ is any finite minimizer of $f$.
\end{mdframed}

\begin{proof}[Proof of Theorem \ref{VARAG-cord-theorem-1}]
	Assumption Eq.~\eqref{VARAG-cord-lemma-3-assump-1} and Eq.~\eqref{VARAG-cord-lemma-3-assump-2} are satisfied since 
	\begin{align}
	1+\tau \gamma_{s} - L\alpha_{s}\gamma_{s} = 1 - \frac{1}{12} > 0,
	\end{align}
	\begin{align}
	p_s - \frac{4L\alpha_{s}\gamma_{s}}{1+\tau \gamma_{s} - L\alpha_{s}\gamma_{s}} = \frac{1}{2}-\frac{1}{3}\cdot \frac{1}{1-\frac{1}{12}} > 0.
	\end{align}
	We define
	\begin{align}
	w_s := \mathcal{L}_s - \mathcal{R}_{s+1}.
	\end{align}
	As in \cite{lan2019unified}, if $1\le s <s_0$, 
	\begin{align*}
	w_s = \mathcal{L}_s - \mathcal{R}_{s+1} = \frac{\gamma_{s}}{\alpha_s}\big[1+(T_s-1)(\alpha_{s}+p_s)-(1-\alpha_{s})-(2T_s-1)p_s\big] = \frac{\gamma_{s}}{\alpha_{s}}\big[T_s(\alpha_{s}-\gamma_{s})\big] = 0;
	\end{align*}
	else, if $s \ge s_0$,
	\begin{align*}
	w_s & = \mathcal{L}_s - \mathcal{R}_{s+1} = \frac{\gamma_{s}}{\alpha_{s}} -  \frac{\gamma_{s+1}}{\alpha_{s+1}}(1-\alpha_{s+1})+(T_{s_0}-1)\left[\frac{\gamma_{s}(\alpha_{s}+p_s)}{\alpha_{s}}-\frac{\gamma_{s+1}p_{s+1}}{\alpha_{s+1}}\right]\\
	& = \frac{1}{48L}+\frac{(T_{s_0}-1)\left[2(s-s_0+4)-1\right]}{96L} > 0.
	\end{align*}
	Hence, $w_s \ge 0$ for all $s$. Using Lemma \ref{VARAG-cord-lemma-4} iteratively,
	\begin{align}
	& \mathcal{L}_s \mathbb{E} \big[f(\tilde{x}^s)-f(x^*)\big] \notag\\
	\le \  & \mathcal{R}_1 \cdot \mathbb{E} \big[f(\tilde{x}^{0})-f(x^*)\big] + \mathbb{E}\left[\frac{1}{2}\|x^{0}-x^*\|^2-\frac{1}{2}\|x^s-x^*\|^2\right] + \sum_{j=1}^{s} T_j \cdot \frac{\gamma_{j}}{\alpha_{j}} \cdot \frac{6\alpha_{j}\gamma_{j}\nu^2L^2d}{1-L\alpha_{j}\gamma_{j}}\notag\\
	& + \sum_{j=1}^{s} T_j \cdot \frac{\gamma_{j}}{\alpha_{j}}(2-\alpha_j)L \sqrt{d}Z \nu \notag\\
	= \  & \frac{1}{6L} \big[f(\tilde{x}^{0})-f(x^*)\big] + \frac{1}{2}\|x^{0}-x^*\|^2 + \sum_{j=1}^{s} T_j \cdot \frac{\gamma_{j}}{\alpha_{j}} \cdot \frac{6\alpha_{j}\gamma_{j}\nu^2L^2d}{1-L\alpha_{j}\gamma_{j}} + \sum_{j=1}^{s} T_j \cdot \frac{\gamma_{j}}{\alpha_{j}}(2-\alpha_j)L \sqrt{d}Z \nu \notag\\
	= \  & \frac{1}{12L} D_0' + \sum_{j=1}^{s} \frac{1}{2} T_j \cdot \frac{\gamma_{j}}{\alpha_{j}} \cdot \frac{\nu^2Ld}{1-L\alpha_{j}\gamma_{j}} + \sum_{j=1}^{s} T_j \cdot \frac{\gamma_{j}}{\alpha_{j}}(2-\alpha_j)L \sqrt{d}Z \nu \notag\\
	\le \  & \frac{1}{12L} D_0' + \sum_{j=1}^{s} \frac{1}{2}T_j \cdot \frac{\gamma_{j}}{\alpha_{j}} \cdot \nu^2Ld + \sum_{j=1}^{s} 2 T_j \cdot \frac{\gamma_{j}}{\alpha_{j}}L \sqrt{d}Z \nu. \label{VARAG-cord-theorem-1-ineq-1}
	\end{align}
	The last equality holds since $\alpha_{j}\gamma_{j} = \frac{1}{12L}$. We proceed in two cases:\\
	\\
	\textbf{Case I:} If $s \le s_0$, $\mathcal{L}_s = \frac{2^{s+1}}{12L}$, $\mathcal{R}_s = \frac{2^{s}}{12L} = \frac{\mathcal{L}_s}{2}$, $\frac{\gamma_{s}}{\alpha_{s}} = \frac{1}{3L}$, $T_s = 2^{s-1}$. Hence, we have
	\begin{align}
	& \mathbb{E} \big[f(\tilde{x}^s)-f(x^*)\big] \le \frac{1}{2^{s+1}} D_0' + \nu^2Ld + 4L\sqrt{d}Z\nu, \qquad 1 \le s \le s_0.
	\label{VARAG-cord-theorem-1-ineq-2}
	\end{align}
	\textbf{Case II:} If $s > s_0$, we have
	\begin{align*}
	\mathcal{L}_s = \  & \frac{1}{12\alpha_s^2}\big[(T_s-1)\alpha_{s} + \frac{1}{2}(T_s+1)\big]\\
	= \  & \frac{(s-s_0+4)^2}{48L} \cdot \big[(T_{s_0}-1)\alpha_{s} + \frac{1}{2}(T_{s_0}+1)\big]\\
	\ge \ & \frac{(s-s_0+4)^2}{96L} \cdot (T_{s_0}+1)\\
	\ge \ & \frac{n \cdot (s-s_0+4)^2}{192L}.
	\end{align*}
	where the last inequality holds since $T_{s_0} = 2^{\lfloor \log_{2}n\rfloor} \ge \frac{n}{2}$, i.e. $2^{s_0} \ge n$. Hence, Eq.~\eqref{VARAG-cord-theorem-1-ineq-1} implies
	\begin{align}
	\mathbb{E} \big[f(\tilde{x}^s)-f(x^*)\big]
	\le \frac{16 D_0'}{n(s-s_0+4)^2} + \mathcal{O}(s-s_0) \cdot \nu^2Ld+ \mathcal{O}(s-s_0) \cdot L\sqrt{d}Z\nu.
	\label{VARAG-cord-theorem-1-ineq-3}
	\end{align}
\end{proof}

We can now derive the final complexity result.
\begin{proof}[Proof of Corollary \ref{VARAG-cord-corollary-1}]
Using the same technique as for the proof of Corollary~\ref{VARAG-corollary-1}, we can make the error terms depending on $\nu$ vanishing. This requires $\nu = \mathcal{O}\big(\frac{\epsilon^{1/2}}{L^{1/2} d^{1/2}} \big)$, $\nu = \mathcal{O}\big(\frac{\epsilon}{Ld^{1/2}Z}\big)$ for the first case ($1 \le s \le s_0$) while $\nu = \mathcal{O}\big(\frac{n^{1/4}\epsilon^{3/4}}{L^{1/2} d^{1/2} D_0'^{1/4}} \big)$, $\nu = \mathcal{O}\big(\frac{n^{1/2}\epsilon^{3/2}}{Ld^{1/2}Z D_0'^{1/2}}\big)$ for the second case ($s > s_0$) to ensure $\epsilon$-optimality, $\frac{\epsilon}{2}$ more specifically. Hence, we can proceed as in~\cite{lan2019unified} , neglecting the errors coming from the DFO framework~(note that a similar procedure is adopted also in~\cite{nesterov2017random} and~\cite{liu2018stochastic,liu2018stochasticb}). If $n \ge \frac{D_0'}{\epsilon}$, we require
	\begin{align*}
	\frac{D_0'}{2^{s_0+1}} \le \frac{D_0'}{2n}\le \frac{\epsilon}{2}.
	\end{align*}
	Therefore, the number of epochs can be bounded by $S_l = \min \left\{\log\left(\frac{D_0'}{\epsilon}\right),s_0\right\}$, achieving $\epsilon$ optimality inside Case I~(see proof of Theorem~\ref{VARAG-cord-theorem-1}). The total number of function queries is bounded by
	\begin{align*}
	 d\left(nS_l + \sum_{s=1}^{S_l}T_s\right) = d \cdot \mathcal{O} \bigg\{\min \bigg(n\log\left(\frac{D_0'}{\epsilon}\right), n \log (n), n\bigg)\bigg\} = d \cdot \mathcal{O} \bigg\{\min \left( n \log \left(\frac{D_0'}{\epsilon}\right),n \right)\bigg\},
	\end{align*}
	where the coefficient $d$ corresponds to the number of function queries for each gradient estimation. All in all, the number of function queries is $\mathcal{O} \left\{dn \log \left(\frac{D_0'}{\epsilon}\right)\right\}$.\\
	\\
	If $n < \frac{D_0'}{\epsilon}$ (Case II), we have $S_h = \left\lceil \sqrt{\frac{32D_0'}{n\epsilon}} + s_0 -4 \right\rceil$, ensuring the first term in Eq.~\eqref{VARAG-cord-theorem-1-ineq-3}  is not bigger than $\frac{\epsilon}{2}$. We can achieve $\epsilon$ optimality. Hence, the total number of function queries is
	\begin{align*}
	d\left[n s_0 + \sum_{s=1}^{s_0}T_s + (T_{s_0}+n)(S_h-s_0)\right] \le d \left[\sum_{s=1}^{s_0}T_s + (T_{s_0}+n)S_h\right]
	= \mathcal{O}\bigg\{d\sqrt{\frac{nD_0'}{\epsilon}} + dn \log(n) \bigg\}.
	\end{align*}
\end{proof}

\subsection{Proof of Theorem~\ref{VARAG-cord-theorem-2}}
In this section, we consider $f$ to be strongly convex, which we denoted as \textbf{(A3)}. We rewrite below Theorem \ref{VARAG-cord-theorem-2}, for convenience of the reader:
\begin{mdframed}
\textbf{Theorem 8.} Consider the coordinate-wise variant of Algorithm~\ref{algorithm-VARAG}. Assume \textbf{(A1)}, \textbf{(A2$_{\boldsymbol{\nu}}$)} and \textbf{(A3)}. Let us denote $s_0 := \lfloor \log n \rfloor+1$ and assume that the
	weights $\{\theta_t\}$ are set to Eq.~\eqref{VARAG-def-theta-1} if $1\le s \le s_0$. Otherwise, they are set to
	\vspace{-1mm}
	\begin{align}
	\theta_t =
	\begin{cases}
	\Gamma_{t-1} - (1 - \alpha_s - p_s) \Gamma_{t}, & 1 \le t \le T_s-1,\\
	\Gamma_{t-1}, & t = T_s,
	\end{cases}
	\end{align}
	where $\Gamma_t= \big(1+\tau\gamma_s\big)^t$.
	If the parameters $\{T_s\}$, $\{\gamma_s\}$ and $\{p_s\}$ set  to
		\vspace{-2mm}
	Eq.~\eqref{cord-parameter-deter-smooth1} with 
	\begin{align}
	\alpha_s =
	\begin{cases}
	\tfrac{1}{2}, & s \le s_0,\\
	\min\{\sqrt{\frac{n \tau}{12L}}, \tfrac{1}{2}\},& s > s_0.
	\end{cases}
	\end{align}
	\vspace{-1.5mm}
	We obtain
	\begin{align*}
	& \mathbb{E} \big[f(\tilde{x}^{s})-f^*\big] \le  \begin{cases}
	\cfrac{1}{2^{s+1}} D_0' + \varsigma_1 + \varsigma_2, & \ \ \ \ 1 \le s \le s_0\\
	& \\
	(4/5)^{s-s_0}\cfrac{D_0'}{n} + \varsigma_1 + \varsigma_2,  &  \ \ \ \ s > s_0 \text{ and } n \ge \frac{3L}{\tau}\\
	& \\
	\left(1+\frac{1}{4} \sqrt{\frac{n\tau}{3L}}\right)^{-(s-s_0)} \cfrac{D_0'}{n} + \big(2\sqrt{\frac{3L}{n\tau}}+1\big)(\varsigma_1 + \varsigma_2), &  \ \ \ \  s > s_0 \text{ and } n < \frac{3L}{\tau}
	\end{cases}
	\end{align*}
	where $\varsigma_1=9\nu^2Ld$, $\varsigma_2 = 24L\sqrt{d}Z\nu$ and $D_0'$ is defined as in Eq.~\eqref{VARAG-cord-def-D_0}.
\end{mdframed}
We start with a lemma.
\begin{mdframed}
\begin{lemma}\label{VARAG-cord-lemma-5}
	Consider the coordinate-wise variant of Algorithm~\ref{algorithm-VARAG}. Assume \textbf{(A1)}, \textbf{(A2$_{\boldsymbol{\nu}}$)} and \textbf{(A3)}.  Under the choice of parameters from Theorem~\ref{VARAG-cord-theorem-2}, we have
	\begin{align}
	& \mathbb{E}_{i_t}\bigg[\frac{\gamma_{s}}{\alpha_{s}}\big[f(\bar{x}_t)-f^{*}\big] + \frac{(1+\tau \gamma_{s})}{2}\|x_t-x^*\|^2\bigg] \notag\\
	\le \  & \frac{\gamma_{s}}{\alpha_{s}}(1-\alpha_{s}-p_s)\big[f(\bar{x}_{t-1})-f^{*}\big]+\frac{\gamma_{s}p_s}{\alpha_{s}} \big[f(\tilde{x})-f^{*}\big] + \frac{1}{2}\|x_{t-1}-x^*\|^2 + \frac{\gamma_{s}}{\alpha_{s}} \cdot \frac{3}{4}\nu^2Ld + \frac{\gamma_s}{\alpha_s} (2-\alpha_s)L \sqrt{d}Z \nu. \label{VARAG-cord-lemma-5-result}
	\end{align}
\end{lemma}
\end{mdframed}

\begin{proof}[Proof of Lemma \ref{VARAG-cord-lemma-5}]
	For strongly convex $f$, Eq.~\eqref{VARAG-cord-lemma-3-result} becomes,
	\begin{align*}
	& \mathbb{E}_{i_t}\bigg[\frac{\gamma_{s}}{\alpha_{s}}\big[f(\bar{x}_t)-f^{*}\big] + \frac{(1+\tau \gamma_{s})}{2}\|x_t-x^*\|^2\bigg] \notag\\
	\le \  & \frac{\gamma_{s}}{\alpha_{s}}(1-\alpha_{s}-p_s)\big[f(\bar{x}_{t-1})-f^{*}\big]+\frac{\gamma_{s}p_s}{\alpha_{s}} \big[f(\tilde{x})-f^{*}\big] + \frac{1}{2}\|x_{t-1}-x^*\|^2 \notag\\
	& + \frac{\gamma_{s}}{\alpha_{s}} \cdot \frac{6\alpha_{s}\gamma_{s}\nu^2L^2d}{1+\tau\gamma_{s}-L\alpha_{s}\gamma_{s}} - \frac{\gamma_{s}}{\alpha_{s}} \cdot \alpha_{s}\langle g_{\nu}(\underline{x}_t) - \nabla f(\underline{x}_t), x_{t-1}^+ - x \rangle\\
	\le \  & \frac{\gamma_{s}}{\alpha_{s}}(1-\alpha_{s}-p_s)\big[f(\bar{x}_{t-1})-f^{*}\big]+\frac{\gamma_{s}p_s}{\alpha_{s}} \big[f(\tilde{x})-f^{*}\big] + \frac{1}{2}\|x_{t-1}-x^*\|^2 \notag\\
	& + \frac{\gamma_{s}}{\alpha_{s}} \cdot \frac{3}{4}\nu^2Ld + \frac{\gamma_s}{\alpha_s} \cdot (2-\alpha_s)L \sqrt{d}Z \nu.
	\end{align*}
	The last inequality holds when $\alpha_{s}$ and $\gamma_{s} $ are as defined in Theorem \ref{VARAG-cord-theorem-2} and Eq.~\eqref{VARAG-cord-surplus-error}.
\end{proof}

We divide the proof of Theorem \ref{VARAG-cord-theorem-2} into three cases, corresponding to Lemma \ref{VARAG-cord-lemma-6}, Lemma \ref{VARAG-cord-lemma-7}, Lemma \ref{VARAG-cord-lemma-8}.
\begin{mdframed}
\begin{lemma}\label{VARAG-cord-lemma-6}
	Consider the coordinate-wise variant of Algorithm~\ref{algorithm-VARAG}. Assume \textbf{(A1)}, \textbf{(A2$_{\boldsymbol{\nu}}$)} and \textbf{(A3)}. Under the choice of parameters from Theorem~\ref{VARAG-cord-theorem-2}, if $s \le s_0$, for any $x \in \mathbb{R}^d$ we have
	\begin{align*}
	\mathbb{E}\big[f(\tilde{x}^s) - f^{*}\big] \le \frac{1}{2^{s+1}} D_0' + \frac{3}{2} \nu^2Ld + 4L \sqrt{d}Z \nu,
	\end{align*}
	where $D_0'$ is defined in Eq.~\eqref{VARAG-cord-def-D_0}.
\end{lemma}
\end{mdframed}

\begin{proof}[Proof of Lemma \ref{VARAG-cord-lemma-6}]
	For this case, $\alpha_s = p_s = \frac{1}{2}$, $\gamma_{s} = \frac{1}{6L}$, $T_s = 2^{s-1}$.
	For Lemma \ref{VARAG-cord-lemma-5}, sum it up from $t = 1$ to $T_s$, we have 
	\begin{align*}
	& \sum_{t=1}^{T_s} \frac{\gamma_{s}}{\alpha_{s}}\mathbb{E} \big[f(\bar{x}_t)-f^{*}\big] +  \frac{1}{2} \mathbb{E} \big[\|x_{T_s}-x^*\|^2\big]\notag\\
	\le \  & \frac{\gamma_{s}}{2\alpha_{s}} \cdot T_s \big[f(\tilde{x})-f^{*}\big] + \frac{1}{2}\|x_{0}-x^*\|^2 + T_s \cdot \frac{\gamma_{s}}{\alpha_{s}} \cdot \frac{3}{4}\nu^2Ld + T_s \cdot \frac{\gamma_s}{\alpha_s}\cdot (2-\alpha_s)L \sqrt{d}Z \nu.
	\end{align*}
	Since $f$ is convex, we have
	\begin{align*}
	& \frac{1}{3L} \cdot T_s \cdot \mathbb{E} \big[f(\tilde{x}^{s})-f^{*}\big] +  \frac{1}{2} \mathbb{E} \big[\|x^{s}-x^*\|^2\big] \notag\\
	\le \  & \frac{1}{6L} \cdot T_s \big[f(\tilde{x}^{s-1})-f^{*}\big] + \frac{1}{2}\|x^{s-1}-x^*\|^2 + T_s \cdot \frac{1}{3L} \cdot \frac{3}{4}\nu^2Ld + T_s \cdot \frac{1}{3L} \cdot (2-\alpha_s)L \sqrt{d}Z \nu \notag\\
	= \  & \frac{1}{3L} \cdot T_{s-1} \big[f(\tilde{x}^{s-1})-f^{*}\big] + \frac{1}{2}\|x^{s-1}-x^*\|^2 + T_s \cdot \frac{1}{3L} \cdot \frac{3}{4}\nu^2Ld + T_s \cdot \frac{1}{3L} \cdot (2-\alpha_s)L \sqrt{d}Z \nu,
	\end{align*}
	where $x_{T_s} = x^s$, $x_{0} = x^{s-1}$, $\tilde{x} = \tilde{x}^{s-1}$. From using the last inequality iteratively, we obtain
	\begin{align}
	& \frac{1}{3L} \cdot T_s \cdot \mathbb{E} \big[f(\tilde{x}^{s})-f^{*}\big] +  \frac{1}{2} \mathbb{E} \big[\|x^{s}-x^*\|^2\big] \notag\\
	\le \  & \frac{1}{3L} \cdot T_{0} \big[f(\tilde{x}^{0})-f^{*}\big] + \frac{1}{2}\|x^{0}-x^*\|^2 + \frac{1}{3L} \sum_{j=1}^{s} T_j \cdot \frac{3}{4} \nu^2Ld + \frac{1}{3L}\sum_{j=1}^{s} T_j \cdot (2-\alpha_j)L \sqrt{d}Z \nu \notag
	\end{align}
	where $T_0 = \frac{1}{2}$ is in accordance with the definition of $T_s = 2^{s-1}, \ s> 0$. Hence, we obtain
	\begin{align}
	& \mathbb{E} \big[f(\tilde{x}^{s})-f^{*}\big] + \frac{3L}{T_s} \cdot \frac{1}{2} \mathbb{E} \big[\|x^{s}-x^*\|^2\big] \notag\\
	\le \  & \frac{1}{2^s} \big[f(\tilde{x}^{0})-f^{*} + 3L\|x^{0}-x^*\|^2 \big] + \frac{1}{T_s} \sum_{j=1}^{s} T_j \cdot \frac{3}{4} \nu^2Ld + \frac{1}{T_s}\sum_{j=1}^{s} T_j \cdot (2-\alpha_j)L \sqrt{d}Z \nu \notag\\
	\le \  & \frac{1}{2^{s+1}} D_0' + \frac{3}{2} \nu^2Ld + 4L \sqrt{d}Z \nu. \label{VARAG-cord-lemma-6-ineq-1}
	\end{align}
	We conclude the proof by observing that $\frac{1}{2^{s-1}}\sum_{j=1}^{s} T_j \leq 2$ when $s \le s_0$.
\end{proof}

\begin{mdframed}
\begin{lemma}\label{VARAG-cord-lemma-7}
	Consider the coordinate-wise variant of Algorithm~\ref{algorithm-VARAG}. Assume \textbf{(A1)}, \textbf{(A2$_{\boldsymbol{\nu}}$)} and \textbf{(A3)}. Under the choice of parameters from Theorem~\ref{VARAG-cord-theorem-2}, if $s \ge s_0$ and $n \ge \frac{3L}{\tau}$, then for any $x \in \mathbb{R}^d$ we have:
	\begin{align*}
	\mathbb{E} \big[f(\tilde{x}^{s})-f^{*}\big] \le \left(\frac{4}{5}\right)^{s-s_0} \frac{D_0'}{n} + 9\nu^2Ld + 24L\sqrt{d}Z\nu. 
	\end{align*}
\end{lemma}
\end{mdframed}

\begin{proof}[Proof of Lemma \ref{VARAG-cord-lemma-7}]
	For this case, $\alpha_s = \alpha = p_s = \frac{1}{2}$, $\gamma_{s} = \gamma = \frac{1}{6L}$, $T_s = 2^{s_0-1}$ when $s \ge s_0$. Based on Lemma \ref{VARAG-cord-lemma-5}, we have
	\begin{align*}
	\mathbb{E}_{i_t}\bigg[\frac{\gamma}{\alpha}\big[f(\bar{x}_t)-f^{*}\big] + \big(1+ \tau \gamma\big) \cdot \frac{1}{2}\|x_t-x^*\|^2\bigg]
	\le \frac{\gamma}{2\alpha} \big[f(\tilde{x})-f^{*}\big] + \frac{1}{2}\|x_{t-1}-x^*\|^2 + \frac{\gamma}{\alpha} \cdot \frac{3}{4}\nu^2Ld + \frac{\gamma}{\alpha} \cdot 2L \sqrt{d}Z \nu.
	\end{align*}
	Multiplying both sides by $\Gamma_{t-1} = (1 + \tau \gamma)^{t-1}$, we obtain
	\begin{align*}
	& \mathbb{E}_{i_t}\bigg[\frac{\gamma}{\alpha} \Gamma_{t-1} \big[f(\bar{x}_t)-f^{*}\big] + \frac{\Gamma_{t}}{2}\|x_t-x^*\|^2\bigg] \notag\\
	\le \  & \frac{\gamma}{2\alpha} \Gamma_{t-1} \big[f(\tilde{x})-f^{*}\big] + \frac{\Gamma_{t-1}}{2}\|x_{t-1}-x^*\|^2 + \frac{\gamma}{\alpha}\Gamma_{t-1} \cdot \frac{3}{4}\nu^2Ld + \frac{\gamma}{\alpha}\Gamma_{t-1} \cdot 2L \sqrt{d}Z \nu.
	\end{align*}
	Since $\theta_t = \Gamma_{t-1}$, as defined in Eq.~\eqref{VARAG-cord-def-theta-2}, the last inequality can be rewritten as
	\begin{align*}
	\mathbb{E}_{i_t}\bigg[\frac{\gamma}{\alpha} \theta_{t} \big[f(\bar{x}_t)-f^{*}\big] + \frac{\Gamma_{t}}{2}\|x_t-x^*\|^2\bigg]
	\le \frac{\gamma}{2\alpha} \theta_{t} \big[f(\tilde{x})-f^{*}\big] + \frac{\Gamma_{t-1}}{2}\|x_{t-1}-x^*\|^2 + \frac{\gamma}{\alpha}\theta_{t} \cdot \frac{3}{4}\nu^2Ld + \frac{\gamma}{\alpha}\theta_{t} \cdot 2L \sqrt{d}Z \nu.
	\end{align*}
	Summing up the inequality above from $t = 1$ to $T_s$, we obtain
	\begin{align*}
	& \frac{\gamma}{\alpha} \sum_{t=1}^{T_s}\theta_{t} \mathbb{E} \big[f(\bar{x}_t)-f^{*}\big] + \frac{\Gamma_{T_s}}{2} \mathbb{E} \|x_{T_s}-x^*\|^2\notag\\
	\le \  & \frac{\gamma}{2\alpha} \sum_{t=1}^{T_s}\theta_{t} \mathbb{E} \big[f(\tilde{x})-f^{*}\big] + \frac{1}{2}\|x_{0}-x^*\|^2 + \frac{\gamma}{\alpha} \cdot \frac{3}{4}\nu^2Ld \sum_{t=1}^{T_s}\theta_{t} + \frac{\gamma}{\alpha}\cdot 2L \sqrt{d}Z \nu \sum_{t=1}^{T_s} \theta_{t},
	\end{align*}
	and then
	\begin{align}
	& \frac{5}{4} \bigg[\frac{\gamma}{2\alpha} \sum_{t=1}^{T_s}\theta_{t} \mathbb{E} \big[f(\bar{x}_t)-f^{*}\big] + \frac{1}{2} \mathbb{E} \|x_{T_s}-x^*\|^2\bigg]\notag\\
	\le \  & \frac{\gamma}{2\alpha} \sum_{t=1}^{T_s}\theta_{t} \mathbb{E} \big[f(\tilde{x})-f^{*}\big] + \frac{1}{2}\|x_{0}-x^*\|^2 + \frac{\gamma}{\alpha} \cdot \frac{3}{4}\nu^2Ld \sum_{t=1}^{T_s}\theta_{t} + \frac{\gamma}{\alpha}\cdot 2L \sqrt{d}Z \nu \sum_{t=1}^{T_s} \theta_{t}, \label{VARAG-cord-lemma-7-ineq-1}
	\end{align}
	which is based on the fact that, for $s \ge s_0$, $\frac{n}{2}\le T_s = T_{s_0} \le n$, we have
	\begin{align*}
	\Gamma_{T_s} = \big(1+\tau \gamma\big)^{T_s} = \big(1+\tau \gamma\big)^{T_{s_0}} \ge 1+ \tau \gamma \cdot T_{s_0} \ge 1 + \tau \gamma \cdot \frac{n}{2} = 1 + \frac{\tau n}{12 L} \ge \frac{5}{4},
	\end{align*}
	where the last inequality is conditioned on $n \ge \frac{3L}{\tau}$. Since $\tilde{x}^{s} = \sum_{t=1}^{T_s} (\theta_t \bar{x}_{t})/ \sum_{t=1}^{T_s} \theta_t$, $\tilde{x} = \tilde{x}^{s-1}$, $x_0 = x^{s-1}$, $x_{T_s} = x^s$ in the epoch $s$ and thanks to the convexity of $f$, Eq.~\eqref{VARAG-cord-lemma-7-ineq-1} implies
	\begin{align*}
	& \frac{5}{4} \bigg[\frac{\gamma}{2\alpha} \mathbb{E} \big[f(\tilde{x}^{s})-f^{*}\big] + \frac{1}{2 \sum_{t=1}^{T_s}\theta_{t}} \mathbb{E} \|x^s-x^*\|^2\bigg]\notag\\
	\le \  & \frac{\gamma}{2\alpha} \mathbb{E} \big[f(\tilde{x}^{s-1})-f^{*}\big] + \frac{1}{2 \sum_{t=1}^{T_s}\theta_{t}}\|x^{s-1}-x^*\|^2 + \frac{\gamma}{\alpha} \cdot \frac{3}{4}\nu^2Ld + \frac{\gamma}{\alpha}\cdot 2L \sqrt{d}Z \nu.
	\end{align*}
	Multiplying both sides with $\frac{2 \alpha}{\gamma}$ and applying this inequality recursively for $s \ge s_0$, we obtain
	\begin{align*}
	& \mathbb{E} \big[f(\tilde{x}^{s})-f^{*}\big] + \frac{2\alpha}{\gamma \sum_{t=1}^{T_s}\theta_{t}} \cdot \frac{1}{2}\mathbb{E} \big[\|x^s-x^*\|^2\big] \notag\\
	\le \  & \left(\frac{4}{5}\right)^{s-s_0} \bigg[\mathbb{E} \big[f(\tilde{x}^{s_0})-f^{*}\big] + \frac{2\alpha}{\gamma \sum_{t=1}^{T_s}\theta_{t}} \cdot \frac{1}{2} \mathbb{E} \big[\|x^{s_0}-x^*\|^2\big] \bigg] + \sum_{j = s_0+1}^{s} \left(\frac{4}{5}\right)^{s+1-j} \bigg[ \frac{3}{2}\nu^2Ld + 4L\sqrt{d}Z\nu\bigg] \notag\\
	\le \  & \left(\frac{4}{5}\right)^{s-s_0} \bigg[\mathbb{E} \big[f(\tilde{x}^{s_0})-f^{*}\big] + \frac{2\alpha}{\gamma T_{s_0}} \cdot \frac{1}{2}\mathbb{E} \big[\|x^{s_0}-x^*\|^2\big]\bigg] + 6\nu^2Ld + 16L\sqrt{d}Z\nu.
	\end{align*}
	where the last inequality holds since $\sum_{j = s_0+1}^{s} \left(\frac{4}{5}\right)^{s+1-j} \le \frac{4}{5} \cdot \frac{1}{1-\frac{4}{5}} = 4$ and $\sum_{t=1}^{T_s}\theta_t \ge T_s = T_{s_0}$. Hence
	\begin{align*}
	& \mathbb{E} \big[f(\tilde{x}^{s})-f^{*}\big] + \frac{2\alpha}{\gamma \sum_{t=1}^{T_s}\theta_{t}} \cdot \frac{1}{2}\mathbb{E} \big[\|x^s-x^*\|^2\big] \notag\\
	\le \  & \left(\frac{4}{5}\right)^{s-s_0} \bigg[\mathbb{E} \big[f(\tilde{x}^{s_0})-f^{*}\big] + \frac{6L}{T_{s_0}} \cdot \frac{1}{2}\mathbb{E} \big[\|x^{s_0}-x^*\|^2\big]\bigg] + 6\nu^2Ld + 16L\sqrt{d}Z\nu \notag\\
	\le \  & \left(\frac{4}{5}\right)^{s-s_0} 2\bigg[\mathbb{E} \big[f(\tilde{x}^{s_0})-f^{*}\big] + \frac{3L}{ T_{s_0}} \cdot \frac{1}{2}\mathbb{E} \big[\|x^{s_0}-x^*\|^2\big]\bigg] + 6\nu^2Ld + 16L\sqrt{d}Z\nu \notag\\
	\le \  & \left(\frac{4}{5}\right)^{s-s_0} 2 \cdot \big[\frac{1}{2^{s_0+1}} D_0' + \frac{3}{2} \nu^2Ld + 4L \sqrt{d}Z \nu \big] + 6\nu^2Ld + 16L\sqrt{d}Z\nu \notag\\
	\le \  & \left(\frac{4}{5}\right)^{s-s_0} \cdot \frac{D_0'}{2^{s_0}} + 9\nu^2Ld + 24L\sqrt{d}Z\nu \notag\\
	= \  & \left(\frac{4}{5}\right)^{s-s_0} \frac{D_0'}{2T_{s_0}} + 9\nu^2Ld + 24L\sqrt{d}Z\nu \notag\\
	\le \  & \left(\frac{4}{5}\right)^{s-s_0} \frac{D_0'}{n} + 9\nu^2Ld + 24L\sqrt{d}Z\nu,
	\end{align*}
	where the third inequality comes from Eq.~\eqref{VARAG-cord-lemma-6-ineq-1} and the last inequality from the fact that $T_{s_0} \ge \frac{n}{2}$.
\end{proof}

\begin{mdframed}
\begin{lemma}\label{VARAG-cord-lemma-8}
	Consider the coordinate-wise variant of Algorithm~\ref{algorithm-VARAG}. Assume \textbf{(A1)}, \textbf{(A2$_{\boldsymbol{\nu}}$)} and \textbf{(A3)}. If $s \ge s_0$ and $n < \frac{3L}{\tau}$, then for any $x \in \mathbb{R}^d$,
	\begin{align*}
	\mathbb{E}\big[f(\tilde{x}^s) - f^{*}\big] \le \bigg(1+\frac{1}{4} \sqrt{\frac{n\tau}{3L}}\bigg)^{-(s-s_0)} \frac{D_0'}{n} + \left(2\sqrt{\frac{3L}{n\tau}}+1\right)\bigg[ 3\nu^2Ld + 8L \sqrt{d}Z \nu \bigg].
	\end{align*}
\end{lemma}
\end{mdframed}

\begin{proof}[Proof of Lemma \ref{VARAG-cord-lemma-8}]
	For this case, $\alpha_s = \alpha = \sqrt{\frac{n \tau}{12L}}$, $p_s = p = \frac{1}{2}$, $\gamma_{s} = \gamma = \frac{1}{\sqrt{12nL\tau}}$, $T_s = T_{s_0} = 2^{s_0-1}$ when $s \ge s_0$. Based on Lemma \ref{VARAG-cord-lemma-5}, we have
	\begin{align*}
	\mathbb{E}_{i_t}\bigg[\frac{\gamma}{\alpha}\big[f(\bar{x}_t)-f^{*}\big] + \frac{(1+\tau \gamma)}{2}\|x_t-x^*\|^2\bigg] \le \  & \frac{\gamma}{\alpha}(1-\alpha-p)\big[f(\bar{x}_{t-1})-f^{*}\big]+\frac{\gamma}{2\alpha} \big[f(\tilde{x})-f^{*}\big] + \frac{1}{2}\|x_{t-1}-x^*\|^2 \notag\\
	& + \frac{\gamma}{\alpha} \cdot \frac{3}{4}\nu^2Ld + \frac{\gamma}{\alpha} \cdot (2-\alpha)L \sqrt{d}Z \nu.
	\end{align*}
	Multiplying both sides by $\Gamma_{t-1} = (1+ \tau \gamma)^{t-1}$, we obtain
	\begin{align*}
	\mathbb{E}_{i_t}\bigg[\frac{\gamma}{\alpha} \Gamma_{t-1} \big[f(\bar{x}_t)-f^{*}\big] + \frac{\Gamma_{t}}{2}\|x_t-x^*\|^2\bigg]
	\le \  & \frac{\Gamma_{t-1} \gamma}{\alpha}(1-\alpha-p)\big[f(\bar{x}_{t-1})-f^{*}\big] + \frac{\Gamma_{t-1} \gamma p}{\alpha} \big[f(\tilde{x})-f^{*}\big] \notag\\
	& + \frac{\Gamma_{t-1}}{2}\|x_{t-1}-x^*\|^2 + \frac{\gamma}{\alpha} \Gamma_{t-1} \cdot \frac{3}{4}\nu^2Ld + \frac{\gamma}{\alpha}\Gamma_{t-1} \cdot (2-\alpha)L \sqrt{d}Z \nu.
	\end{align*}
	Summing up the inequality above from $t = 1$ to $T_s$, we obtain
	\begin{align*}
	& \frac{\gamma}{\alpha} \sum_{t=1}^{T_s}\theta_{t} \mathbb{E} \big[f(\bar{x}_t)-f^{*}\big] + \frac{\Gamma_{T_s}}{2} \mathbb{E} \|x_{T_s}-x^*\|^2\notag\\
	\le \  & \frac{\gamma}{\alpha} \big[1-\alpha - p + p\sum_{t=1}^{T_s}\Gamma_{t-1}\big] \mathbb{E} \big[f(\tilde{x})-f^{*}\big] + \frac{1}{2}\|x_{0}-x^*\|^2 + \frac{\gamma}{\alpha}\sum_{t=1}^{T_s}\Gamma_{t-1} \cdot \frac{3}{4}\nu^2Ld + \frac{\gamma}{\alpha}\sum_{t=1}^{T_s}\Gamma_{t-1} \cdot (2-\alpha)L \sqrt{d}Z \nu.
	\end{align*}
	Since $\tilde{x}^{s} = \sum_{t=1}^{T_s} (\theta_t \bar{x}_{t})/ \sum_{t=1}^{T_s} \theta_t$, $\tilde{x} = \tilde{x}^{s-1}$, $x_0 = x^{s-1}$, $x_{T_s} = x^s$ in the epoch $s$ and the convexity of $f_{\nu}$, it implies, for $s > s_0$,
	\begin{multline}
	\frac{\gamma}{\alpha} \sum_{t=1}^{T_s}\theta_{t} \mathbb{E} \big[f(\tilde{x}^{s})-f^{*}\big] + \frac{\Gamma_{T_{s_0}}}{2} \mathbb{E} \|x^s-x^*\|^2\\
	\le \frac{\gamma}{\alpha} \big[1-\alpha - p + p\sum_{t=1}^{T_s}\Gamma_{t-1}\big] \mathbb{E} \big[f(\tilde{x}^{s-1})-f^{*}\big] + \frac{1}{2}\|x^{s-1}-x^*\|^2 + \frac{\gamma}{\alpha}\sum_{t=1}^{T_s}\Gamma_{t-1} \cdot \frac{3}{4}\nu^2Ld + \frac{\gamma}{\alpha}\sum_{t=1}^{T_s}\Gamma_{t-1} \cdot (2-\alpha)L \sqrt{d}Z \nu. \label{VARAG-cord-lemma-8-ineq-1}
	\end{multline}
	Moreover, we have
	\begin{align*}
	\sum_{t=1}^{T_{s_0}} \theta_{t} = \  & \Gamma_{T_{s_0}-1} + \sum_{t=1}^{T_{s_0}-1} \big(\Gamma_{t-1}-(1-\alpha - p)\Gamma_{t}\big) \\
	= \  & \Gamma_{T_{s_0}}(1-\alpha -p) + \sum_{t=1}^{T_{s_0}} \big(\Gamma_{t-1}-(1-\alpha - p)\Gamma_{t}\big)\\
	= \  & \Gamma_{T_{s_0}}(1-\alpha -p) + \big[1-(1-\alpha -p)(1+\tau \gamma)\big]\sum_{t=1}^{T_{s_0}} \Gamma_{t-1}.\\
	\end{align*}
	Considering the range of $\alpha_{s}$, since $T_{s_0} \le n$,
	\begin{align*}
	\alpha & = \sqrt{\frac{n\tau}{12L}} \ge \sqrt{\frac{T_{s_0} \tau}{12L}} = \tau \cdot \frac{1}{\sqrt{12nL\tau}} \cdot \sqrt{T_{s_0}n} \\
	& = \tau \gamma \cdot \sqrt{T_{s_0}n} \ge \tau \gamma T_{s_0}.
	\end{align*}
	Also note that, for any $T > 1$ and $0 \le \delta T \le 1$, $(1+T\delta) \le (1+\delta)^T \le (1+2T \delta)$. If we set $\delta = \tau \gamma$ and $T = T_{s_0}$ here,
	\begin{align*}
	\delta T = \tau \gamma T_{s_0} \le \alpha < 1.
	\end{align*}
	Then, we have
	\begin{align*}
	1-(1-\alpha -p)(1+ \tau \gamma) = \  & (1+\tau \gamma)(\alpha+p-\tau \gamma) + \tau^2\gamma^2 \\
	\ge \ & (1+\tau \gamma)(\tau \gamma T_{s_0} + p- \tau \gamma)\\
	= \  & p(1+ \tau \gamma)(1+ 2(T_{s_0}-1)\cdot \tau \gamma)\\
	\ge \ &  p(1+ \tau \gamma)^{T_{s_0}} = p\Gamma_{T_{s_0}}.
	\end{align*}
	Hence, we obtain $\sum_{t=1}^{T_{s_0}} \theta_{t} \ge \Gamma_{T_{s_0}} \cdot \big[1-\alpha - p + p\sum_{t=1}^{T_s}\Gamma_{t-1}\big]$. Moreover, thanks to Eq.~\eqref{VARAG-cord-lemma-8-ineq-1} and $f(\tilde{x}^{s}) - f^{*} \ge 0$, the last inequality implies that
	\begin{align*}
	& \Gamma_{T_{s_0}} \cdot \bigg[\frac{\gamma}{\alpha}\big[1-\alpha - p + p\sum_{t=1}^{T_s}\Gamma_{t-1}\big] \mathbb{E} \big[f(\tilde{x}^{s})-f^{*}\big] + \frac{1}{2} \mathbb{E} \|x^s-x^*\|^2\bigg]\notag\\
	\le \  & \frac{\gamma}{\alpha} \big[1-\alpha - p + p\sum_{t=1}^{T_s}\Gamma_{t-1}\big] \mathbb{E} \big[f(\tilde{x}^{s-1})-f^{*}\big] + \frac{1}{2}\|x^{s-1}-x^*\|^2 + \frac{\gamma}{\alpha}\sum_{t=1}^{T_s}\Gamma_{t-1} \cdot \frac{3}{4}\nu^2Ld\notag\\
	& + \frac{\gamma}{\alpha}\sum_{t=1}^{T_s}\Gamma_{t-1} \cdot (2-\alpha)L \sqrt{d}Z \nu.
	\end{align*}
	Applying this inequality iteratively for $s > s_0$, we obtain
	\begin{align*}
	& \frac{\gamma}{\alpha}\big[1-\alpha - p + p\sum_{t=1}^{T_s}\Gamma_{t-1}\big] \mathbb{E} \big[f(\tilde{x}^{s})-f^{*}\big] + \frac{1}{2} \mathbb{E} \|x^s-x^*\|^2 \notag\\
	\le \  & \left(\frac{1}{\Gamma_{T_{s_0}}}\right)^{s-s_0} \bigg[\frac{\gamma}{\alpha} \big[1-\alpha - p + p\sum_{t=1}^{T_s}\Gamma_{t-1}\big] \mathbb{E} \big[f(\tilde{x}^{s_0})-f^{*}\big] + \frac{1}{2}\|x^{s_0}-x^*\|^2\bigg]\notag\\
	& + \sum_{j=1}^{s-s_0} \left(\frac{1}{\Gamma_{T_{s_0}}}\right)^{j}\bigg[ \frac{\gamma}{\alpha}\sum_{t=1}^{T_s}\Gamma_{t-1} \cdot \frac{3}{4}\nu^2Ld + \frac{\gamma}{\alpha}\sum_{t=1}^{T_s}\Gamma_{t-1} \cdot (2-\alpha)L \sqrt{d}Z \nu\bigg].
	\end{align*}
	Note that, since 
	$$\frac{\gamma}{\alpha}\big[1-\alpha - p + p\sum_{t=1}^{T_s}\Gamma_{t-1}\big] \ge \frac{\gamma p}{\alpha}\sum_{t=1}^{T_s}\Gamma_{t-1} \ge \frac{\gamma p T_s}{\alpha} = \frac{\gamma p T_{s_0}}{\alpha}$$
	and $p = \frac{1}{2}$, the inequality above implies
	\begin{align*}
	& \mathbb{E} \big[f(\tilde{x}^{s})-f^{*}\big] \notag\\
	\le \  & \left(\frac{1}{\Gamma_{T_{s_0}}}\right)^{s-s_0} \bigg[\mathbb{E} \big[f(\tilde{x}^{s_0})-f^{*}\big] + \frac{ \alpha}{\gamma T_{s_0}} \mathbb{E}\big[\|x^{s_0}-x^*\|^2\big]\bigg] + \sum_{j=1}^{s-s_0} \big(\frac{1}{\Gamma_{T_{s_0}}}\big)^{j}\bigg[ \frac{3}{2}\nu^2Ld + (4-2\alpha)L \sqrt{d}Z \nu \bigg]\\
	\le \  & \left(\frac{1}{\Gamma_{T_{s_0}}}\right)^{s-s_0} \bigg[\mathbb{E} \big[f(\tilde{x}^{s_0})-f^{*}\big] + \frac{ \alpha}{\gamma T_{s_0}} \mathbb{E}\big[\|x^{s_0}-x^*\|^2\big]\bigg] + \frac{1}{\Gamma_{T_{s_0}}-1}\bigg[ \frac{3}{2}\nu^2Ld + (4-2\alpha)L \sqrt{d}Z \nu \bigg].
	\end{align*}
	As $\Gamma_{T_{s_0}} = \big(1+ \tau \gamma\big)^{T_{s_0}} \ge 1+ \tau \gamma T_{s_0} \ge 1+\frac{\tau \gamma n}{2} = 1 + \frac{1}{2}\cdot \sqrt{\frac{n\tau}{12L}}$, it implies, for $s > s_0$,
	\begin{align*}
	\mathbb{E} \big[f(\tilde{x}^{s})-f^{*}\big]
	\le \  & \left(\frac{1}{\Gamma_{T_{s_0}}}\right)^{s-s_0} \bigg[\mathbb{E} \big[f(\tilde{x}^{s_0})-f^{*}\big] + \frac{ \alpha}{\gamma T_{s_0}} \mathbb{E}\big[\|x^{s_0}-x^*\|^2\big] \bigg] + 4\sqrt{\frac{3L}{n\tau}}\bigg[ \frac{3}{2}\nu^2Ld + (4-2\alpha)L \sqrt{d}Z \nu \bigg].
	\end{align*}
	Note that, since $n < \frac{3L}{\tau}$, we have $\frac{\alpha}{\gamma} = 12L\alpha^2 \le 3L$. Hence, for $s > s_0$, we have
	\begin{align*}
	& \mathbb{E} \big[f(\tilde{x}^{s})-f^{*}\big] \notag\\
	\le \  & \left(\frac{1}{\Gamma_{T_{s_0}}}\right)^{s-s_0} 2\bigg[\mathbb{E} \big[f(\tilde{x}^{s_0})-f^{*}\big] + \frac{ 3L}{T_{s_0}} \cdot \frac{1}{2} \mathbb{E}\big[\|x^{s_0}-x^*\|^2\big]\bigg] + 4\sqrt{\frac{3L}{n\tau}}\bigg[ \frac{3}{2}\nu^2Ld + (4-2\alpha)L \sqrt{d}Z \nu \bigg]\\
	\le \  & \left(\frac{1}{\Gamma_{T_{s_0}}}\right)^{s-s_0} 2\bigg[\frac{1}{2^{s_0+1}} D_0' + \frac{3}{2}\nu^2Ld + 4L \sqrt{d}Z \nu\bigg] + 4\sqrt{\frac{3L}{n\tau}}\bigg[ \frac{3}{2}\nu^2Ld + (4-2\alpha)L \sqrt{d}Z \nu \bigg]\\
	\le \  & \left(\frac{1}{\Gamma_{T_{s_0}}}\right)^{s-s_0} \frac{D_0'}{2^{s_0}} + \big(2\sqrt{\frac{3L}{n\tau}}+1\big)\bigg[ 3\nu^2Ld + 8L \sqrt{d}Z \nu \bigg]\\
	= \  & \left(\frac{1}{\Gamma_{T_{s_0}}}\right)^{s-s_0} \frac{D_0'}{2T_{s_0}} + \big(2\sqrt{\frac{3L}{n\tau}}+1\big)\bigg[ 3\nu^2Ld + 8L \sqrt{d}Z \nu \bigg]\\
	\le \  & \left(\frac{1}{\Gamma_{T_{s_0}}}\right)^{s-s_0} \frac{D_0'}{n} + \big(2\sqrt{\frac{3L}{n\tau}}+1\big)\bigg[ 3\nu^2Ld + 8L \sqrt{d}Z \nu \bigg]\\
	= \  & \bigg(1+\frac{1}{2} \cdot \sqrt{\frac{\tau}{3nL}}\bigg)^{-T_{s_0}(s-s_0)} \frac{D_0'}{n} +  \big(2\sqrt{\frac{3L}{n\tau}}+1\big)\bigg[ 3\nu^2Ld + 8L \sqrt{d}Z \nu \bigg]\\
	\le \  & \bigg(1+\frac{1}{2} \cdot \sqrt{\frac{\tau}{3nL}}\bigg)^{-\frac{n(s-s_0)}{2}} \frac{D_0'}{n} + \big(2\sqrt{\frac{3L}{n\tau}}+1\big)\bigg[ 3\nu^2Ld + 8L \sqrt{d}Z \nu \bigg]\\
	\le \  & \bigg(1+\frac{1}{4} \cdot \sqrt{\frac{n\tau}{3L}}\bigg)^{-(s-s_0)} \frac{D_0'}{n} + \big(2\sqrt{\frac{3L}{n\tau}}+1\big)\bigg[ 3\nu^2Ld + 8L \sqrt{d}Z \nu \bigg],
	\end{align*}
	where the second inequality is based on Eq.~\eqref{VARAG-cord-lemma-6-ineq-1} and the fourth and fifth inequalities rely on $T_{s_0} \ge \frac{n}{2}$. The last inequality comes from $1+T\delta \le (1+\delta)^T$ when $\delta \ge 0$.
\end{proof}

Now, we can finish the proof of Theorem \ref{VARAG-cord-theorem-2}:

\begin{proof}[Proof of Theorem \ref{VARAG-cord-theorem-2}]
	To summarize, we have obtained
	\begin{align}
	\mathbb{E} \big[f(\tilde{x}^{s})-f^{*}\big] := 
	\begin{cases}
	\frac{1}{2^{s+1}} D_0' + \frac{3}{2} \nu^2Ld + 4L \sqrt{d}Z \nu, & \quad \quad \quad 1 \le s \le s_0\\
	& \\
	\left(\frac{4}{5}\right)^{s-s_0} \frac{D_0'}{n} + 9\nu^2Ld + 24L\sqrt{d}Z\nu,  & \quad \quad \quad s > s_0 \text{ and } n \ge \frac{3L}{\tau}\\
	& \\
	\bigg(1+\frac{1}{4}\sqrt{\frac{n\tau}{3L}}\bigg)^{-(s-s_0)} \frac{D_0'}{n} &  \quad \quad \quad s > s_0 \text{ and } n < \frac{3L}{\tau} \\
	 \ \ \ \ \ \ \ + \big(2\sqrt{\frac{3L}{n\tau}}+1\big)\bigg[ 3\nu^2Ld + 8L \sqrt{d}Z \nu \bigg], &
	\end{cases}
	\end{align}
	from  Lemma \ref{VARAG-cord-lemma-6}, Lemma \ref{VARAG-cord-lemma-7}, Lemma \ref{VARAG-cord-lemma-8}.
\end{proof} 

We conclude once again by proving the complexity result.

\begin{proof}[Proof of Corollary \ref{VARAG-cord-corollary-2}]
Using the same technique as for the proof of Corollary~\ref{VARAG-corollary-2}, we can make the error terms depending on $\nu$ vanishing. It requires $\nu = \mathcal{O}\big(\frac{\epsilon^{1/2}}{L^{1/2} d^{1/2}} \big)$, $\nu = \mathcal{O}\big(\frac{\epsilon}{Ld^{1/2}Z}\big)$ for the first two cases ($1 \le s \le s_0$ or $s > s_0 \text{ and } n \ge \frac{3L}{\tau} $) while we need to take the extra conditional number $\frac{L}{\tau}$ into account for the last case ($s > s_0 \text{ and } n < \frac{3L}{\tau} $) to ensure $\epsilon$-optimality, $\frac{\epsilon}{2}$ more specifically. Hence, we can proceed as in~\cite{lan2019unified} , neglecting the errors coming from the DFO framework~(note that a similar procedure is adopted also in~\cite{nesterov2017random} and~\cite{liu2018stochastic,liu2018stochasticb}).
    	For the first case above ($1 \le s \le s_0$), the total number of function queries is given in Theorem \ref{VARAG-cord-theorem-1}. In the second case ($s > s_0 \text{ and } n \ge \frac{3L}{\tau}$), the algorithm runs at most $S := \mathcal{O}\big\{\log \big(\frac{D_0'}{\epsilon}\big)\big\}$ epochs to ensure $\epsilon$-optimality. Thus, the total number of function queries in this case is bounded by
	\begin{align}
	dnS + \sum_{s=1}^{S}d \cdot T_s \le dnS + dnS = \mathcal{O}\left\{dn\log \left(\frac{D_0'}{\epsilon}\right)\right\}.
	\end{align}
	Finally, to achieve $\epsilon$-error for the last case ($s > s_0 \text{ and } n < \frac{3L}{\tau} $), our algorithm needs to run at most $S^{'} := s_0 + \sqrt{\frac{3L}{n\tau}} \log \big(\frac{D_0'}{n\epsilon}\big)$ epochs. Therefore,  the total number of function queries is bounded by
	\begin{align}
	\sum_{s=1}^{S^{'}}(dn + dT_s) = \  & \sum_{s=1}^{s_0}(dn + dT_s) + (dn + dT_{s_0})(S^{'}-s_0) \notag\\
	\le \  & 2dns_0 + \big(dn+dn\big)\sqrt{\frac{3L}{n\tau}} \log \left(\frac{D_0'}{n\epsilon}\right) \notag\\
	= \  & \mathcal{O} \bigg\{dn \log(n) + d \sqrt{\frac{nL}{\tau}} \log \left(\frac{D_0'}{n\epsilon}\right)\bigg\}.
	\end{align}
\end{proof}

\section{Experiments}
\label{app:exp}

\subsection{Parameter settings for Fig.~\ref{fig:basic_test}}

Here, we compare our method~(Algorithm~\ref{algorithm-VARAG}) with ZO-SVRG-Coord-Rand~\cite{ji2019improved}, with the accelerated method in~\cite{nesterov2017random} and with a zero-order version of Katyusha inspired from~\cite{fanhua2017simKatyusha} --- which is a simplified version of the original algoerithm presented in~\citep{allen2017katyusha}. We define this method in Algorithm~\ref{algorithm-Katyusha}.

\begin{algorithm}[ht]
\caption{Simplified ZO-Katyusha}
\label{algorithm-Katyusha}
\renewcommand{\algorithmicoutput}{\textbf{Output:}}
\begin{algorithmic}[1]
\REQUIRE $x^0\in \mathbb{R}^d, \{T_s\}, \{\gamma_s\}, \{\alpha_s\}$.
\FOR{$s=1,2,\ldots,S$}
\STATE {$\tilde{x} = \tilde{x}^{s-1}$, $x^{s}_{0}=y^{s}_{0}=\tilde{x}^{s-\!1}$;}
\STATE \textbf{Pivotal ZO gradient} $\tilde{g} = g_{\nu}(\tilde{x})$ using the coordinate-wise approach by Eq.~\eqref{DFO-framework-cord-finite-difference}.
\FOR{$t=1,2,\ldots,T_{s}$}
\STATE {Pick $i_t$ uniformly at random from $\{1,\ldots,n\}$;}
\STATE {$G_t= g_{\mu}(x_{t-1}^s,u_t,i_t)- g_{\mu}(\tilde{x},u_t,i_t)+\tilde{g}$;}
\STATE {$y^{s}_{t}=y^{s}_{t-\!1}-\gamma_s G_t$;}
\STATE {$x^{s}_{t}=\tilde{x}+\alpha_{s}(y^{s}_{t}-\tilde{x})$;}
\ENDFOR
\STATE {$\tilde{x}^{s}=\frac{1}{T_{s}}\!\sum^{T_{s}}_{t=1}\!x^{s}_{t}$;}
\ENDFOR
\OUTPUT {$\tilde{x}^{S}$}
\end{algorithmic}
\end{algorithm}

We recall some notation from the main paper.
\vspace{-2mm}
\begin{itemize}
\item $n$ is the data-set size;
\item $d$ is the problem dimension;
\item $b$ is the mini-batch size used to compute stochastic ZO-gradients.
\item $\nu$ is the coordinate-smoothing parameter~(see Section~\ref{sec:background});
\item $\mu$ is the Gaussian-smoothing parameter~(see Section~\ref{sec:background});
\item $\{\alpha_s\}, \{\gamma_s\}$, $\{T_s\}$ and $\{\theta_t\}$ are parameters defined for ZO-Varag~(Algorithm~\ref{algorithm-VARAG}), which also appear in ZO-Katyusha~(Algorithm~\ref{algorithm-Katyusha}). $\alpha_k, \gamma_k, \theta_k$ also appear in the algorithm by~\citet{nesterov2017random}, but have different definitions~(see Eq. 60 in their paper);
\item $\{p_s\}$ is the Katyusha momentum parameter in Algorithm~\ref{algorithm-VARAG}, which can be seen as a \textit{Katyusha momentum} even though it is defined differently in the simplified framework of~\cite{fanhua2017simKatyusha} (see definition in the original paper by~\citet{allen2017katyusha});
\item $\eta$ is the step size for ZO-SVRG~\cite{ji2019improved}.
\end{itemize}

Next, we specify some parameter settings used for the experiments in Fig.~\ref{fig:basic_test}. What is not specified here directly appears in the corresponding figures.
\begin{itemize}
\item $\mu = \nu = 0.001$.
\item $T_s$ are set as in Theorem~\ref{VARAG-theorem-1}.
\item In Algorithm~\ref{algorithm-VARAG} and Algorithm~\ref{algorithm-Katyusha} we set $\alpha_s$ according to Eq.~\eqref{parameter-deter-alpha-sm} and Eq.~\eqref{parameter-deter-alpha-unified} in the main paper. For the accelerated method by \citep{nesterov2017random}, we used the choice of $\alpha_k$ reccomended in their paper.
\item Note that, for both Algorithm~\ref{algorithm-VARAG} and Algorithm~\ref{algorithm-Katyusha}, the gradient estimate $G_t$ is actually multiplied~\footnote{In Algorithm~\ref{algorithm-VARAG}, this is actually $\alpha_s \gamma_s/(1+\mu \gamma_s) \approx \alpha_s \gamma_s$.} by $\alpha_s\gamma_s$. Hence, $\alpha_s\gamma_s$ acts like a step-size. Therefore, in ZO-SVRG, we choose the \textit{equivalent stepsize} $\eta_s = \alpha_s\gamma_s$. Also note that, as one can note in Eq.~\eqref{parameter-deter-smooth1}, $\alpha_s\gamma_s$ is actually constant and inversely proportional to $d$ (see also next bullet-point).
\item We choose $\gamma_s$ such that $\eta = \alpha_s \gamma_s = 0.001 \cdot b/d$ for logistic regression and $\eta = \alpha_s \gamma_s = b/d$ for ridge regression when testing on the diabetes dataset (python sklearn). For the ijcnn1 dataset (python LIBSVM), we instead choose $\gamma_s$ such that $\eta = \alpha_s \gamma_s = 0.1 \cdot b/d$ for logistic regression and $\eta = \alpha_s \gamma_s = 0.001 \cdot b/d$ for ridge regression.
\item We pick $p_s = 0.5$, as specified in Eq.~\eqref{parameter-deter-smooth1}.
\end{itemize}

\subsection{Additional experiments}
Next, we discuss potential variations of the parameters discussed in the last subsection.

\paragraph{Options for pivotal point.}
We tested two options for pivot computation in Algorithm~\ref{algorithm-VARAG}:
\begin{center}
    Option I: $\tilde{x} = \tilde{x}^{s-1} = \tilde{x}^s =\sum_{t=1}^{T_s}(\theta_t \bar x_t)/\big(\sum_{t=1}^{T_s} \theta_t\big)$ (as used in our analysis), \quad or \quad  Option II:  $\tilde{x} = \bar{x}^{s-1}$ . 
\end{center}

In addition to the experimental results on the ijcnn1 dataset~(LIBSVM) provided in Fig.~\ref{fig:pivot_selection_ijcnn1}, we also provide results on the diabetes dataset (sklearn) here: from Fig. \ref{fig:pivot_selection_diabetes}, we observe that Option II also achieves faster convergence than Option I on the diabetes dataset in practice. Overall, our empirical evidence seems to indicate that Option II works better than Option I.

\begin{figure}[htbp]
 \centering          \begin{tabular}{c@{}c@{}c@{}c@{}}%c@{}}
        &   \scriptsize{diabetes, S = 200, b = 5, $\lambda = 0$} & \scriptsize{diabetes, S = 200, b = 5, $\lambda = 1e^{-5}$} \\%& \tiny{autoencoder} \\
        \rotatebox[origin=c]{90}{\scriptsize{Logistic}} &  
             \includegraphics[width=0.35\linewidth,valign=c]{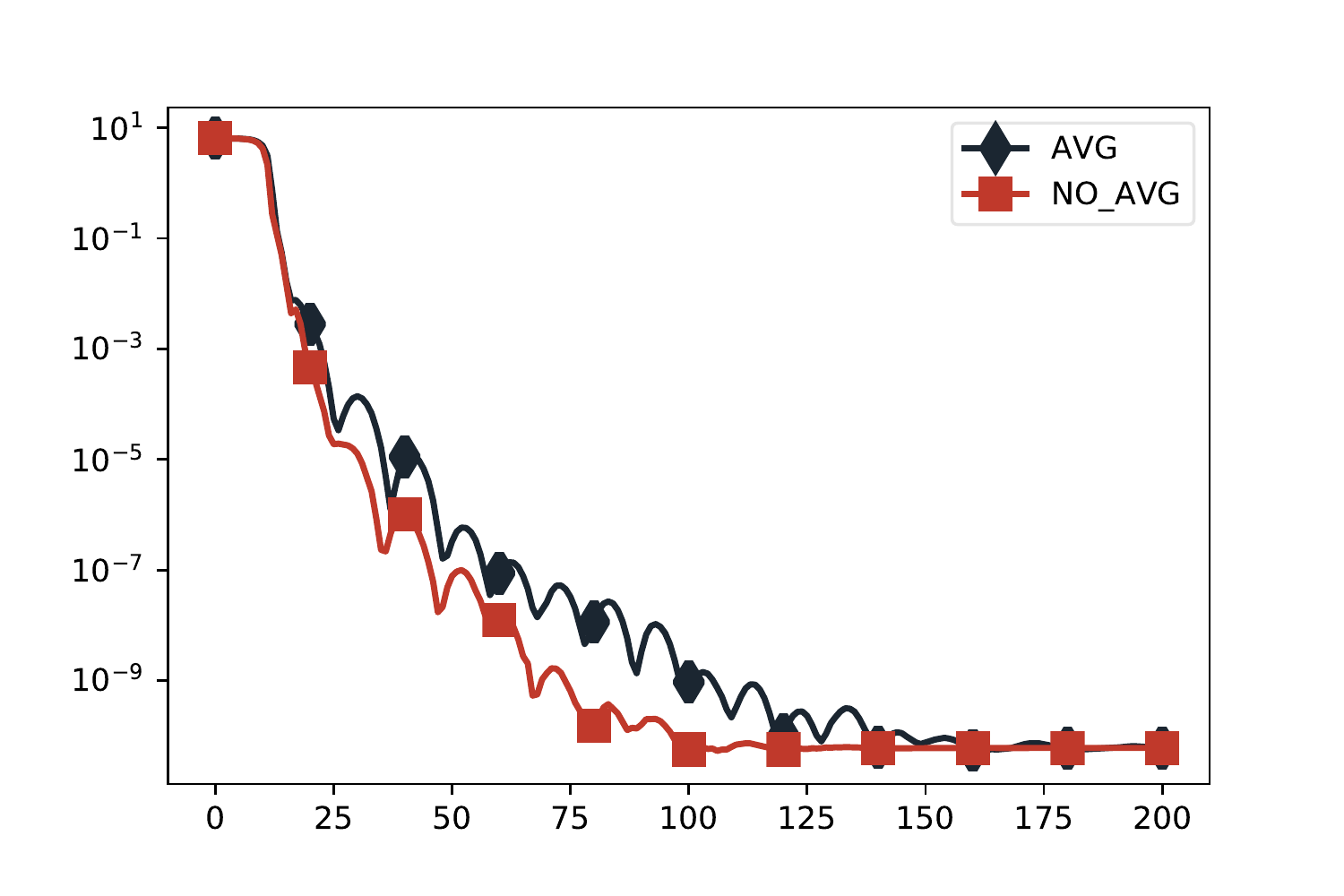}  &  \includegraphics[width=0.35\linewidth,valign=c]{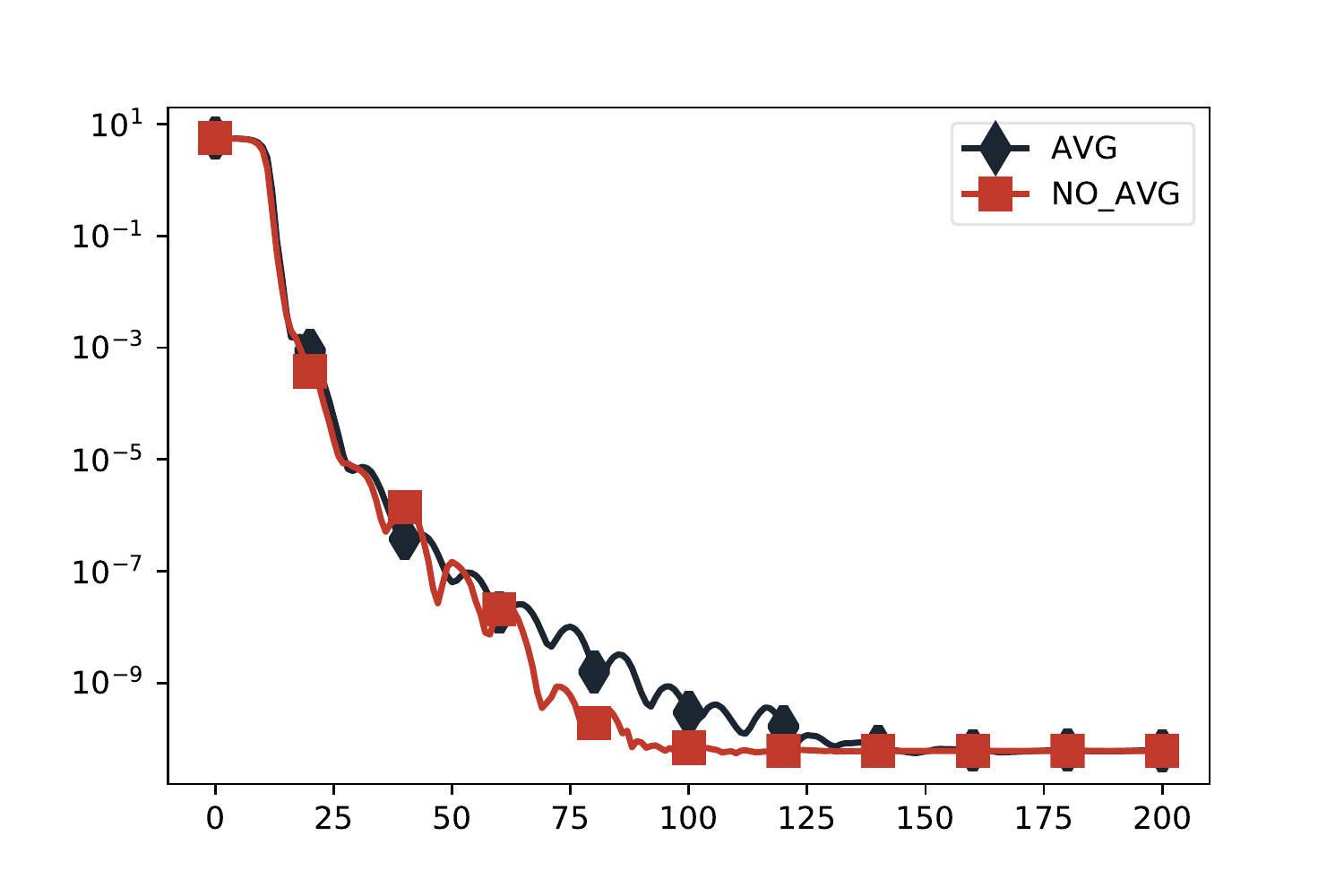}
             \\
                               \vspace{-3mm}
        \rotatebox[origin=c]{90}{\scriptsize{Ridge}}&
             
              \includegraphics[width=0.35\linewidth,valign=c]{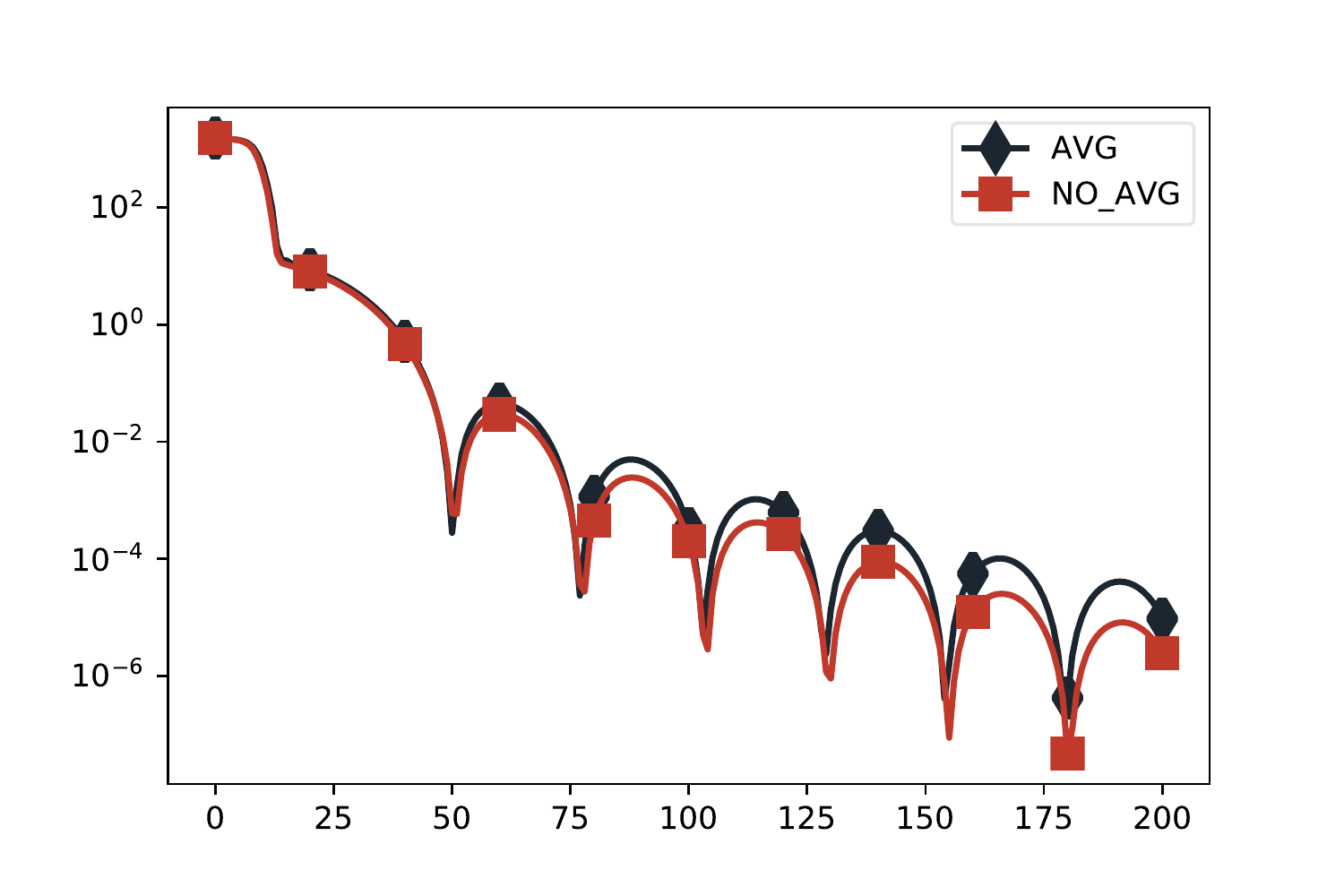}& \includegraphics[width=0.35\linewidth,valign=c]{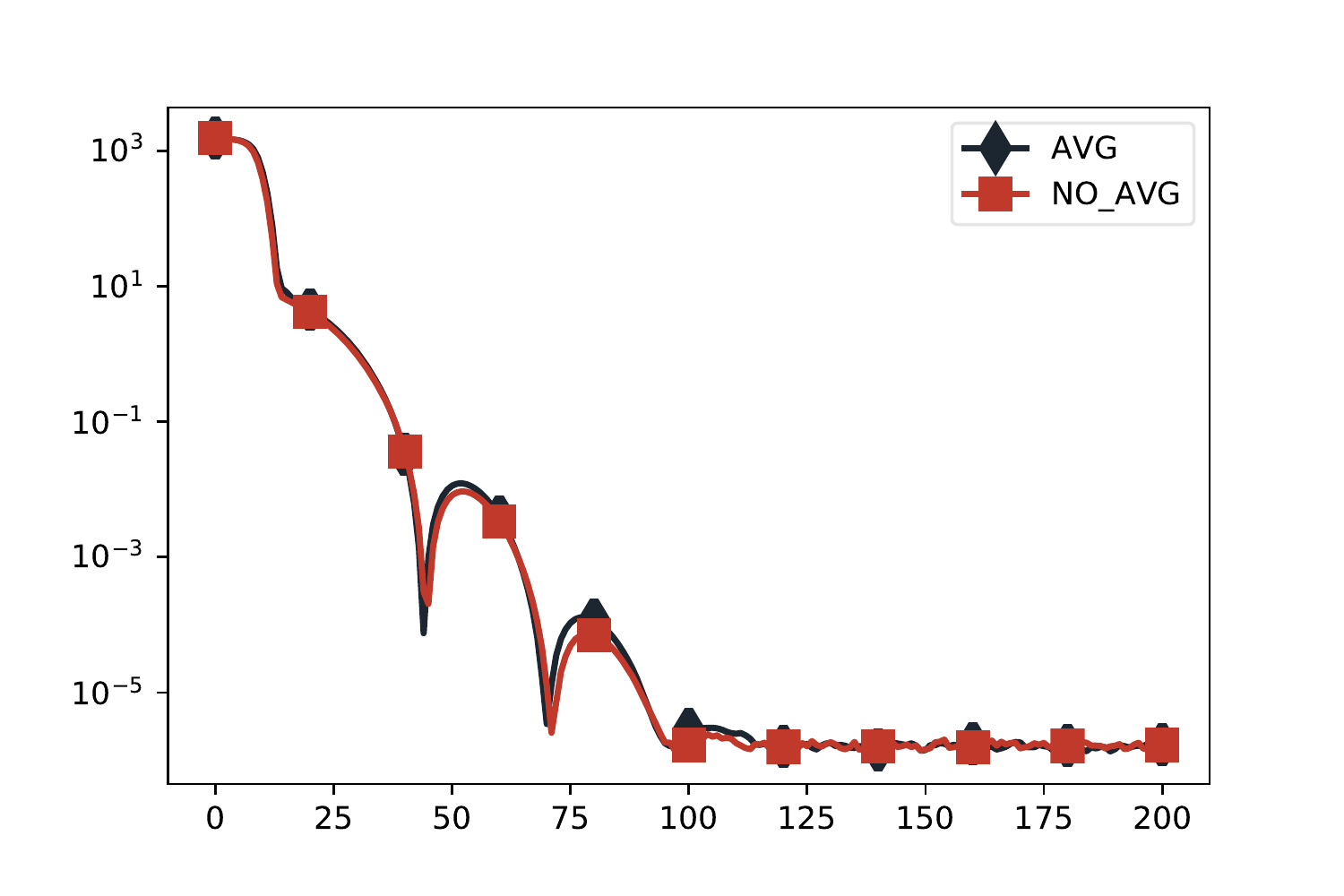}
            \\
        &   \scriptsize{epochs} & \scriptsize{epochs}
   \end{tabular}
          \caption{ \footnotesize{ZO-Varag, averaging~(Option I) vs. no-averaging~(Option II).}}
          \label{fig:pivot_selection_diabetes}
\end{figure}

\vspace{-2mm}
\paragraph{Effect of the momentum $\boldsymbol{p_s}$.}

%Since the accelerated method doesn't have the step as defined for non-accelerated method, we set $\alpha \gamma$ as the "so-called" step for the ZO-Varag and the simplified ZO-Katyusha.

The effect of $p_s$~(a.k.a Katyusha momentum) varies depending on the data set. From Fig.~\ref{fig:diabetes_momentum_test} and Fig.~\ref{fig:ijcnn1_momentum_test}, we find that increasing values of $p$ can either accelerate or slow down the convergence of the algorithm. Moreover, the algorithm may not converge when $p < 0.5$, since the constraint from Eq.~\eqref{VARAG-lemma-3-assump-2} is not guaranteed anymore (recall the proof we provide is based on $p=0.5$). 

\begin{figure}[htbp]
 \centering          \begin{tabular}{c@{}c@{}c@{}c@{}}%c@{}}
        &   \scriptsize{regularizer $\lambda = 0$} & \scriptsize{regularizer $\lambda = 1e^{-5}$} \\%& \tiny{autoencoder} \\
        \rotatebox[origin=c]{90}{\scriptsize{Logistic}} &    
             \includegraphics[width=0.35\linewidth,valign=c]{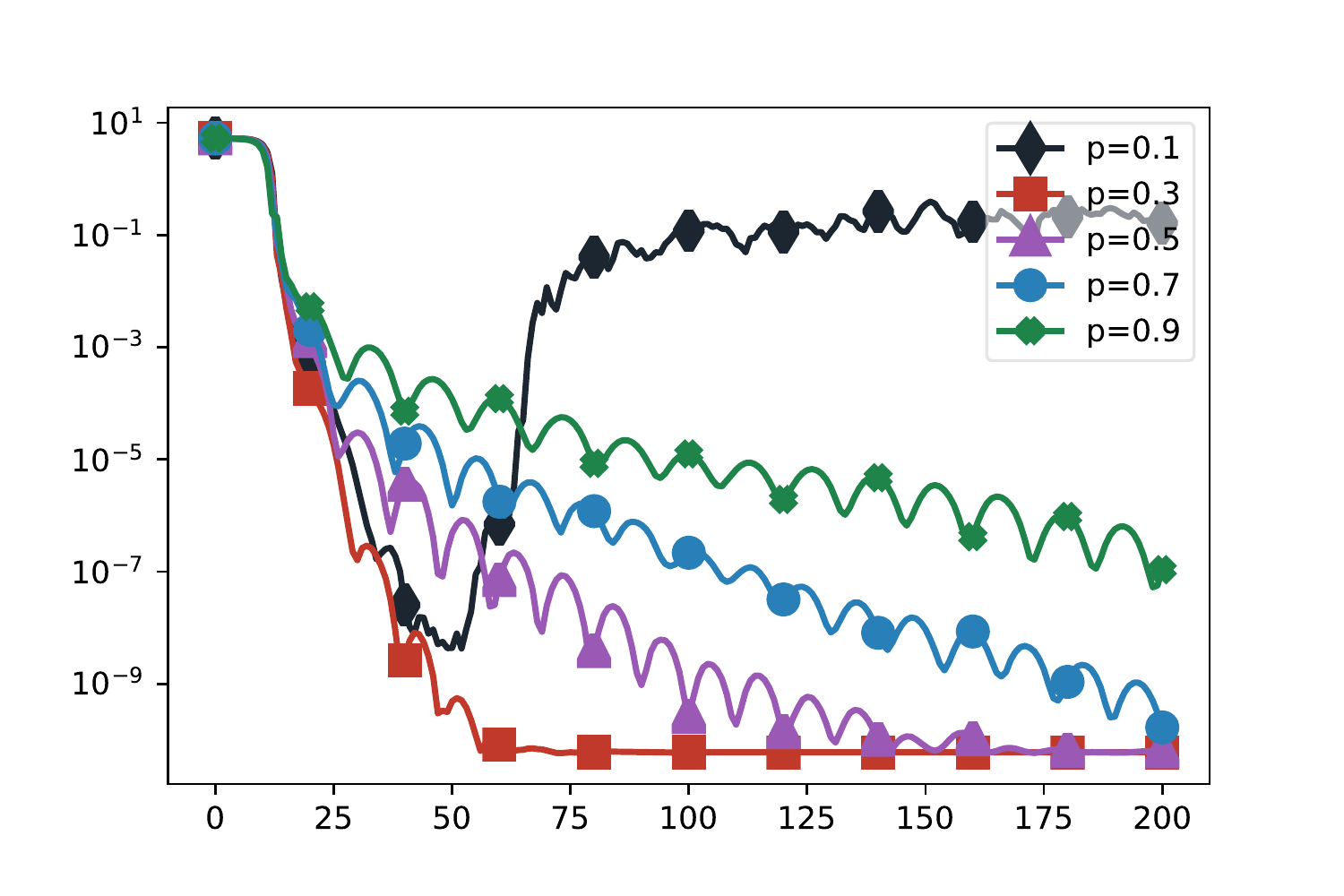}  &  \includegraphics[width=0.35\linewidth,valign=c]{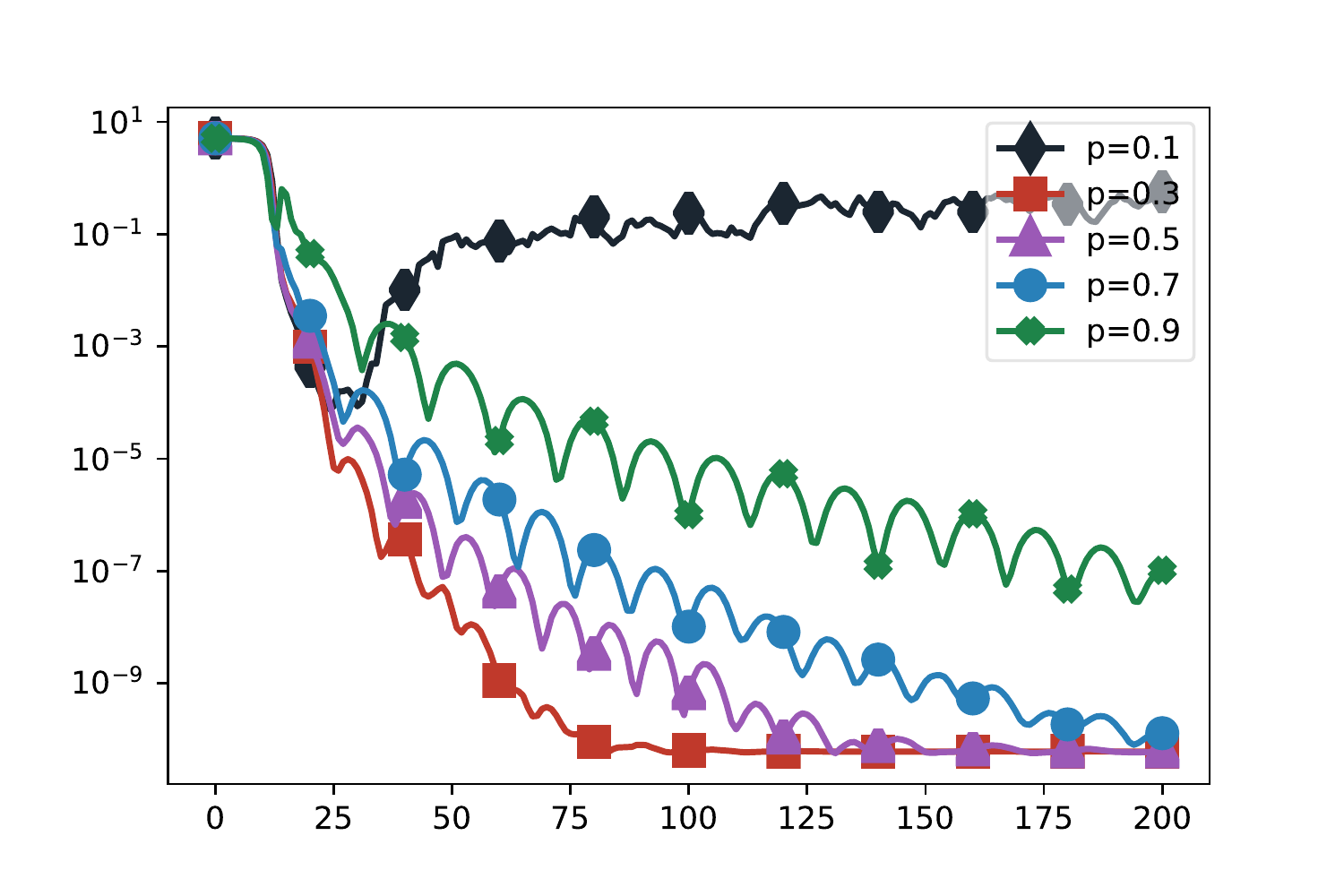}
             \\
                               \vspace{-3mm}
        \rotatebox[origin=c]{90}{\scriptsize{Ridge}}&
             
              \includegraphics[width=0.35\linewidth,valign=c]{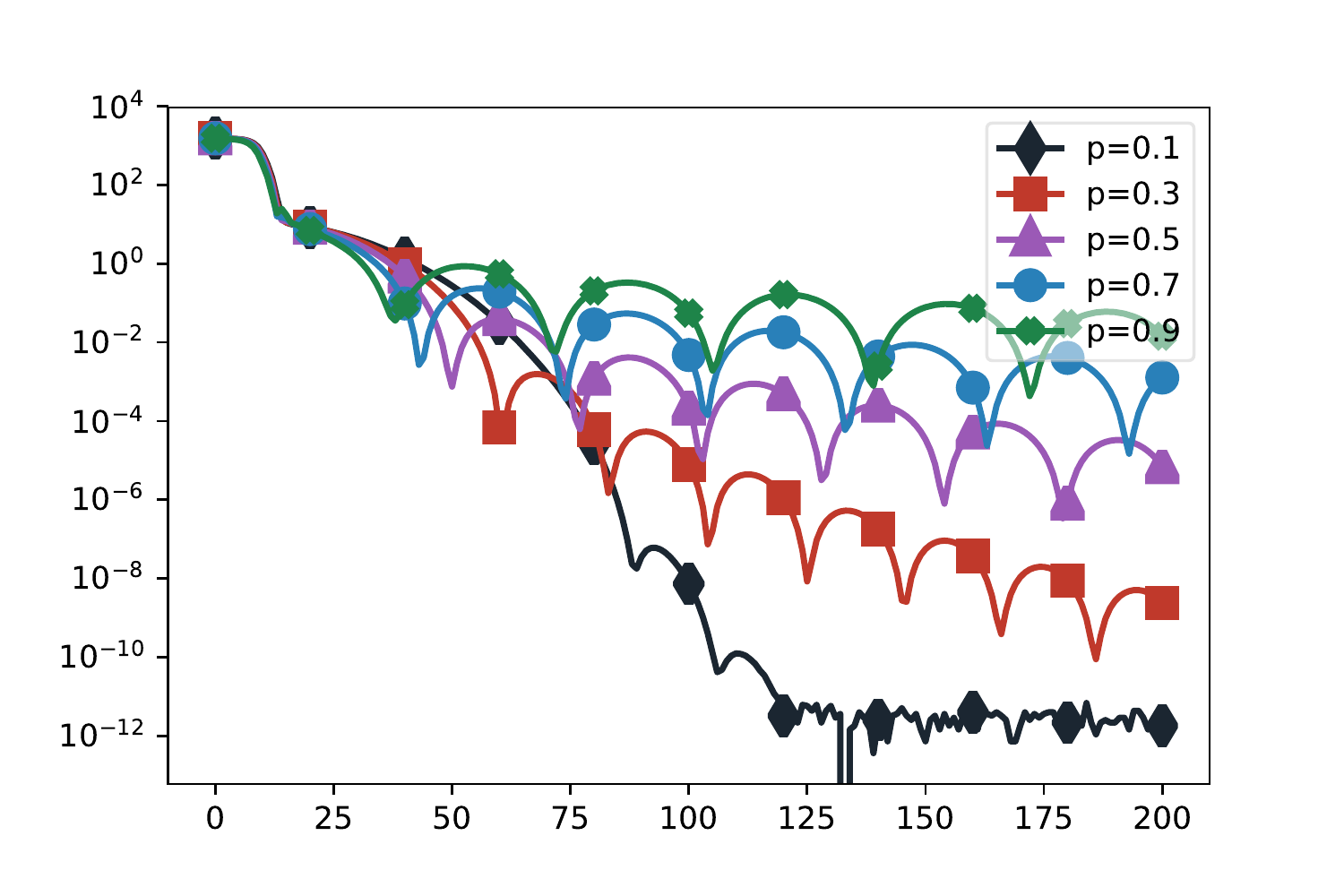}& \includegraphics[width=0.35\linewidth,valign=c]{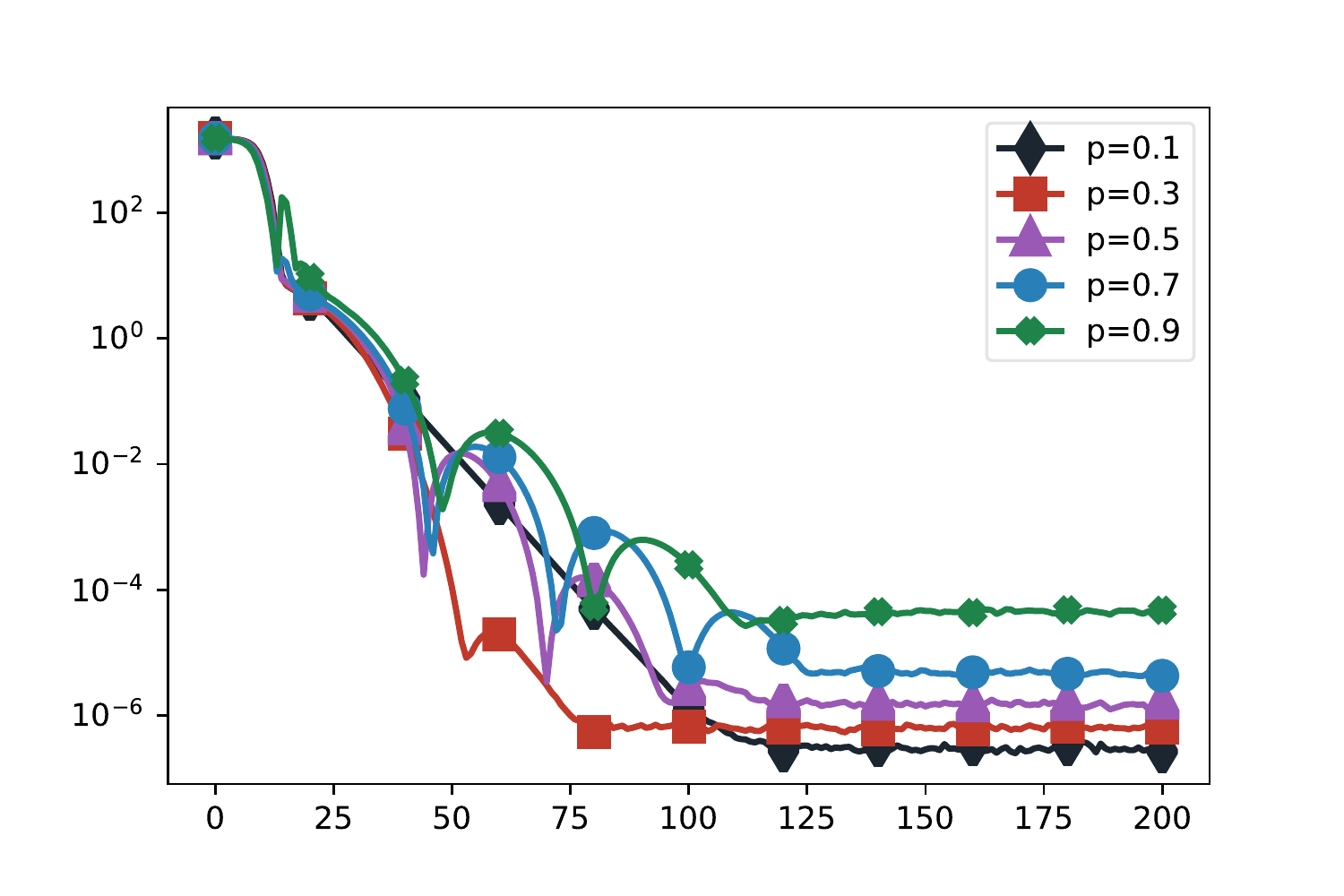}
            \\
         & \scriptsize{epochs} & \scriptsize{epochs}
   \end{tabular}
	      \caption{Effect of $p$ on the diabetes dataset. Recall that our theoritical guarantees hold for $p=0.5$.}
          \label{fig:diabetes_momentum_test}
\end{figure}

\begin{figure}[htbp]
 \centering          \begin{tabular}{c@{}c@{}c@{}c@{}}%c@{}}
        &   \scriptsize{regularizer $\lambda = 0$} & \scriptsize{regularizer $\lambda = 1e^{-5}$} \\%& \tiny{autoencoder} \\
        \rotatebox[origin=c]{90}{\scriptsize{Logistic}} &    
             \includegraphics[width=0.35\linewidth,valign=c]{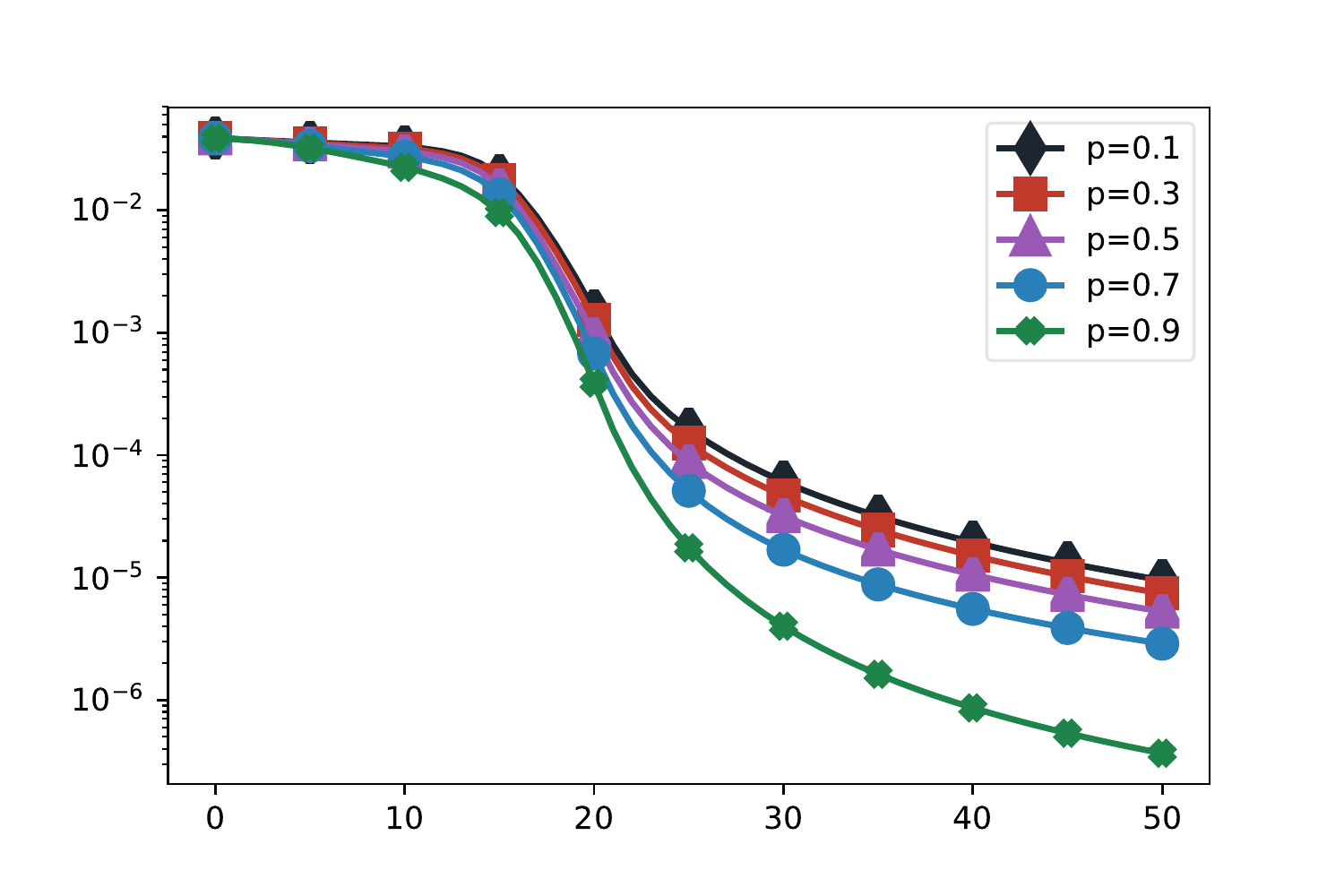}  &  \includegraphics[width=0.35\linewidth,valign=c]{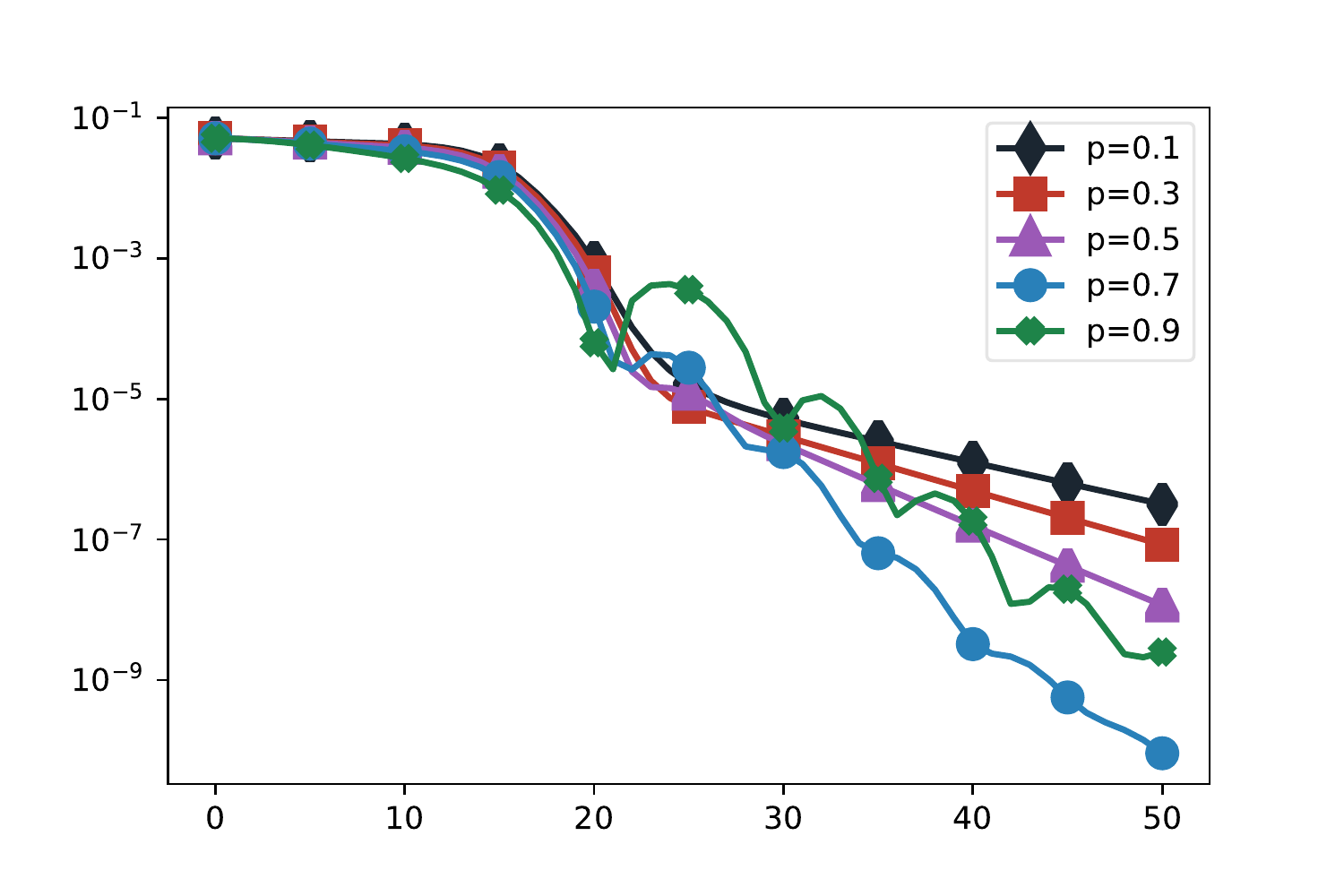}
             \\
                               \vspace{-3mm}
        \rotatebox[origin=c]{90}{\scriptsize{Ridge}}&
             
              \includegraphics[width=0.35\linewidth,valign=c]{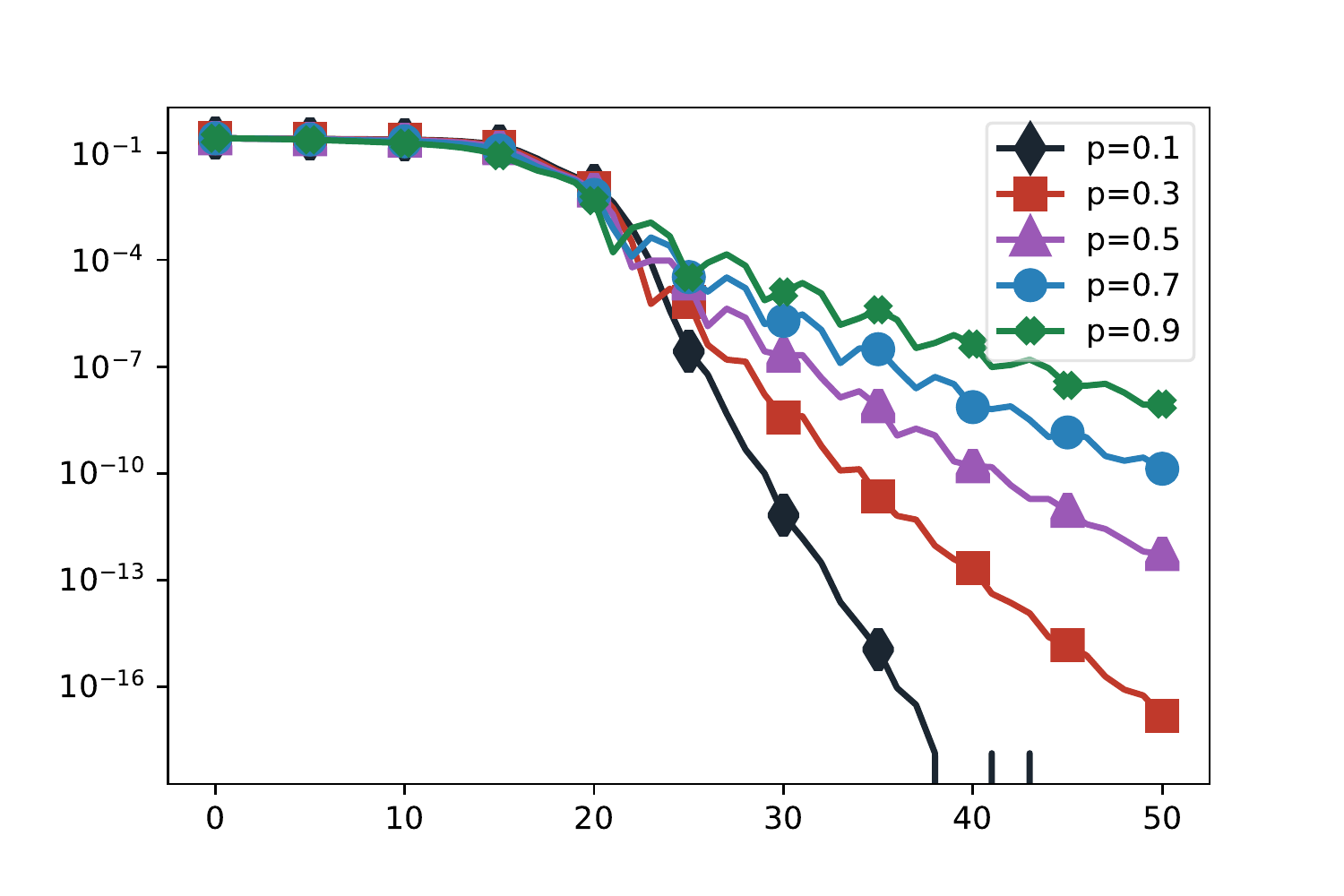}& \includegraphics[width=0.35\linewidth,valign=c]{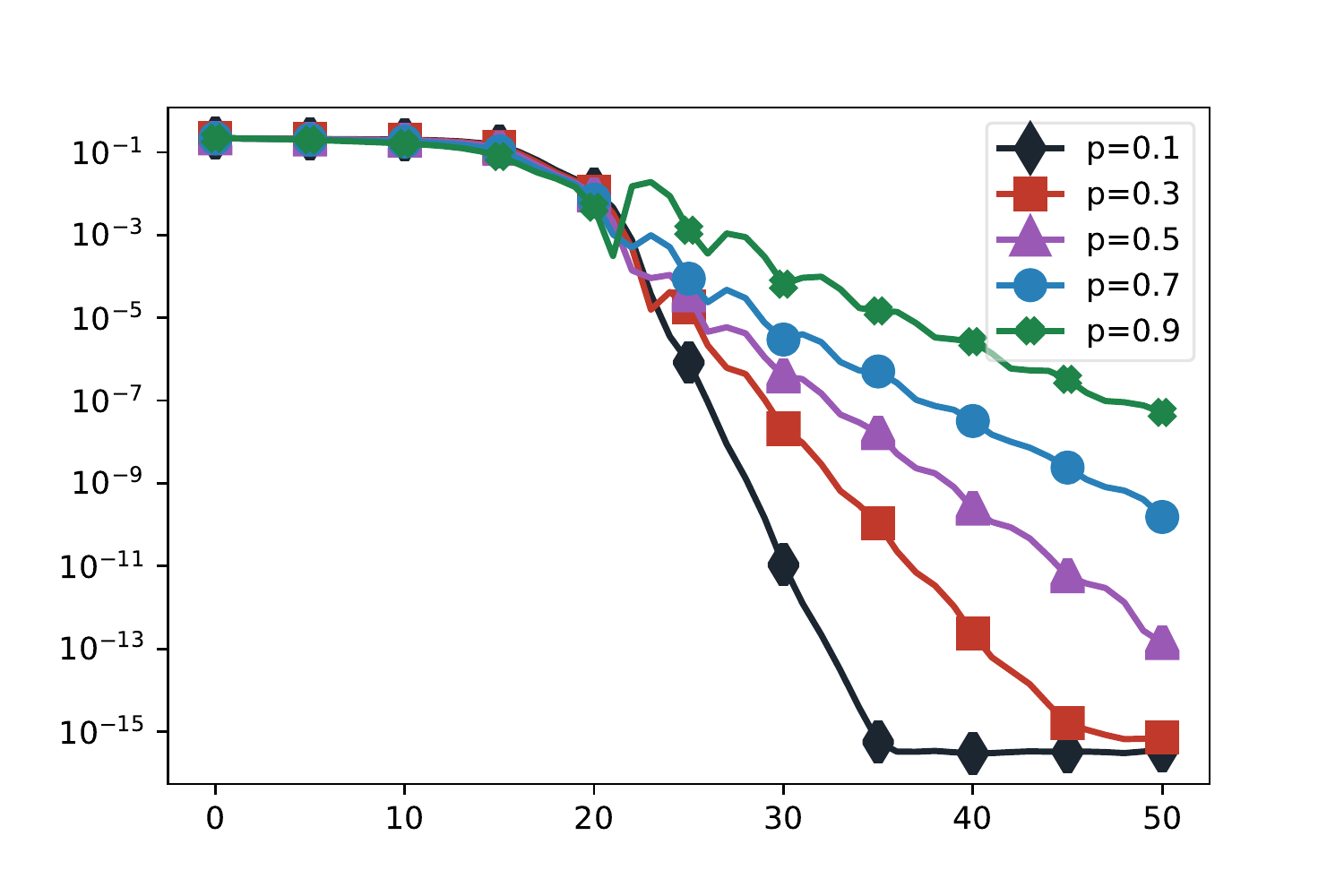}
            \\
        &   \scriptsize{epochs} & \scriptsize{epochs}
   \end{tabular}
	      \caption{Effect of $p$ on the ijcnn1 dataset. Recall that our theoritical guarantees hold for $p=0.5$.}
          \label{fig:ijcnn1_momentum_test}
\end{figure}

\vspace{-2mm}
\paragraph{Effect of the step-size $\boldsymbol{\alpha_s\gamma_s}$.} As discussed in Fig.~\ref{fig:vary_regularizer} from the main paper, we find that the suboptimality stalling effect is related to the magnitude of the regularizer $\lambda$, which influences  the strong-convexity constant of the objective function. Here, we show how such stalling effect of ZO-Varag can be controlled by tuning the step size $\alpha_s \gamma_s$. From Fig. \ref{fig:varying_step}, we see that the final suboptimality decreases if we decrease the magnitude of the step size $\alpha_s \gamma_s$, which however also affects speed of convergence.

\begin{figure}[htbp]
 \centering          \begin{tabular}{c@{}c@{}c@{}c@{}}%c@{}}
        &   \scriptsize{diabetes, S = 600, b = 5, $\lambda = 1e^{-5}$} & \scriptsize{ijcnn1, S = 200, b = 500, $\lambda = 1e^{-5}$} \\%& \tiny{autoencoder} \\
        \rotatebox[origin=c]{90}{\scriptsize{Logistic}} &    
             \includegraphics[width=0.35\linewidth,valign=c]{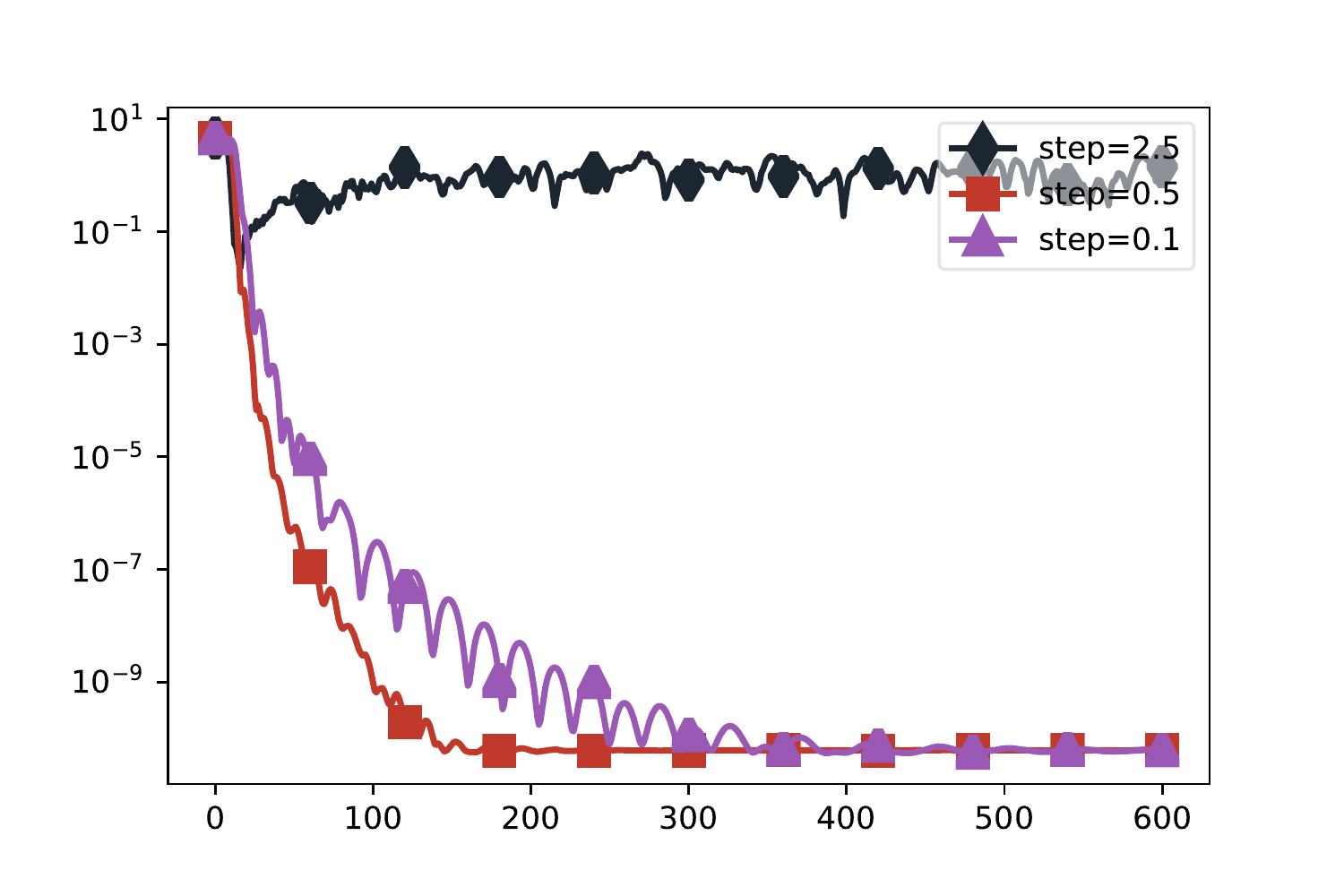}  &  \includegraphics[width=0.35\linewidth,valign=c]{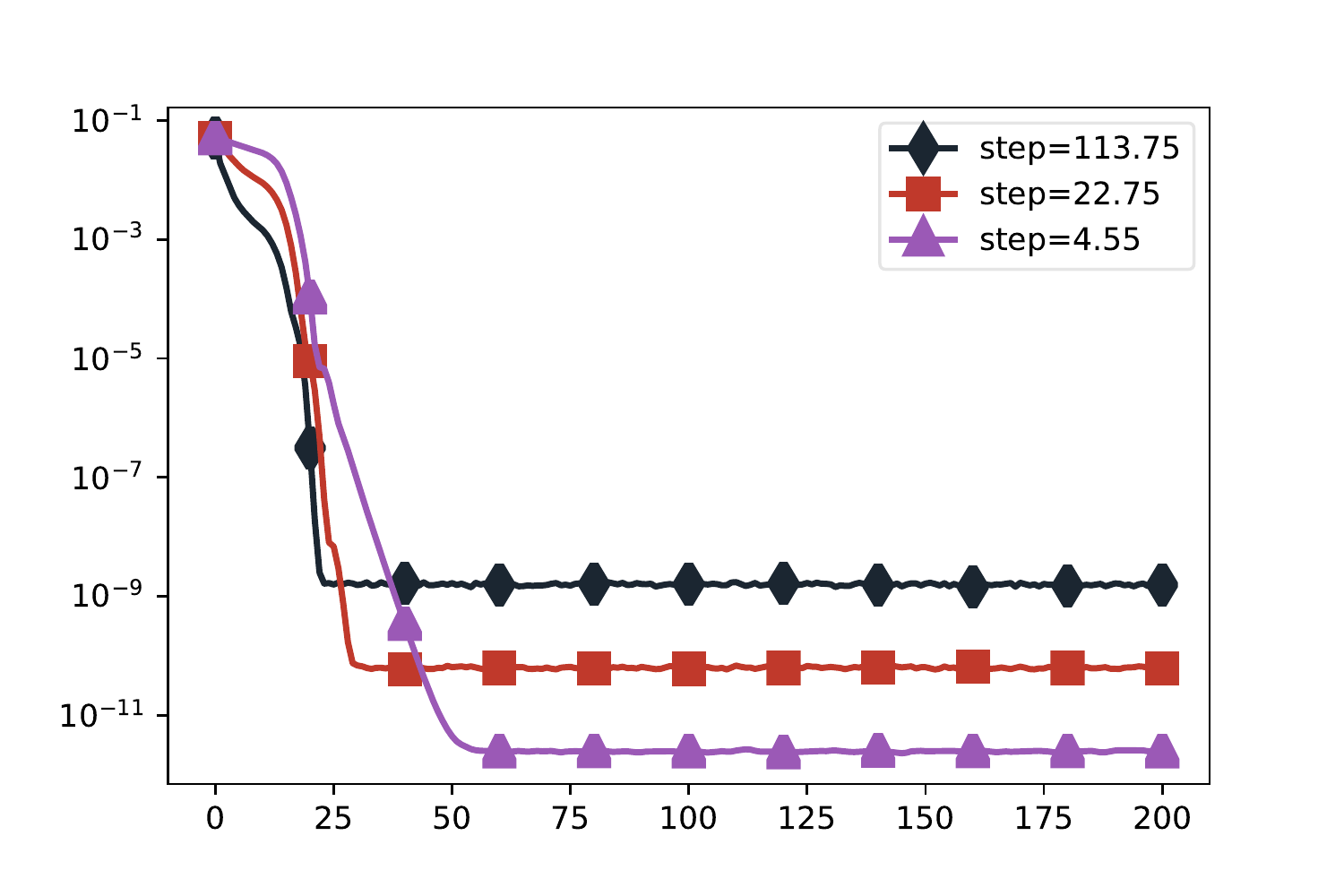}
             \\
        \vspace{-3mm}
        \rotatebox[origin=c]{90}{\scriptsize{Ridge}}&
             
              \includegraphics[width=0.35\linewidth,valign=c]{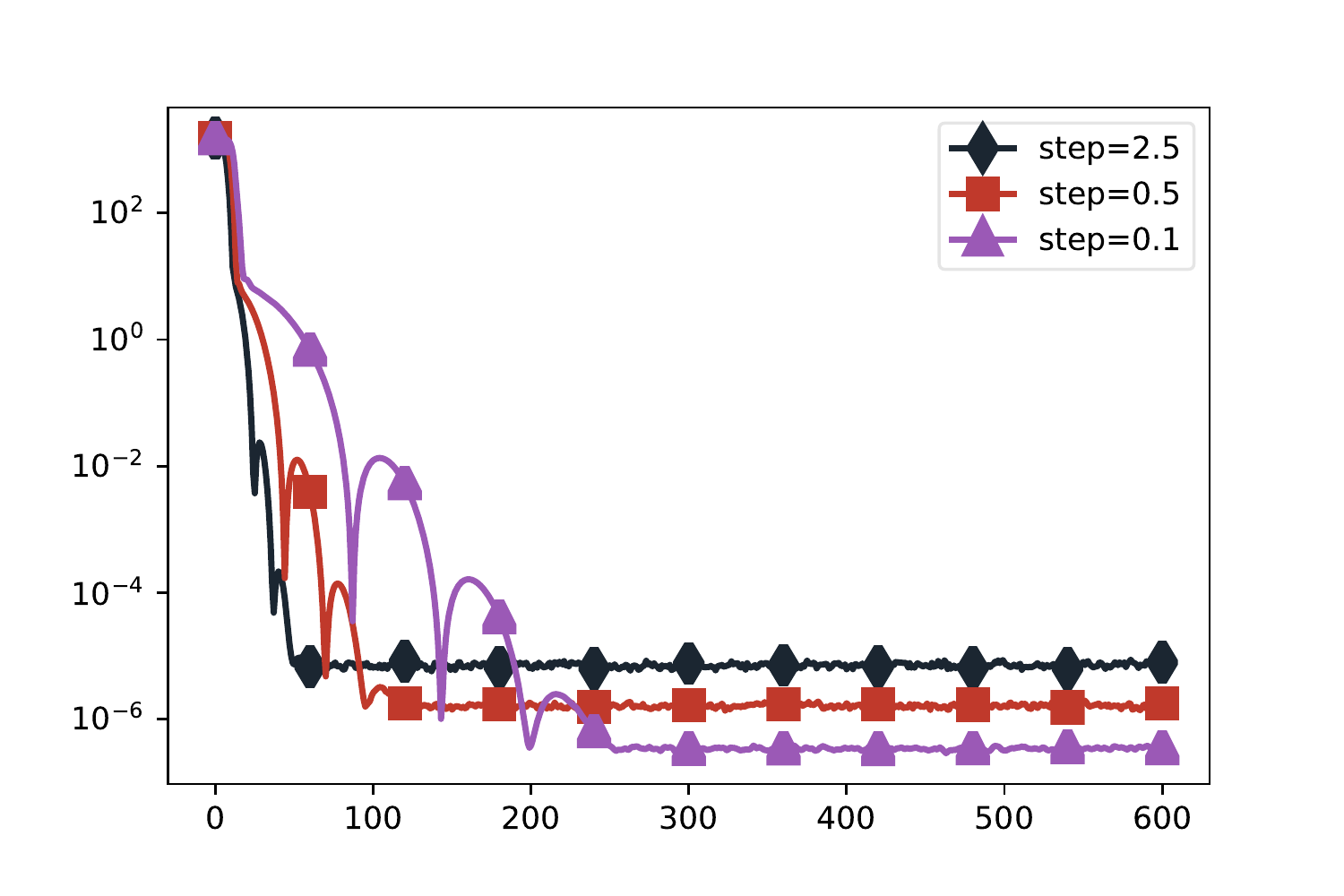}& \includegraphics[width=0.35\linewidth,valign=c]{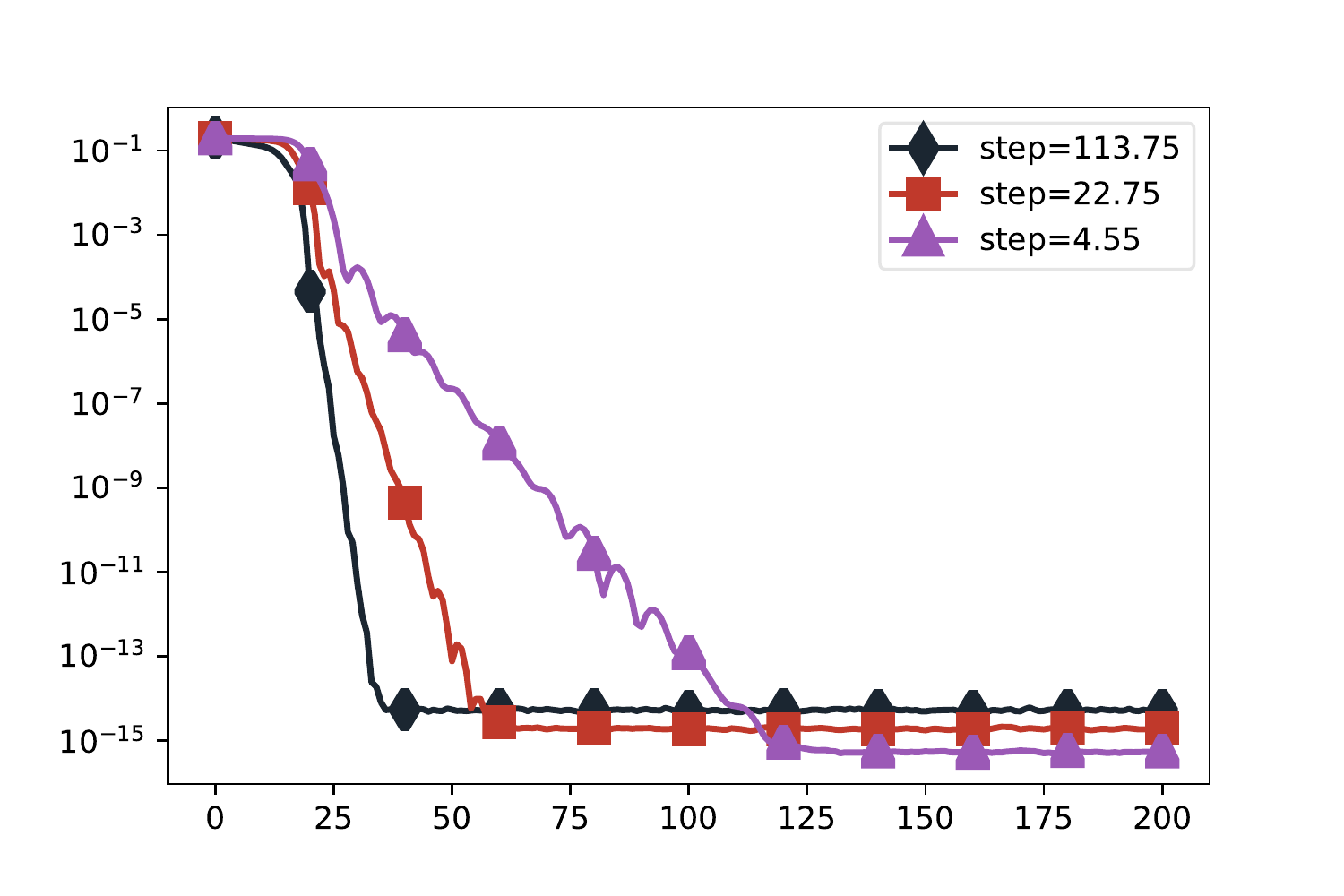}
            \\
        &   \scriptsize{epochs} & \scriptsize{epochs}
   \end{tabular}
   \vspace{-2mm}
          \caption{ \footnotesize{ZO-Varag with varying step-sizes$ =\alpha_s\gamma_s$.}}
          \label{fig:varying_step}
\end{figure}

\vspace{-2mm}

\paragraph{Effect of the smoothing parameter $\boldsymbol{\nu}$.}In this test, we set steps as the biggest step for each scenario in Fig. \ref{fig:varying_step} as we only care about the stalling effects. In Fig. \ref{fig:smoothing_parameter}, we verify that the final error our ZO algorithm is dependent on the smoothing parameter $\nu$ at the pivotal point, i.e. smaller $\nu$ yields smaller error deviating from the optimum. However, we also find that this effect varies depending on the datasets and models being used, and is sometimes negligible: the logistic regression is sensitive to the values of the smoothing parameters, while the ridge regression is not. Note that, as expected, $\mu$ does not influence the steady-state error.
\begin{figure}[htbp]
 \centering          \begin{tabular}{c@{}c@{}c@{}c@{}}%c@{}}
        &   \scriptsize{diabetes, S = 300, b = 5, $\lambda = 1e^{-5}$} & \scriptsize{ijcnn1, S = 100, b = 500, $\lambda = 1e^{-5}$} \\%& \tiny{autoencoder} \\
        \rotatebox[origin=c]{90}{\scriptsize{Logistic}} &    
             \includegraphics[width=0.35\linewidth,valign=c]{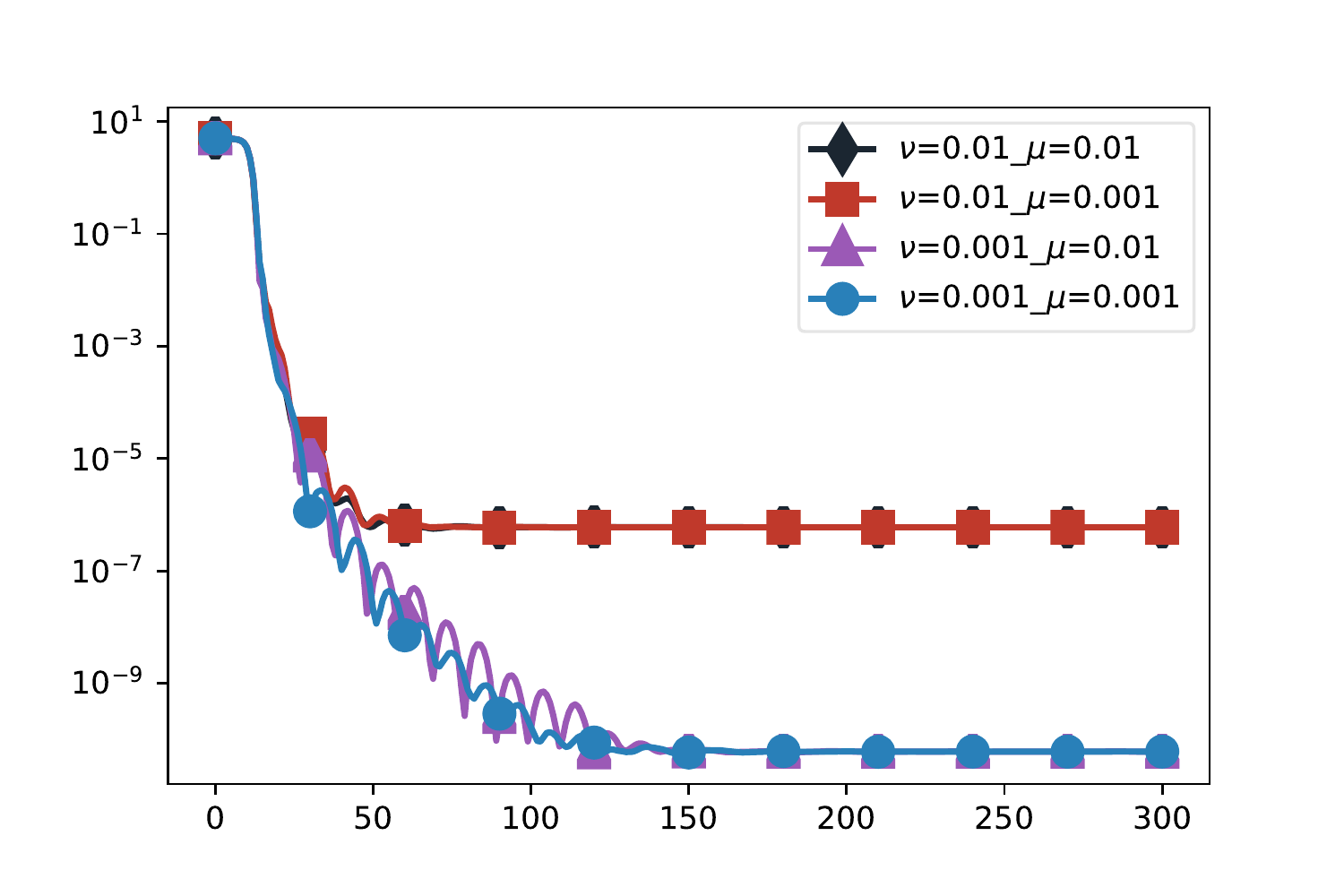}  &  \includegraphics[width=0.35\linewidth,valign=c]{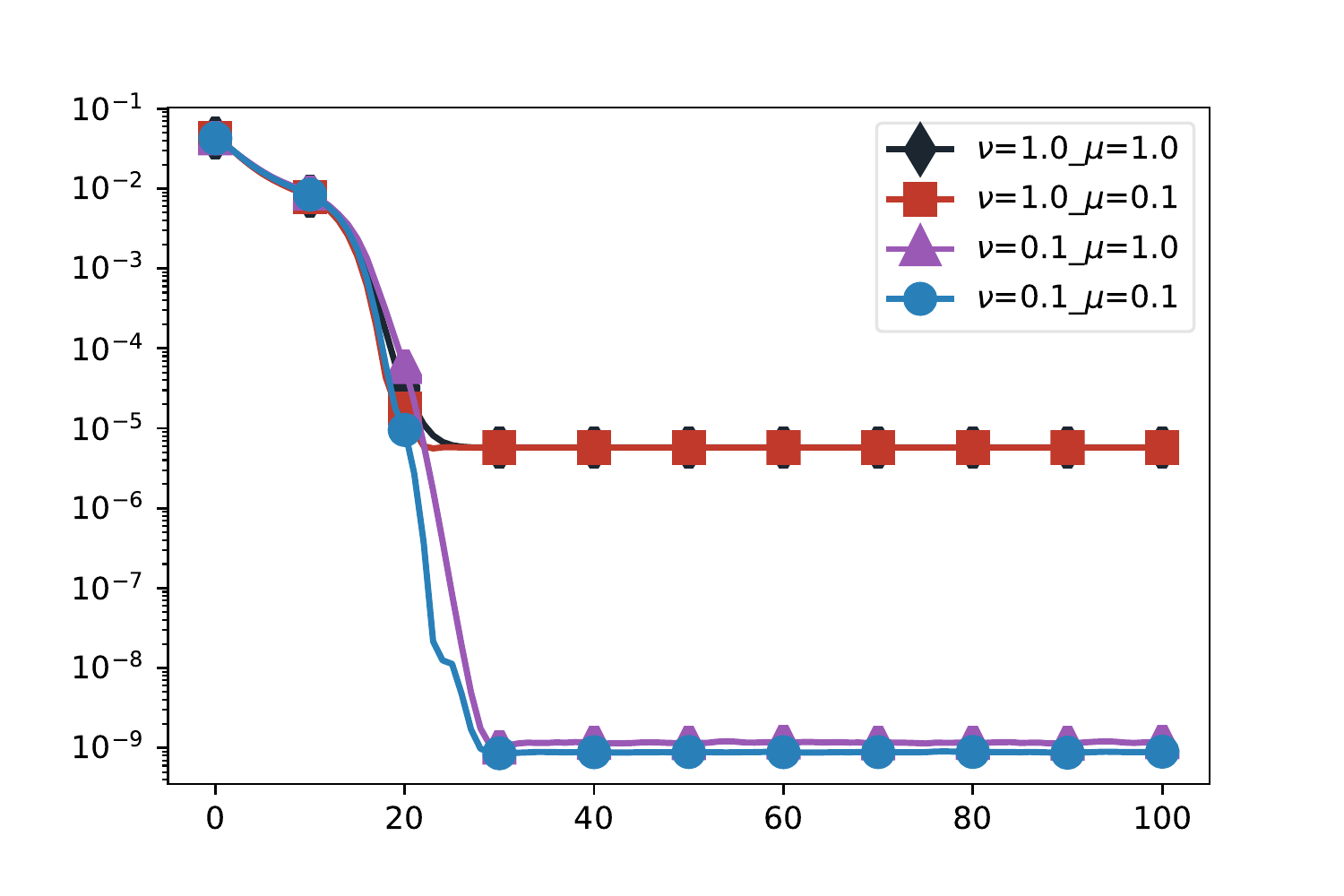}
             \\
              \vspace{-3mm}
        \rotatebox[origin=c]{90}{\scriptsize{Ridge}}&
             
              \includegraphics[width=0.35\linewidth,valign=c]{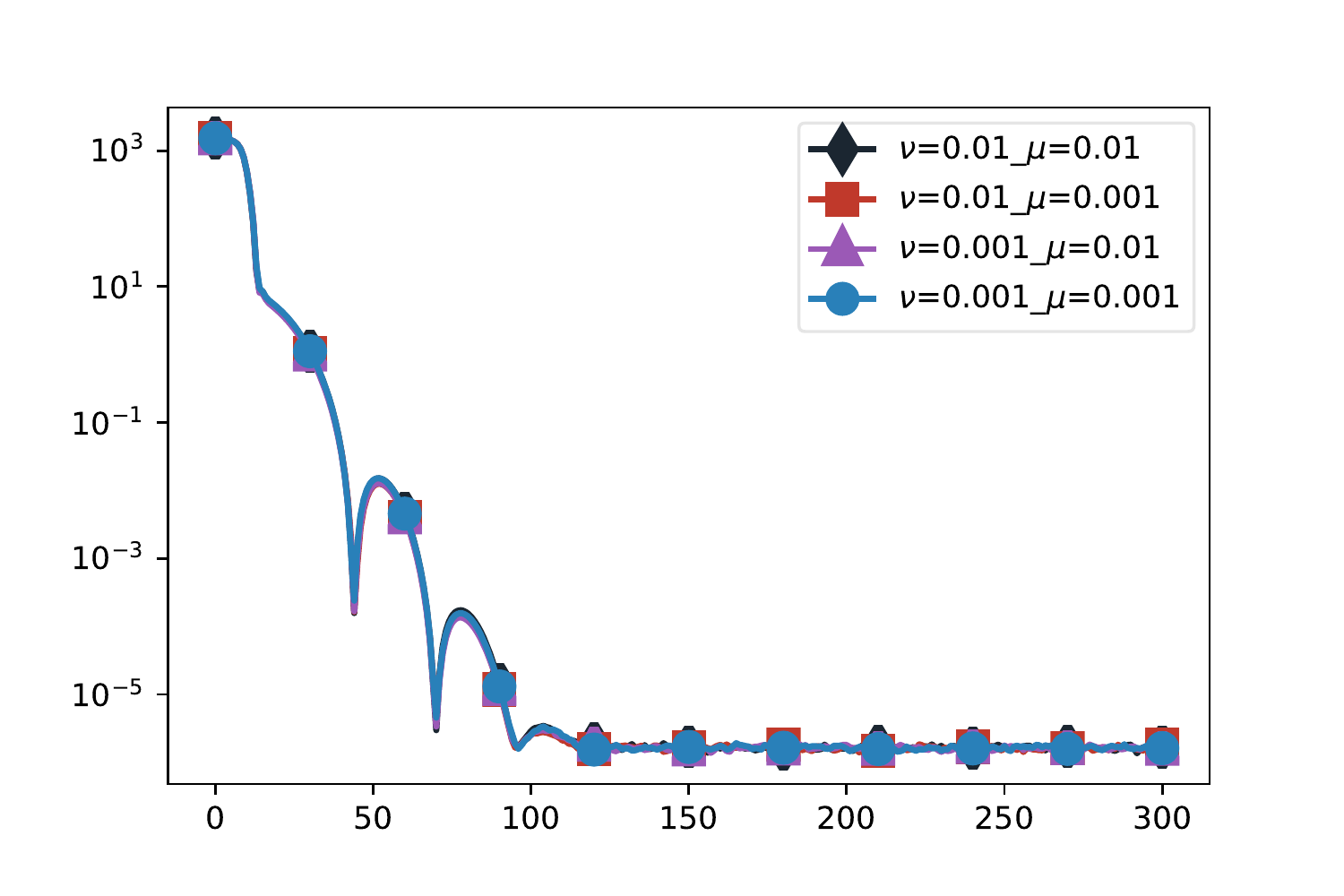}& \includegraphics[width=0.35\linewidth,valign=c]{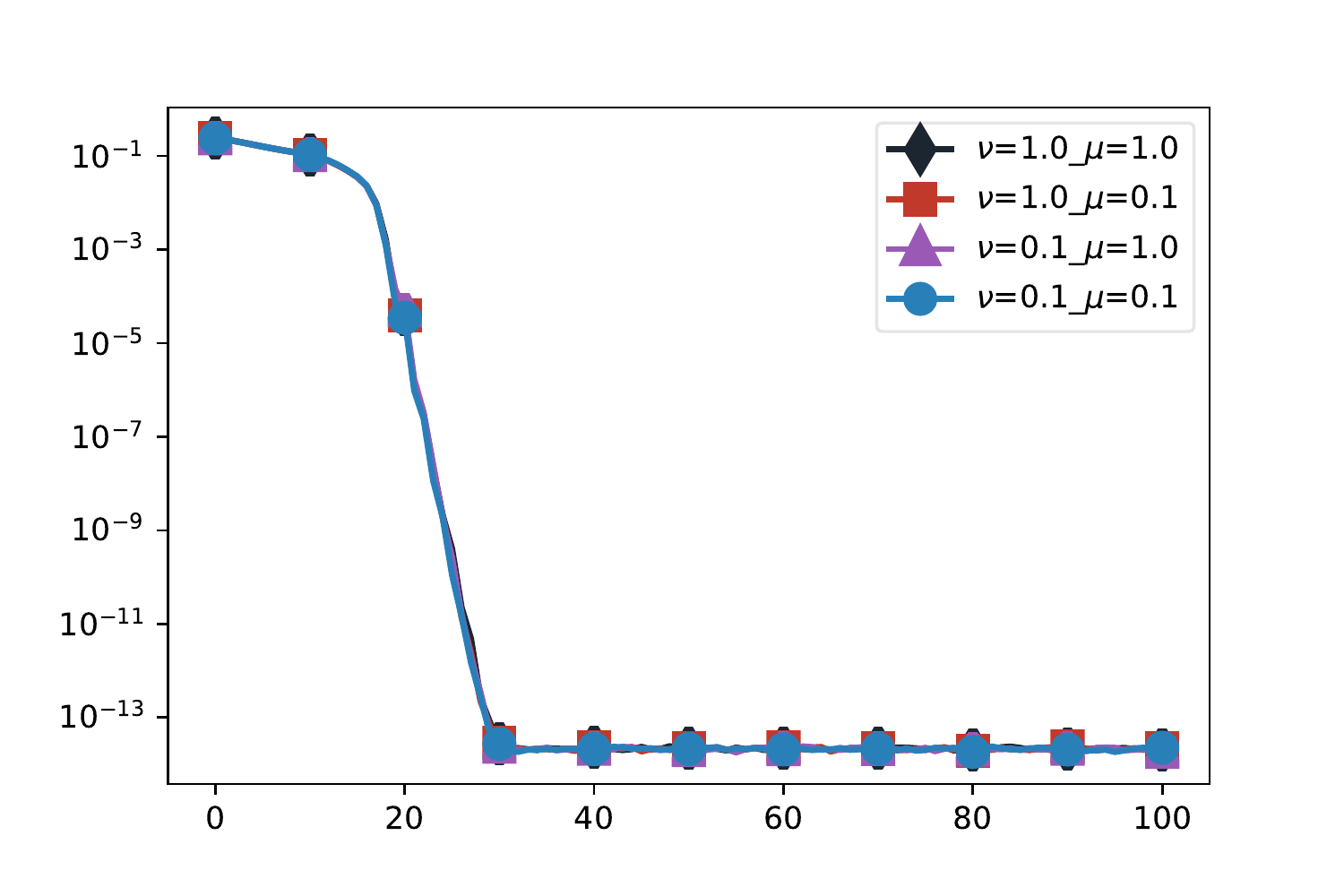}
            \\
        &   \scriptsize{epochs} & \scriptsize{epochs}
   \end{tabular}
          \caption{ \footnotesize{ZO-Varag, varying smoothing parameter $\mu$ and coordinate-wise paramater $\nu$.}}
          \label{fig:smoothing_parameter}
\end{figure}

%\vspace{-2mm}

\paragraph{Comparison with the Coordinate-wise Variant}
\begin{figure}[htbp]
 \centering          \begin{tabular}{c@{}c@{}c@{}c@{}}%c@{}}
        &   \scriptsize{regularizer $\lambda = 0$} & \scriptsize{regularizer $\lambda = 1e^{-5}$} \\
                \vspace{-3mm}
        \rotatebox[origin=c]{90}{\scriptsize{Logistic}} &    
             \includegraphics[width=0.35\linewidth,valign=c]{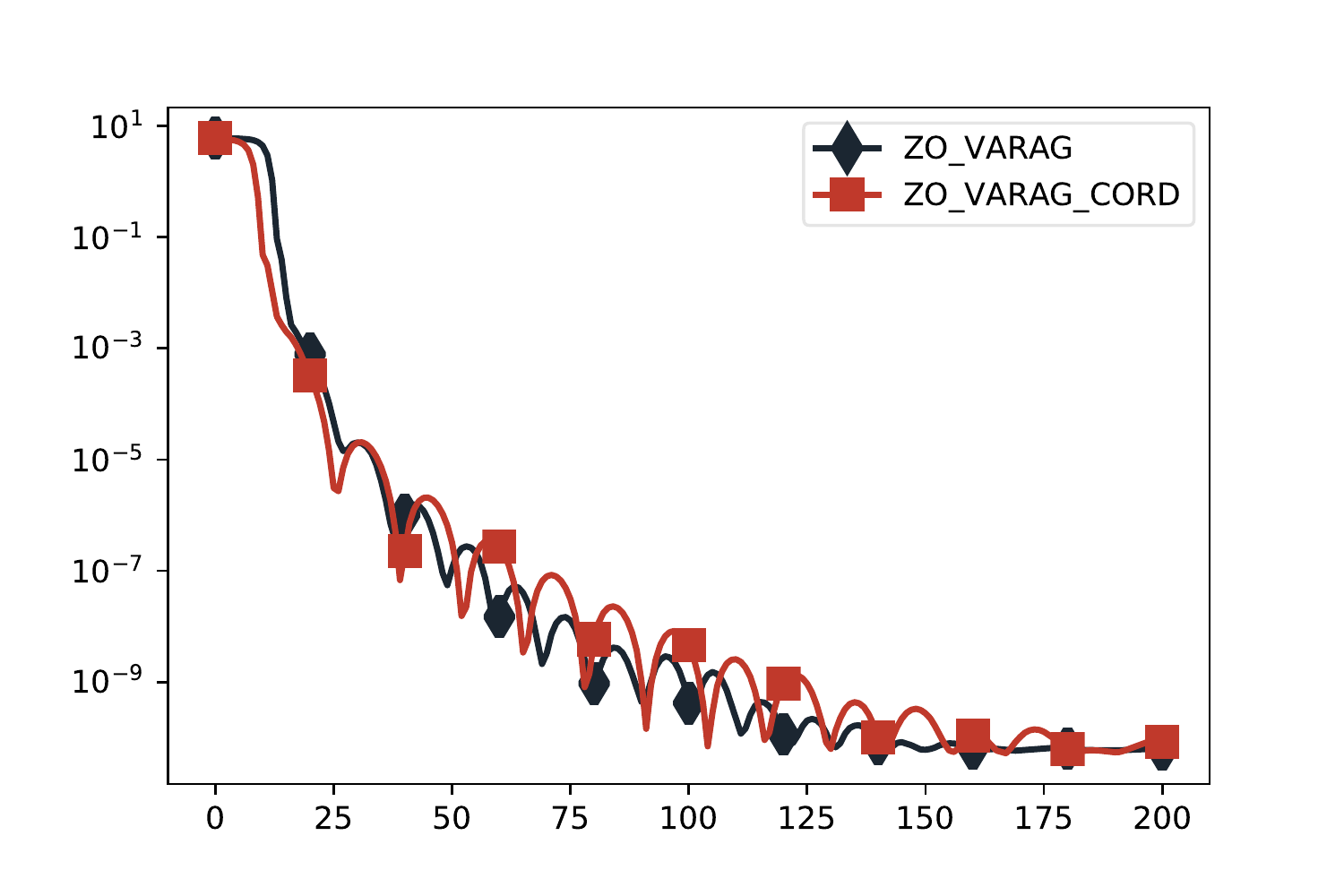}  &  \includegraphics[width=0.35\linewidth,valign=c]{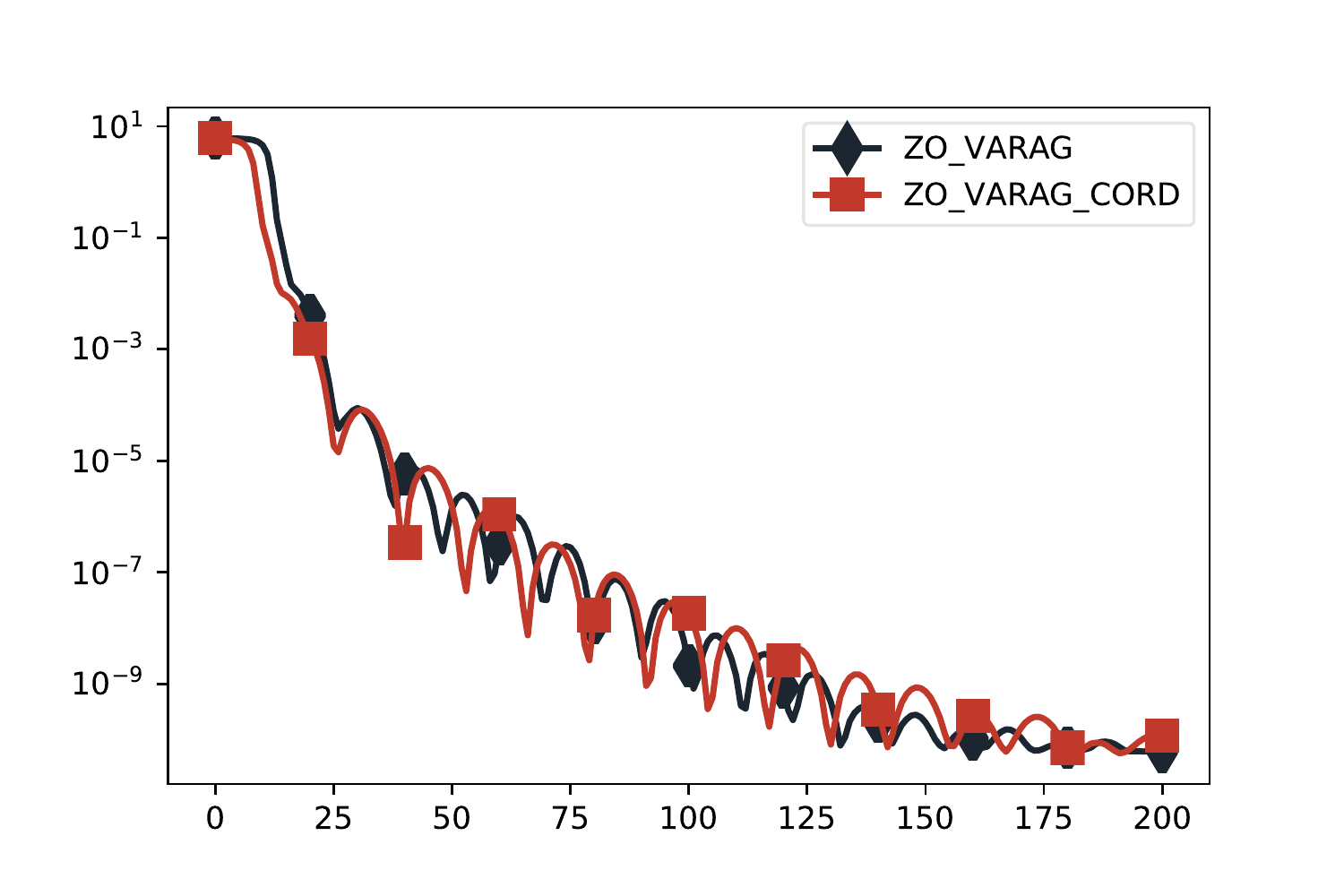}
             \\
                                \vspace{-3mm}
        \rotatebox[origin=c]{90}{\scriptsize{Ridge}}&
             
             \includegraphics[width=0.35\linewidth,valign=c]{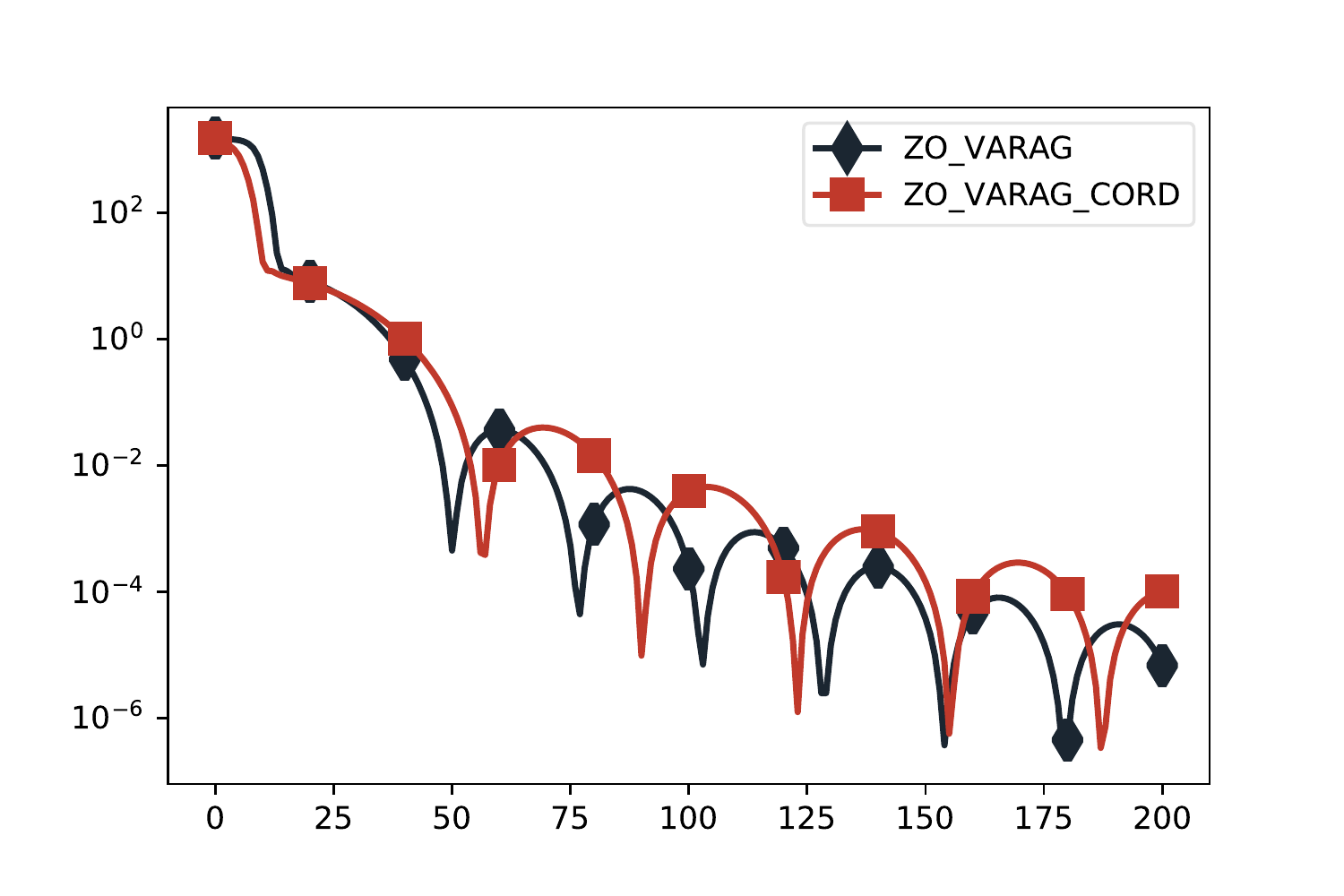}  &  \includegraphics[width=0.35\linewidth,valign=c]{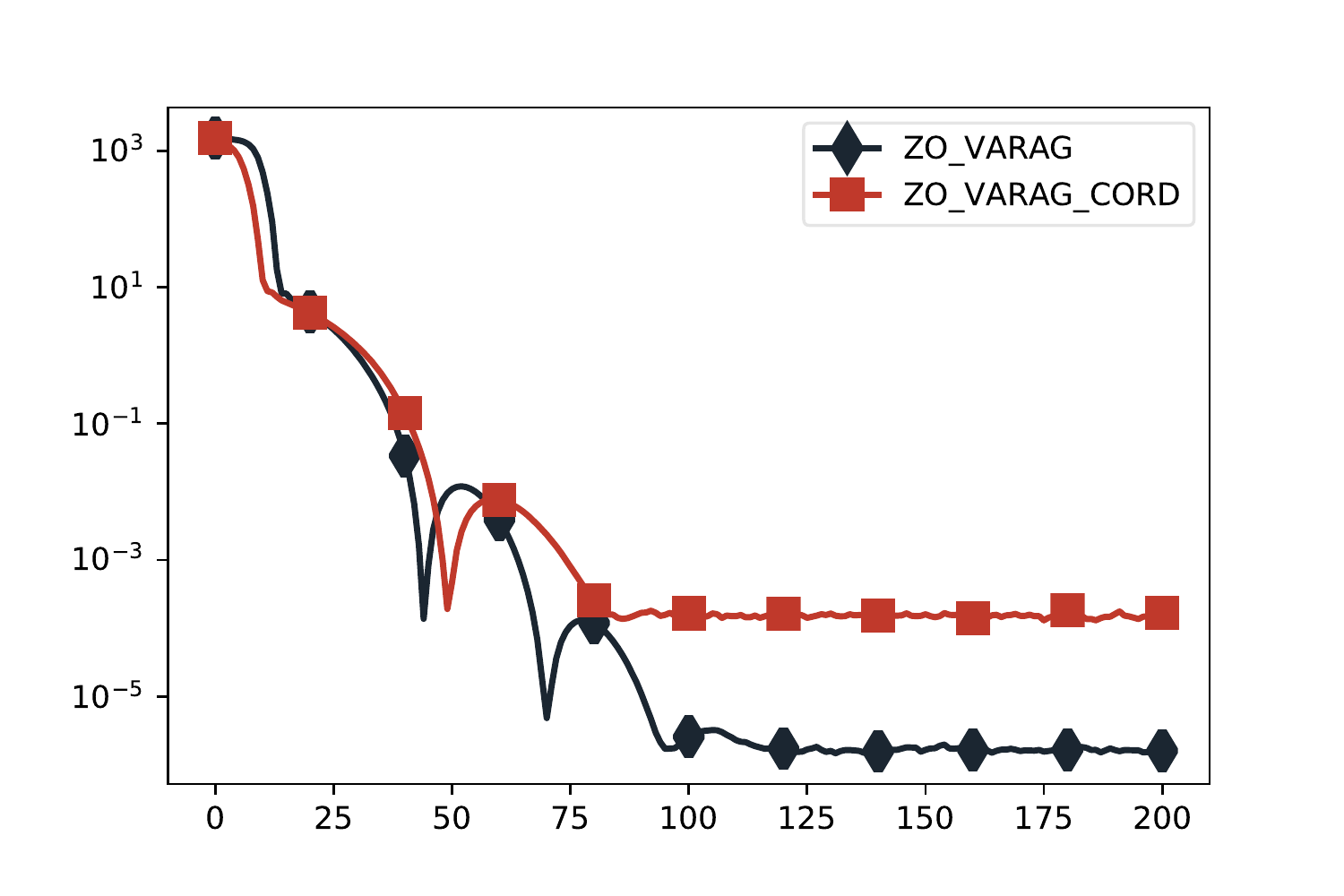}
            \\
        &   \scriptsize{epochs} & \scriptsize{epochs}
   \end{tabular}
          \caption{ \footnotesize{ZO-Varag vs. Coordinate-wise Variant of ZO-Varag (Diabetes)}}
          \label{fig:diabetes_coordinate_variant}
\end{figure}
\vspace{-2mm}
\begin{figure}[htbp]
 \centering          \begin{tabular}{c@{}c@{}c@{}c@{}}%c@{}}
        &   \scriptsize{regularizer $\lambda = 0$} & \scriptsize{regularizer $\lambda = 1e^{-5}$} \\
        \vspace{-3mm}
        \rotatebox[origin=c]{90}{\scriptsize{Logistic}} &    
             \includegraphics[width=0.35\linewidth,valign=c]{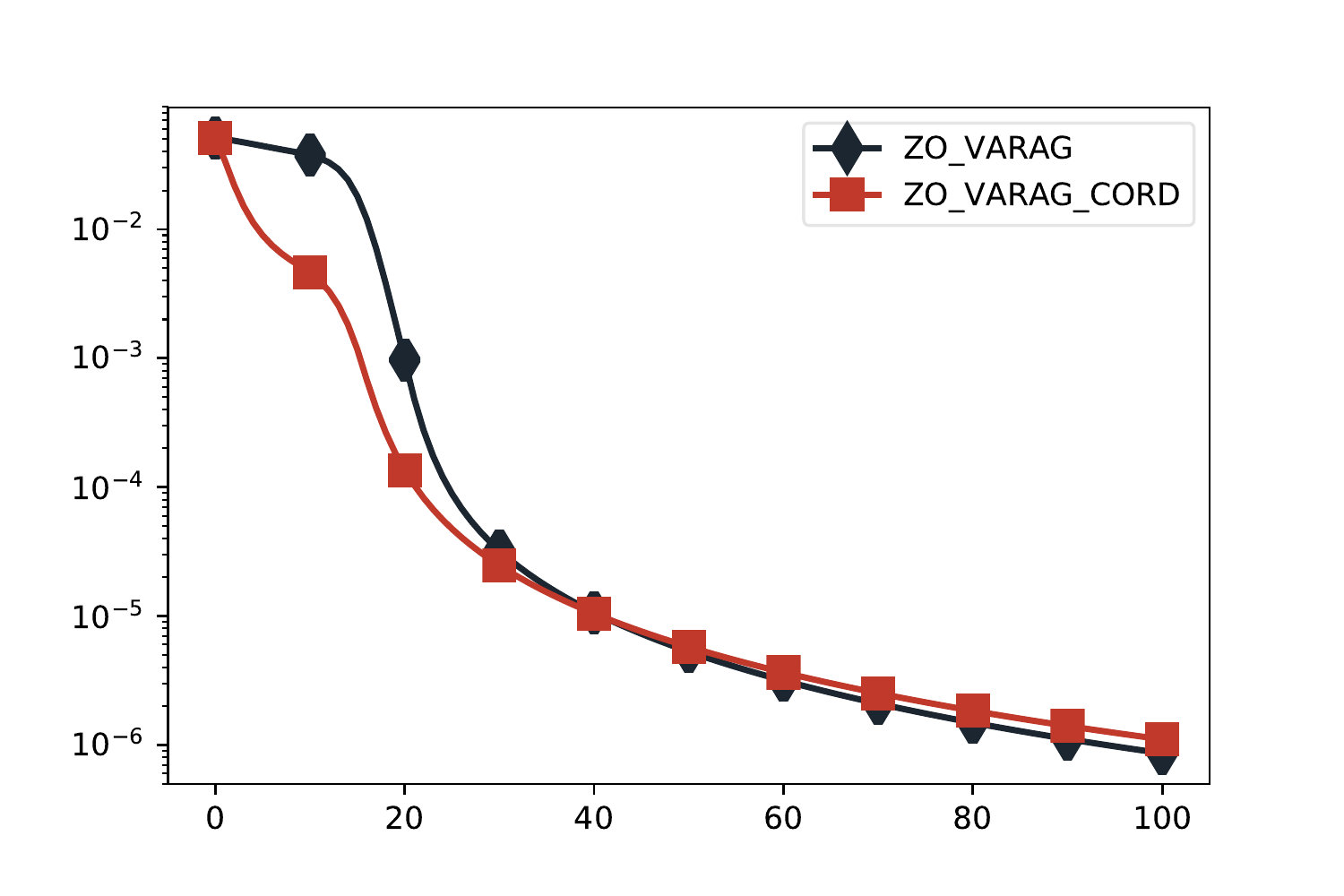}  &  \includegraphics[width=0.35\linewidth,valign=c]{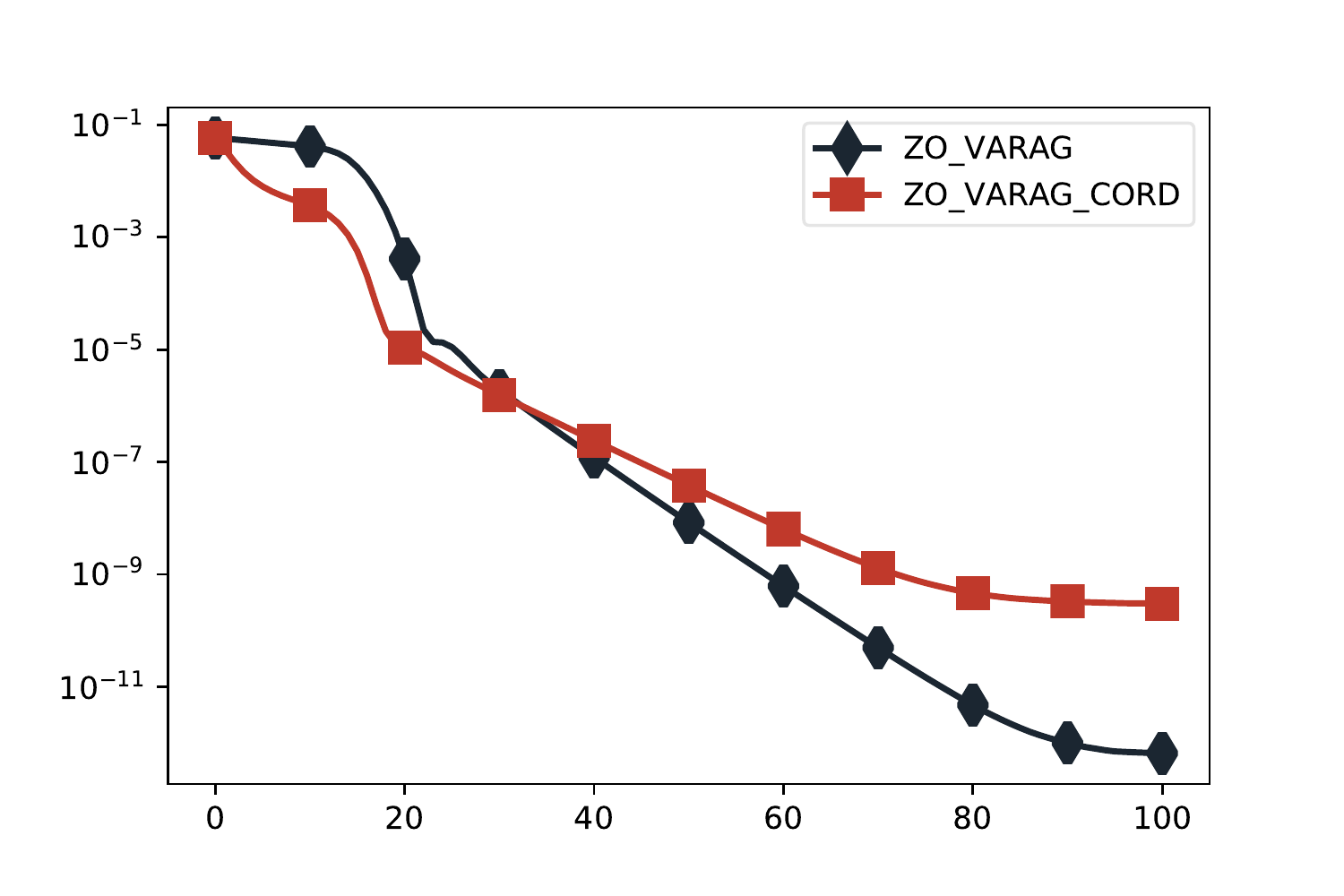}
             \\
                                \vspace{-3mm}
        \rotatebox[origin=c]{90}{\scriptsize{Ridge}}&
             
             \includegraphics[width=0.35\linewidth,valign=c]{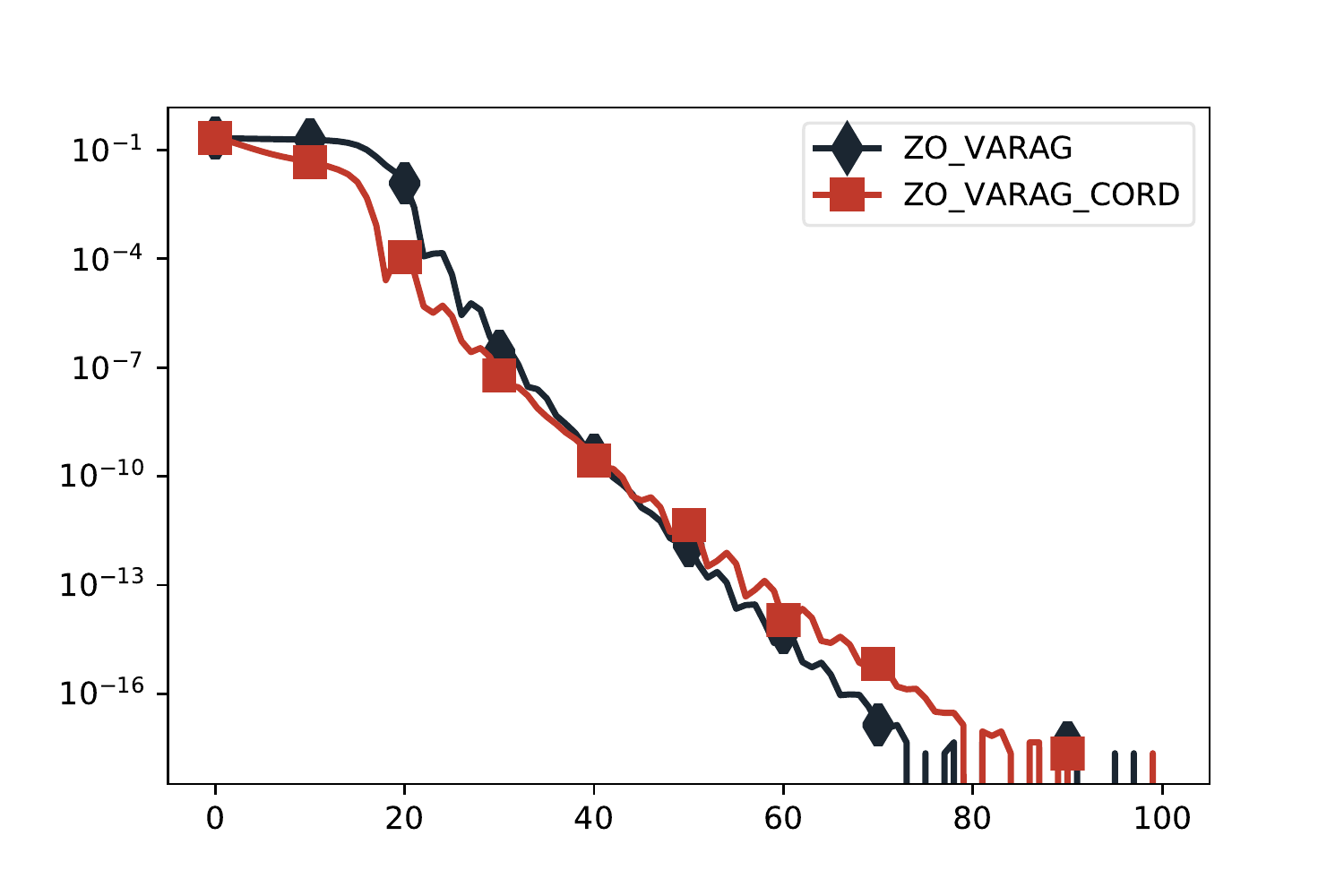}  &  \includegraphics[width=0.35\linewidth,valign=c]{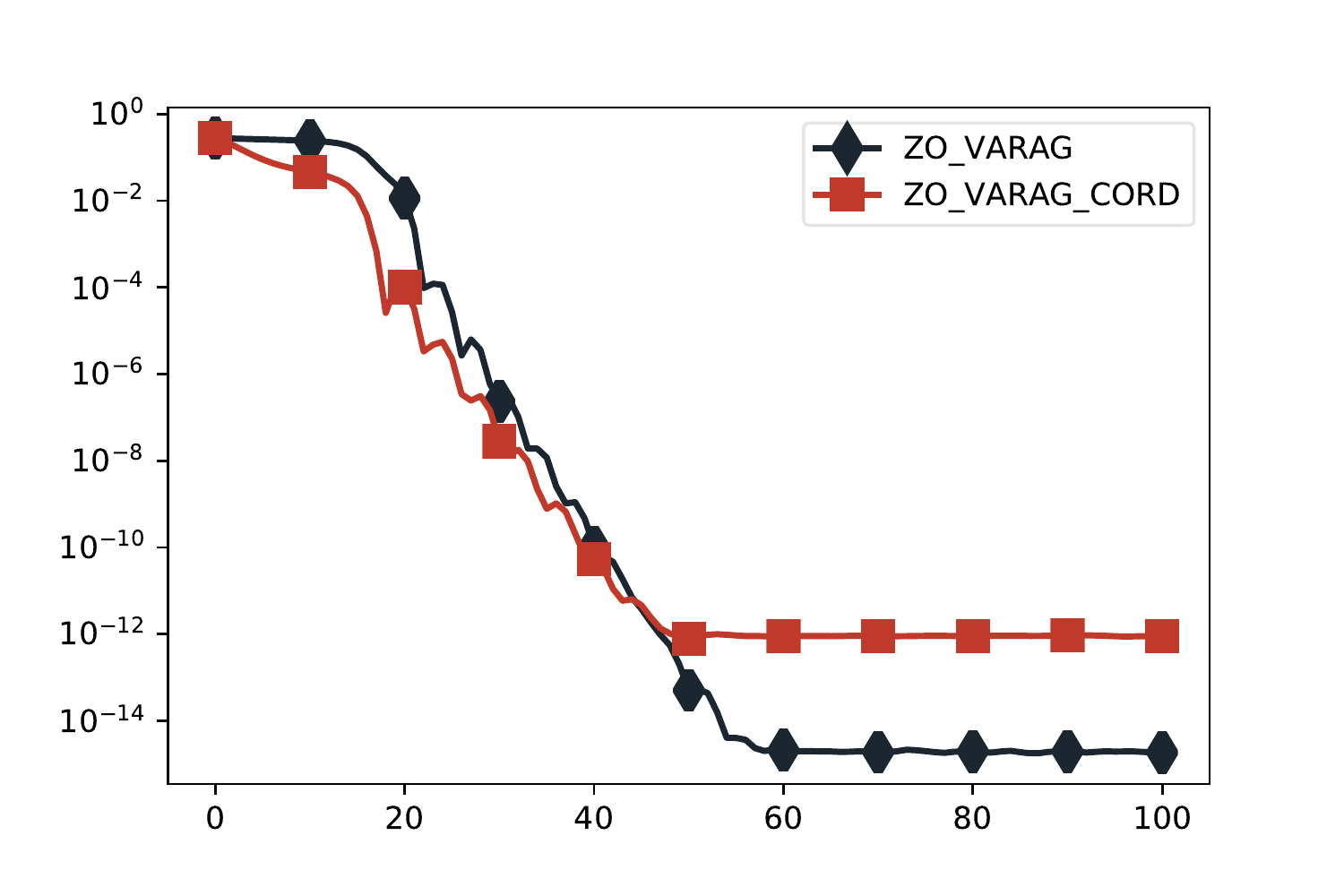}
            \\
        &   \scriptsize{epochs} & \scriptsize{epochs}
   \end{tabular}
          \caption{ \footnotesize{ZO-Varag vs. Coordinate-wise Variant of ZO-Varag (ijcnn1)}}
          \label{fig:ijcnn1_coordinate_variant}
\end{figure}

Finally, we also provide a preliminary test between the ZO-Varag algorithm and its coordinate-wise variant which is introduced in Section~\ref{sec:coordinate-wise}. Although the length of inner loops are not the same for these two algorithms, see different definitions of $s_0$ in Theorem~\ref{VARAG-theorem-1},~\ref{VARAG-theorem-2},~\ref{VARAG-cord-theorem-1},~\ref{VARAG-cord-theorem-2}, we only need to compare the function values at the pivotal points as the function queries are the same inside each inner loop after $s_0$ iterations (defined in Theorem~\ref{VARAG-theorem-1}). The experiments are carried out in Figure~\ref{fig:diabetes_coordinate_variant} and Figure~\ref{fig:ijcnn1_coordinate_variant}, and show that there is almost no difference between the performance of ZO-Varag and the performance of its coordinate-wise variant, except the magnitude of stalling errors. This comes from the fact that the step size for the coordinate-wise variant is $d$ times larger than that for ZO-Varag.

\end{document}